\documentclass[a4paper,reqno]{amsart}
\usepackage[paperheight=11in,paperwidth=8.5in]{geometry}
\usepackage{amsthm,amssymb,amsmath,hyperref}
\usepackage{color, mdframed}
\usepackage{tikz}
\usetikzlibrary{decorations.pathreplacing}
\usepackage[utf8]{inputenc}
\usepackage{graphicx,enumerate}
\usepackage{mathtools,bm}
\usepackage{enumitem}

\makeatletter
\def\l@section{\@tocline{1}{12pt plus2pt}{0pt}{}{\bfseries}}
\def\l@subsection{\@tocline{2}{0pt}{2pc}{2pc}{}}
\makeatother
\makeatletter
\def\subsection{\@startsection{subsection}{2}{\z@}%
	{-3.25ex\@plus -1ex \@minus -.2ex}%
	{1.5ex \@plus .2ex}%
	{\normalfont\bfseries\boldmath}}
\def\subsubsection{\@startsection{subsubsection}{3}%
	\z@{.5\linespacing\@plus.7\linespacing}{-.5em}%
	{\normalfont\bfseries\boldmath}}
\renewcommand\paragraph{\@startsection{paragraph}{4}{\z@}%
	{3.25ex \@plus1ex \@minus.2ex}%
	{-1em}%
	{\normalfont\normalsize\bfseries}}
\makeatother

\theoremstyle{plain}
\newtheorem{thm}{Theorem}[section]

\newtheorem{cor}[thm]{Corollary}
\newtheorem{lem}[thm]{Lemma}
\newtheorem{prop}[thm]{Proposition}

\theoremstyle{definition}

\theoremstyle{remark}
\newtheorem{rem}[thm]{Remark}
\newtheorem{obs}[thm]{Observation}

\theoremstyle{plain}
\newtheorem{conj}[thm]{Conjecture}
\numberwithin{equation}{section}

\theoremstyle{plain} % just in case the style had changed
\newcommand{\thistheoremname}{}
\newtheorem{genericthm}[thm]{\thistheoremname}

  \newtheorem*{genericthm*}{\thistheoremname}
\newenvironment{namedthm*}[1]
  {\renewcommand{\thistheoremname}{#1}%
   \begin{genericthm*}}
  {\end{genericthm*}}

\newcommand{\ep}{\epsilon}
\newcommand{\al}{\alpha}

\newcommand{\del}{\delta}

\newcommand{\R}{{\mathbb R}}

\newcommand{\Q}{{\mathbb Q}}
\newcommand{\N}{{\mathbb N}}

\newcommand{\Z}{{\mathbb Z}}

\newcommand{\TT}{{\mathbb T}}

\newcommand{\calA}{{\mathcal A}}
\newcommand{\calB}{{\mathcal B}}

\newcommand{\calF}{{\mathcal F}}

\newcommand{\calH}{{\mathcal H}}
\newcommand{\calI}{{\mathcal I}}
\newcommand{\calJ}{{\mathcal J}}

\newcommand{\calL}{{\mathcal L}}

\newcommand{\calN}{{\mathcal N}}

\newcommand{\calR}{{\mathcal R}}
\newcommand{\calS}{{\mathcal S}}

\newcommand{\calU}{{\mathcal U}}
\newcommand{\calV}{{\mathcal V}}

\newcommand{\supp}{{\textrm{supp}}}

\makeatletter
\newcommand{\vast}{\bBigg@{4}}
\newcommand{\Vast}{\bBigg@{5}}
\makeatother

\newcommand{\frakA}{{\mathfrak A}}
\newcommand{\frakB}{{\mathfrak B}}
\newcommand{\frakC}{{\mathfrak C}}
\newcommand{\frakD}{{\mathfrak D}}
\newcommand{\frakm}{{\mathfrak m}}
\newcommand{\frakM}{{\mathfrak M}}
\newcommand{\frakE}{{\mathfrak E}}
\newcommand{\frakF}{{\mathfrak F}}
\newcommand{\frakG}{{\mathfrak G}}
\newcommand{\frakH}{{\mathfrak H}}

\newcommand{\frakJ}{{\mathfrak J}}
\newcommand{\frakj}{{\mathfrak j}}

\newcommand{\sfp}{{\mathsf p}}
\newcommand{\bbS}{{\mathbb S}}
\newcommand{\bbA}{{\mathbb A}}
\newcommand{\bbB}{{\mathbb B}}

\newcommand{\bbU}{{\mathbb U}}

\newcommand{\frakl}{{\mathfrak l}}

\newcommand{\TFC}{\textnormal{{\bf TFC}}}
\newcommand{\sTFC}{\textnormal{{$\mathsf{TFC}$}}}
\newcommand{\cTFC}{\textnormal{{$\mathcal{TFC}$}}}
\newcommand{\sfa}{\textnormal{{$\mathsf{a}$}}}
\newcommand{\sfb}{\textnormal{{$\mathsf{b}$}}}
\newcommand{\sfc}{\textnormal{{$\mathsf{c}$}}}
\newcommand{\sfS}{\textnormal{{$\mathsf{S}$}}}
\newcommand{\sfU}{\textnormal{{$\mathsf{U}$}}}

\def\udot#1{\ifmmode\oalign{$#1$\crcr\hidewidth.\hidewidth
    }\else\oalign{#1\crcr\hidewidth.\hidewidth}\fi}

\def\R{\mathbb{R}}
\def\Z{\mathbb{Z}}
\def\T{\mathbb{T}}

\def\Q{\mathbb{Q}}
\def\beq{\begin{equation}}
\def\eeq{\end{equation}}
\makeatletter
\newcommand{\doublewidetilde}[1]{{%
  \mathpalette\double@widetilde{#1}%
}}
\newcommand{\double@widetilde}[2]{%
  \sbox\z@{$\m@th#1\widetilde{#2}$}%
  \ht\z@=.9\ht\z@
  \widetilde{\box\z@}%
}
\makeatother

\def\one{\mbox{1\hspace{-4.25pt}\fontsize{12}{14.4}\selectfont\textrm{1}}}

\usepackage{tikz}

\makeatletter
\def\@makefnmark{%
  \leavevmode
  \raise.9ex\hbox{\fontsize\sf@size\z@\normalfont\tiny\@thefnmark}}
\makeatother

\begin{document}
	
\title[On the curved Trilinear Hilbert transform]{On the curved Trilinear Hilbert transform}
\author{Bingyang Hu and Victor Lie}

\address{Bingyang Hu: Department of Mathematics, Purdue University, 150 N. University St, W. Lafayette, IN 47907, U.S.A.}%
\email{hu776@purdue.edu}%%

\address{Victor D. Lie: Department of Mathematics, Purdue University, 150 N. University St, W. Lafayette, IN 47907, U.S.A.  and
The ``Simion Stoilow" Institute of Mathematics of the Romanian Academy, Bucharest, RO 70700, P.O. Box 1-764, Romania}%
\email{vlie@purdue.edu}%%

\begin{abstract} Building on the (Rank I) LGC-methodology introduced by the second author in \cite{Lie19} and on the novel perspective employed in the time-frequency discretization of the non-resonant bilinear Hilbert--Carleson operator, \cite{BBLV21}, we develop a new, versatile method---referred to as Rank II LGC---that has as a consequence the resolution of the $L^p$ boundedness of the trilinear Hilbert transform along the moment curve. More precisely, we show that the operator
\begin{equation*}
H_{C}(f_1, f_2, f_3)(x):= \textrm{p.v.}\,\int_{\R} f_1(x-t)f_2(x+t^2)f_3(x+t^3) \frac{dt}{t}, \quad x \in \R\,,
\end{equation*}
is bounded from $L^{p_1}(\R)\times L^{p_2}(\R)\times L^{p_3}(\R)$ into $L^{r}(\R)$ within the Banach H\"older range
$\frac{1}{p_1}+\frac{1}{p_2}+\frac{1}{p_3}=\frac{1}{r}$ with $1<p_1,p_3<\infty$, $1<p_2\leq \infty$ and $1\leq r <\infty$.

A crucial difficulty in approaching this problem is the lack of absolute summability for the linearized discretized model (derived via Rank I LGC method) of the quadrilinear form associated to $H_{C}$. In order to overcome this, we develope a so-called \emph{correlative} time-frequency model whose control is achieved via the following interdependent elements: (1) a sparse-unform decomposition of the input functions adapted to an appropriate time-frequency foliation of the phase-space, (2) a structural analysis of suitable maximal ``joint Fourier coefficients", and (3)  a level set analysis with respect to the time-frequency correlation set.
\end{abstract}
\date{\today}

\keywords{Trilinear Hilbert transform, correlative time-frequency/wave-packet analysis, joint Fourier coefficients,  zero/non-zero curvature, LGC method, sparse-uniform decomposition, level set analysis, time-frequency correlation set.}

\thanks{}

\dedicatory{Dedicated to Nicolae Popa on the\\ occasion of his $80^{\textrm{th}}$ birthday celebration.}

\maketitle

\setcounter{tocdepth}{2}
\tableofcontents

\newpage
\section{Introduction}

In this paper, we study the boundedness properties of the \emph{curved} tri-linear Hilbert transform defined by
\begin{equation}\label{CurvedTHTdef}
H_{C}(f_1, f_2, f_3)(x):= \textrm{p.v.}\,\int_{\R} f_1(x-t)f_2(x+t^2)f_3(x+t^3) \frac{dt}{t}, \quad x \in \R.
\end{equation}
where here $f_1, f_2, f_3 \in \calS(\R)$ are Schwarz functions.

The interest in considering this topic is multifaceted---with motivations from the areas of harmonic analysis, ergodic theory,  additive combinatorics and number theory---and will be detailed later in Section \ref{Motivation}. For the time being, embracing only the time-frequency perspective, we mention that this is the first proof in the literature of the natural Banach range\footnote{In the sense that we require both the domain and the codomain of $H_{C}$ to be Banach spaces.} bound for a trilinear singular oscillatory integral operator with translation and dilation symmetries and kernel singularity that is directly related to the celebrated open problem regarding the boundedness of the trilinear Hilbert transform defined by
\begin{equation}\label{THT}
{\bf T}(f_1, f_2, f_3)(x):= \textrm{p.v.}\,\int_{\R} f_1(x+t)f_2(x+2t)f_3(x+3t) \frac{dt}{t}, \quad x \in \R.
\end{equation}

For approaching the operator $H_{C}$ we develop a new method---referred to as {Rank II LGC---that
builds on the work on the (non-resonant) bilinear Hilbert--Carleson operator, \cite{BBLV21}, and further originates in the (Rank I) LGC-methodology\footnote{The abbreviation ``LGC" stands for the three main steps composing this method: (L)inearization, (G)abor frame decomposition, (C)ancellation via (C)orrelation.} introduced by the second author in 2019 in \cite{Lie19}. For more on the origin, motivation, and main steps in the development of this method, we invite the reader to consult Section \ref{LGCphil}, Section \ref{Keyideas}, and also Section \ref{SDTmotiv}. Indeed, readers who are less familiar with the
literature--in particular with the role of the zero/non-zero curvature dichotomy in the context of multilinear
singular integral operators and the challenges presented by modulation symmetries in the time-frequency
discretization of such operators--may wish to consult those subsections before proceeding further.

\subsection{Main results; future directions of investigation}\label{MR}
\medskip

The main result of our paper is given by the following

\begin{thm} \label{MAINthm}
Let $\vec{f}:=(f_1,f_2,f_3,f_4)\in \calS^4(\R)$ and, via a dualization argument, let us  represent \eqref{CurvedTHTdef} by the quadrilinear form\footnote{Throughout the paper, for notational simplicity, we ignore in our writing the effect of the conjugation. Also, in several instances, when convenient, we ignore the principal value symbol.}
\begin{equation}\label{DefL}
\Lambda(\vec{f}):=\int_{\R}\int_{\R} f_1(x-t)f_2(x+t^2)f_3(x+t^3)f_4(x)\frac{dtdx}{t}\,.
\end{equation}

Then, for $\vec{p}:=(p_1,p_2,p_3,p_4)$ and $\|\vec{f}\|_{L^{\vec{p}}}:=\prod_{j=1}^4 \|f_j\|_{L^{p_j}}$, the following holds
\begin{equation}\label{boundQ}
|\Lambda(\vec{f})|\lesssim_{\vec{p}} \|\,\vec{f}\,\|_{L^{\vec{p}}}\,,
\end{equation}
for any $\vec{p}\in\textbf{B}:=\left\{(p_1,p_2,p_3,p_4)\:|\:\sum_{j=1}^4\frac{1}{p_j}=1,\,\substack{1<p_1,p_3<\infty\\ 1<p_2,p_4\leq \infty}\right\}$.
\end{thm}

We next present the following

\begin{obs}\textsf{[Extensions]}\label{Mgen}  One can extend the above result as follows:
\begin{enumerate}[label=(\roman*)]
\item Our particular choice for the moment curve $\gamma_0(t)=(-t,\,t^2,\,t^3)$ associated with the $t$-argument of the input functions $(f_1,f_2,f_3)$ is just for expository reasons. Indeed, the conclusion of our theorem holds (up to possibly endpoints, \textit{i.e}, requiring $1<p_j<\infty$ for all $j$) if one replaces $\gamma_0$ by any curve of the form $\gamma(t)=(t^{\al_1},\,t^{\al_2},\,t^{\al_3})$ with $\al_j\in\R\setminus\{0\}$, $1\leq j \leq 3$, distinct. Using the techniques in \cite{Lie15} and \cite{Lie19}, this result can be further extended to the situation when $\gamma(t)=(P_1(t),\,P_2(t),\,P_3(t))$ with $P_j(t)=\sum_{r=1}^{n_j}a_{j,r}\, t^{\al_{j,r}}$ generalized polynomials/fewnomials such that $\{a_{j,r}\}_{j,r},\,\{\al_{j,r}\}_{j,r}\subset\R\setminus\{0\}$ and the $\al_{j,r}$ are distinct. Moreover, with some supplementary adjustments, the latter condition can be further relaxed by only requiring that the three sets formed with the smallest and highest degrees for each of the three fewnomials are pairwise disjoint. This latter condition is required in order to avoid resonances\footnote{For more on the implications/consequences of the existence of resonances please see \cite{BBLV21} and \cite{GL22}.} among the input functions that are usually indicative of (possibly hidden) modulation invariance properties of the operator/form under analysis. However, if one combines the method developed here with (linear) modulation invariance techniques, then one can even control the case of two-input resonances as revealed in Observation \ref{FW} below!

\item The boundedness range for our theorem can be improved to cover suitable subunitary values of $p_4$ (see the general Conjecture \ref{Mult} and Remark \ref{Boundrangeconj} below for the maximal projected range). This requires the usage of multilinear interpolation techniques beyond the Banach range involving exceptional/major sets.\footnote{For the unfamiliar reader, a good reference for this subject is the monograph \cite{MS13}.} Due to the volume and highly involved conceptual and technical nature of this paper we did not pursue this direction here. Even so, the most difficult result needed for such an extension is in fact provided here in the form of Theorem \ref{mainthm}.
\item We believe that the new methods/techniques introduced in the present paper are suitable for approaching both the singular and the maximal variant of the higher multilinear case--see Conjecture \ref{Mult}--and we hope to address (at least part of) these topics in the not too distant future.
\end{enumerate}
\end{obs}

\begin{obs}\textsf{[Forthcoming work]}\label{FW}
Once one understands the complexity of the object defined in \eqref{CurvedTHTdef} in this \emph{pure-curvature} setting it is natural to go one step further, and construct another key model for the trilinear Hilbert transform \eqref{THT} that captures not only the \emph{kernel singularity} and the \emph{dilation and translation symmetries together with the multilinear singular} nature of ${\bf T}$ but also its \emph{linear modulation invariant} property. A successful embodiment of all these features is represented by the so-called \emph{hybrid} trilinear Hilbert transform\footnote{The monomial $t^3$ in the definition of $H_{H}$ can be replaced with any $t^{\alpha}$ under the restriction $\alpha\in\R\setminus\{0,1,2\}$ in order to avoid resonances, or more generally, with any nonconstant (generalized)  polynomial having no linear and no quadratic terms.}
\begin{equation}\label{HTHTdef}
H_{H}(f_1, f_2, f_3)(x):= \textrm{p.v.}\,\int_{\R} f_1(x-t)f_2(x+t)f_3(x+t^3) \frac{dt}{t}, \quad x \in \R.
\end{equation}
Indeed, notice that $H_{H}$ extends the classical bilinear Hilbert transform
\begin{equation}\label{BHTdef}
B(f_1, f_2)(x):= \textrm{p.v.}\,\int_{\R} f_1(x-t)f_2(x+t) \frac{dt}{t}, \quad x \in \R\,,
\end{equation}
and thus it inherits the natural linear modulation invariance properties  of $B$ (relative to the inputs $f_1$ and $f_2$); moreover, $H_{H}$ also exhibits pure-curvature traits derived from the third input $f_3$, whence its truly \emph{hybrid} nature.

Combining our newly developed method here with the approach in \cite{BBLV21} the authors in \cite{BBL23} provide the boundendness of $H_{H}$ within a range that extends the Banach case discussed in the present paper.
\end{obs}

\begin{rem}\label{RM} The program dedicated to the study of the boundedness properties for the operators $H_{C}$ and $H_{H}$---regarded primarily but not exclusively as ``curved" models for {\bf T} (see Section \ref{Motivation})---was envisioned more than a decade\footnote{The original arXiv submission date of \cite{Lie15} was October 2011.} ago by the second author here--see the Remarks section in \cite{Lie15}.  In light of Theorem \ref{MAINthm}, Observation \ref{FW} and subject to the resolution of the work in \cite{BBL23}, this program is essentially complete\footnote{Up to the maximal boundedness range--see Conjecture \ref{Mult} and Remark \ref{Boundrangeconj}.}. However, the methods developed along the way for achieving it open now new avenues and appear promising for approaching various other difficult ``purely curved" and hybrid problems, among which, the most natural in the context of the present paper is the higher multilinear case, $n\geq 4$, of the following:
\end{rem}

\begin{conj}\textsf{[Curved/Hybrid $n-$linear Hilbert transform and Maximal operators]}\label{Mult}
\medskip

Fix $n\geq 2$ and, as before, let $\vec{p}:=(p_1,\ldots, p_n)$, $\vec{f}=(f_1,\ldots, f_n)\in \calS^n(\R)$ and $\|\vec{f}\|_{L^{\vec{p}}}:=\prod_{j=1}^n \|f_j\|_{L^{p_j}}$.

Define the curved $n-$linear Hilbert transform as
\begin{equation}\label{nCurvedTHTdef}
H_{C,n}(\vec{f})(x):= \textrm{p.v.}\,\int_{\R} f_1(x-t)f_2(x+t^2)\ldots f_n(x+t^n) \frac{dt}{t}, \quad x \in \R\,,
\end{equation}
and the curved $n-$(sub)linear maximal operator as
\begin{equation}\label{nCurvedMTHTdef}
M_{C,n}(\vec{f})(x):= \sup_{\ep>0}\,\frac{1}{2\ep}\,\int_{-\ep}^{\ep} \left| f_1(x-t)f_2(x+t^2)\ldots f_n(x+t^n)\right|\,dt, \quad x \in \R\,.
\end{equation}
Then
\begin{equation}\label{nboundHM}
\|H_{C,n}(\vec{f})\|_{L^{p_{n+1}}},\,\|M_{C,n}(\vec{f})\|_{L^{p_{n+1}}}\lesssim_{\vec{p}} \|\vec{f}\|_{L^{\vec{p}}}\,,
\end{equation}
for any $\vec{p}:=(p_1,\ldots, p_n)\in (1,\infty]^n$ with $\sum_{j=1}^n\frac{1}{p_j}=\frac{1}{p_{n+1}}$ and $\frac{1}{n}<p_{n+1}<\infty$.

The same range is conjectured to hold in the hybrid situation, \textit{i.e.}, when the argument $x+t^2$ for $f_2$ above is replaced by $x+t$, case in which the corresponding operators for \eqref{nCurvedTHTdef} and \eqref{nCurvedMTHTdef} are referred to as $H_{H,n}$ and $M_{H,n}$, respectively.
\end{conj}

\begin{rem}\label{Boundrangeconj} Notice that for $n\in\{2,3\}$ the above conjecture is, at least in part, settled:  The case $n=2$ is completely solved in the (purely) non-zero curvature setting, (\cite{Li13}, \cite{Lie15}, \cite{LX16}, \cite{Lie18}, \cite{GL20}, \cite{GL22}), while, in the (hybrid) zero curvature context is only partly known under the restriction\footnote{For more on this please see Section \ref{LGCphil} and Section \ref{CTFrev} including Observation \ref{BHTinv} therein.} $p_3>\frac{2}{3}$,(\cite{LT97}, \cite{LT99}). For the case $n=3$, referring only to the singular version(s) \eqref{nCurvedTHTdef}, we have: the purely curved situation is the topic of the present paper and covers the situation $p_4\geq 1$ while the hybrid situation is treated within a partial range, in the forthcoming work \cite{BBL23}.
\end{rem}

\subsection{Main ideas behind the proof}\label{Keyideas}
\medskip

The leitmotif of the Rank II LGC method that we develop in this paper is the interplay between \emph{cancellation} and \emph{correlation}.\footnote{With the latter understood here in a loosely manner.}  This manifests throughout the proof via the following elements, from its strategy conception, (i), to its realization and conclusion, (ii):
\begin{enumerate}[label=(\roman*)]
\item  \emph{correlative time-frequency discretized model};

\item  \emph{time-frequency foliation} and a \emph{sparse-uniform dichotomy} in order to control the behavior of
\begin{enumerate}
\item the \emph{linearizing supremum function} $\lambda$---in the uniform case;

\item the \emph{time-frequency correlation sets}---in the sparse case.
\end{enumerate}
\end{enumerate}
In what follows, we elaborate on the above itemization.

\subsubsection{Correlative time-frequency discretized model}\label{CTF}

The essence of what will be discussed in this section may be expressed by the following antithesis:
\begin{itemize}
\item Except for the recent work in \cite{BBLV21}, all the previous literature involving linear wave-packet discretization and addressing singular integral operators--irrespective of treating zero or non-zero curvature problems\footnote{For a more elaborate discussion, please see Section \ref{LGCphil}.}--have in common the following procedure: given the operator $T$ to be analyzed, after a dualization argument, $T$ is replaced by the corresponding form $\Lambda$; then, the first fundamental step in the proof, relies on identifying the proper discretized model for $\Lambda$.
    After a preliminary discretization of both the kernel/physical side (providing the physical scale) and of the multiplier/frequency side (providing the location or/and height of the oscillatory phase) this procedure relies on
\begin{itemize}
\item a Gabor decomposition of \emph{all} the input functions, and, subject to the scale and frequency location, the input wave-packet discretizations are to be performed \emph{independently} of one another;\footnote{This relative independency refers strictly to the discretization process and is not to be confused with the later process of controlling/gluing the tiles together when tree selection algorithms based on concepts such as size and mass fundamentally rely on input interdependencies.}

\item obtaining control on the resulting discretized model of $\Lambda$ via \emph{absolute values} of the input Gabor coefficients, thus requiring this model to be \emph{absolutely summable}.
\end{itemize}

\item In contrast with the above, in the context of our problem, if one attempts a similar procedure\footnote{This amounts to an application of the standard (Rank I) LGC-method that passes through linearization and Gabor decomposition of the inputs.} for the quadrilinear form $\Lambda$ defined in \eqref{DefL} the resulting discretized model is \emph{not absolutely summable}, \textit{i.e}, the sum of the absolute values of the input Gabor coefficients is divergent. In order to circumvent this key difficulty, we will develop a different discretization strategy that will involve a Gabor decomposition for only \emph{one} of the inputs, allowing for the other three inputs to mix their frequency information along spatial fibers and thus exploit a hidden cancellation. The process through which we obtain this new summable discretized model  will be referred to as \emph{correlative time-frequency analysis}, on
which we now elaborate through a staged approach.
\end{itemize}

\newpage

\noindent\textbf{$1^{st}$ stage: The issue of absolute summability; joint Fourier coefficient decomposition}
$\newline$

Departing from the quadrilinear form $\Lambda(\vec{f})$ defined in \eqref{DefL}, after some standard considerations involving  discretizations on both the kernel and multiplier sides followed by stationary phase analysis we find that the main contribution to our form is brought by the diagonal component
\begin{equation}\label{DiagtermI}
\Lambda^{D}\approx\sum_{k \in \Z}\:\: \sum_{m \in \N} \Lambda^k_{m}\,,
\end{equation}
where here
\begin{equation}\label{mfI}
\Lambda_m^k(\vec{f})=\int_{\R}\int_{\R} f_{1,m+k}(x-t)f_{2,m+2k}(x+t^2)f_{3,m+3k}(x+t^3)f_{4}(x)\,2^k \varrho(2^k t) dtdx,
\end{equation}
with $\widehat{f_{j,m+j k}}(\xi):=\hat{f}_{j}(\xi)\,\phi(\frac{\xi}{2^{m+j k}})$, $j\in\{1,2,3\}$ and $\phi, \varrho$ are smooth and compactly supported away from the origin.

Taking $m\in\N$ and assuming without loss of generality that $k\geq 0$ is fixed, it is only natural to restrict the $x-$domain of integration to an interval of the form $[2^{-k},2^{-k+1}]:=I^k$ whose shape is induced by the range of $t$ and some standard almost orthogonality arguments.

Applying now the (Rank I) LGC method introduced in \cite{Lie19}, we deduce that\footnote{See Section \ref{LGC(I)} for more details.}
\begin{eqnarray} \label{LGC1I}
&&\Lambda_m^k(\vec{f})\approx 2^{\frac{3k}{2}}\sum_{\substack{p \sim 2^{\frac{m}{2}+2k} \\ q \sim 2^{\frac{m}{2}}}} \sum_{(n_1, n_2, n_3)\in \TFC^{I}_q} C_{q, n_1, n_2, n_3}\,\left\langle f_1, \phi_{\frac{p}{2^{2k}}-q, n_1}^k \right\rangle \left\langle f_2, \phi_{\frac{p}{2^{k}}+\frac{q^2}{2^{\frac{m}{2}}}, n_2}^{2k} \right \rangle  \\
&&\qquad\qquad\qquad\qquad\qquad\qquad\qquad\qquad\qquad\qquad\qquad\cdot\left\langle f_3, \phi_{p+\frac{q^3}{2^m}, n_3}^{3k} \right \rangle \left\langle f_4, \phi_{p, n_3+\frac{n_2}{2^k}+\frac{n_1}{2^{2k}}}^{3k} \right \rangle\,,\nonumber
\end{eqnarray}
where here $\phi_{r, n}^{j k}(x):=2^{\frac{m}{4}+\frac{j k}{2}} \phi \left(2^{\frac{m}{2}+jk}x-r \right) e^{i 2^{\frac{m}{2}+j k}xn}$ are the standard Gabor frames (wave-packets), the coefficients $C_{q, n_1, n_2, n_3}$ are some unitary complex numbers that only depend on $q, n_1, n_2, n_3$ but not on the input functions $\{f_j\}_j$, and
\begin{equation} \label{TFCqqI}
\TFC^{I}_q:=\left\{ \substack{\left(n_1, n_2, n_3 \right)\\n_1, n_2, n_3 \sim 2^{\frac{m}{2}}}\, :\, n_1-\frac{2q}{2^{\frac{m}{2}}} n_2- \frac{3q^2}{2^m} n_3\approx 0 \right\}
\end{equation}
stands for the \emph{time-frequency correlation set}, which captures the inter-dependencies between the spatial and frequency parameters---a direct manifestation of the presence of curvature in our problem.
\medskip

The first major obstacle for controlling \eqref{DiagtermI} is that the model discretization derived in \eqref{LGC1I} is \emph{not absolutely summable}. Thus we have to design a new type of discretization that exploits the hidden cancelations encoded within the local Fourier coefficients of the input functions.
\medskip

\textit{The key insight is to perform a \emph{partial} time-frequency decomposition in which we apply a Gabor frame discretization of \emph{only one} of the input functions---that, for $k\geq 0$, corresponds to the frequency dominant input $f_3$, and treat the remaining inputs---in a first stage three and later, in the second stage, reduced to just two inputs---as \emph{a single entity} on which we act via a discretization of the \emph{physical} side only, allowing thus for the frequencies aligned along spatial fibers to freely interact and produce the desired cancellation.}
\medskip

Concretely, in the first stage of our time-frequency discretization, we proceed as follows: adopting the \emph{linearization} step in Rank I LGC that provides us with the correct scale for the spatial cuttings of the $(x,t)-$domain so that the $t$ arguments for each $f_j$ behave linearly on each of the resulting subintervals, we first write
\begin{equation}\label{linI}
\Lambda_m^k(\vec{f})\approx 2^k\,\sum_{\substack{p \sim 2^{\frac{m}{2}+2k} \\ q \sim 2^{\frac{m}{2}}}} \int_{I_p^{m, 3k}} \int_{I_q^{m, k}} f_1(x-t) f_2(x+t^2) f_3(x+t^3) f_4(x) dtdx
\end{equation}
where here  for $m\in \N$ and $k,\,q \in \Z$ we let $I_q^{m, k}:=\left[\frac{q}{2^{\frac{m}{2}+k}}, \frac{q+1}{2^{\frac{m}{2}+k}} \right]\,.$ Next, we perform a Gabor decomposition of the third input $f_3(x)=\sum_{{r \sim 2^{\frac{m}{2}+2k}}\atop{n \sim 2^{\frac{m}{2}}}}  \left\langle f_3, \phi_{r, n}^{3k} \right \rangle \phi_{r, n}^{3k}(x)$. Substituting this back into \eqref{linI} and exploiting the spatial localization and the linearization effect (this results in suitable correlations of the spatial parameters), we deduce
\begin{equation} \label{discI}
 \Lambda_m^k(\vec{f}) \sim 2^k \sum_{\substack{p \sim 2^{\frac{m}{2}+2k} \\ q, n \sim 2^{\frac{m}{2}}}} e^{-2i n \cdot \frac{q^3}{2^m}}\,\left\langle f_3, \phi_{p+\frac{q^3}{2^m}, n}^{3k}  \right\rangle \int_{I_p^{m, 3k}} {J}^{k}_{q, n}(f_1,f_2)(x)\,\phi_{p, n}^{3k} (x)\,f_4(x)dx
\end{equation}
with
\begin{equation} \label{joint}
{J}^{k}_{q, n}(f_1,f_2)(x) := \int_{I_q^{m, k}} f_1(x-t)f_2(x+t^2) e^{3i \cdot 2^{\frac{m}{2}+k}t \left(\frac{q}{2^\frac{m}{2}} \right)^2 n} dt\,.
\end{equation}

The important realization (and main goal of our paper) is that the substitution of \eqref{discI} into \eqref{DiagtermI} produces an \emph{absolutely summable model}. Of course, there is a price to pay for this: encoded in the ``joint Fourier coefficient" \eqref{joint}, we now have to deal with the fuzziness and mixture of the frequency localization of the input functions $f_1$, $f_2$ and $f_4$. In order to better control this phenomenon, we perform a second stage time-frequency analysis that is discussed immediately below.

$\newline$
\noindent\textbf{$2^{nd}$ stage: Decoupling of the pair $(f_1,f_2)$ from $f_4$}\label{TFfoliation}
$\newline$

The crux of this second stage may be explained as follows: in order to analyze the behavior of
$$\int_{I_p^{m, 3k}} {J}^{k}_{q, n}(f_1,f_2)(x)\,\phi_{p, n}^{3k} (x)\,f_4(x)dx$$
we intend to  simultaneously accomplish:
\begin{itemize}\label{KeYprinc}
\item the preservation of the joint behavior for the pair $(f_1,f_2)$;
\item the decoupling of the information carried by $f_4$ from that corresponding to the pair $(f_1,f_2)$.
\end{itemize}

To achieve this we appeal to a frequency partition of the inputs $f_1$ and $f_2$ adapted to the Heisenberg localization of the wave-packet $\phi_{p, n}^{3k}$. As it turns out, this partition is influenced by the relative positions of the physical scale $k\geq 0$ and of the frequency scale $m\in\N$:
\medskip

\noindent{\underline{\textsf{The case $k\geq \frac{m}{2}$}}.}
\medskip
In this situation, exploiting the hypothesis $\supp \ \widehat{f_i} \subseteq \left[2^{m+ik}, 2^{m+ik+1} \right]$ for $i\in\{1,2\}$
and using the relation $|I_p^{m, 3k}|=2^{-\frac{m}{2}-3k} \le 2^{-m-2k}$, we deduce that\footnote{See the Key Heuristic 1 in Section \ref{IIlgcklarge}.}
\begin{equation} \label{Key1I}
\textrm{ the map}\:\:x \mapsto {J}^{k}_{q, n}(f_1,f_2)(x)\:\textrm{is morally constant on the interval}\:I_p^{m, 3k}\,,
\end{equation}
and thus, in this regime there is no need to further subdivide the frequency locations on the spatial fibers associated with $f_1$ and $f_2$ in order to achieve the above declared goal.

Consequently, we obtain the following (correlative) time-frequency model:
\begin{eqnarray} \label{Case1}
&& \left|\Lambda_m^k(\vec{f})\right|  \lesssim 2^k \sum_{\substack{p \sim 2^{\frac{m}{2}+2k} \\ q, n \sim 2^{\frac{m}{2}}}} \left| \left\langle f_3, \phi_{p+\frac{q^3}{2^m}, n}^{3k}  \right\rangle \right| \left|  \left \langle f_4, \phi_{p, n}^{3k} \right \rangle \right| \\
&& \quad \quad \quad \quad \cdot \frac{1}{\left|I_p^{m, 3k} \right|} \left| \int_{I_p^{m, 3k}} \int_{I_q^{m, k}} f_1(x-t)f_2(x+t^2) e^{3i \cdot 2^{\frac{m}{2}+k}t \left(\frac{q}{2^\frac{m}{2}} \right)^2 n} dt dx \right| . \nonumber
\end{eqnarray}

\medskip
\noindent{\underline{\textsf{The case $0\leq k<\frac{m}{2}$}}.}
\medskip
In this situation, guided by the Heisenberg principle, our strategy is to discretize the frequency information of $f_1$ and $f_2$ into equidistant intervals such that both of the items below hold:
\begin{itemize}
\item the length of the frequency intervals is dominated by the length of the frequency support of the wave-packet $\phi_{p, n}^{3k}$: this guarantees the desired decoupling goal;

\item the number of the subdivisions corresponding to $f_1$ agrees with the corresponding number for $f_2$: this is crucial for controlling the time-frequency correlation set--see Subsection \ref{RescTFC}.
\end{itemize}

Concretely,  we proceed as follows: for $j\in\{1,2,3\}$ we divide
$$\supp \ \widehat{f_j} \subseteq\left[2^{m+jk}, 2^{m+jk+1} \right]=\bigcup_{\ell_j \sim 2^{\frac{m}{2}-k}} \omega_j^{\ell_j}:=\bigcup_{\ell_j\sim 2^{\frac{m}{2}-k}} \left[\ell_j 2^{\frac{m}{2}+(j+1)k}, \left(\ell_j+1 \right) 2^{\frac{m}{2}+(j+1)k}\right]\,,$$
write
\begin{equation}\label{freqdec}
f_j=\sum_{\ell_j \sim 2^{\frac{m}{2}-k}} f_j^{\omega_j^{l_j}}\,,
\end{equation}
with $f_j^{\omega_j^{l_j}}$ designating the Fourier projection of $f_j$ onto $\omega_j^{l_j}$ and finally let $\underline{f}_j^{\omega_j^{\ell_j}}(x):= f_j^{\omega_j^{\ell_j}}(x)\, e^{-i \ell_j 2^{\frac{m}{2}+(j+1)k} x}$ be the shift at the zero frequency of $f_j^{\omega_j^{l_j}}$.

With these notations, recalling \eqref{joint}, we have
\begin{equation}\label{DecIntr}
{J}^{k}_{q, n}(f_1,f_2)(x)=\sum_{\ell_1, \ell_2 \sim 2^{\frac{m}{2}-k}} e^{i \left(\ell_1 2^{\frac{m}{2}+2k}+\ell_2 2^{\frac{m}{2}+3k} \right)x} \underline{J}^k_{q,n} \left( \underline{f}_1^{\omega_1^{\ell_1}}, \underline{f}_2^{\omega_2^{\ell_2}} \right)(x),
\end{equation}
with
\begin{equation}
\underline{J}^k_{q,n} \left( \underline{f}_1^{\omega_1^{\ell_1}}, \underline{f}_2^{\omega_2^{\ell_2}} \right)(x):=\int_{I_q^{m, k}}  \underline{f}_1^{\omega_1^{\ell_1}}(x-t) \underline{f}_2^{\omega_2^{\ell_2}}(x+t^2) e^{i \left(-\ell_1 2^{\frac{m}{2}+2k}t+\ell_2 2^{\frac{m}{2}+3k} t^2+3 \cdot 2^{\frac{m}{2}+k} \cdot \frac{q^2}{2^m} n t \right)} dt.
\end{equation}
Once at this point, one can apply a similar argument with the one in \eqref{Key1I}:
\begin{equation} \label{Key2I}
x \to \underline{J}^k_{q,n} \left( \underline{f}_1^{\omega_1^{\ell_1}}, \underline{f}_2^{\omega_2^{\ell_2}} \right)(x)
\:\textrm{is morally constant on the interval}\:I_p^{m, 3k}\,.
\end{equation}
With these, exploiting the dependencies among the spatial and frequency parameters--see Section \ref{Diagkpsmall}, we obtain the following (correlative) time-frequency model:
\begin{eqnarray} \label{CTFM2}
&&|\Lambda_m^k(\vec{f})|\lesssim 2^k \sum_{\substack{p \sim 2^{\frac{m}{2}+2k}\\ q \sim 2^{\frac{m}{2}} \\ n \sim 2^k}} \sum_{(\ell_1, \ell_2, \ell_3) \in \TFC_q}  \left| \left\langle f_3^{\omega_3^{\ell_3}}, \phi_{p+\frac{q^3}{2^m}, \ell_3 2^k+n}^{3k} \right \rangle \right|  \left|\left \langle f_4, \phi^{3k}_{p, \ell_1 2^{-k}+\ell_2+\ell_32^k+n} \right \rangle \right| \nonumber \\
&&\frac{1}{\left|I_p^{m, 3k} \right|} \bigg| \int_{I_p^{m, 3k}} \int_{I_q^{m, k}}  \underline{f}_1^{\omega_1^{\ell_1}}(x-t) \underline{f}_2^{\omega_2^{\ell_2}}(x+t^2) e^{i 2^{\frac{m}{2}+k}t \cdot \frac{3q^2}{2^m}n}  e^{i 2^{\frac{m}{2}+k}t \cdot 2^k \left(-\ell_1 +\ell_2 \frac{2q}{2^{\frac{m}{2}}}+ \ell_3 \frac{3q^2}{2^m}\right)} dt dx \bigg|
\end{eqnarray}
with \emph{the time-frequency correlation set} $\TFC_q$ defined as follows:
for each $q \sim 2^{\frac{m}{2}}$ we have
\begin{equation} \label{TFCq}
\TFC_q:=\left\{ \substack{\left(\ell_1, \ell_2, \ell_3 \right)\\\ell_1, \ell_2, \ell_3 \sim 2^{\frac{m}{2}-k}}\, :\, \ell_1-\frac{2q}{2^{\frac{m}{2}}} \ell_2- \frac{3q^2}{2^m} \ell_3\approx 0\right\}.
\end{equation}

\subsubsection{Time-frequency foliation and a sparse-unform dichotomy}

Once we have understood the correct perspective on the discretization of our form $\Lambda_m^k(\vec{f})$, the main challenge is to obtain a good control on the models \eqref{Case1} and \eqref{CTFM2}. This is provided by Theorem \ref{mainthm} in Section \ref{SDTmotiv} in the form: there exists $\epsilon>0$ and $c(\vec{p})>0$ such that
\begin{equation} \label{Keyyy}
\left|\Lambda_m^k(\vec{f})\right| \lesssim_{\vec{p}} 2^{-\epsilon\,c(\vec{p})\,\min\{|k|,\frac{m}{2}\}} \,\|\vec{f}\|_{L^{\vec{p}}}\,,
\end{equation}
for any $p$ within the range of Theorem \ref{MAINthm}. A key underlying realization that contrasts with the typical curvature estimates is the formulation of the correct decaying factor in \eqref{Keyyy} that captures the interplay between the physical scale $k$ and the frequency scale $m$.

The proof of this result involves the analysis of how the information carried by the input functions $f_j$ distributes within the time-frequency space. The specifics of this analysis depend though on the two cases discussed above:

\medskip
\noindent{\textsf{\underline{The case $k\geq \frac{m}{2}$}: control over the linearizing supremum function $\lambda$}}
\medskip

In this situation, since we are not subdividing the frequency locations corresponding to $f_1$ and $f_2$ along the spatial fibers, we only focus on the spatial distribution of the inputs: each $f_j$ is decomposed into a sparse $f_j^{\calS}$ and a uniform component $f_j^{\calU}$ that, at the heuristic level, correspond to the locations where the function $f_j$ is above or below the expected value in the $L^2-$average sense, respectively. Accordingly, one decomposes the form $\Lambda_m^k(\vec{f})$ into a sparse $\Lambda_m^{k, \calS}(f)$ and a uniform component $\Lambda_m^{k, \calU}(f)$, with the latter addressing the case when, informally, all the input functions are uniform.

Now, relying on the smallness of the spatial support of the inputs, the sparse component $\Lambda_m^{k, \calS}(f)$ can be easily controlled. Thus, we are left with the treatment of the uniform case--the key task in the context $k\geq\frac{m}{2}$. Departing from \eqref{Case1}, after successive Cauchy-Schwarz arguments, a final $l^2\times l^2\times l^{\infty}$ H\"older application in the $n-$parameter, reduces matters to
\begin{eqnarray}\label{Holder}
&&\left| \Lambda_m^{k, \calU} (\vec{f}) \right| \lesssim 2^k \sum_{\ell \sim 2^{2k}} \left( \sum_{n \sim 2^{\frac{m}{2}}} \sum_{p=2^{\frac{m}{2}}\ell}^{2^{\frac{m}{2}}(\ell+1)} \left| \left \langle f_3^{\calU}, \phi_{p, n}^{3k} \right \rangle \right |^2 \right)^{\frac{1}{2}} \left( \sum_{n \sim 2^{\frac{m}{2}}} \sum_{p=2^{\frac{m}{2}} \ell}^{2^{\frac{m}{2}}(\ell+1)} \left| \left \langle f_4^{\calU}, \phi_{p, n}^{3k} \right \rangle \right |^2 \right)^{\frac{1}{2}}\nonumber \\
&&\left( \frac{1}{\left|I_0^{m,3k} \right|}  \sup_{n \sim 2^{\frac{m}{2}}} \int_{I_\ell^{0, 3k}} \left( \sum_{q \sim 2^{\frac{m}{2}}}  \left| \int_{I_q^{m, k}} f^{\calU}_1(x-t)f^{\calU}_2(x+t^2) e^{3i \cdot 2^{\frac{m}{2}+k}t \left(\frac{q}{2^\frac{m}{2}} \right)^2 n} dt \right|^2 \right) dx \right)^{\frac{1}{2}}.\qquad
\end{eqnarray}
Exploiting now the uniformity of the input functions one concludes that
$$
\left|\Lambda_m^{k, \calU}(\vec{f}) \right| \lesssim 2^k\,\left\|f_3 \right\|_{L^2 \left(3 I^k \right)} \left\|f_4 \right\|_{L^2 \left(3 I^k \right)} \textbf{L}_{m,k}(f_1, f_2),
$$
with
\begin{equation}\label{Lmk}
\left(\textbf{L}_{m,k}(f_1, f_2)\right)^2 := \frac{1}{\left|I_0^{m,k} \right|} \int_{I^{k}} \left( \sum_{q \sim 2^{\frac{m}{2}}}  \left| \int_{I_q^{m, k}} f^{\calU}_1(x-t)f^{\calU}_2(x+t^2) e^{3i \cdot 2^{\frac{m}{2}+k}t \left(\frac{q}{2^\frac{m}{2}} \right)^2 \lambda(x)} dt \right|^2 \right) dx\,,
\end{equation}
where $\lambda$ is the linearizing function that achieves the supremum in \eqref{Holder}. Notice that the inner expression in \eqref{Lmk} may be interpreted as a suitable maximal joint Fourier coefficient of the pair $(f^{\calU}_1,f^{\calU}_2)$ relative to a fixed spatial location $x\in I^k$ hence the \emph{correlative} nature of our approach.

The crux of the matter and the most technically involved reasoning is the proof of the statement: there exists $\tilde{\epsilon}>0$ such that for any $k,m\in\N$ with $k\geq \frac{m}{2}$ the following holds:
\begin{equation} \label{keyCase1}
\textrm{\textbf{L}}_{m,k}(f_1, f_2) \lesssim 2^{-\tilde{\epsilon} m} \left\|f_1 \right\|_{L^2 \left(3 I^k \right)} \left\|f_2 \right\|_{L^2 \left(3 I^k \right)}.
\end{equation}

This is the content of Theorem \ref{20230202thm01}, whose treatment covers most of Section \ref{MaindiagKpositive}. Relying on the interplay between the \emph{local} and the \emph{global} properties of \eqref{Lmk}, the main novelty in this proof is represented by a \emph{bootstrap argument involving the structure of the linearizing function}--see Key Heuristic 4 and Observation \ref{constancy}. This may be regarded as the analog of an induction on scale argument in which the starting  hypothesis is the constancy of $\lambda$ at a fixed suitably small scale and the conclusion is that either one can propagate the constancy of $\lambda$ to a larger interval (scale) or  the phase in \eqref{Lmk} oscillates fast enough to provide local control on the (outer) $x-$integrand of $\textrm{\textbf{L}}_{m,k}(f_1, f_2)$ that further implies \eqref{keyCase1}. This bootstrap algorithm is only run twice in order to reach the situation when $\lambda$ may be assumed constant at the level of the linearizing scale $|I_0^{m, k}|$. Once at this point--see Proposition \ref{mainprop02}--the desired control is obtained at the global level by exploiting the curvature between the spatially discretized $(x,t)-$fibers, which, heuristically speaking, forces a suitable ``discrete mixed second order derivative of the phase"\footnote{Note the parallelism with a H\"ormander type condition.} to be large on average--this is the underlying principle that is exploited along the way from \eqref{20230215eq10} to the final control obtained in \eqref{Jkey3}.

\medskip
\noindent{\textsf{\underline{The case $0\leq k<\frac{m}{2}$}: control over the time-frequency correlation sets}}
\medskip

Before elaborating on the philosophy behind this case, we mention that the entire approach described below necessitates, in reality, a split in two subcases: the case $0\leq k\leq\frac{m}{4}$ that is treated in Section \ref{Sec05} and the case $\frac{m}{4} \le k<\frac{m}{2}$ that is addressed in Section \ref{Sec06}.

With this caveat noted, we start by saying that in the regime $0\leq k<\frac{m}{2}$, the sparse-uniform decomposition requires a more sophisticated procedure due to the fact that the model \eqref{CTFM2} involves a summation over suitable frequency fibers corresponding to the input decomposition(s) in \eqref{freqdec}.

Indeed, for each input $f_j$, we design a carefully chosen tiling--according to a rescaled grid--$\mathcal{P}_j$ of the phase-space plane (depending on the parameters $j,m,k$) that gives the decomposition\footnote{This decomposition is only implicit in our proof but can be easily recovered from the definition of the sparse and uniform sets--see Sections \ref{SU1} and \ref{SU2}.}
\begin{equation}\label{freqdecc1}
f_j=\sum_{P_j\in\mathcal{P}_j} f_j(P_j)\,,
\end{equation}
where $f_j(P_j)$ is adapted to $P_j$ with $P_j=\omega_{P_j}\times I_{P_j}$ having the same scale within each family $\mathcal{P}_j$ and of a fixed area  $\geq 2^{\frac{m}{2}}$.

Now, referring to a spatial fiber as the collection of all tiles having the same spatial interval, we employ the sparse-uniform decomposition of each $f_j$, as follows: the sparse component $f^{\calS}_j$ is formed by those pieces $ f_j(P_j)$ whose $L^2$ mass captures a big chunk from the $L^2$ mass of the whole fiber to which they belong, while the uniform component $f^{\calU}_j$ is the superposition of all the other remaining pieces $ f_j(P_j)$. It is worth mentioning here that in order for this sparse-uniform dichotomy to be effective it is crucial to work with appropriate rescaled versions of the standard (\emph{i.e.}, of area one) Heisenberg tiles and thus go beyond the more common Fourier coefficient size analysis that in this case would run into the very delicate problem of understanding the \emph{structure}\footnote{In the spirit of Freiman's theorem for small sum sets, \cite{Fre66}, \cite{Fre73}.} of the set representing  intermediate size Fourier coefficients aligned along a given spatial fiber and how these interact for different inputs.

Once at this point, similarly with what we have seen for the case $k\geq \frac{m}{2}$, one decomposes the form $\Lambda_m^k(\vec{f})$ into a sparse $\Lambda_m^{k, \calS}(\vec{f})$ and a uniform component $\Lambda_m^{k, \calU}(\vec{f})$, but this time, the accent falls on the former component that covers the situation when \emph{all} the input functions are sparse.

Indeed, in contrast with the first case above, in the current context it is the uniform component $\Lambda_m^{k, \calU}(\vec{f})$ that has a simpler treatment. This is a direct consequence of: 1) the model form \eqref{CTFM2} involves a summation over frequency fibers, and, 2) the information carried by an input of the form $f^{\calU}_j$ is, essentially, uniformly distributed along the frequency locations belonging to a given spatial fiber.

The treatment of the sparse component $\Lambda_m^{k, \calS}(\vec{f})$ relies fundamentally on the interplay between the
time-frequency correlation set defined in \eqref{Keyyy} and the properly defined sparse input functions  $\{f^{\calS}_j\}_j$. After suitable rescalings, exploiting the fact that for each $j$, and each spatial fiber, $f^{\calS}_j$ has essentially only one possible frequency location, one reaches the fundamental step in the proof of this case--that of obtaining a good control over a new/special type of \emph{time-frequency correlation set} that now involves suitable dependencies between measurable functions, thus reducing matters to a \emph{level set analysis}--see Sections \ref{20230306subsec01} and \ref{SparseII}.

\subsection{Motivation and historical background}\label{Motivation}

In this section, we discuss the origin and motivation for our problem and how it relates to the existing literature. The presentation below is tailored to reflect the transformation/evolution of this problem: from its formulation in \cite{Lie15} as a ``curved" prototype for the fundamental open question regarding the boundedness of the trilinear Hilbert transform \eqref{THT} to its present resolution serving as a main driving force in the development of the Rank II LGC method.

\subsubsection{The trilinear Hilbert Transform}\label{THTdes}

For $\vec{f}:=(f_1,f_2,f_3)\in \calS^3(\R)$, define the $\gamma-$Trilinear Hilbert Transform by the expression
\begin{equation}\label{THTgama}
{\bf T_{\gamma}}(\vec{f})(x):= \textrm{p.v.}\,\int_{\R} f_1(x+\gamma_1(t))f_2(x+\gamma_2(t))f_3(x+\gamma_3(t)) \frac{dt}{t}, \quad x \in \R\,,
\end{equation}
where $\gamma=(\gamma_1,\gamma_2,\gamma_3)\,:\R\mapsto\R^3$ is a suitable nondegenerate, smooth curve in $\R^3$.

The problem of understanding the boundedness properties of ${\bf T_{\gamma}}$ in the zero-curvature setting constitutes one of the major open questions in the Fourier analysis area and may be formulated as follows:

\begin{conj}\textsf{[The classical/zero-curvature Trilinear Hilbert Transform]}\label{THTZ}
\medskip

Let $\gamma(t)=(a_1 t,a_2 t, a_3 t)$ be a nondegenerate\footnote{In this context, by nondegenerate we mean $\{a_j\}_{j=1}^3\subset \R\setminus\{0\}$ pairwise distinct.} line in $\R^3$. Then ${\bf T_{\gamma}}$ defined in \eqref{THTgama} obeys
\begin{equation}\label{boundTHT}
\|{\bf T_{\gamma}}(\vec{f})\|_{L^{p_4}}\lesssim_{\vec{p},\gamma} \|\,\vec{f}\,\|_{L^{\vec{p}}}\,,
\end{equation}
for any $\vec{p}:=(p_1, p_2, p_3)\in (1,\infty]^3$ with $\sum_{j=1}^3\frac{1}{p_j}=\frac{1}{p_4}$ and $\frac{1}{2}<p_{4}<\infty$.
\end{conj}

\begin{rem}\label{BoundrangeconjTHT} Notice the difference between the conjectured boundedness range above and the one corresponding to the hybrid, $n=3$ case in Conjecture \ref{Mult}. In fact, on the negative side, the failure of \eqref{boundTHT} outside the range specified in the above conjecture was verified for the specific case $p_1=p_2=p_3=p>1$ as follows: first, in \cite{Dem08}, for some $1<p<\frac{3}{2}$ and the choice $(a_1,a_2,a_3)=(1,2,3)$; then, in \cite{KL12}, for all $1<p<\frac{3}{2}$ and the choice $(a_1,a_2,a_3)=(1,-1,c)$ with $c$ irrational and $(c-1)(c+1)>0$; and finally, in \cite{Kom21}, for some $1<p<\frac{3}{2}$ and all nondegenerate triples $(a_1,a_2,a_3)$.
\end{rem}

The interest in studying the boundedness properties of the trilinear Hilbert transform\footnote{In what follows, for expository reasons we fix $(a_1,a_2,a_3)=(1,2,3)$ and thus identify ${\bf T_{\gamma}}$ with ${\bf T}$ defined in \eqref{THT}.} may be motivated from multiple angles:
\begin{itemize}
\item on the time-frequency/classical harmonic analysis side ${\bf T}$  is part of a natural hierarchy of objects that has as the foundational case the classical Hilbert transform followed, in the bilinear setting, by the bilinear Hilbert transform \eqref{BHTdef} originating in A. Calder\'on's study of the Cauchy transform on Lipschitz curves;

\item on the analytic number theory/additive combinatorics side the study of ${\bf T}$ is connected with problems of counting additive patterns in subsets of $\Z$ in the spirit of Roth's and Szemer\'edi's theorems on arithmetic progressions;

\item on the ergodic theory side the study of ${\bf T}$ relates to the behavior of suitable Furstenberg non-conventional averages as well as to the pointwise convergence of bilinear Birkhoff averages for two commuting measure-preserving transformations.
\end{itemize}

In what follows, given the main focus of our present paper, we will only discuss the first item above; for the remaining two items we refer the interested reader to the extensive Introduction in \cite{Lie20}--see Sections 1.4.6 and 1.4.7 therein and also the related discussion in Section \ref{ErNT} below.

We start our time-frequency perspective saga by letting  $M_{l,b}f(x):=e^{i b x^l}\,f(x)$, $b\in\R$, be the generalized modulation of order $l\in\N$. Given the fact that the trilinear Hilbert transform ${\bf T}$ extends the behavior of the bilinear Hilbert transform
$B$ we obviously have that ${\bf T}$ commutes with translation and dilation symmetries and is also linear modulation invariant as seen from the following relation:
\begin{equation}\label{LinModInv}
\textbf{T}(M_{1,3b_1} f_1, M_{1,3b_2} f_2, M_{1,-b_1-2 b_2} f_3 )=M_{1, 2 b_1+b_2} \textbf{T} (f_1, f_2, f_3)\qquad\forall\;b_1,b_2\in\R \,.
\end{equation}
This immediately suggests that any successful approach in the affirmative to Conjecture \ref{THTZ} must involve modulation invariance techniques, and, likely, a wave-packet analysis that subsumes the methods developed for proving the boundendness of the Carleson operator (\cite{Car66}, \cite{Feff73}, \cite{LT00}), and, in a partial range, of the  bilinear Hilbert transform (\cite{LT97}, \cite{LT99}).

However, this is only one side of the story: on top of the above,  $\textbf{T}$ also obeys a \emph{quadratic modulation invariance}:
\begin{equation}\label{QuadModInv}
\textbf{T}(M_{2,3b} f_1, M_{2,-3b} f_2, M_{2,b} f_3)=M_{2, b} \textbf{T} (f_1, f_2, f_3)\qquad\forall\;b\in\R \,.
\end{equation}

In this way, we are naturally brought to the following closing commentary:
\medskip
\begin{rem}\textsf{[Challenges for bounding $\textbf{T}$]}\label{ChallengesTHT} Below we list what we believe are some major difficulties in approaching Conjecture \ref{THTZ}:
\begin{itemize}
\item \textit{[Additive combinatorics]}: The arithmetic structure of the input-function arguments in the $t-$variable is not only determinative of the modulation invariance properties in \eqref{LinModInv} and \eqref{QuadModInv} but also responsible for subtle connections with additive combinatorics, such as Gowers's work, \cite{Gow98}, on Szemer\'edi's theorem regarding arithmetic progressions on length $\geq 4$, \cite{Sze69}, \cite{Sze75}.\footnote{In particular the concept of Gowers uniformity norm of order 2 seems naturally tailored for evaluating expressions similar to a single scale piece of $\textbf{T}$.} Another manifestation of the connections between the study of $\textbf{T}$-related objects\footnote{Though single scale and positive in nature.} and additive combinatorics appears in \cite{Chr01}.

\item \textit{[Lack of absolute summability for discretized linear wave-packet models]}: If one performs a scale adapted (linear) wave-packet decomposition of each of the input functions of $\textbf{T}$ the resulting model operator is \emph{not absolutely summable}. A similar phenomenon manifests when attempting to show the $L^{p_1}\times L^{p_2}\rightarrow L^{p_3}$ boundedness of the bilinear Hilbert transform $B$ when $\frac{1}{p_1}+\frac{1}{p_2}=\frac{1}{p_3}$ with $p_1,p_2>1$ and $\frac{1}{2}<p_3<\frac{2}{3}$. While the method that we develop in the present paper suggests a possible route towards approaching the lack of the summability issue for the latter problem---see Observation \ref{BHTinv}, it is unlikely that such an approach could be transferred to $\textbf{T}$ without accounting for the topic discussed next.

\item \textit{[Higher order wave-packet analysis]}: The presence of the quadratic modulation invariance property in \eqref{QuadModInv} is an indication that the problem of providing bounds for $\textbf{T}$ belongs to a higher, more difficult class of problems than that addressing operators that only share a linear modulation invariance structure--such as the classical Carleson operator or the bilinear Hilbert transform. Consequently, any time-frequency approach to $\textbf{T}$ is expected to involve a higher order wave-packet discretization. Unfortunately, to date, we have a very limited understanding of how to implement such a higher order wave-packet analysis for treating quadratic modulation invariant operators. The only exception to this is the case of polynomial (Quadratic) Carleson operator--see \cite{Lie09} and \cite{Lie20}.
\end{itemize}
\end{rem}

\subsubsection{Zero versus non-zero curvature dichotomy and the LGC-methodology}\label{LGCphil}

In this section we discuss based on some fundamental examples one of the oldest and most consequential phenomena in harmonic analysis---the zero versus non-zero curvature dichotomy. We start with a very brief heuristic.
\begin{itemize}
\item \textsf{The zero curvature case}: In this situation, the operator under scrutiny, say $T$, obeys some suitable (linear) modulation invariance that, regarded on the frequency side, manifests into a translation invariant property of the multiplier. Many of the most interesting and challenging problems in harmonic analysis---see the discussion below---belong to this case. The methods developed for approaching the zero-curvature situation require the usage of wave-packet analysis and involve a careful analysis of the tiles representing the time-frequency discretization of $T$. The information stored in these tiles is then grouped depending on concepts such as mass or size, in the spirit developed for approaching the Carleson operator and the bilinear Hilbert transform.

\item  \textsf{The non-zero curvature case}: In this situation, the operator $T$ has no modulation invariance structure, with the zero frequency playing a favorite role in the analysis of its multiplier. The latter usually encompasses a highly oscillatory phase whose basic understanding relies on a stationary phase analysis. Usually, some sort of scale type decay is expected, where here the reference to ``scale" must be suitably interpreted and is most frequently related to the height of the phase of the multiplier. Besides stationary phase analysis, the classical approach for a large set of problems belonging to this category involved orthogonality methods such as the $TT^{*}-$method or Littlewood-Paley theory, square functions and vector valued inequalities, etc. In general the problems pertaining to this class are easier in nature than those corresponding to the first item, though there are a number of non-zero curvature problems that are still open and quite difficult---see in particular Conjecture \ref{Mult} above.
\end{itemize}

Now, in order to add more concrete elements/content to the above philosophy while still keeping the exposition very focused, we follow the discussion in \cite{Lie19} while omitting the bibliographical/historical details\footnote{For this, the interested reader is invited to visit the extensive Introduction in \cite{Lie19}.} by providing only a brief outline of three main themes in harmonic analysis:

\begin{itemize}

\item (I) \emph{The linear Hilbert transform along $\gamma$}, with $x=(x_1,x_2)\in\R^2$:
\begin{eqnarray}\label{HVT}
H_{\gamma}(f)(x):= \textrm{p.v.}\int_{\R}f(x_1-t,\,x_2+\gamma(x,t))\,\frac{dt}{t}\,.
\end{eqnarray}
The open problem of providing $L^p$ bounds in the zero curvature case $\gamma(x,t)=a(x)\,t$ with $a$ Lipschitz is the singular integral version of the so called Zygmund conjecture on differentiability along Lipschitz vector fields - one of the deep open questions in harmonic analysis. The non-zero curvature case $\gamma(x,t)=\sum_{j=2}^n a_j(x)\,t^j$, $n\geq2$, with $a_j(\cdot)$ merely measurable functions originates in PDE in the study of constant coefficient parabolic equations.
\medskip

\item (II) \emph{The $\gamma$ - Carleson operator}:

\begin{eqnarray}\label{gC}
 C_{\gamma}f(x):=\textit{p.v.}\,\int_{\R} f(x-t)\, e^{i\,\gamma(x,t)}\,\frac{dt}{t}\,.
\end{eqnarray}

The zero curvature case $\gamma(x,t)=a(x)\,t$ with $a(\cdot)$ measurable represents precisely the classical Carleson operator addressing the old problem of the pointwise convergence of Fourier series.  The non-zero curvature case $\gamma(x,t)=\sum_{j=2}^n a_j(x)\,t^j$ with $a_j(\cdot)$ measurable and $n\geq 2$ is related to the study of oscillatory singular integral operators on the Heisenberg group. Finally, the hybrid case $\gamma(x,t)=\sum_{j=1}^n a_j(x)\,t^j$ (again with $a_j(\cdot)$ measurable, $n\geq 1$) is referred to as the Polynomial Carleson operator. From the three themes discussed here this is the only one for which we have a complete understanding in all three regimes: zero-curvature, non-zero curvature and hybrid. Also, as mentioned in the discussion about the trilinear Hilbert transform, the latter regime provides the only known situation to date for which we are able to provide $L^p$ bounds for a singular integral operator having linear but also higher order modulation invariance.
\medskip

\item (III) \emph{The bilinear Hilbert transform along $\gamma$}:

\begin{eqnarray}\label{defbHVT}
 B_{\gamma}(f,g)(x):= \textrm{p.v.}\int_{\R}f(x-t)\:g(x+\gamma(t))\:\frac{dt}{t}\,.
\end{eqnarray}

The zero curvature case $\gamma(t)= a t$ with $a\in\R\setminus\{-1,0\}$ represents the classical bilinear Hilbert transform. The study of its $L^p$ boundedness properties was proposed by A. Calder\'on in his search for understanding the behavior of the Cauchy integral along Lipschitz curves and was solved by M. Lacey and C. Thiele within the range $\frac{1}{p}+\frac{1}{q}=\frac{1}{r}$ for $1<p,q\leq \infty$ with $\frac{2}{3}<r<\infty$. As mentioned earlier, the case $\frac{1}{2}<r<\frac{2}{3}$ is still open due to the lack of absolute summability of the wave-packet discretized model.
The purely non-zero curvature case $\gamma(t)= \sum_{j=2}^n a_j t^j$, $n\geq 2$ is now completely understood and was motivated, in part, by analogous models studied in ergodic theory and number theory.
\end{itemize}

With this panoramic view on some of the pre-existent methods and type of problems addressing the zero/non-zero curvature realm  we are ready to outline the origin and evolution of the recently developed LGC method with its two layers.

The main motivation for this method was to create a unified perspective and set of tools for approaching both zero and non-zero curvature problems. In order to be able to achieve this the first key requirement was the integration of wave-packet analysis/time-frequency discretization within the non-zero curvature setting. This naturally required a preliminary \emph{linearization} process applied either on the multiplier side or on the spatial side followed by a \emph{Gabor frame decomposition} of the input functions. The resulting time-frequency discretized model operator presented then the following key feature: as a direct consequence of the curvature of the original operator one is able to relate via some algebraic identities the spatial and frequency parameters appearing in the above discretization, thus creating the so-called \emph{time-frequency correlations}. Understanding the latter proves to be an essential task in obtaining the desired global control on the operator under discussion.

This approach was developed in \cite{Lie19} and provided a first unified treatment for both the singular and the maximal versions of the themes (I) and (II) above in the non-zero curvature setting $\gamma(x,t)=\sum_{j=1}^n a_j(x) t^{\alpha_j}$ with $a_j(\cdot)$ measurable and $\alpha_j\in\R\setminus\{0,1\}$, $1\leq j\leq n$, distinct, where here for the first theme we imposed the extra restriction $a_j(x)=a_j(x_1)$.

The next step accomplished in \cite{GL22} came as a confirmation that the LGC approach provides a natural environment in which zero and non-zero curvature techniques can be integrated; indeed, the authors therein prove several results, among which we mention
\begin{itemize}
\item the $L^2$-boundedness of \eqref{HVT} and \eqref{gC} in the hybrid case $\gamma(x,t)=a_1(x)t+\sum_{j=2}^n a_j(x) t^{\alpha_j}$ with $a_j(\cdot)$ measurable and $\alpha_j\in \R\setminus\{0,1,2\}$, $2\leq j\leq n$, distinct, where for the first theme we operate under the same assumption $a_j(x)=a_j(x_1)$;
\item  the $L^p\times L^q\rightarrow L^r$ boundedness of \eqref{defbHVT} with $\frac{1}{p}+\frac{1}{q}=\frac{1}{r}$, $p,q>1$ and $\frac{2}{3}<r<\infty$ in the hybrid case $\gamma(t)=at +b t^{\alpha}$ for $a\in\R\setminus\{1\}$, $b\in\R$ and $\alpha\in(0,\infty)\setminus\{1,2\}$;
\item  a short, simple(r), self-contained proof for the key ingredient used in the boundedness of the curved triangular Hilbert transform in \cite{CDR20};
\item a first rigorous explanation of why in the case of quadratic modulation invariance the higher order wave-packet analysis introduced in \cite{Lie09} and \cite{Lie20} is in fact a necessity and thus why a linearization argument followed by a standard (linear) wave-packet analysis, as in the case of LGC, does not apply; this is also the deep reason for which at the first two items above one is required to eliminate the value $2$ from the range of the admissible $\alpha-$exponents.
\end{itemize}

Now, while all of these more recent results in \cite{Lie19} and \cite{GL22} demonstrated the versatility of this new method, they all had in common the fact that after the linearization and adapted input Gabor decomposition steps the resulting discretized wave-packet models become \emph{absolutely summable} - a feature that if incorporated in the above method is referred to as ``\emph{Rank I LGC}".

This latter fact however changes when confronting the next hierarchy of problems: the Bilinear Hilbert--Carleson operator, \cite{BBLV21}, and the curved and hybrid trilinear Hilbert discussed here and in \cite{BBL23}, respectively. Indeed, in \cite{BBLV21}, the authors therein studied the boundedness properties of the Bilinear Hilbert--Carleson operator
\beq\label{Top}
BC^{\alpha}(f,g)(x):= \sup_{\lambda\in \R}  \Big|\int f(x-t)\, g(x+t)\, e^{i\lambda t^{\alpha}} \, \frac{dt}{t} \Big|
\eeq
in the non-resonant case $\alpha\in(0,\,\infty)\setminus\{1,\,2\}$, which is a hybrid operator in the sense that it enjoys both modulation invariant structure inherited from the bilinear Hilbert transform \eqref{BHTdef} and non-zero curvature features derived from the shape of the oscillatory phase. In this situation the application of the standard (Rank I) LGC method above runs into the key difficulty of the lack of absolute summability of the resulting discretized wave-packet model.

The underlying philosophy of the Rank II LGC approach developed in our current paper builds on the insight in \cite{BBLV21} and makes the crucial observation that in order to address this lack of absolute summability issue, it is necessary to treat the input functions $(f,g)$ as a single entity, which in turn entails the associated analysis of the resulting ``joint" (maximal) Fourier coefficient. However, unlike in the setting of the Bilinear Hilbert--Carleson operator---where the linearizing function arises \emph{inherently} in the context of a maximal oscillatory integral operator---the curved trilinear Hilbert transform has \emph{no immediately apparent maximal structure}, and accordingly the aptitude of this insight in the present setting is surprising.

On top of this, there are also numerous additional major difficulties that must be overcome in our present approach, including, to name only a few: the quadrilinear nature of the dualized form that requires an interdependent analysis of the four input functions (see the $1^{st}$ and $2^{nd}$ stages in the model discretization part of Section \ref{Keyideas}); the necessity of a structural analysis for the linearizing function $\lambda$ in the setting of the iterative constancy propagation approach; and the symbiosis between the time-frequency foliation algorithm and the analysis of the time-frequency correlation set.

\subsubsection{Ergodic and number theory perspectives}\label{ErNT}

In this section we summarize some natural connections between the main topic of our paper---see Theorem \ref{MAINthm} but also Observation \ref{FW} and Conjecture \ref{Mult}---and problems in ergodic theory and number theory/additive combinatorics.

We start our brief discussion within the realm of ergodic theory: Let $(X, \mu, T)$ be a $(\sigma-)$finite measure preserving system with $T:\:X\,\rightarrow\,X$ an invertible bimeasurable map. For any $k,\,N\in\N$, any polynomials\footnote{Here $\Z[n]$ designates the space of all polynomials $P(n)$ in one indeterminate $n$ with
integer coefficients.} $P_1,\ldots, P_k\in\Z[n]$ and any measurable functions $f_1,\ldots, f_k\in L^{0}(X)$ define the associated non-conventional polynomial multiple ergodic average as

\begin{eqnarray}\label{Nonconvmultergodav}
A_{N}^{P_1,\ldots,\,P_k}(f_1,\ldots,\,f_k)(x):=\frac{1}{N} \sum_{n=1}^N f_1(T^{P_1(n)}x)\ldots f_k(T^{P_k(n)}x)\,.
\end{eqnarray}

These averages play an important role in detecting recurrence point patterns and their study has a long history, as (only partly) reflected in what follows:

\begin{itemize}
\item the base case $k=1$, $P_1(n)=n$: the norm convergence of \eqref{Nonconvmultergodav} in $L^2(X)$ is the classical von Neumann mean ergodic theorem, \cite{Neu32}, while the pointwise convergence counterpart result is Birkhoff’s ergodic theorem in \cite{Bir31};

\item the base case $k=1$, general polynomial: the question on the pointwise convergence of \eqref{Nonconvmultergodav} for an $L^{\infty}$ input was proposed independently in the early 1980s by A. Bellow and H. Furstenberg and was positively solved by J. Bourgain in \cite{Bou88b}, \cite{Bou88a}, \cite{Bou89};

\item the case $k=2$, $P_1(n)=a\,n$ and $P_2(n)=b\,n$, $a,b\in\Z$: the norm convergence in $L^2(X)$ was settled by H. Furstenberg in \cite{Fur77a}, while the much harder pointwise convergence result was established by J. Bourgain in \cite{Bou90}, see also \cite{Dem07};

\item the case $k=2$, $P_1(n)=n$ and $P_2(n)$ of degree at least two: the pointwise convergence result was recently obtained in \cite{KMT22};

\item the case $k=3$, $P_i(n)=a_i\,n$ with $\{a_i\}_{i=1}^3\subset\Z$: the norm convergence in $L^2(X)$ was proved by  H. Furstenberg and B. Weiss in \cite{FW96}.

\item the general case $k\in\N$, linear polynomials, \textit{i.e.} $P_i(n)=a_i\,n$ with $\{a_i\}_{i=1}^k\subset\Z$: B. Host and B. Kra, \cite{HK05} and  T. Ziegler, \cite{Zieg07}, established norm convergence in $L^2(X)$;

\item the general case $k\in\N$ for general polynomials: norm convergence in $L^2(X)$ was settled by A. Leibman in \cite{Leib05a};

\item for more connections between the study of \eqref{Nonconvmultergodav} and other topics such as Furstenberg’s ergodic-theoretic proof \cite{Fur77a} of Szemerédi’s theorem \cite{Sze75} and Furstenberg--Bergelson--
Leibman conjecture one is invited to consult the Introduction in \cite{KMT22}.
\end{itemize}

With these being said it is worth noticing that the key question concerning the pointwise convergence of \eqref{Nonconvmultergodav} for $k\geq 3$ is still completely open. In this context, our main result in Theorem \ref{MAINthm} may be regarded as the natural continuous singular integral analogue for the pointwise convergence problem in the case $k=3$ with $P_i(t)=a_i\,t^i$, $a_i\in\Z$ and $i\in\{1,2,3\}$.

We now move to the number theory/additive combinatorics realm. The generic problem of interest for us may be heuristically formulated as follows: let as before $k\in \N$ and $P_1,\ldots,\,P_k$ polynomials (with no constant terms) in $K[y]$ where here one may take $K$ as $\Z$, a finite field $F_q$ or $\R$. We then want to identify qualitative and/or quantitative conditions on a set $S\subset K$ that guarantee the existence of and/or lower bounds for polynomial progressions $x,\,x+P_1(y),\ldots, x+P_k(y)$ lying within the set $S$. This generic problem has a very rich history and again, without aiming to be exhaustive, we mention some of the most relevant results:

\begin{itemize}

\item  in the linear polynomial, $K=\Z$ setting, the case $k=2$  was first studied by K. Roth, \cite{Ro53}, while the case $k\geq 3$ was famously addressed by E. Szemer\'edi in \cite{Sze75} and later T. Gowers in \cite{Gow01};

\item for the general polynomial, $K=\Z$ setting and general $k\geq 2$ qualitative results were obtained by Bergelson, Furstenberg and Weiss in \cite{BFW86} and  Bergelson and Leibman in \cite{BL96};

\item for $k=2$ and the real setting $K=\R$, Bourgain, \cite{Bou88c}, studied the problem for $P_1(y)=y$ and $P_2(y)=y^d$, $d\in\N, d\geq 2$; for more general polynomials $P_1$ and $P_2$ similar results were obtained in \cite{CGL21}, \cite{DGR19};

\item in the setting $K=F_q$ the easier case $k=2$ was first addressed in \cite{BC17} and more recently in \cite{Pel18}, \cite{DLS20};

\item in the same setting $K=F_q$ and general $k$, S. Peluse \cite{Pel19} obtained via additive combinatorics techniques quantitative bounds under the hypothesis that $\{P_i\}_{i=1}^k\subset \Z[y]$ are linearly independent over $\Q$;

\item the situation $K=\Z$ for general $k$ and $\{P_i\}_{i=1}^k\subset \Z[y]$ of distinct degrees was studied in \cite{Pel20} building upon the work in \cite{PP19};

\item finally, the topological field situation $K=\R$ for general $k$ was recently addressed in \cite{KBPW22} under the hypothesis that $\{P_i\}_{i=1}^k\subset \R[y]$ have distinct degrees.

\end{itemize}

It is now conceivable that the techniques developed in this paper together with those in the forthcoming \cite{BBL23} might offer a different avenue in treating the problem above in the $K=\R$ and $k=3$ setting and possibly under more relaxed hypotheses.

\subsection*{Acknowledgments}
The second author was supported by the NSF grant DMS-1900801 and by the Simons Foundation grant no. 682242. Also the second author wishes to thank the Department of Mathematics at Princeton University for hosting him during the Spring 2022 semester as part of the Simons fellowship award mentioned above.

\section{Basic reductions}\label{basicred}

In this section we will focus on a preliminary discretization of our main object of interest represented by the quadrilinear form
\begin{equation}\label{DefL}
\Lambda(\vec{f})=\int_{\R}\int_{\R} f_1(x-t)f_2(x+t^2)f_3(x+t^3)f_4(x)\frac{dtdx}{t}\,.
\end{equation}

\subsection{Discretization of the multiplier}

Rewriting \eqref{DefL} in the form
\begin{equation}\label{DefL1}
\Lambda(\vec{f})=\int_{\R^3} \frakm(\xi, \eta, \tau)\,\widehat{f_1}(\xi) \widehat{f_2}(\eta) \widehat{f_3}(\tau) \widehat{f_4}(\xi+\eta+\tau) d\xi d\eta d\tau\,,
\end{equation}
with
\begin{equation}
\frakm(\xi, \eta, \tau):=\int_\R e^{-i \xi t} e^{i \eta t^2} e^{i \tau t^3}  \frac{dt}{t}\,,
\end{equation}
we proceed with two types of decompositions:

\subsubsection{A physical discretization}

In the usual fashion, dictated by the singularities at zero and infinity of the kernel $\frac{1}{t}$ we perform a Whitney decomposition relative to the origin, that is, taking $\rho \in C_0^\infty (\R)$ an odd function with $\textrm{supp}\, \rho \subset \left[-4, 4 \right] \backslash \left[ -\frac{1}{4}, \frac{1}{4} \right]$ we write
$$
\frac{1}{t}=\sum_{k \in \Z} \rho_{k}(t),
$$
where $\rho_k(t):=2^k \rho \left( 2^k t\right)$. This yields a first decomposition
\begin{equation}\label{spatialdecL}
\Lambda=\sum_{k \in \Z} \Lambda^k,
\end{equation}
where $\Lambda_k$ has as a multiplier
\begin{equation}
\frakm^k(\xi, \eta, \tau):=\int_\R e^{-i \xi t} e^{i \eta t^2} e^{i \tau t^3} \rho_{k}(t) dt=\int_\R e^{-i \frac{\xi}{2^k} t} e^{i \frac{\eta}{2^{2k}} t^2} e^{i \frac{\tau}{2^{3k}} t^3} \rho(t)dt.
\end{equation}

\subsubsection{A frequency discretization}

We now move our attention on the frequency side. Let us denote the phase function by
\begin{equation} \label{20230307eq01}
\varphi_{\xi, \eta, \tau}^k(t):=-\frac{\xi}{2^k}t+\frac{\eta}{2^{2k}} t^2+\frac{\tau}{2^{3k}}t^3.
\end{equation}
Since the derivative of the phase function plays--via (non)stationary phase principle--a fundamental role in the behavior of our operator, it is natural to analyze the relative size of the coefficients of \eqref{20230307eq01}. More precisely, we let $\phi$ be an even function with $\textrm{supp}\,\phi \subseteq \left[-4, 4 \right] \backslash \left[-\frac{1}{4}, \frac{1}{4} \right]$ which satisfies the partition of unity
$$
1=\sum_{j, l, m \in \Z} \phi \left(\frac{\xi}{2^{j+k}} \right) \phi \left(\frac{\eta}{2^{l+2k}} \right) \phi\left(\frac{\tau}{2^{m+3k}} \right).
$$
This invites us to further consider for every $k, j, l, m \in \Z$
$$
\frakm_{j, l, m}^k(\xi, \eta, \tau):=\left( \int_{\R} e^{i\varphi_{\xi, \eta, \tau}^k(t)} \rho(t) dt \right) \phi \left(\frac{\xi}{2^{j+k}} \right) \phi \left(\frac{\eta}{2^{l+2k}} \right) \phi\left(\frac{\tau}{2^{m+3k}} \right)
$$
and thus to correspondingly decompose $\Lambda^k$ as
\begin{equation}\label{freqdeconmpL}
\Lambda^k=\Lambda^k_{j, l, m},
\end{equation}
with $\frakm_{j, l, m}^k$ being the multiplier of $\Lambda^k_{j, l, m}$.

Finally, we decompose
\begin{equation}\label{Llohi}
\Lambda=\Lambda^{Hi}\,+\,\Lambda^{Lo}=\Lambda^{Hi}_{+}\,+\,\Lambda^{Hi}_{-}\,+\,\Lambda^{Lo}\,,\qquad\textrm{with}
\end{equation}
\begin{itemize}
\item the \emph{low oscillatory}\footnote{Here we are making a slight notational abuse: while strictly speaking, in the regime $(j, l, m) \in \Z^3\backslash \N^3$ the phase of the multiplier can still oscillate, this can only happen in at most two out of the three directions $\xi, \eta$ and $\tau$.} component defined by
$$\Lambda^{Lo}:=\sum_{k \in \Z}\:\: \sum_{(j, l, m) \in \Z^3 \backslash \N^3} \Lambda^k_{j, l, m}\,;$$
\item the \emph{high oscillatory} component defined by
$$\Lambda^{Hi}:=\sum_{k \in \Z}\:\: \sum_{(j, l, m) \in \N^3} \Lambda^k_{j, l, m}\,, \qquad
\textrm{and}\footnote{Here, by convention, we set $\Z_{-}:=\Z\setminus\Z_{+}$ with $\Z_{+}:=\N\cup\{0\}$.}
\qquad \Lambda^{Hi}_{\pm}:=\sum_{k \in \Z_{\pm}}\:\: \sum_{(j, l, m) \in \N^3} \Lambda^k_{j, l, m}\,.$$
\end{itemize}

\subsection{Phase analysis for the high oscillatory component $\Lambda^{Hi}$}

This section is aimed to better understand the phase behavior corresponding to the multiplier of $\Lambda^{Hi}$.

Thus, we start by computing the derivatives of the phase function $\varphi_{\xi, \eta, \tau}^k$: for each $k \in \Z$,
\begin{equation} \label{20230307eq02}
\frac{d}{dt} \varphi_{\xi, \eta, \tau}^k (t)=-\frac{\xi}{2^k}+\frac{2\eta}{2^{2k}}t+\frac{3\tau}{2^{3k}} t^2,
\end{equation}
\begin{equation} \label{20230307eq03}
\frac{d^2}{dt^2} \varphi_{\xi, \eta, \tau}^k(t)=\frac{2\eta}{2^{2k}}+\frac{6\tau}{2^{3k}} t
\end{equation}
and
$$
\frac{d^3}{dt^3} \varphi_{\xi, \eta, \tau}^k(t)=\frac{6\tau}{2^{3k}}.
$$
Now we notice that \eqref{20230307eq02} represents a second degree polynomial whose reduced discriminant is
\begin{equation} \label{20230307eq04}
\Delta:=\frac{\eta^2+3\xi \tau}{2^{4k}}
\end{equation}
Next, from \eqref{20230307eq03}, we set $t_0$ be the unique root of $\frac{d^2}{dt^2} \varphi_{\xi, \eta, \tau}^k$, that is
\begin{equation} \label{20230307eq05}
t_0:=-\frac{2^k \eta}{3\tau} \quad \textrm{and} \quad \frac{d^2}{dt^2} \varphi_{\xi, \eta, \tau}^k(t_0)=0.
\end{equation}
Using \eqref{20230307eq04} and \eqref{20230307eq05}, we can rewrite $\frac{d}{d t} \varphi_{\xi, \eta, \tau}^k$ as
\begin{equation} \label{deriv}
\frac{d}{dt} \varphi_{\xi, \eta, \tau}^k(t)=\frac{3\tau}{2^{3k}} \cdot \left[ \left(t-t_0 \right)^2- \left(\frac{2^{3k}}{3\tau} \right)^2 \Delta \right]
\end{equation}
and hence
\begin{equation*}
\frac{d}{dt} \varphi_{\xi, \eta, \tau}^k(t_0)=-\frac{2^{3k}}{3\tau} \Delta.
\end{equation*}
In the case when $\Delta \ge 0$, we let $t_1$ and $t_2$ be the roots of the equation $\frac{d}{dt} \varphi_{\xi, \eta, \tau}^k(t)=0$, that is
\begin{equation} \label{20230307eq11}
t_1=-\frac{2^k \eta}{3\tau}+\frac{2^{3k} \sqrt{\Delta}}{3\tau}  \quad \textrm{and} \quad t_2=-\frac{2^k \eta}{3\tau}-\frac{2^{3k} \sqrt{\Delta}}{3\tau} .
\end{equation}
Moreover, we have
$$
\left|t_1-t_0 \right|=\left|t_2-t_0 \right|=\frac{2^{3k} \sqrt{\Delta}}{3|\tau|}.
$$

\subsection{Decomposing $\Lambda^{Hi}$ into a diagonal $\Lambda^{D}$ and an off-diagonal $\Lambda^{\not{D}}$ component}

As a consequence of the discussion in the previous two sections we further decompose
\begin{equation}\label{LlohiDN}
\Lambda^{Hi}=\Lambda^{D}\,+\,\Lambda^{\not{D}}\qquad\textrm{with}
\end{equation}
\begin{itemize}
\item $\Lambda^{D}$ the \emph{diagonal} component having the frequency index range obey
\begin{equation} \label{diag}
\max\{j, l, m\}-\min\{j, l, m\} \le 300\,;
\end{equation}
\item $\Lambda^{\not{D}}$ the \emph{off-diagonal} component  having the frequency index range obey
\begin{equation} \label{20230308eq01}
\max\{j, l, m\}-\min\{j, l, m\} > 300\,.
\end{equation}
\end{itemize}
Obvious extensions apply for the definitions of $\Lambda^{D}_{\pm}$ and $\Lambda^{\not{D}}_{\pm}$.

We further subdivide the off-diagonal component as
\begin{equation}\label{LodSNS}
\Lambda^{\not{D}}=\Lambda^{\not{D},S}\,+\,\Lambda^{\not{D},NS}\qquad\textrm{with}
\end{equation}
\begin{itemize}
\item $\Lambda^{\not{D},S}$ the \emph{stationary} off-diagonal component corresponding to the stationary phase regime, \emph{i.e.}, when at least one of the roots $\left\{t_1, t_2 \right\}$ has the absolute value of size $\approx 1$.

\item $\Lambda^{\not{D},NS}$ the  \emph{non-stationary} off-diagonal component corresponding to the non-stationary phase regime, \emph{i.e.},  when the first item above is not satisfied.
\end{itemize}

Next, we characterize $\Lambda^{\not{D},S}$ in terms of the frequency index range:

 \begin{lem} \label{20230308lem01}
The stationary off-diagonal regime corresponds precisely to one of the below situations\footnote{For any $j, l \in \Z$, $j \cong l$ means $|j-l|<20$.}:
\begin{enumerate}
    \item [(1)] $j \cong l$ and $j>m+100$;
    \item [(2)] $l \cong m$ and $l>j+100$;
    \item [(3)] $m \cong j$ and $m>l+100$.
\end{enumerate}
 \end{lem}
\begin{proof}
Assuming without loss of generality that $\eta>0$ in \eqref{20230307eq11}, we deduce that
\begin{equation}\label{t12}
|t_1| \sim \min\{2^{j-l},2^{\frac{j-m}{2}}\}\qquad\textrm{and}\qquad |t_2| \sim \max\{2^{l-m},2^{\frac{j-m}{2}}\}\,.
\end{equation}

We now consider two distinct cases:
\medskip

\noindent\textsf{Case i)} $j+m\geq 2l$. In this situation, from \eqref{t12} we have $\left|t_{1, 2} \right| \sim 2^{\frac{j-m}{2}}$, and thus we must have $j \cong m$. Adding now condition \eqref{20230308eq01} we see that this case corresponds to (3) above.
\medskip

\medskip
\noindent\textsf{Case ii)} $j+m<2l$. In this situation, from \eqref{t12} we have $|t_1| \sim 2^{j-l}$. Therefore, if $|t_1| \sim 1$, we have $j \cong l$, which together with \eqref{20230308eq01} gives (1) above. Otherwise, if $|t_1| \not\sim 1$, then $|t_2| \sim 2^{l-m} \sim 1$, or equivalently, $l \cong m$, which together with \eqref{20230308eq01} gives (2) above.
\end{proof}

\subsection{Treatment of the non-stationary off-diagonal term $\Lambda^{\not{D},NS}$}

In this section we provide a quick treatment for $\Lambda^{\not{D},NS}$. Indeed, we have the following:

\begin{lem}
For any $(j, l, m)\in\N^3$ within the non-stationary off-diagonal regime, \textit{i.e.}, satisfying \eqref{20230308eq01} but not the three cases described in Lemma \ref{20230308lem01}, one has for any $|t| \sim 1$ that
\begin{equation} \label{20230308eq10}
\left| \frac{d}{dt} \varphi_{\xi, \eta, \tau}^k(t) \right| \gtrsim 2^j+2^l+2^m\,.
\end{equation}
\end{lem}

\begin{proof}
We proceed with a case discussion:
\medskip

\noindent\textsf{Case I}: $\Delta\geq 0\qquad$   In this situation, from \eqref{deriv} and \eqref{20230307eq11}, for $|t| \sim 1$ we have

$$
\left| \frac{d}{dt} \varphi_{\xi, \eta, \tau}^k(t) \right| \sim \left| \frac{\tau}{2^{3k}} \right| \left|t-t_1 \right| \left|t-t_2 \right|\,.
$$

Assuming again without loss of generality that $\eta>0$ we have that \eqref{t12} holds and also that $|t_1|\leq |t_2|$.
Based on these observations we now have three possible subcases:
\begin{enumerate}
\item [$\bullet$] If $|t_1|\leq |t_2| \ll 1$ then $\left| \frac{d}{dt} \varphi_{\xi, \eta, \tau}^k(t) \right| \sim \left|\frac{\tau}{2^{3k}} \right| \sim 2^m$ and $\max\{j,l\}<m$.

\item [$\bullet$] If $|t_1| \ll 1\ll |t_2|$ then $\left| \frac{d}{dt} \varphi_{\xi, \eta, \tau}^k(t) \right| \sim \left|\frac{\tau}{2^{3k}} \right| \left|t_2 \right| \gtrsim 2^m \cdot 2^{l-m}=2^l$ and $\max\{j,m\}<l$.

\item [$\bullet$] If $1 \ll |t_1|\leq |t_2|$ then $\left| \frac{d}{dt} \varphi_{\xi, \eta, \tau}^k(t) \right| \sim \left|\frac{\tau}{2^{3k}} \right| \left|t_1 \right| \left|t_2 \right|\gtrsim 2^m \min\{2^{j-l},2^{\frac{j-m}{2}}\}(2^{l-m}+ 2^{\frac{j-m}{2}})\gtrsim 2^{j}$ and $\max\{l,m\}<j$.
\end{enumerate}

\medskip

\noindent\textsf{Case II}: $\Delta< 0\qquad$ In this situation we must have

\begin{equation} \label{delneg}
j+m>2l-20\,.
\end{equation}

Using now relation \eqref{deriv} we have three subcases:
\begin{enumerate}
\item [$\bullet$] If $|l-m|<10$ then combining \eqref{20230308eq01} and \eqref{delneg} we must have  $j+m>2l+100$ and hence
$\left|\frac{d}{dt} \varphi_{\xi, \eta, \tau}^k(t)\right|\gtrsim \frac{2^{3k}}{|\tau|}  |\Delta|\gtrsim 2^j$.

\item [$\bullet$] If $|l-m|\geq 10$ and $2l-20<j+m<2l+20$ then $\left|\frac{d}{dt} \varphi_{\xi, \eta, \tau}^k(t)\right|\gtrsim 2^m\max\{1, 2^{2(l-m)}\}$.

\item [$\bullet$] If $|l-m|\geq 10$ and $j+m\geq 2l+20$ then $\left| \frac{d}{dt} \varphi_{\xi, \eta, \tau}^k(t) \right| \gtrsim 2^m\max\{1, 2^{2(l-m)}\}+2^j$.
\end{enumerate}

It is now immediate to see that any of the above subcases satisfy the desired estimate \eqref{20230308eq10}.
\end{proof}

We are now ready to state and prove the following:

\begin{lem}\label{offdiagnonst}
With the previous notations, for any
\begin{equation}\label{defBbar}
\vec{p}\in\overline{\textbf{B}}:=\left\{(p_1,p_2,p_3,p_4)\,\Big|\,\sum_{j=1}^4\frac{1}{p_j}=1\:\:\textrm{and}\:\:1<p_j\le \infty\right\}\,,
\end{equation}
we have:
\begin{equation}
\left|\Lambda^{\not{D},NS}(\vec{f})\right|\lesssim_{\vec{p}} \|\vec{f}\|_{L^{\vec{p}}}\,.
\end{equation}
\end{lem}

\begin{proof}
Fix $k\in\Z$ and $(j,l,m)\in\N^3$ in the non-stationary off-diagonal regime. Using integration by parts, we see that
$$
\frakm_{j, l, m}^k(\xi, \eta, \tau)=\frakA_{j, l, m}^k(\xi, \eta, \tau)+\frakB_{j, l, m}^k(\xi, \eta, \tau),
$$
where
$$
\frakA_{j, l, m}^k(\xi, \eta, \tau):=i\left(\int_\R e^{i\varphi_{\xi, \eta, \tau}^k(t)} \frac{\rho'(t)}{\frac{d}{dt} \varphi^k_{\xi, \eta, \tau}(t))} dt \right) \phi_{j, l, m}^k(\xi, \eta, \tau)
$$
and
$$
\frakB_{j, l, m}^k(\xi, \eta, \tau):=\frac{1}{i}\left(\int_\R e^{i\varphi_{\xi, \eta, \tau}^k(t)} \frac{\rho(t) \frac{d^2}{dt^2} \varphi_{\xi, \eta, \tau}^k(t)}{\left(\frac{d}{dt} \varphi_{\xi, \eta, \tau}^k(t)\right)^2} dt \right) \phi_{j, l, m}^k(\xi, \eta, \tau).
$$
This further allows us to write
$$
\Lambda^k_{j, l, m}=\Lambda^{k;  {\frakA}}_{j, l, m}\,+\,\Lambda^{k;  {\frakB}}_{j, l, m},
$$
where $\Lambda^{k;  {\frakA}}_{j, l, m}$ and $\Lambda^{k;  {\frakB}}_{j, l, m}$ have the multipliers $\frakA_{j, l, m}^k$ and $\frakB_{j, l, m}^k$, respectively.

Let us denote $n:=\max\{j, l, m\}$. Recall that for $|t| \simeq 1$
\begin{equation}
\left|\frac{d^2}{dt^2} \varphi_{\xi, \eta, \tau}^k(t)\right|
=\frac{6\tau}{2^{3k}} \cdot \left| t-t_0 \right| \lesssim 2^m \cdot \max\{2^{l-m}, 1 \}
\leq  2^n \lesssim \left|\frac{d}{dt}\varphi_{\xi, \eta, \tau}^k(t) \right|,
\end{equation}
Hence, we have
\begin{equation} \label{20211101eq01}
\left|\frac{\rho'(t)}{\frac{d}{dt} \varphi_{\xi, \eta, \tau}^k(t)} \right|, \quad \left| \frac{\rho(t) \frac{d^2}{dt^2} \varphi_{\xi, \eta, \tau}^k(t)}{\left(\frac{d}{dt} \varphi_{\xi, \eta, \tau}^k(t)\right)^2}  \right| \lesssim \frac{1}{2^n}
\end{equation}
and therefore it suffices for us to analyse the term $\Lambda^{k;  {\frakA}}_{j, l, m}$. Applying a Taylor series argument we have\footnote{We ignore here the error terms derived from higher order terms in the Taylor approximation since these error terms have a similar structure with the main term displayed in our estimate.}
\begin{eqnarray}\label{prepsq}
|\Lambda^{k;  {\frakA}}_{j, l, m}(\vec{f})|
&=&\left|\int_{\R}f_4(x)\int_{\R^3} \widehat{f_1}(\xi) \widehat{f_2}(\eta) \widehat{f_3}(\tau) \left[\int_{\R} e^{i \left(-\frac{\xi t}{2^k}+\frac{\eta t^2}{2^{2k}}+\frac{\tau t^3}{2^{3k}} \right)} \frac{\frac{d}{dt} \rho(t)}{\frac{d}{dt} \varphi_{\xi, \eta, \tau}^k(t)} dt \right] \phi_{j, l, m}^k(\xi, \eta, \tau) e^{i(\xi+\eta+\tau)x}\,d\xi d\eta d\tau dx\right| \nonumber\\
&\lesssim&\frac{1}{2^n}\,\int_{\R\times\{\frac{1}{4}\leq |t|\leq 4\}} \left|\left( \widehat{f_1} (\cdot) \phi \left(\frac{\cdot}{2^{j+k}} \right) \right)^{\vee} \left(x-\frac{t}{2^k} \right)\right| \left|\left( \widehat{f_2} (\cdot) \phi \left(\frac{\cdot}{2^{l+2k}} \right) \right)^{\vee} \left(x+\frac{t^2}{2^{2k}} \right) \right| \nonumber\\
&& \quad \quad \quad \cdot \left|\left( \widehat{f_3} (\cdot) \phi \left(\frac{\cdot}{2^{m+3k}} \right) \right)^{\vee} \left(x+\frac{t^3}{2^{3k}} \right)\right| \left|\left( \widehat{f_4} (\cdot) \chi_{A(k,j,l,m)}(\cdot)\right)^{\vee} \left(x\right)\right| dt dx\,,
\end{eqnarray}
where $A(k,j,l,m):=\textrm{supp}\,\phi\left(\frac{\cdot}{2^{j+k}} \right)\,+\,\textrm{supp}\,\phi\left(\frac{\cdot}{2^{l+2k}} \right)\,+\,\textrm{supp}\,\phi\left(\frac{\cdot}{2^{m+3k}} \right)$.
\medskip

Once at this point, we recall the following useful result proved in \cite[Lemma 4.8]{Lie19} and \cite[Lemma 3.3, Lemma 3.4]{GL20}:

\begin{lem}\label{sqfunct}
Fix $t\in\R$ and for $\phi,\varphi \in C_0^\infty(\R)$ with $\phi(0)=0$ define
\begin{itemize}
\item the \emph{$t$-shifted square function} $S_t^{\phi}$ by
\begin{equation} \label{sq}
S_t^\phi f(x):=\left(\sum_{k \in \Z} \left| \left( f* \left(2^k \check{\phi}(2^k \cdot) \right) \right) \left(x-\frac{t}{2^k} \right)\right|^2 \right)^{\frac{1}{2}};
\end{equation}
\item the \emph{$t$-shifted Hardy-Littlewood maximal function} $M_t^{\varphi}$ by
\begin{equation} \label{maxsh}
M_t^\varphi f(x):=\sup_{k \in \Z} \left| \left( f* \left(2^k \check{\varphi}(2^k \cdot) \right) \right) \left(x-\frac{t}{2^k} \right)\right|\,;
\end{equation}
\end{itemize}
Then, setting $r^{*}=\min\{r,r'\}$ with $r'$ the H\"older conjugate of $r$, we have
\begin{itemize}
\item for any $1<r<\infty$ one has
\begin{equation} \label{sq1}
\|S_t^\phi f\|_{L^r}\lesssim_{r,\phi} \left(\log (10+|t|)\right)^{\frac{2}{r^{*}}-1}\,\|f\|_{L^r}\,.
\end{equation}
\item  for any $1<r\leq\infty$ one has
\begin{equation} \label{sq1}
\|M_t^\varphi f\|_{L^r}\lesssim_{r,\varphi} \left(\log (10+|t|)\right)^{\frac{1}{r}}\,\|f\|_{L^r}\,.
\end{equation}
\end{itemize}
\end{lem}

Fix now $\vec{p}=(p_1,p_2,p_3,p_4)\in\overline{\textbf{B}}$ and assume without loss of generality that $p_1,p_2\not=\infty$. Then, returning to \eqref{prepsq} and making use of \eqref{sq} and \eqref{maxsh} above, we have:
\begin{equation*}
\sum_{k \in \Z} |\Lambda^{k;  {\frakA}}_{j, l, m}(\vec{f})|\lesssim\frac{1}{2^n}\,\int_{\R}\int_{\frac{1}{4}<|t|<4} S_{2^j t}^{\phi}f_1(x) S_{-2^l t^2}^{\phi} f_2(x)\,M_{-2^m t^3}^{\phi} f_3(x)\,M f_4(x)\, dt\,dx
\end{equation*}
which, from Fubini, Minkowski, H\"older inequalities and Lemma \ref{sqfunct}, implies
\begin{equation}\label{FMH}
\left| \sum_{k \in \Z} \Lambda^{k;  {\frakA}}_{j, l, m}(\vec{f})\right|\lesssim_{\vec{p}} \frac{1}{2^n}\,
\int_{\frac{1}{4}<|t|<4} \left\|S_{2^j t}^{\phi} f_1 \right\|_{L^{p_1}}\, \left\|S_{-2^l t^2}^{\phi} f_2\right\|_{L^{p_2}}\,
\left\|M_{-2^m t^3}^{\phi} f_3\right\|_{L^{p_3}}\,\left\|M f_4\right\|_{L^{p_4}} dt
\lesssim_{\vec{p}} \frac{j\,l\,m}{2^n} \|\vec{f}\|_{L^{\vec{p}}}\,.
\end{equation}

Summing up the above we get
\begin{equation}\label{Easyprange}
\left|\Lambda^{\not{D},NS}(\vec{f})\right|
\lesssim_{\vec{p}} \left( \sum_{n \ge 1} \frac{n^6}{2^n} \right)\,\|\vec{f}\|_{L^{\vec{p}}}\lesssim \|\vec{f}\|_{L^{\vec{p}}}\,.
\end{equation}

The conclusion of our theorem follows now from the symmetry of our proof in $f_1, f_2, f_3$ and $f_4$.
\end{proof}

We are now left with treating the main terms represented by the diagonal component $\Lambda^{D}$ and the stationary off-diagonal component $\Lambda^{\not{D},S}$.

\section{The diagonal term $\Lambda^{D}$: Strategy and motivation}\label{SDTmotiv}

We start by recalling that without loss of generality one has that
\begin{equation}\label{Diagterm}
\Lambda^{D}\approx\sum_{k \in \Z}\:\: \sum_{m \in \N} \Lambda^k_{m}\,,
\end{equation}
where here, for notational simplicity we set $\Lambda^k_{m}=\Lambda^k_{m,m,m}$.

The main ingredient on which the proof of the boundedness of the curved trilinear Hilbert transform relies is given by the following

\begin{thm} \label{mainthm}
Let $k\in\Z$ and $m\in\N$ be given, and, for  $j\in\{1,2,3\}$, set $\widehat{f_{j,m+j k}}(\xi):=\hat{f}_{j}(\xi)\,\phi(\frac{\xi}{2^{m+j k}})$. Then
\begin{equation}\label{mf}
\Lambda_m^k(\vec{f})=\int_{\R}\int_{\R} f_{1,m+k}(x-t)f_{2,m+2k}(x+t^2)f_{3,m+3k}(x+t^3)f_{4}(x)\,2^k \varrho(2^k t) dtdx,
\end{equation}
obeys: there exists some $\epsilon>0$ absolute constant such that for any $\vec{p}\in\textbf{B}=\left\{\vec{p}\in\overline{\textbf{B}}\,\Big|\,1<p_1,\,p_3<\infty\right\}$ there exists $c=c(\vec{p})>0$ depending only on $\vec{p}$ such that
\begin{equation}\label{KEYestim}
\left|\Lambda_m^k(\vec{f})\right| \lesssim_{\vec{p}} 2^{-\epsilon\,c(\vec{p})\,\min\{|k|,\frac{m}{2}\}} \,\|\vec{f}\|_{L^{\vec{p}}}\,.
\end{equation}
\end{thm}

The aim of this entire section is to reveal the key idea part of what we will refer from now on as \emph{correlative time-frequency analysis}--an approach that defines and guides our entire discretization process of $\Lambda_m^k$ and plays a crucial role in proving Theorem \ref{mainthm} above.

\subsection{Preparatives}

We start by splitting
\begin{equation}\label{Diagtermdec}
\Lambda^{D} \approx \Lambda^{D}_{+}\,+\,\Lambda^{D}_{-}\,,
\end{equation}
with
\begin{equation}\label{Diagtermpm}
\Lambda^{D}_{\pm}:=\sum_{k \in \Z_{\pm}}\:\: \sum_{m \in \N} \Lambda^k_{m}\,.
\end{equation}

For the reminder of the section we will focus on $\Lambda^{D}_{+}$ as the analysis performed for this case can be naturally adapted to the $-$ counterpart.\footnote{For a detailed discussion on the adaptations required for treating  $\Lambda^{D}_{-}$ we invite the reader to consult Section \ref{Negative}.}

Thus, throughout the reminder of this section we will fix $k,\,m \in \N$. Guided by the fact that in this regime $|t|\approx 2^{-k}<1$ and hence $|t|^3\leq t^2\leq |t|$, we are invited to divide
\begin{equation} \label{20230130eq01}
\Lambda_m^k(\vec{f})=\sum_{n\in\Z}\Lambda_m^{k, I_n^{0, k}}(\vec{f})\,,
\end{equation}
with
\begin{equation}\label{mf}
\Lambda_m^{k, I_n^{0, k}}(\vec{f}):=\int_{I_n^{0, k}}\int_{\R} f_{1,m+k}(x-t)f_{2,m+2k}(x+t^2)f_{3,m+3k}(x+t^3)f_{4,m+3k}(x)\,2^k \varrho(2^k t) dtdx,
\end{equation}
where here we use the notation: for $m\in \N$ and $k,\,q \in \Z$ we set
$$I_q^{m, k}:=\left[\frac{q}{2^{\frac{m}{2}+k}}, \frac{q+1}{2^{\frac{m}{2}+k}} \right]\,.$$

With these, Theorem \ref{mainthm} will be a consequence of the following:
$\newline$

\begin{thm} \label{mainthm-2022}
There exists some $\epsilon>0$, such that for any $k\geq 0$, $m,\,n\in\N$ and
$$\forall\:\:\vec{p}\in {\bf H}^{+}:=\{\vec{p}\in \overline{\textbf{B}}\,|\,p_1=2\:\:\&\:\:(p_2,p_3,p_4)\in\{2,\infty\}\}$$
the following holds:
\begin{equation}\label{keyestim}
\left|\Lambda_m^{k, I_n^{0, k}}(\vec{f})\right| \lesssim 2^{-\epsilon \min\{k,\frac{m}{2}\}} \,\|\vec{f} \|_{L^{\vec{p}}(3\, I_n^{0, k})}\,.
\end{equation}
\end{thm}

$\newline$

Since $n\in\N$ is fixed, without loss of generality we may focus only on $n=0$. For notational simplicity, throughout the reminder of the paper we identify $I_0^{0, k}$ with $I^k$ and (with a slight notational abuse) $\Lambda_m^{k, I_0^{0, k}}$ with $\Lambda_m^k$. Also, we will drop the second index in the definition of $f's$ above.

\subsection{Rank I LGC - an overview}\label{LGC(I)}

Our goal in this section is two folded:
\begin{itemize}
\item to outline the main steps in the standard (Rank I) LGC-methodology developed in \cite{Lie19};
\item to extract the curvature information encoded in $\Lambda_m^k$ in the form of certain time-frequency correlations, most prominently expressed in \eqref{TFCqq} and \eqref{LGC1}.
\end{itemize}

\noindent\textsf{Step 1. Linearization}
\smallskip

The aim of this first step is the linearization of the \emph{physical} argument, that is the $t-$integrand. In order to do so one has two account for two key parameters:
\begin{itemize}
\item the \emph{physical scale} represented by $k$: this captures the size--on the physical side--of the kernel $\frac{1}{t}$ associated with the quadrilinear form $\Lambda_m^k$;

\item the \emph{frequency scale} represented by $m$: this captures the size--on the frequency side--of the phase of the multiplier associated with the quadrilinear form $\Lambda_m^k$.
\end{itemize}
Thus, bringing together the homogeneity for each of the $t$-arguments and the frequency location corresponding to each of the corresponding functions $f_j$, via Heisenberg principle, we are able to identify the scale that has the linearizing effect in the discretization of our input functions (below $x\in\R$ is a fixed parameter):
\begin{itemize}
\item for $f_1(x-t)$ we use discretization in spatial intervals of length $2^{-\frac{m}{2}-k}$;
\item for $f_2(x+t^2)$ we use discretization in spatial intervals of length $2^{-\frac{m}{2}-2k}$;
\item for $f_3(x+t^3)$ we use discretization in spatial intervals of length $2^{-\frac{m}{2}-3k}$.
\end{itemize}

With these, we are naturally brought to the next step of our method:

\smallskip
\noindent\textsf{Step 2. Adapted Gabor wave-packet decomposition}
\smallskip

For each $j\in\{1,2,3\}$ and each input $f_j$ we apply a Gabor wave-packet decomposition adapted to the corresponding linearizing scale identified at the first step:
\begin{equation}\label{Gabordec}
f_j=\sum_{ \substack{p_j \sim 2^{\frac{m}{2}+(j-1)k} \\ n_j \sim 2^{\frac{m}{2}}}} \left \langle f_j, \phi_{p_j, n_j}^{j k} \right \rangle \phi_{p_j, n_j}^{j k}
\end{equation}
where here
\begin{equation}\label{GF3}
\phi_{r, n}^{j k}(x):=2^{\frac{m}{4}+\frac{j k}{2}} \phi \left(2^{\frac{m}{2}+jk}x-r \right) e^{i 2^{\frac{m}{2}+j k}xn}
\end{equation}
with $\phi$ a smooth cut-off function supported away from the origin.

Next, given the $k\geq 0$ regime, we partition the $t$-integral domain in intervals of length $2^{-\frac{m}{2}-k}$ and the $x$-integral domain in intervals of length $2^{-\frac{m}{2}-3k}$.

From the above, we deduce
\begin{eqnarray} \label{20230222eq10}
&& \Lambda_m^k(\vec{f})= 2^k \sum_{\substack{p \sim 2^{\frac{m}{2}+2k} \\ q \sim 2^{\frac{m}{2}}}}  \sum_{\substack{p_j \sim 2^{\frac{m}{2}+\left(j-1 \right)k}, \\ n_j \sim 2^{\frac{m}{2}},  \ j\in\{1, 2, 3\}}} \left \langle f_1, \phi_{p_1, n_1}^k \right \rangle \left \langle f_2, \phi_{p_2, n_2}^{2k} \right \rangle \left \langle f_3, \phi_{p_3, n_3}^{3k} \right \rangle \nonumber \\
&&\quad \cdot \int_{I_p^{m, 3k}} \int_{I_q^{m, k}} \phi_{p_1, n_1}^k \left(x-t \right) \phi_{p_2, n_2}^{2k} \left(x+t^2 \right) \phi_{p_3, n_3}^{3k} \left(x+t^3 \right) f_4(x) dtdx\nonumber\\
&& \approx 2^{\frac{m}{2}+\frac{5k}{2}}  \sum_{\substack{p \sim 2^{\frac{m}{2}+2k} \\ q \sim 2^{\frac{m}{2}}}}  \sum_{\substack{p_j \sim 2^{\frac{m}{2}+\left(j-1 \right)k}, \\ n_j \sim 2^{\frac{m}{2}},  \ j\in\{1, 2, 3\}}} \left \langle f_1, \phi_{p_1, n_1}^k \right \rangle \left \langle f_2, \phi_{p_2, n_2}^{2k} \right \rangle \left \langle f_3, \phi_{p_3, n_3}^{3k} \right \rangle \\
&& \quad \cdot \int_{I_p^{m, 3k}} \int_{I_q^{m, k}} 2^{\frac{m}{4}+\frac{3k}{2}} f_4(x) \phi \left(2^{\frac{m}{2}+k}x-q-p_1 \right) \phi \left(2^{\frac{m}{2}+2k} x+\frac{q^2}{2^{\frac{m}{2}}}-p_2 \right) \nonumber \\
&& \quad \quad \quad  \quad \quad \quad \cdot \phi \left(2^{\frac{m}{2}+3k} x+\frac{q^3}{2^m}-p_3 \right) e^{i2^{\frac{m}{2}+3k} x \left(n_3+\frac{n_2}{2^k}+\frac{n_1}{2^{2k}} \right)}  e^{-i \left(2^{\frac{m}{2}+k}tn_1-2^{\frac{m}{2}+2k}t^2 n_2-2^{\frac{m}{2}+3k} t^3 n_3 \right)} dtdx. \nonumber
\end{eqnarray}

\medskip
\noindent\textsf{Step 3. Time-frequency correlations}
\medskip

Relation \eqref{20230222eq10} allows us to identify now the interdependencies between the time and frequency parameters.  We first analyze the restrictions imposed on the spatial parameters:

Indeed, since $x \in I_p^{m, 3k}$, we have
\begin{equation} \label{20230222eq03}
\left|2^{\frac{m}{2}+3k}x-p \right| \lesssim 1.
\end{equation}
From \eqref{20230222eq10}, \eqref{20230222eq03} and the Schwartz localization of $\phi$ we immediately deduce the spatial correlations
\begin{equation} \label{20230222eq11}
    p_3\cong p+\frac{q^3}{2^m}\,,
\end{equation}

\begin{equation} \label{20230222eq12}
    p_2\cong\frac{p}{2^k}+\frac{q^2}{2^{\frac{m}{2}}}\,,
\end{equation}
and
\begin{equation} \label{20230222eq13}
p_1\cong\frac{p}{2^{2k}}-q.
\end{equation}

Substituting \eqref{20230222eq11}, \eqref{20230222eq12} and \eqref{20230222eq13} back to \eqref{20230222eq10}, we get
\begin{eqnarray} \label{20230222eq21}
\Lambda_m^k(\vec{f}) &\approx& 2^{\frac{m}{2}+\frac{5k}{2}}\sum_{\substack{p \sim 2^{\frac{m}{2}+2k} \\ q \sim 2^{\frac{m}{2}}}} \sum_{n_1, n_2, n_3, \sim 2^{\frac{m}{2}}}\left\langle f_1, \phi_{\frac{p}{2^{2k}}-q, n_1}^k \right \rangle \left\langle f_2, \phi_{\frac{p}{2^{k}}+\frac{q^2}{2^{\frac{m}{2}}}, n_2}^{2k} \right \rangle  \\
&\cdot& \left\langle f_3, \phi_{p+\frac{q^3}{2^m}, n_3}^{3k} \right \rangle \left\langle f_4, \phi_{p, n_3+\frac{n_2}{2^k}+\frac{n_1}{2^{2k}}}^{3k} \right \rangle \cdot \int_{I_q^{m, k}} e^{-i \varphi_{n_1, n_2, n_3}(t)} dt, \nonumber
\end{eqnarray}
where
$$
\varphi_{n_1, n_2, n_3}(t):=2^{\frac{m}{2}+k}tn_1-2^{\frac{m}{2}+2k} t^2 n_2-2^{\frac{m}{2}+3k} t^3 n_3.
$$
Notice that as a consequence of the $x-$integral in \eqref{20230222eq10} we obtain in \eqref{20230222eq21} the frequency correlation
\begin{equation} \label{TFC-01}
n_4=\frac{n_1}{2^{2k}}+\frac{n_2}{2^k}+n_3\,,
\end{equation}
where here  $n_4$ is the parameter representing the frequency localization of $f_4$.

We now pass to the induced time-frequency conditionalities. This is the crucial point where we make use of the \emph{linearizing} effect of the scale identified at Step 1 and involved in the Gabor decomposition at Step 2. Indeed, by writing $t=\left(t-\frac{q}{2^{\frac{m}{2}+k}}\right)+\frac{q}{2^{\frac{m}{2}+k}}$ and using $t \in I_q^{m, k}$ we linearize the phase function $\varphi_{n_1, n_2, n_3}$ via a Taylor series argument:
$$
\varphi_{n_1, n_2, n_3}(t)=2^{\frac{m}{2}+k} t \left(n_1-\frac{2q}{2^{\frac{m}{2}}} n_2-\frac{3q^2}{2^m} n_3 \right)+\frac{q^2n_2}{2^{\frac{m}{2}}}+\frac{2q^3n_3}{2^m}+O(1).
$$
Therefore, up to admissible error terms
\begin{equation} \label{20230422eq01}
\int_{I_q^{m, k}} e^{-i \varphi_{n_1, n_2, n_3}(t)} dt \approx 2^{-\frac{m}{2}-k} \cdot  C_{q, n_1, n_2, n_3} \cdot \int_0^1 e^{-it  \left(n_1-\frac{2q}{2^{\frac{m}{2}}} n_2-\frac{3q^2}{2^m} n_3 \right)} dt
\end{equation}
where $C_{q, n_1, n_2, n_3}= e^{-i \left[\frac{q^2 n_2}{2^{\frac{m}{2}}}+\frac{2q^3 n_3}{2^m}+q  \left(n_1-\frac{2q}{2^{\frac{m}{2}}} n_2-\frac{3q^2}{2^m} n_3 \right)\right]}$.

As a consequence, for a fixed $q\sim 2^{\frac{m}{2}}$, the main contribution arises when the triple $(n_1,n_2,n_3)$ belongs to the \emph{time-frequency correlation set}
\begin{equation} \label{TFCqq}
\TFC^{I}_q:=\left\{ \substack{\left(n_1, n_2, n_3 \right)\\n_1, n_2, n_3 \sim 2^{\frac{m}{2}}}\, :\, n_1-\frac{2q}{2^{\frac{m}{2}}} n_2- \frac{3q^2}{2^m} n_3\cong 0 \right\}.
\end{equation}
With these, we conclude that\footnote{For some more details justifying this step the reader in invited to consult Section \ref{AppendixB} in the Appendix.}
\begin{eqnarray} \label{LGC1}
&&\Lambda_m^k(\vec{f})\approx 2^{\frac{3k}{2}}\sum_{\substack{p \sim 2^{\frac{m}{2}+2k} \\ q \sim 2^{\frac{m}{2}}}} \sum_{(n_1, n_2, n_3)\in \TFC^{I}_q} C_{q, n_1, n_2, n_3}\,\left\langle f_1, \phi_{\frac{p}{2^{2k}}-q, n_1}^k \right\rangle \left\langle f_2, \phi_{\frac{p}{2^{k}}+\frac{q^2}{2^{\frac{m}{2}}}, n_2}^{2k} \right \rangle  \\
&&\qquad\qquad\qquad\qquad\qquad\qquad\qquad\qquad\qquad\qquad\qquad\cdot\left\langle f_3, \phi_{p+\frac{q^3}{2^m}, n_3}^{3k} \right \rangle \left\langle f_4, \phi_{p, n_3+\frac{n_2}{2^k}+\frac{n_1}{2^{2k}}}^{3k} \right \rangle\,.\nonumber
\end{eqnarray}

\subsection{Standard versus correlative time-frequency analysis}\label{CTFrev}

In this section, we are going to present a brief antithetical discussion on the time-frequency analysis of multi-linear forms. In the context of our story, the ``original sin" is expressed by the fundamental fact that the trigonometric system is \emph{not an unconditional basis} for $L^p(\TT)$ unless $p=2$. We will illustrate this via three examples starting from a more philosophical perspective and ending up with a model that explains the obstacle in proving the full conjectural boundedness range for the Bilinear Hilbert transform.

The first example, may be thought at the philosophical level as a key motivation for the development of the Littlewood--Paley theory: indeed, the central piece of the latter, expressed in the well-known form\footnote{Here, in the standard fashion, $\psi$ is assumed to be a smooth real function with zero mean and $\psi_k(x)=2^k \psi \left( 2^k x\right)$.}
\begin{equation} \label{LP}
\|f\|_{L^p(\R)}\approx \|Sf\|_{L^p(R)}:=\left\|\left(\sum_{k\in\Z} |f*\psi_k|^2 \right)^{\frac{1}{2}}\right\|_{L^p(\R)}\,,\qquad\qquad 1<p<\infty\,,
\end{equation}
could be thought as an upgrade of (the partial) Young's inequality in the effort of controlling the spatial information encoded in $\|f\|_{L^p(\R)}$ by properly regrouping the frequency information carried by $\hat{f}$. Thus, a first lesson that we are learning out of this is the very delicate balance in controlling the spatial properties of an object when deciding the structure of the cuttings on the frequency side.

A second delicate example confirming and adding to the above lesson, this time in the multilinear setting, was offered in \cite{BBLV21}. For reader's convenience we summarize it below: define\footnote{Here, $\T=[0, 1]$.}
\begin{equation}
\label{Example1}
\Lambda(F_1,F_2,F_3):=((F_1 F_2)*F_3)(0)=\int_{\TT} F_1(t) F_2(t) F_3(-t) dt\,,
\end{equation}
where here $F_1, F_2\in L^2(\TT)$ and $F_3\in L^\infty(\TT)$. For $j\in\{1,2,3\}$ let $F_j(t)=\sum\limits_{n\in\Z} a_{j,n} e^{i n t}$ be the Fourier Series representation of $F_j$. Then
\begin{itemize}
\item on the one hand, from Parseval and H\"older, we have
\begin{equation*}
|\Lambda(F_1,F_2,F_3)|\simeq \left|\sum_{n,k\in\Z}a_{1,n-k} a_{2,k} a_{3,n} \right|\lesssim \|F_1\|_{L^2(\TT)} \|F_2\|_{L^2(\TT)} \|F_3\|_{L^{\infty}(\TT)};
\end{equation*}
\item on the other hand, for generic $\|F_1\|_{L^2(\TT)}=\|F_2\|_{L^2(\TT)}=\|F_3\|_{L^{\infty}(\TT)}=1$, we have
\begin{equation*}
\sum_{n,k\in\Z}\left|a_{1,n-k} a_{2,k} a_{3,n}\right|=\infty.
\end{equation*}
\end{itemize}
 We end this revelatory discussion with yet another example: define
\begin{equation}
\label{Example3}
T(F_1,F_2)(x):=\int_{\TT} F_1(x-t) F_2(x+t) dt\,,
\end{equation}
where here $F_1\in L^{p_1}(\TT)$, $F_2\in L^{p_2}(\TT)$ with $p_1,\,p_2\geq 1$ and, as before, for $j\in\{1,2\}$ we let $F_j(t)\approx\sum\limits_{n\in\Z} a_{j,n} e^{i n t}$ be the Fourier Series representation of $F_j$. Then, we have
\begin{itemize}
\item on the one hand, from H\"older, for any $\frac{1}{p_3}=\frac{1}{p_1}+\frac{1}{p_2}$, we have
\begin{equation}\label{cex}
\|T(F_1,F_2)\|_{L^{p_3}(\TT)}\simeq \left(\int_{\TT}\left|\sum_{n\in\Z} a_{1,n}\,a_{2,n}\,e^{i 2 n x}\,dx\right|^{p_3}\right)^{\frac{1}{p_3}}\lesssim \|F_1\|_{L^{p_1}(\TT)} \|F_2\|_{L^{p_2}(\TT)};
\end{equation}
\item on the other hand, \eqref{cex} can not hold for $p_3<\frac{2}{3}$ if we allow a random signum change of the Fourier coefficients for $F_1$ and/or $F_2$ as one can easily deduce via an application of Hincin's inequality.
\end{itemize}

All these examples above explicitly emphasize the challenge brought by the \emph{conditional} frequency discretizations.
This aspect proves to be a ``red thread" of modern Fourier analysis: until very recently, all the advancement in the time-frequency area was strictly dependent on the existence of an \emph{unconditional discretized wave-packet model} of the object under investigation. One of the most prominent example confirming these is offered by the bilinear Hilbert transform:
$$B(f,g)(x)=\int_{\R} f(x-t)g(x+t)\,\frac{dt}{t}$$
The gap between the conjectured boundedness range $B:\,L^p\times L^q\,\mapsto L^r$ for $\frac{1}{p}+\frac{1}{q}=\frac{1}{r}$ with
$1<p,q\leq \infty$, $r>\frac{1}{2}$ and the currently known range imposing the restriction $r>\frac{2}{3}$ derives precisely from the fact that the discretized  wave-packet model of $B$ stops being unconditional when $r$ drops below $\frac{2}{3}$. This latter aspect is in fact a direct consequence of the third example discussed above - see \eqref{Example3} and \eqref{cex} (for more on this one can for example consult \cite{MS13}).

The first instance of controlling a nontrivial operator that does not admit a standard \emph{unconditional discretized wave-packet model} was offered in \cite{BBLV21} where the authors therein discuss the non-resonant Bilinear Hilbert Carleson operator defined earlier in \eqref{Top}.

The present paper offers a second such instance: as it turns out, the representation in \eqref{LGC1} is not an $m-$decaying unconditional (summable) model, in the sense that, for any $\vec{p}\in {\bf H}^{+}$

 \begin{eqnarray} \label{LGC1abs}
&&(\Lambda_m^k)^{*}(\vec{f}):= 2^{\frac{3k}{2}}\sum_{\substack{p \sim 2^{\frac{m}{2}+2k} \\ q \sim 2^{\frac{m}{2}}}} \sum_{(n_1, n_2, n_3)\in \TFC^{I}_q} \left|\left\langle f_1, \phi_{\frac{p}{2^{2k}}-q, n_1}^k \right\rangle \right| \left|\left\langle f_2, \phi_{\frac{p}{2^{k}}+\frac{q^2}{2^{\frac{m}{2}}}, n_2}^{2k} \right \rangle \right| \\
&&\qquad\qquad\qquad\qquad\qquad\cdot\left|\left\langle f_3, \phi_{p+\frac{q^3}{2^m}, n_3}^{3k} \right \rangle\right|  \left|\left\langle f_4, \phi_{p, n_3+\frac{n_2}{2^k}+\frac{n_1}{2^{2k}}}^{3k} \right \rangle\right|\approx \|\vec{f}\|_{L^{\vec{p}}}\,.\nonumber
\end{eqnarray}
Indeed, \eqref{LGC1abs} can be easily checked by considering the situation when for each $j$ the local Fourier (Gabor) coefficients associated with $f_j$ are equidistributed.

Thus, the standard time-frequency analysis approach that would rely on manipulating absolute values of Gabor coefficients can not be relied upon in order to achieve Theorem \ref{mainthm}.

This justifies a new approach that we will refer to from now on as \emph{correlative} time-frequency analysis: unlike the standard situation in which the model operator/form relies on the discretization of each of its input functions based strictly on the spatial and frequency properties\footnote{These are derived from the spatial and frequency properties of the kernel/multiplier.} of the input but \emph{independent} of the structure of the other inputs, in the new situation the discretization of the input functions is allowed to depend on the other inputs in the form of joint local Fourier coefficients--see the discussion in Section \ref{CTF}--thus creating correlations between the input functions, hence the name of the approach.\footnote{The specifics of this new approach will become transparent in the next section.} In view of the above it is only natural to end this commentary with the following

\begin{obs}\textsf{[A natural direction of investigation for BHT with $r<\frac{2}{3}$]}\label{BHTinv} As was already revealed above, we know that the absolute summability of the wave-packet model used to address the boundedness of the Bilinear Hilbert transform for the $r>\frac{2}{3}$ case, (\cite{LT97}, \cite{LT99}), admits a single scale counterexample in the situation $r<\frac{2}{3}$ - see the third example defined by \eqref{Example3}. An inspection of this counterexample reveals the discrepancy between  controlling an expression of the form
\begin{equation}\label{jointmax}
x\,\longrightarrow\,\int_{\R} f_1(x-t) f_2(x+t)\,\rho(t)\,dt\,,
\end{equation}
versus controlling the absolute values of (local) Fourier coefficients arising in the Fourier decomposition of the corresponding input functions above.

In view of the Rank II LGC method developed in this paper which addresses a very similar situation for the case of the curved trilinear Hilbert transform, it is quite natural to propose as a line of inquiry a \emph{correlative time-frequency model} for BHT in which instead of analyzing the original interaction of \emph{individual Gabor decompositions} (single function) to operate with \emph{joint Gabor coefficients} corresponding to (multiscale versions of) \eqref{jointmax}.
\end{obs}

\subsection{Rank II LGC: preliminaries}\label{IIPrelim}

Our goal in this section is to provide a concrete shape to the philosophy of correlative time-frequency analysis evoked earlier.
This will be done by exploiting a cancelation hidden in a suitable correlation between the input functions $f_1$ and $f_2$.

We start by recalling \eqref{GF3} and then decompose
$$f_3(x)=\sum_{{r \sim 2^{\frac{m}{2}+2k}}\atop{n \sim 2^{\frac{m}{2}}}}  \left\langle f_3, \phi_{r, n}^{3k} \right \rangle \phi_{r, n}^{3k}(x)\,.$$
Now, we substitute the above in \eqref{mf}:
\begin{eqnarray} \label{20230131eq01}
\Lambda_m^k(\vec{f})&=&\sum_{\substack{p \sim 2^{\frac{m}{2}+2k} \\ q \sim 2^{\frac{m}{2}}}} \int_{I_p^{m, 3k}} \int_{I_q^{m, k}} f_1(x-t) f_2(x+t^2) \left(\sum_{\substack{ r \sim 2^{\frac{m}{2}+2k} \\ n \sim 2^{\frac{m}{2} }}}\left\langle f_3, \phi_{r, n}^{3k} \right \rangle \phi_{r, n}^{3k}(x+t^3) \right)  f_4(x) 2^k \varrho(2^kt) dtdx \nonumber \\
&=&\sum_{\substack{p \sim 2^{\frac{m}{2}+2k} \\ q \sim 2^{\frac{m}{2}}}} 2^k \sum_{\substack{r \sim 2^{\frac{m}{2}+2k} \\ n \sim 2^{\frac{m}{2}}}} \langle f_3, \phi_{r, n}^{3k} \rangle  \int_{I_p^{m, 3k}} \widetilde{J}_{q, r, n}^k(f_1,f_2)(x) f_4(x)dx,
\end{eqnarray}
where
\begin{eqnarray} \label{20230131eq02}
&&\qquad\qquad\widetilde{J}_{q, r, n}^k(f_1,f_2)(x):=  \int_{I_q^{m, k}} f_1(x-t)f_2(x+t^2) \phi_{r, n_3}^{3k} (x+t^3) \varrho(2^k t) dt\\
&&= \int_{I_q^{m, k}} f_1(x-t)f_2(x+t^2) 2^{\frac{\frac{m}{2}+3k}{2}}
\phi \left(2^{\frac{m}{2}+3k} (x+t^3)-r \right) e^{i 2^{\frac{m}{2}+3k}(x+t^3) n}  \varrho(2^k t) dt\nonumber\,.
\end{eqnarray}
We now take aside \eqref{20230131eq02} and model it according to the linearization step of the Rank I LGC methodology, that is: for $t \in I_q^{m, k}$, at the phase level, one uses
\begin{eqnarray} \label{20221229eq02}
2^{\frac{m}{2}+3k} t^3=3\cdot 2^{\frac{m}{2}+k} t \left( \frac{q}{2^{\frac{m}{2}}} \right)^2-2 \cdot 2^{\frac{m}{2}}  \left( \frac{q}{2^{\frac{m}{2}}} \right)^3+O \left(2^{-\frac{m}{2}} \right),
\end{eqnarray}
while  at the spatial argument level, one only needs the zero order approximation
\begin{equation} \label{20221229eq03}
2^{\frac{m}{2}+3k} t^3=\frac{q^3}{2^m}+O(1).
\end{equation}
Substituting \eqref{20221229eq02} and \eqref{20221229eq03} back to \eqref{20230131eq02} and applying a Taylor expansion argument, we see that it suffices to treat\footnote{Here, for notational simplicity, we continue to refer to the main term with the original label  $\widetilde{J}_{q, r, n}^k(f_1,f_2)(x)$.} \footnote{Recall that throughout the paper, the compactly supported  function $\phi$, is allowed to change from line to line.}
\begin{equation} \label{20221229eq04}
\widetilde{J}_{q, r, n}^k(f_1,f_2)(x) \approx \int_{I_q^{m, k}} f_1(x-t)f_2(x+t^2) 2^{\frac{\frac{m}{2}+3k}{2}} \phi \left(2^{\frac{m}{2}+3k}x+\frac{q^3}{2^m}-r \right) e^{-2i n \cdot \frac{q^3}{2^m}} e^{i 2^{\frac{m}{2}+3k}xn} e^{3i \cdot 2^{\frac{m}{2}+k}t \left(\frac{q}{2^\frac{m}{2}} \right)^2 n} dt.
\end{equation}

Now since we have that $x \in I_p^{m, 3k}$, we notice the \textit{time correlation} identity
\begin{equation} \label{20221229eq05}
r \cong p+\frac{q^3}{2^m}.
\end{equation}
As a consequence, for $x \in I_p^{m, 3k}$, we deduce that
\begin{equation} \label{20221229eq06}
\widetilde{J}_{q, r, n}^k(f_1,f_2)(x) \approx \widetilde{J}^{k}_{q, p+\frac{q^3}{2^m}, n}(f_1,f_2)(x) = \left(\int_{I_q^{m, k}} f_1(x-t)f_2(x+t^2) e^{3i \cdot 2^{\frac{m}{2}+k}t \left(\frac{q}{2^\frac{m}{2}} \right)^2 n} dt \right) \cdot \phi_{p, n}^{3k} (x) \cdot e^{-2i n \cdot \frac{q^3}{2^m}}.
\end{equation}
Substituting \eqref{20221229eq05} and \eqref{20221229eq06} back to the formula of $\Lambda_m^k(\vec{f})$, we conclude
\begin{equation} \label{20221229eq07}
 \Lambda_m^k(\vec{f}) \sim 2^k \sum_{\substack{p \sim 2^{\frac{m}{2}+2k} \\ q, n \sim 2^{\frac{m}{2}}}} \left\langle f_3, \phi_{p+\frac{q^3}{2^m}, n}^{3k}  \right\rangle \int_{I_p^{m, 3k}} \widetilde{J}^{k}_{q, p+\frac{q^3}{2^m}, n}(f_1,f_2)(x) f_4(x)dx
\end{equation}

\subsection{Rank II LGC: the case $k\geq \frac{m}{2}$}\label{IIlgcklarge}

Following the preview offered in Section \ref{Keyideas}, our goal in this section is to obtain the final discretization of our form that allows us to decouple the
the information carried by the pairs $(f_1,f_2)$ from that of $f_4$.

As it turns out, in order to achieve this we will need to split the range of $k\geq 0$ in three regimes.

Indeed, the first regime for $k$ is invited by \eqref{20221229eq06}--\eqref{20221229eq07} as revealed by the following:
$\newline$

\noindent\textbf{Key Heuristic 1.} \textit{Since $\supp \ \widehat{f_i} \subseteq \left[2^{m+ik}, 2^{m+ik+1} \right]$ for $i\in\{1,2\}$, via Heisenberg uncertainty principle, we have that
\begin{enumerate}
\item [(1)] $f_1$ is morally a constant on intervals of length $2^{-m-k}$;
\item [(2)] $f_2$ is morally a constant on intervals of length $2^{-m-2k}$.
\end{enumerate}
Imposing now the condition $k \ge \frac{m}{2}$ we deduce that $|I_p^{m, 3k}|=2^{-\frac{m}{2}-3k} \le 2^{-m-2k}$. As a consequence, at the moral level, we have that the mapping
\begin{equation} \label{20230131eq04}
x \mapsto {J}^{k}_{q, n}(f_1,f_2)(x) := \int_{I_q^{m, k}} f_1(x-t)f_2(x+t^2) e^{3i \cdot 2^{\frac{m}{2}+k}t \left(\frac{q}{2^\frac{m}{2}} \right)^2 n}dt
\end{equation}
is constant on the interval $I_p^{m, 3k}$. (In effect the expression in \eqref{20230131eq04} is morally constant on intervals of length $2^{-m-2k}$.)}

\begin{obs}\textsf{[Error terms I]}\label{ETerm}
Throughout this paper, due to expository reasons and space limitation, we will only focus on the \emph{main terms} arising in our estimates. Thus, for example, we will not discuss in detail how to make rigorous statements like ``a certain expression is morally constant on a specific interval" as these kind of things can be easily handled via standard tools such as inverse Fourier representation, Fourier series,  Taylor series expansions, etc.. In some rare occasions we will defer such reasonings to the discussion in the Appendix. In particular, the constancy of the term \eqref{20230131eq04} is just briefly touched in Section \ref{AppendixA}.
\end{obs}

Using the above heuristic together with \eqref{20221229eq07} (see also Observation \ref{ETerm}), we obtain the following \emph{correlative time-frequency model}:
\begin{eqnarray} \label{20230131main01}
&& \left|\Lambda_m^k(\vec{f})\right|  \lesssim 2^k \sum_{\substack{p \sim 2^{\frac{m}{2}+2k} \\ q, n \sim 2^{\frac{m}{2}}}} \left| \left\langle f_3, \phi_{p+\frac{q^3}{2^m}, n}^{3k}  \right\rangle \right| \left|  \left \langle f_4, \phi_{p, n}^{3k} \right \rangle \right| \\
&& \quad \quad \quad \quad \cdot \frac{1}{\left|I_p^{m, 3k} \right|} \left| \int_{I_p^{m, 3k}} \int_{I_q^{m, k}} f_1(x-t)f_2(x+t^2) e^{3i \cdot 2^{\frac{m}{2}+k}t \left(\frac{q}{2^\frac{m}{2}} \right)^2 n} dt dx \right| . \nonumber
\end{eqnarray}

\subsection{Rank II LGC: the case $0<k<\frac{m}{2}$}\label{Diagkpsmall}

Again, as outlined earlier in Section \ref{CTF} (Second stage), we would next like to \emph{decouple} the action of the pair $(f_1,f_2)$ from $f_4$ by achieving a similar conclusion with the one displayed in \eqref{20230131eq04}. However, the latter can not be fulfilled in the range $k\leq \frac{m}{2}$ without a reorganization of the time-frequency information carried by $f_1$ and $f_2$. This process requires a further analysis that will be performed in several steps below:

\subsubsection{Preliminary frequency foliation}\label{Timefreqfoliation}

We return now to \eqref{20221229eq06}--\eqref{20221229eq07}. Thus, the strategy--guided by the Heisenberg principle--will be to discretize (simultaneously) each spatial fibers of $f_1$ and $f_2$ into frequency locations such that we achieve \emph{both} of the following key requirements: (i) the length of the frequency subdivisions has to be dominated by the length of the frequency support of the wave-packet $\phi_{p, n}^{3k}$, and, (ii) the number of the subdivision corresponding to $f_1$ and $f_2$ is the same.\footnote{The reasons behind will become transparent in the next two subsections.}

Concretely, we proceed as follows: for $j\in\{1, 2\}$ we write
$$
\supp \ \widehat{f_j} \subseteq\left[2^{m+jk}, 2^{m+jk+1} \right]=\bigcup_{\ell_l \sim 2^{\frac{m}{2}-k}} \omega_j^{\ell_j}:=\bigcup_{\ell_j \sim 2^{\frac{m}{2}-k}} \left[\ell_j 2^{\frac{m}{2}+(j+1)k}, \left(\ell_j+1 \right) 2^{\frac{m}{2}+(j+1)k}\right]\,.$$

Further on, in view of the time-frequency correlation relation that will become one of the main pillars in our approach, we also subdivide the frequency supports of $f_3$ and $f_4$ as follows: for $j\in\{3, 4\}$,
$$\supp \ \widehat{f_j}\subseteq\left[2^{m+3k}, 2^{m+3k+1} \right]=\bigcup_{\ell_j \sim 2^{\frac{m}{2}-k}} \omega_j^{\ell_j}:=\bigcup_{\ell_j \sim 2^{\frac{m}{2}-k}} \left[\ell_j 2^{\frac{m}{2}+4k}, \left(\ell_j+1 \right) 2^{\frac{m}{2}+4k}\right].
$$
In summary, we achieved the following: if $j\in\{1,\ldots,4\}$ and $\underline{j}=\min\{j,3\}$ then
\begin{equation} \label{freq}
  \omega_j^{\ell_j} \subset \left[2^{m+k\underline{j}}, 2^{m+k\underline{j}+1} \right]\quad\textrm{with}\quad \left|\omega_j^{\ell_j}\right|=2^{\frac{m}{2}+(\underline{j}+1)k}\quad\textrm{and}\quad\# \ell_j \sim 2^{\frac{m}{2}-k}\,.
\end{equation}

\subsubsection{Action on the ``thickened" frequency fibers}\label{Decksmall}
$\newline$

In this section we utilize the frequency foliation explained above in order to obtain an analogue of the Key Heuristic 1. Thus, we let $\chi$ be a smooth cut-off function with $\supp \ \chi \subset \{\frac{1}{2}<|s|<2\}$ and for $j\in\{1, 2\}$ one writes
$$
f_j=\sum_{\ell_j \sim 2^{\frac{m}{2}-k}} f_j^{\omega_j^{l_j}},
$$
where
\begin{equation} \label{20230505eq01}
f_j^{\omega_j^{\ell_j}}(x)=\int_{\R} \widehat{f_j}(\xi) \chi \left( \frac{\xi-\ell_j 2^{\frac{m}{2}+(j+1)k}}{\left|\omega_j^{\ell_j} \right|} \right) e^{i \xi x} d\xi=\underline{f}_j^{\omega_j^{\ell_j}}(x) \cdot e^{i \ell_j 2^{\frac{m}{2}+(j+1)k} x},
\end{equation}
with
\begin{equation} \label{20230408eq01}
\underline{f}_j^{\omega_j^{\ell_j}}(x):=\int_{\R} \widehat{f}_j \left(\xi+\ell_j 2^{\frac{m}{2}+(j+1)k} \right) \chi \left( \frac{\xi}{\left|\omega_j^{\ell_j} \right|} \right) e^{i\xi x} d\xi.
\end{equation}

\begin{obs} \label{keyobs-01}
Note that since $\supp\,\calF \left(\underline{f}_1^{\omega_1^{\ell_1}} \right) \subseteq (-2|\omega_1^{\ell_1}|,\,2|\omega_1^{\ell_1}|)$ with $\left|\omega_1^{\ell_1} \right|=2^{\frac{m}{2}+2k}$ and $\supp\, \calF\left(\underline{f}_2^{\omega_2^{\ell_2}} \right) \subseteq (-2|\omega_2^{\ell_2}|,\,2|\omega_2^{\ell_2}|)$  with $\left|\omega_2^{\ell_2} \right|=2^{\frac{m}{2}+3k}$, we have
$$
\underline{f}_1^{\omega_1^{\ell_1}}(\cdot),  \ \underline{f}_2^{\omega_2^{\ell_2}}(\cdot) \ \textnormal{are morally constant on} \ \left\{I_p^{m, 3k} \right\}_{p \sim 2^{\frac{m}{2}+2k}}.
$$
\end{obs}
Applying the above decomposition, we have
\begin{eqnarray}\label{Dec}
{J}^{k}_{q, n}(f_1,f_2)(x)&=&\int_{I_q^{m, k}} f_1(x-t)f_2(x+t^2)e^{3i \cdot 2^{\frac{m}{2}+k} \cdot \frac{q^2}{2^m} n t} dt\nonumber \\
&=&\sum_{\ell_1, \ell_2 \sim 2^{\frac{m}{2}-k}} e^{i \left(\ell_1 2^{\frac{m}{2}+2k}+\ell_2 2^{\frac{m}{2}+3k} \right)x} \underline{J}^k_{q,n} \left( \underline{f}_1^{\omega_1^{\ell_1}}, \underline{f}_2^{\omega_2^{\ell_2}} \right)(x),
\end{eqnarray}
where
\begin{equation}
\underline{J}^k_{q,n} \left( \underline{f}_1^{\omega_1^{\ell_1}}, \underline{f}_2^{\omega_2^{\ell_2}} \right)(x):=\int_{I_q^{m, k}}  \underline{f}_1^{\omega_1^{\ell_1}}(x-t) \underline{f}_2^{\omega_2^{\ell_2}}(x+t^2) e^{i \left(-\ell_1 2^{\frac{m}{2}+2k}t+\ell_2 2^{\frac{m}{2}+3k} t^2+3 \cdot 2^{\frac{m}{2}+k} \cdot \frac{q^2}{2^m} n t \right)} dt.
\end{equation}

\begin{obs} \label{keyobs-02}
 As a consequence of Observation \ref{keyobs-01}, we obtained the desired analogue of Key Heuristic 1, that is,  we have that the mapping
$$
x \to \underline{J}^k_{q,n} \left( \underline{f}_1^{\omega_1^{\ell_1}}, \underline{f}_2^{\omega_2^{\ell_2}} \right)(x)
$$
is morally constant on intervals $\left\{I_p^{m, 3k} \right\}_{p \sim 2^{\frac{m}{2}+2k}}$.
\end{obs}

\subsubsection{Rescaled time-frequency correlation set}\label{RescTFC}
$\newline$

In this section we focus on obtaining a rescaled version of \eqref{TFCqq} that is adapted to the frequency foliation accomplished earlier and that will prove quintessential in controlling our quadrilinear form $\Lambda_m^k$ in the case $0\leq k\leq \frac{m}{2}$.

\begin{lem}\label{Rtfc}
Let $\ell_1, \ell_2, \ell_3$ and $q$ be defined as in \eqref{20230505eq01}, then one has
\begin{equation} \label{TFC-03}
\ell_1-\frac{2q}{2^{\frac{m}{2}}} \ell_2- \frac{3q^2}{2^m} \ell_3\cong 0.
\end{equation}
\end{lem}

\begin{proof}
Let us fix the choice of $\ell_1, \ell_2$ and $\ell_3$ first. Recall now the decomposition that we performed in Section \ref{LGC(I)}, specifically \eqref{LGC1}. Note that for any $n_j \sim 2^{\frac{m}{2}}, j\in\{1, 2, 3\}$, if $n_j 2^{\frac{m}{2}+jk} \in \omega_j^{\ell_j}$, then
\begin{equation} \label{20230203eq40}
n_j \in
\widetilde{\omega_j^{\ell_j}}:=\left[\ell_j 2^k, \left(\ell_j+1 \right) 2^k \right].
\end{equation}
By \eqref{TFCqq}, that is, $n_1-\frac{2q}{2^{\frac{m}{2}}}n_2-\frac{3q^2}{2^m}n_3=O(1)$, we see that
$$
2^k \ell_1-\frac{2q}{2^{\frac{m}{2}}} \cdot 2^k \ell_2-\frac{3q^2}{2^m} \cdot 2^k \ell_3=O(2^k).
$$
The desired claim then follows by dividing by $2^k$ on both sides of the above equation.
\end{proof}

As a result of the above lemma, it is now natural to introduce \emph{the time-frequency correlation set} as follows:
for each $q \sim 2^{\frac{m}{2}}$ we let
\begin{equation} \label{TFCq}
\TFC_q:=\left\{ \substack{\left(\ell_1, \ell_2, \ell_3 \right)\\\ell_1, \ell_2, \ell_3 \sim 2^{\frac{m}{2}-k}}\, :\, \ell_1-\frac{2q}{2^{\frac{m}{2}}} \ell_2- \frac{3q^2}{2^m} \ell_3\cong 0\right\}.
\end{equation}

\subsubsection{The correlative time-frequency model}\label{Model}
$\newline$

With the above settled, we are now ready to refine the time-frequency decomposition \eqref{20221229eq07} of our quadrilinear form  $\Lambda_m^k$. Indeed, using \eqref{Dec}, Observation \ref{keyobs-02}, Lemma \ref{Rtfc} and \eqref{TFCq}, we have
\begin{eqnarray} \label{20230223main01a}
&& |\Lambda_m^k(\vec{f})|\lesssim 2^k \sum_{\substack{p \sim 2^{\frac{m}{2}+2k}\\ q \sim 2^{\frac{m}{2}}}} \sum_{(\ell_1, \ell_2, \ell_3) \in \TFC_q} \sum_{n \in \widetilde{\omega_3^{\ell_3}}} \left| \left\langle f_3^{\omega_3^{\ell_3}}, \phi_{p+\frac{q^3}{2^m}, n}^{3k} \right \rangle \right|  \left| \left \langle f_4 e^{i \left(\ell_12^{\frac{m}{2}+2k}+\ell_22^{\frac{m}{2}+3k} \right) \left( \cdot \right) }, \phi_{p, n}^{3k} \right \rangle \right| \nonumber \\
&& \quad  \cdot \frac{1}{\left|I_p^{m, 3k} \right|} \bigg| \int_{I_p^{m, 3k}} \int_{I_q^{m, k}}  \underline{f}_1^{\omega_1^{\ell_1}}(x-t) \underline{f}_2^{\omega_2^{\ell_2}}(x+t^2) e^{i \left(-\ell_1 2^{\frac{m}{2}+2k}t+\ell_2 2^{\frac{m}{2}+3k} t^2+3 \cdot 2^{\frac{m}{2}+k} \cdot \frac{q^2}{2^m}n t \right)} dt dx \bigg|\,.
\end{eqnarray}
Next, linearizing the phase of the double integral in the main term \eqref{20230223main01a}, for $t \in I_q^{m, k}$, we have $t^2= \left(t-\frac{q}{2^{\frac{m}{2}+k}}+\frac{q}{2^{\frac{m}{2}+k}} \right)^2 =\frac{2tq}{2^{\frac{m}{2}+k}}-\frac{q^2}{2^{m+2k}}+O \left(\frac{1}{2^{m+2k}} \right)$, hence
\begin{eqnarray} \label{20230223main02}
|\Lambda_m^k(\vec{f})|&\lesssim&2^k \sum_{\substack{p \sim 2^{\frac{m}{2}+2k}\\ q \sim 2^{\frac{m}{2}}}} \sum_{(\ell_1, \ell_2, \ell_3) \in \TFC_q} \sum_{n \in \widetilde{\omega_3^{\ell_3}}} \left| \left\langle f_3^{\omega_3^{\ell_3}}, \phi_{p+\frac{q^3}{2^m}, n}^{3k} \right \rangle \right|  \left| \left \langle f_4 e^{i \left(\ell_12^{\frac{m}{2}+2k}+\ell_22^{\frac{m}{2}+3k} \right) \left( \cdot \right) }, \phi_{p, n}^{3k} \right \rangle \right| \nonumber \\
&\cdot& \frac{1}{\left|I_p^{m, 3k} \right|} \bigg| \int_{I_p^{m, 3k}} \int_{I_q^{m, k}}  \underline{f}_1^{\omega_1^{\ell_1}}(x-t) \underline{f}_2^{\omega_2^{\ell_2}}(x+t^2)\,e^{i 2^{\frac{m}{2}+k}t \left(-\ell_1 2^k+\ell_2 2^k \frac{2q}{2^{\frac{m}{2}}}+\frac{3q^2}{2^m} n\right)} dt dx \bigg|.  \nonumber
\end{eqnarray}
Finally, in order to decouple the information carried by $\ell_3$ and $n$, we apply the change variable $n \to n+\ell_3 2^k$ in order to conclude
\begin{eqnarray} \label{20230224main01}
&&|\Lambda_m^k(\vec{f})|\lesssim 2^k \sum_{\substack{p \sim 2^{\frac{m}{2}+2k}\\ q \sim 2^{\frac{m}{2}} \\ n \sim 2^k}} \sum_{(\ell_1, \ell_2, \ell_3) \in \TFC_q}  \left| \left\langle f_3^{\omega_3^{\ell_3}}, \phi_{p+\frac{q^3}{2^m}, \ell_3 2^k+n}^{3k} \right \rangle \right|  \left| \left \langle f_4 e^{i 2^{\frac{m}{2}+3k} \left(\ell_1 2^{-k}+\ell_2 \right) \left( \cdot \right) }, \phi_{p, \ell_3 2^k+n}^{3k} \right \rangle \right| \nonumber \\
&&\quad\cdot \frac{1}{\left|I_p^{m, 3k} \right|} \bigg| \int_{I_p^{m, 3k}} \int_{I_q^{m, k}}  \underline{f}_1^{\omega_1^{\ell_1}}(x-t) \underline{f}_2^{\omega_2^{\ell_2}}(x+t^2) e^{i 2^{\frac{m}{2}+k}t \cdot \frac{3q^2}{2^m}n}  e^{i 2^{\frac{m}{2}+k}t \cdot 2^k \left(-\ell_1 +\ell_2 \frac{2q}{2^{\frac{m}{2}}}+ \ell_3 \frac{3q^2}{2^m}\right)} dt dx \bigg|.
\end{eqnarray}
We now reach the conclusion of this section\footnote{Recall that throughout our paper, for notational and expository simplicity, we ignore the action of the conjugation.}
$\newline$

\noindent\textbf{Key Heuristic 2.} \textit{Observe that in \eqref{20230224main01} above
\begin{equation}\label{KH2}
\left \langle f_4 e^{i 2^{\frac{m}{2}+3k} \left(\ell_1 2^{-k}+\ell_2 \right) \left( \cdot \right) }, \phi_{p, \ell_3 2^k+n}^{3k} \right \rangle=\left \langle f_4, \phi^{3k}_{p, \ell_1 2^{-k}+\ell_2+\ell_32^k+n} \right \rangle.
\end{equation}
In analyzing the Gabor (local Fourier) coefficient in \eqref{KH2} we have
\begin{enumerate}
\item [$\bullet$] $\ell_2$ ranges over an interval of length $\sim 2^{\frac{m}{2}-k}$;
\item [$\bullet$] $n$ ranges over an interval of length $\sim 2^k$.
\end{enumerate}
This naturally invites to divide our discussion into two different cases depending on the regime of dominance:\vspace{-0.1in}
$\newline$
\noindent\textsf{Case I:} $0 \le k \le \frac{m}{4}$; this corresponds to $\ell_2$ dominance.
$\newline$
\noindent\textsf{Case II:} $\frac{m}{4} \le k \le \frac{m}{2}$; this corresponds to $n$ dominance.}

\section{The diagonal term $\Lambda^{D}$: The case $k \ge \frac{m}{2}$}\label{MaindiagKpositive}

With the notations from Section \ref{SDTmotiv} our goal in this section is to prove

\begin{thm} \label{mainthIb}
There exists some $\epsilon_1>0$, such that for any $k \ge \frac{m}{2}$ and $\vec{p}\in {\bf H}^{+}$
\begin{equation}\label{keyestim}
\left|\Lambda_m^{k, I_0^{0, k}}(\vec{f})\right| \lesssim 2^{-\epsilon_1 m} \,\|\vec{f} \|_{L^{\vec{p}}(3\, I_0^{0, k})}\,.
\end{equation}
\end{thm}

\subsection{A spatial sparse--uniform dichotomy}

In what follows we will need to decompose our quadrilinear form $\Lambda_m^k(\vec{f})$ depending on how the information carried by the input functions is spatially distributed. Thus, our analysis will involve a spatial \emph{sparse--uniform} dichotomy.

In order to proceed with this we pick $\mu \in (0, 1)$ sufficiently small (to be chosen later) and define the \emph{sparse index sets} relative to the input functions as follows:
\begin{equation}\label{sparse12}
\calS_{\mu} (f_i):= \left\{ r \sim 2^{m+(i-1)k}: \frac{1}{|I_{r}^{0, ik+m}|}\int_{3 I_{r}^{0, ik+m}} |M f_i|^2 \gtrsim \frac{2^{\mu m}}{|I^k|} \int_{3I^k} |f_i|^2  \right\}, \quad i\in\{1, 2\}\,,
\end{equation}
and
\begin{equation}\label{sparse34}
\calS_{\mu}(f_i):= \left\{ \ell \sim 2^{\frac{m}{2}+2k}: \frac{1}{|I_\ell^{m, 3k}|}\int_{3I_\ell^{m, 3k}} |f_i|^2 \gtrsim \frac{2^{\mu m}}{|I^k|} \int_{3I^k} |f_i|^2 \right\}, \quad i\in\{3, 4\}.
\end{equation}
The \emph{uniform  index sets} $\{\calU_{\mu} (f_i)\}_{i=1}^4$ are introduced in an obvious manner as complements of their respective sparse index sets.

Using the above sets we decompose our input functions accordingly:
\begin{equation}\label{decsu}
f_i=f_i^{\calS_\mu}+f_i^{\calU_\mu}, \quad i\in \{1,\ldots, 4\}
\end{equation}
where
\begin{equation} \label{20230202eq02}
f_i^{\calS_{\mu}}:=\sum_{r \in \calS_\mu(f_i)} \one_{I_r^{0, ik+m}} f_i , \quad  f_i^{\calU_{\mu}}:=\sum_{r \in \calU_\mu(f_i)} \one_{I_r^{0, ik+m}} f_i, \quad i\in\{1, 2\};
\end{equation}
and
\begin{equation} \label{20230202eqq}
f_i^{\calS_\mu}:=\sum_{\ell \in \calS_\mu(f_i)} \one_{I_\ell^{m, 3k}} f_i , \quad \textrm{and} \quad f_i^{\calU_{\mu}}:=\sum_{\ell \in \calU_\mu(f_i)} \one_{I_\ell^{m, 3k}} f_i, \quad i\in\{3, 4\}.
\end{equation}

With these done, taking $\mu_0>\mu>0$ suitable small, we decompose our main form $\Lambda_m^k$ into
\begin{itemize}
\item the \textit{uniform} component $\Lambda_m^{k, \calU_\mu}(\vec{f})$ given by
\begin{equation} \label{uniformeq01}
\Lambda_m^{k, \calU_\mu}(\vec{f}):=\Lambda_m^k \left(f_1^{\calU_{\mu_0}}, f_2^{\calU_{\mu_0}}, f_3^{\calU_\mu}, f_4^{\calS_\mu} \right)+ \Lambda_m^k \left(f_1^{\calU_{\mu_0}}, f_2^{\calU_{\mu_0}}, f_3^{\calS_\mu}, f_4^{\calU_\mu} \right)
 + \Lambda_m^k \left(f_1^{\calU_{\mu_0}}, f_2^{\calU_{\mu_0}}, f_3^{\calU_\mu}, f_4^{\calU_\mu} \right),
\end{equation}
\item the \textit{sparse} component $\Lambda_m^{k, \calS_\mu}(\vec{f})$ given by
\begin{equation} \label{sparseeq01}
\Lambda_m^{k, \calS_\mu}(\vec{f}):=\Lambda_m^k(\vec{f})-\Lambda_m^{k, \calU_\mu}(\vec{f}).
\end{equation}
\end{itemize}

We end this short subsection with the following

\begin{obs}\textsf{[Spatial sparse-uniform dichotomy: philosophy]}\label{CZtype}  This sparse-uniform decomposition of the input functions $\{f_i\}_{i=1}^{4}$ may be regarded as a scale adapted Marcinkiewicz type decomposition\footnote{This heuristic applies to the situation $i\in\{1,2\}$; notice that as opposed to Calderon-Zygmund decomposition, we do not require any mean zero condition for the ``bad" function.} in which
\begin{itemize}
\item the \emph{good part} represented by the uniform component obeys:
\begin{itemize}
\item if $i\in\{1,2\}$ then $f_i^{\calU_{\mu_0}}\in L^\infty$ and satisfies:
\begin{equation} \label{goodpart}
\|f_i^{\calU_{\mu_0}}\|_{L^{\infty}(3 I^k)}\lesssim 2^{\frac{\mu_0 m}{2}}\,2^{\frac{k}{2}}\,\|f_i\|_{L^{2}(3 I^k)}\,;
\end{equation}

\item if $i\in\{3,4\}$ then $f_i^{\calU_{\mu}}\in L^2$ and for any $I\subseteq 3I^k$ dyadic interval with $|I|\geq |I_0^{m, 3k}|$
\begin{equation} \label{goodpart2}
\|f_i^{\calU_{\mu}}\|_{L^{2}(I)}\lesssim \min\left\{2^{\frac{\mu m}{2}}\,\left(\frac{|I|}{|I^k|}\right)^{\frac{1}{2}},\,1\right\}\,\|f_i\|_{L^{2}(3 I^k)}\,;
\end{equation}
\end{itemize}
\item the \emph{bad part} represented by the sparse component lives in a relatively ``small" spatial region:
\begin{equation} \label{badpart}
\max_{i\in\{1,2\}}|\textrm{supp}\,f_i^{\calS_{\mu_0}}|\lesssim 2^{-\mu_0 m}\,|I^k|\qquad\textrm{and}\qquad \max_{i\in\{3,4\}}|\textrm{supp}\,f_i^{\calS_{\mu}}|\lesssim 2^{-\mu m}\,|I^k|\,.
\end{equation}
\end{itemize}
\end{obs}

Indeed, \eqref{goodpart2} and \eqref{badpart} are immediate from \eqref{sparse12} and \eqref{sparse34} while for \eqref{goodpart} we use the first part of Key Heuristic 1 together with: if $i\in\{1,2\}$, $p_i\sim 2^{m+(i-1)k}$ and $x\in I_{p_i}^{0, m+i k}\subseteq \textrm{supp}\,f_i^{\calS_{\mu_0}}$ fixed, then
\begin{equation*}
\left| f_i^{\calU_{\mu_0}} (x) \right|
  \lesssim  \frac{1}{\left|I_{p_i}^{0, m+ i k}\right|} \int_{3I_{p_i}^{0, m+ i k}} M f_i\left(\alpha\right) d\alpha\lesssim \left( \frac{1}{\left|I_{p_i}^{0, m+ i k}\right|} \int_{3 I_{p_i}^{0, m+ i k}} \left|M f_i\left(\alpha\right)\right|^2 d\alpha \right)^{\frac{1}{2}} \lesssim 2^{\frac{k}{2}} \cdot 2^{\frac{\mu_0 m}{2}} \left\|f_i \right\|_{L^2 \left(3I^k\right)}.
\end{equation*}

\subsection{Treatment of the sparse component $\Lambda_m^{k, \calS_\mu}(\vec{f})$} \label{20230330subsec01}

In what follows, we analyze two subcases for the sparse component $\Lambda_m^{k, \calS_\mu}(\vec{f})$:
\begin{enumerate}
\item [(1)] \textsf{Both $f_3$ and $f_4$ are sparse.} This corresponds to the term
$$
\Lambda_m^{k, \calS_\mu, \textnormal{I}}(\vec{f}):=\Lambda_m^k \left(f_1, f_2, f_3^{\calS_\mu}, f_4^{\calS_\mu} \right)\,;
$$
\item [(2)] \textsf{One of the functions $\left\{f_3, f_4 \right\}$ is uniform and one of the functions $\left\{f_1, f_2 \right\}$ is sparse.} This corresponds to the term
\begin{equation*}
\Lambda_m^{k, \calS_\mu, \textnormal{II}}(\vec{f}):=\Lambda_m^{k, \calS_\mu}(\vec{f})-\Lambda_m^{k, \calS_\mu, \textnormal{I}}(\vec{f})\,.
\end{equation*}
\end{enumerate}

\subsubsection{Treatment of $\Lambda_m^{k, \calS_{\mu}, \textnormal{I}}(\vec{f})$}
$\newline$

To begin with, we decompose the regime of $p$ as follows\footnote{Throughout our paper, for notational simplicity we allow the  following slight notational abuse: when decomposing the range of an integer parameter into a (disjoint) union of intervals, we do not write explicitly the restriction of those intervals to the set of integers.}:
$$\left\{p \sim 2^{\frac{m}{2}+2k} \right\}=\bigcup\limits_{s \sim 2^k} \bigcup\limits_{\ell \in \left[s 2^k, (s+1)2^k \right)} \left[2^{\frac{m}{2}}\ell, 2^{\frac{m}{2}}(\ell+1) \right).$$
Using the above decomposition, we can write
\begin{equation} \label{20230217eq01}
\left|\Lambda_m^{k, \calS_\mu, \textnormal{I}}(\vec{f})\right| \lesssim \sum_{s \sim 2^k} \Lambda_m^{k, \calS_\mu, \textnormal{I}, s}(\vec{f}),\qquad\textrm{with}
\end{equation}
\begin{eqnarray*}
&& \Lambda_m^{k, \calS_\mu, \textnormal{I}, s}(\vec{f}):=2^k \sum_{\substack{\ell \in \left[s2^k, (s+1)2^k \right) \\ q, n \sim 2^{\frac{m}{2}}}} \sum_{p \in \left[ 2^{\frac{m}{2}} \ell, 2^{\frac{m}{2}}(\ell+1) \right)} \left| \left\langle f_3^{\calS_\mu}, \phi_{p+\frac{q^3}{2^m}, n}^{3k}  \right\rangle \right| \left|  \left \langle f_4^{\calS_\mu}, \phi_{p, n}^{3k} \right \rangle \right| \\
&& \quad \quad \quad \quad \quad \quad \cdot \frac{1}{\left|I_p^{m, 3k} \right|} \left| \int_{I_p^{m, 3k}} \int_{I_q^{m, k}} f_1(x-t)f_2(x+t^2) e^{3i \cdot 2^{\frac{m}{2}+k}t \left(\frac{q}{2^\frac{m}{2}} \right)^2 n} dt dx \right|.
\end{eqnarray*}
Applying Cauchy-Schwarz in $t$ first, followed by another Cauchy-Schwarz in $n$ and then Parseval, one has
\begin{eqnarray} \label{20230219eq10}
&& \Lambda_m^{k, \calS_{\mu}, \textnormal{I}, s}(\vec{f}) \lesssim 2^k \sum_{\substack{\ell \in \left[s 2^k, (s+1)2^k \right] \\ q \sim 2^{\frac{m}{2}}}} \sum_{p \in \left[2^{\frac{m}{2}}\ell, 2^{\frac{m}{2}}(\ell+1) \right]} \left( \int_{I_{p+\frac{q^3}{2^m}}^{m, 3k}} \left|f_3^{\calS_{\mu}} \right|^2 \right)^{\frac{1}{2}} \left( \int_{I_{p}^{m, 3k}} \left|f_4^{\calS_{\mu}} \right|^2 \right)^{\frac{1}{2}} \nonumber\\
&&  \quad \quad \quad \quad \quad \quad \cdot \frac{1}{\left|I_p^{m, 3k} \right|}  \int_{I_p^{m, 3k}} \left(\int_{3I_{\frac{p}{2^{2k}}-q}^{m, k}} |f_1|^2 \right)^{\frac{1}{2}} \left(\int_{I_q^{m, k}} \left|f_2(x+t^2)\right|^2 dt \right)^{\frac{1}{2}} dx.
\end{eqnarray}
Using now the information that $f_2$ is morally constant on intervals of length $2^{-m-2k}$, we have, up to error terms, that
\begin{equation}\label{f2est}
\int_{I_q^{m, k}} \left|f_2(x+t^2) \right|^2 dt\lesssim
2^k \int_{3I_0^{m, 2k}} \left|f_2 \left(\frac{p}{2^{\frac{m}{2}+3k}}+t+\frac{q^2}{2^{m+2k}} \right) \right|^2 dt\lesssim 2^k \int_{5I_{\frac{q^2}{2^{\frac{m}{2}}}}^{m, 2k}} \left|f_2 \left(\frac{\ell}{2^{3k}}+t \right) \right|^2 dt\,,
\end{equation}
where for the last inequality we used the observation  that as $p$ covers $\left[2^{\frac{m}{2}} \ell, 2^{\frac{m}{2}} \left(\ell+1 \right) \right]$ the term $\frac{p}{2^{\frac{m}{2}+3k}}$ varies within an interval of length $\frac{1}{2^{3k}}\leq |I_0^{m, 2k}|$.

Plugging the above estimate back to \eqref{20230219eq10} and applying Cauchy-Schwarz in $p$ as well as in $q$, we get
\begin{eqnarray*}
\Lambda_m^{k, \calS_\mu, \textnormal{I}, s}(\vec{f}) &\lesssim& 2^{\frac{3k}{2}} \sum_{\ell \in \left[s 2^k, (s+1)2^k \right]} \left( \sum_{p \in \left[2^{\frac{m}{2}}(\ell-1), 2^{\frac{m}{2}}(\ell+2) \right]} \int_{I_{p}^{m, 3k}} \left|f_3^{\calS_{\mu}} \right|^2 \right)^{\frac{1}{2}} \left( \sum_{p \in \left[2^{\frac{m}{2}}\ell, 2^{\frac{m}{2}}(\ell+1) \right]}  \int_{I_{p}^{m, 3k}} \left|f_4^{\calS_{\mu}} \right|^2 \right)^{\frac{1}{2}} \\
&& \quad \quad \quad \quad \quad \quad\quad \quad \cdot \left(\sum_{q \sim 2^{\frac{m}{2}}}\int_{5I_{\frac{\ell}{2^{2k-\frac{m}{2}}}-q}^{m, k}} |f_1|^2 \right)^{\frac{1}{2}} \left( \sum_{q \sim 2^{\frac{m}{2}}}\int_{5I_{\frac{q^2}{2^{\frac{m}{2}}}}^{m, 2k}} \left|f_2 \left(\frac{\ell}{2^{3k}}+t \right) \right|^2 dt \right)^{\frac{1}{2}}\\
&\lesssim& 2^{\frac{3k}{2}} \left\|f_1\right\|_{L^2 \left(3 I^k \right)} \left\|f_2 \right\|_{L^2 \left(3I_s^{0, 2k} \right)} \sum_{\ell \in \left[s 2^k, (s+1)2^k \right]}  \left\|f_3^{\calS_\mu} \right\|_{L^2 \left(3I_\ell^{0, 3k} \right)} \left\|f_4^{\calS_\mu} \right\|_{L^2 \left(I_\ell^{0, 3k} \right)} \\
&\lesssim& 2^{\frac{3k}{2}} \left\|f_1\right\|_{L^2 \left(3 I^k \right)} \left\|f_2 \right\|_{L^2 \left(3I_s^{0, 2k} \right)} \left\|f_3^{\calS_\mu} \right\|_{L^2 \left(3 I_s^{0, 2k} \right)}\left\|f_4^{\calS_\mu} \right\|_{L^2 \left(3 I_s^{0, 2k} \right)}.
\end{eqnarray*}

Substituting the above estimate back to \eqref{20230217eq01}, we have
\begin{equation*}
\left|\Lambda_m^{k, \calS_\mu, \textnormal{I}}(\vec{f})\right| \lesssim 2^{\frac{3k}{2}} \left\|f_1 \right\|_{L^2 \left(3 I^k \right)} \sum_{s \sim 2^k} \left\|f_2 \right\|_{L^2 \left(3I_s^{0, 2k} \right)} \left\|f_3^{\calS_\mu} \right\|_{L^2 \left(3 I_s^{0, 2k} \right)}\left\|f_4^{\calS_\mu} \right\|_{L^2 \left(3 I_s^{0, 2k} \right)}.
\end{equation*}
This immediately implies
\begin{equation}\label{sparse-I}
\left|\Lambda_m^{k, \calS_\mu, \textnormal{I}}(\vec{f})\right| \lesssim 2^{-\frac{\mu}{2} m} \,\|\vec{f} \|_{L^{\vec{p}}(3 I^{k})}\qquad \forall\:\vec{p}\in {\bf H}^{+}\,.
\end{equation}

\subsubsection{Treatment of $\Lambda_m^{k, \calS_{\mu}, \textnormal{II}}(\vec{f})$}
$\newline$

In what follows we only consider a representative of the second sparse component (as all the other terms can be treated similarly), that is $$\Lambda_m^k \left(f_1^{\calS_{\mu_0}}, f_2, f_3^{\calU_\mu}, f_4 \right)\,,$$ and, with a slight notational abuse, in what follows we will continue to refer to the above as $\Lambda_m^{k, \calS_\mu, \textnormal{II}}(\vec{f})$.

Thus, from \eqref{20230131main01}, we have
\begin{equation} \label{20230201eq04}
\left|\Lambda_m^{k, \calS, \textnormal{II}}(\vec{f}) \right| \lesssim  2^k \sum_{\substack{p \sim 2^{\frac{m}{2}+2k} \\ q, n \sim 2^{\frac{m}{2}}}} \left| \left\langle f^{\calU_\mu}_3, \phi_{p+\frac{q^3}{2^m}, n}^{3k}  \right\rangle \right| \left|  \left \langle f_4, \phi_{p, n}^{3k} \right \rangle \right| \cdot \frac{1}{\left|I_p^{m, 3k} \right|}  \int_{I_p^{m, 3k}} \int_{I_q^{m, k}} \left| f^{\calS_{\mu_0}}_1(x-t) \right| \left| f_2(x+t^2) \right|  dt dx.
\end{equation}
Next, for $I \subseteq I^k$ interval, we set $\widetilde{S_{\mu_0}^{3I}}(f_1):=\bigcup\limits_{\substack{r \in \calS_{\mu_0}(f_1) \\ I_r^{0, k+m} \subset 3I}} I_r^{0, k+m}$.  Then, applying Cauchy-Schwarz twice and Parseval, one has
\begin{eqnarray*}
&& \left|\Lambda_m^{k, \calS, \textnormal{II}}(\vec{f}) \right| \lesssim  2^k \sum_{\substack{p \sim 2^{\frac{m}{2}+2k} \\ q \sim 2^{\frac{m}{2}}}} \left\| f^{\calU_\mu}_3 \right\|_{L^2 \left(I_{p+\frac{q^3}{2^m}}^{m, 3k} \right)} \left\|f_4 \right\|_{L^2 \left(I_p^{m, 3k} \right)} \nonumber  \\
&& \quad \quad \cdot \left|\widetilde{S_{\mu_0}^{3 I_{\frac{p}{2^{2k}}-q}^{m, k}}}(f_1) \right|^{\frac{1}{2}} \left(\frac{1}{\left|I_0^{m,3k}\right|} \int_{I_p^{m, 3k}} \int_{I_q^{m, k}} \left|f_1^{\calS_{\mu_0}} \left( x-t \right) \right|^2 \left| f_2 \left(x+t^2 \right) \right|^2 dt dx \right)^{\frac{1}{2}} .
\end{eqnarray*}
Exploiting now Observation \ref{CZtype}--see \eqref{goodpart2}, we have
$\left\| f^{\calU_\mu}_3 \right\|_{L^2 \left(I_{p+\frac{q^3}{2^m}}^{m, 3k} \right)} \lesssim \left( \frac{2^{\mu m}}{2^{\frac{m}{2}+2k}} \right)^{\frac{1}{2}} \left\|f_3 \right\|_{L^2(3 I^k)}$ and hence, via a Cauchy-Schwarz in $q$ and then in $p$ followed by the application of \eqref{badpart}, we have
\begin{eqnarray} \label{sparse-II}
&&\left|\Lambda_m^{k, \calS, \textnormal{II}}(\vec{f}) \right|\lesssim 2^k \cdot \left( \frac{2^{\mu m}}{2^{\frac{m}{2}+2k}} \right)^{\frac{1}{2}} \sum_{p \sim 2^{\frac{m}{2}+2k}} \left\|f_3 \right\|_{L^2(3 I^k)} \left\|f_4 \right\|_{L^2 \left(I_p^{m, 3k} \right)} \cdot   \left( \sum_{q \sim 2^{\frac{m}{2}}}  \left|\widetilde{S_{\mu_0}^{3I_{\frac{p}{2^{2k}}-q}^{m, k}}}(f_1) \right| \right)^{\frac{1}{2}}\nonumber \\
&& \quad \quad \cdot \left( \sum_{q \sim 2^{\frac{m}{2}}} \frac{1}{\left|I_0^{m,3k}\right|} \int_{I_p^{m, 3k}} \int_{I_q^{m, k}} \left|f_1^{\calS_{\mu_0}} \left( x-t \right) \right|^2 \left| f_2 \left(x+t^2 \right) \right|^2 dtdx \right)^{\frac{1}{2}}\nonumber \\
&& \lesssim 2^k \cdot \left(\frac{2^{\mu m}}{2^{\frac{m}{2}+2k}} \right)^{\frac{1}{2}} \cdot \left| \widetilde{S_{\mu_0}^{3I^k}(f_1)} \right|^{\frac{1}{2}} \|f_3\|_{L^2(3 I^k)} \left(\sum_{p \sim 2^{\frac{m}{2}+2k}} \left\|f_4 \right\|^2_{L^2 \left(I_p^{m, 3k} \right)} \right)^{\frac{1}{2}} \nonumber \\
&& \quad \quad \quad \cdot \left( \sum_{p \sim 2^{\frac{m}{2}+2k}} \frac{1}{\left|I_0^{m,3k}\right|} \int_{I_p^{m, 3k}} \int_{I^k} \left|f_1^{\calS_{\mu_0}} \left( x-t \right) \right|^2 \left| f_2 \left(x+t^2 \right) \right|^2 dt dx \right)^{\frac{1}{2}} \nonumber
\end{eqnarray}

\begin{eqnarray}
&&=2^{\frac{m}{4}+\frac{5k}{2}} \cdot \left(\frac{2^{\mu m}}{2^{\frac{m}{2}+2k}} \right)^{\frac{1}{2}} \cdot \left| \widetilde{S_{\mu_0}^{3I^k}(f_1)} \right|^{\frac{1}{2}} \|f_3\|_{L^2(3I^k)} \left\|f_4 \right\|_{L^2(I^k)} \, \cdot \left( \int_{I^k} \int_{I^k} \left|f_1^{\calS_{\mu_0}} \left( x-t \right) \right|^2 \left| f_2 \left(x+t^2 \right) \right|^2 dt dx \right)^{\frac{1}{2}} \nonumber  \\
&&\lesssim2^{\frac{m}{4}+\frac{5k}{2}}\cdot \left(\frac{2^{\mu m}}{2^{\frac{m}{2}+2k}} \right)^{\frac{1}{2}} \cdot 2^{-\frac{k}{2}} \cdot \left( \frac{ \left| \widetilde{S_{\mu_0}^{3I^k}(f_1)} \right| }{|I^k|} \right)^{\frac{1}{2}}\cdot \prod_{\ell=1}^4 \left\|f_\ell\right\|_{L^2(3I^k)} \lesssim 2^k \cdot 2^{\frac{(\mu-\mu_0)m}{2}}  \cdot \prod_{\ell=1}^4 \left\|f_\ell\right\|_{L^2(3I^k)}\nonumber,
\end{eqnarray}
which further implies (recall that $\mu_0>\mu$)
\begin{equation}\label{sparse-II}
\left|\Lambda_m^{k, \calS_\mu, \textnormal{II}}(\vec{f})\right| \lesssim 2^{-\frac{\mu_0-\mu}{2} m} \,\|\vec{f} \|_{L^{\vec{p}}(3I^{k})}\qquad \forall\:\vec{p}\in {\bf H}^{+}\,.
\end{equation}

\subsection{Treatment of the uniform component $\Lambda_m^{k, \calU_\mu}(\vec{f})$}\label{DiagUnifKlarge}

In the $k\geq \frac{m}{2}$ setting, this is the main part of our quadrilinear form and the most difficult to treat.

As before, for notational convenience we will identify the entire uniform component with one of its representatives, that is
$$\Lambda_m^{k, \calU_\mu}(\vec{f})\equiv\Lambda_m^k \left(f_1^{\calU_{\mu_0}}, f_2^{\calU_{\mu_0}}, f_3^{\calU_\mu}, f_4 \right)\,.$$

We start by noticing from \eqref{20230131main01} that one has
\begin{eqnarray} \label{20230201eq11}
\left|\Lambda_m^{k, \calU_\mu}(\vec{f}) \right| &\lesssim& 2^k \sum_{\substack{p \sim 2^{\frac{m}{2}+2k} \\ q, n \sim 2^{\frac{m}{2}}}} \left| \left\langle f^{\calU_\mu}_3, \phi_{p+\frac{q^3}{2^m}, n}^{3k}  \right\rangle \right| \left|  \left \langle f_4, \phi_{p, n}^{3k} \right \rangle \right|   \\
&\cdot& \frac{1}{\left|I_p^{m, 3k} \right|} \left| \int_{I_p^{m, 3k}} \int_{I_q^{m, k}} f^{\calU_{\mu_0}}_1(x-t)f^{\calU_{\mu_0}}_2(x+t^2) e^{3i \cdot 2^{\frac{m}{2}+k}t \left(\frac{q}{2^\frac{m}{2}} \right)^2 n} dt dx \right| .\nonumber
\end{eqnarray}
Next, we perform a discretization of the range of $p$
\begin{equation} \label{20230219eq50}
\left\{p \sim 2^{\frac{m}{2}+2k} \right\}=\bigcup_{\ell \sim 2^{2k}} \left[2^{\frac{m}{2}} \ell, 2^{\frac{m}{2}}(\ell+1) \right].
\end{equation}
Applying now \eqref{20230219eq50} and then Cauchy-Schwarz in $p$ and $q$, we have
\begin{eqnarray*}
\left| \Lambda_m^{k, \calU_\mu}(\vec{f}) \right|  &\lesssim& 2^k \sum_{\substack{\ell \sim 2^{2k} \\ n \sim 2^{\frac{m}{2}}}} \left( \sum_{p=2^{\frac{m}{2}} \ell}^{2^{\frac{m}{2}}(\ell+1)} \sum_{q \sim 2^{\frac{m}{2}}} \left| \left\langle f^{\calU_\mu}_3, \phi_{p+\frac{q^3}{2^m}, n}^{3k}  \right\rangle \right|^2\left|  \left \langle f_4, \phi_{p, n}^{3k} \right \rangle \right|^2 \right)^{\frac{1}{2}}  \\
&\cdot& \left( \sum_{p=2^{\frac{m}{2}} \ell}^{2^{\frac{m}{2}}(\ell+1)} \sum_{q \sim 2^{\frac{m}{2}}} \frac{1}{\left| I_p^{m, 3k} \right|} \int_{I_p^{m, 3k}} \left| \int_{I_q^{m, k}} f^{\calU_{\mu_0}}_1(x-t)f^{\calU_{\mu_0}}_2(x+t^2) e^{3i \cdot 2^{\frac{m}{2}+k}t \left(\frac{q}{2^\frac{m}{2}} \right)^2 n} dt \right|^2 dx \right)^{\frac{1}{2}} \\
&\lesssim& 2^k \sum_{\substack{\ell \sim 2^{2k} \\ n \sim 2^{\frac{m}{2}}}} \left( \sum_{p=2^{\frac{m}{2}}(\ell-1)}^{2^{\frac{m}{2}}(\ell+2)} \left| \left \langle f_3^{\calU_\mu}, \phi_{p, n}^{3k} \right \rangle \right |^2 \right)^{\frac{1}{2}} \left( \sum_{p=2^{\frac{m}{2}} \ell}^{2^{\frac{m}{2}}(\ell+1)} \left| \left \langle f_4, \phi_{p, n}^{3k} \right \rangle \right |^2 \right)^{\frac{1}{2}} \\
&\cdot& \left( \sum_{q \sim 2^{\frac{m}{2}}} \frac{1}{\left| I_0^{m,3k} \right|} \int_{I_\ell^{0, 3k}} \left| \int_{I_q^{m, k}} f^{\calU_{\mu_0}}_1(x-t)f^{\calU_{\mu_0}}_2(x+t^2) e^{3i \cdot 2^{\frac{m}{2}+k}t \left(\frac{q}{2^\frac{m}{2}} \right)^2 n} dt \right|^2 dx \right)^{\frac{1}{2}} \\
\end{eqnarray*}
Next, applying H\"older's inequality in the form $l^2\times l^2\times l^{\infty}$ relative to $n$, we have
\begin{eqnarray*}
 \left| \Lambda_m^{k, \calU_\mu} (\vec{f}) \right| &\lesssim& 2^k \sum_{\ell \sim 2^{2k}} \left( \sum_{n \sim 2^{\frac{m}{2}}} \sum_{p=2^{\frac{m}{2}}(\ell-1)}^{2^{\frac{m}{2}}(\ell+2)} \left| \left \langle f_3^{\calU_\mu}, \phi_{p, n}^{3k} \right \rangle \right |^2 \right)^{\frac{1}{2}} \left( \sum_{n \sim 2^{\frac{m}{2}}} \sum_{p=2^{\frac{m}{2}} \ell}^{2^{\frac{m}{2}}(\ell+1)} \left| \left \langle f_4, \phi_{p, n}^{3k} \right \rangle \right |^2 \right)^{\frac{1}{2}} \\
&\cdot& \left( \frac{1}{\left|I_0^{m,3k} \right|}  \sup_{n \sim 2^{\frac{m}{2}}} \int_{I_\ell^{0, 3k}} \left( \sum_{q \sim 2^{\frac{m}{2}}}  \left| \int_{I_q^{m, k}} f^{\calU_{\mu_0}}_1(x-t)f^{\calU_{\mu_0}}_2(x+t^2) e^{3i \cdot 2^{\frac{m}{2}+k}t \left(\frac{q}{2^\frac{m}{2}} \right)^2 n} dt \right|^2 \right) dx \right)^{\frac{1}{2}}.
\end{eqnarray*}
Let us denote now by $\lambda$ the measurable function that assumes the supremum in the last term above. Then
\begin{enumerate}\label{lambdainit}
    \item [$\bullet$] $\lambda: I^k \mapsto \left[2^{\frac{m}{2}}, 2^{\frac{m}{2}+1} \right] \cap \Z$;
    \item [$\bullet$] $\lambda$ is constant on intervals $\left\{I_\ell^{0, 3k} \right\}_{\ell \sim 2^{2k}}$.
\end{enumerate}
With these together with Parseval, we have the key relation
\begin{eqnarray*}
\left| \Lambda_m^{k, \calU_\mu}(\vec{f})   \right| &\lesssim& 2^k \sum_{\ell \sim 2^{2k}} \left\| f_3^{\calU_\mu} \right\|_{L^2 \left(3 I_\ell^{0, 3k} \right)}  \left\| f_4 \right\|_{L^2 \left(I_\ell^{0, 3k} \right)} \\
&\cdot& \left( \frac{1}{\left|I_0^{m,3k} \right|} \int_{I_\ell^{0, 3k}} \left( \sum_{q \sim 2^{\frac{m}{2}}}  \left| \int_{I_q^{m, k}} f^{\calU_{\mu_0}}_1(x-t)f^{\calU_{\mu_0}}_2(x+t^2) e^{3i \cdot 2^{\frac{m}{2}+k}t \left(\frac{q}{2^\frac{m}{2}} \right)^2 \lambda(x)} dt \right|^2 \right) dx \right)^{\frac{1}{2}}.
\end{eqnarray*}
From \eqref{goodpart2} we have $\left\|f_3^{\calU_\mu} \right\|_{L^2 \left(3 I_\ell^{0, 3k} \right)} \lesssim \frac{2^{\frac{\mu m}{2}}}{2^k} \left\|f_3\right\|_{L^2 \left(3 I^k \right)}$ which paired with a Cauchy--Schwarz in $\ell$ gives
\begin{eqnarray*}
\left| \Lambda_m^{k, \calU_\mu} (\vec{f})  \right| &\lesssim& 2^k \cdot 2^{\frac{\mu m}{2}} \|f_3\|_{L^2 \left(3 I^k \right)}  \cdot \left( \sum_{\ell \sim 2^{2k}} \left\| f_4 \right\|^2_{L^2 \left(I_\ell^{0, 3k} \right)} \right)^{\frac{1}{2}} \\
&\cdot& \left( \frac{1}{\left|I_0^{m,k} \right|} \sum_{\ell \sim 2^{2k}} \int_{I_\ell^{0, 3k}} \left( \sum_{q \sim 2^{\frac{m}{2}}}  \left| \int_{I_q^{m, k}} f^{\calU_{\mu_0}}_1(x-t)f^{\calU_{\mu_0}}_2(x+t^2) e^{3i \cdot 2^{\frac{m}{2}+k}t \left(\frac{q}{2^\frac{m}{2}} \right)^2 \lambda(x)} dt \right|^2 \right) dx \right)^{\frac{1}{2}}.
\end{eqnarray*}
Thus, we conclude that
\begin{equation}\label{unifc}
\left|\Lambda_m^{k, \calU_\mu}(\vec{f}) \right| \lesssim 2^k \cdot 2^{\frac{\mu m}{2}} \left\|f_3 \right\|_{L^2 \left(3 I^k \right)} \left\|f_4 \right\|_{L^2 \left(I^k \right)} \textbf{L}_{m,k}(f_1, f_2),
\end{equation}
where
\begin{equation}\label{defL}
\left(\textbf{L}_{m,k}(f_1, f_2)\right)^2 := \frac{1}{\left|I_0^{m,k} \right|} \int_{I^{k}} \left( \sum_{q \sim 2^{\frac{m}{2}}}  \left| \int_{I_q^{m, k}} f^{\calU_{\mu_0}}_1(x-t)f^{\calU_{\mu_0}}_2(x+t^2) e^{3i \cdot 2^{\frac{m}{2}+k}t \left(\frac{q}{2^\frac{m}{2}} \right)^2 \lambda(x)} dt \right|^2 \right) dx.
\end{equation}

We are now left with estimating the term $\textbf{L}_{m,k}^2(f_1, f_2)$. This is the subject of our main theorem below:

\begin{thm} \label{20230202thm01}
Let $k \ge \frac{m}{2}$ and recall \eqref{sparse12}--\eqref{20230202eq02}. Then there exists an absolute $\epsilon_2>0$, such that for any $\lambda (\cdot): I^k \mapsto \left[2^{\frac{m}{2}}, 2^{\frac{m}{2}+1} \right] \cap \Z$ measurable function that is constant on intervals $\left\{I_\ell^{0, 3k} \right\}_{\ell \sim 2^{2k}}$ one has
\begin{equation} \label{20230202eq01}
\textrm{\textbf{L}}_{m,k}(f_1, f_2) \lesssim 2^{-\epsilon_2 m} \left\|f_1 \right\|_{L^2 \left(3 I^k \right)} \left\|f_2 \right\|_{L^2 \left(3 I^k \right)}.
\end{equation}
\end{thm}

The proof of this crucial theorem relies on a novel and powerful approach that could prove useful in a number of other settings; it is split in two parts that can be informally summarized as follows:
\begin{itemize}
\item  The first part addresses the \emph{local} behavior of $\lambda$ treating the situation when $\lambda$  has ``high" oscillations on ``small" intervals; this is done via a \emph{bootstrap algorithm} relying on Key Heuristic 3 and Proposition \ref{constantprop} followed by Key Heuristic 4 and Proposition \ref{mainprop01}.

\item The second part addresses the \emph{global} behavior of $\lambda$: based on the first part one reduces matters to the situation when $\lambda$ is morally constant on intervals of length $2^{-k-\frac{m}{2}}$ representing precisely the linearizing scale in $t$ derived via LGC. The cancellation is then obtained at the global level based on the \emph{structure of the time-frequency correlation set}; this is the subject of Proposition \ref{mainprop02}.
\end{itemize}

\medskip
\noindent\textsf{Part 1}. \textsf{The local behavior of $\lambda$: treating the high variational component}
$\newline$

As it turns out, this first part of our proof revolves around the behavior of the joint maximal Fourier coefficients\footnote{Throughout the reminder of this section, for notational simplicity, we set $f_1=f_1^{\calU_{\mu_0}}$ and $f_2=f_2^{\calU_{\mu_0}}$.}
$$\left\{\left| \int_{I_q^{m, k}} f_1(\cdot-t)f_2(\cdot+t^2) e^{3i 2^{\frac{m}{2}+k}t \left(\frac{q}{2^\frac{m}{2}} \right)^2 \lambda(\cdot)} dt \right|\right\}_{q \sim 2^{\frac{m}{2}}}$$
at two critical scales $x\in I_p^{0, k+w},\,p \sim 2^w$: in the first stage for $w=k$ and then in the second one for $w=\frac{m}{2}$.

Thus, setting from now on $w\in\{k,\frac{m}{2}\}$, in view of the above, it becomes natural to rewrite \eqref{defL} as
\begin{equation} \label{Lvmk}
\left(\textbf{L}_{m,k}(f_1, f_2)\right)^2=\frac{1}{2^w} \sum_{p \sim 2^w} V_{m, k}^{p,w}(f_1, f_2)\,,
\end{equation}
where
\begin{equation} \label{vmk}
V_{m, k}^{p,w}  \left(f_1, f_2 \right):=2^{\frac{m}{2}+k+w} \sum_{q \sim 2^{\frac{m}{2}}} \int_{I_p^{0, k+w}} \left| \int_{I_q^{m, k}} f_1(x-t) f_2(x+t^2) e^{3i 2^{\frac{m}{2}+k} \cdot \frac{q^2}{2^m} \lambda(x) t} dt \right|^2 dx.
\end{equation}
With these specifications, we are now ready to describe our strategy for addressing Part 1:

\begin{obs} \label{constancy}[\textsf{Constancy propagation for $\lambda$}]  In what follows we will run the following bootstrap algorithm:
\begin{itemize}
\item from our initial hypothesis--see \eqref{lambdainit}--we know that $\lambda$ is constant on intervals $\left\{I_\ell^{0, 3k} \right\}_{\ell \sim 2^{2k}}$. We will then apply Propositions \ref{constantprop} and \ref{mainprop01} below for $w=k$ in order to reduce matters (see also Observation \ref{boots2}) to the case when $\lambda$ is constant on intervals $\left\{I_\ell^{0, 2k} \right\}_{\ell \sim 2^{k}}$ (notice that in our present regime $|I_\ell^{0, 2k}|\geq |I_p^{m, k}|$).

\item feeding this conclusion back to Propositions \ref{constantprop} and \ref{mainprop01}, this time for $w=\frac{m}{2}$, we will be able to put ourselves in the situation when $\lambda$ is constant on all intervals $\left\{I_p^{m, k} \right\}_{p \sim 2^{\frac{m}{2}}}$ which will become the underlying hypothesis/starting point for Part 2.
\end{itemize}
\end{obs}

The first step for accomplishing the above strategy is based on the following

\medskip
\noindent\textbf{Key Heuristic 3.} \textit{Let $w\in\{k,\,\frac{m}{2}\}$ and assume that
\begin{equation} \label{constancy666}
\lambda(\cdot)\:\textrm{is constant on intervals of length}\:2^{-2k-w}\,.
\end{equation}
Then, for each $q \sim 2^{\frac{m}{2}}$ and $p \sim 2^w$, the continuous function
\begin{equation} \label{20220202eq04}
 F: I_0^{0, k+w}\ni\alpha \longmapsto \int_{I_p^{0, k+w}} \left| \int_{I_q^{m, k}} f_1(x-t+\alpha)f_2(x+t^2)e^{3i 2^{\frac{m}{2}+k} \cdot \frac{q^2}{2^m} \cdot \lambda(x) t} dt \right|^2 dx
\end{equation}
is essentially constant. This will allow us later to remove the joint $x$-dependence of the arguments of $f_1$ and $f_2$, a key step in the constancy propagation of $\lambda$ from the initial scale $2^{-3k}$ to the scale $2^{-2k}$ (for $w=k$) and then finally to the scale $2^{-k-\frac{m}{2}}$ (for $w=\frac{m}{2}$).}
$\newline$

Now let us turn to the rigorous proof of the above heuristic.

\begin{prop} \label{constantprop}
Fix $k, m \in \N$, with $k\geq \frac{m}{2}\geq 100$, $w\in\{k,\,\frac{m}{2}\}$, $q\sim 2^{\frac{m}{2}}$ and $p\sim 2^w$. Let $\del_w>2\mu_0$ be sufficiently small. Then, for any $\alpha, \beta\in I_0^{0, k+w+\del_w m}$, the continuous function $F (\cdot)$ defined in \eqref{20220202eq04} obeys
\begin{equation} \label{20220202eq03}
\left|F(\alpha)-F(\beta) \right| \lesssim 2^{-(\del_w-2\mu_0) m} \cdot 2^{-m-k-w} \cdot \left\|f_1 \right\|_{L^2 \left(3 I^k \right)}^2 \left\|f_2 \right\|_{L^2 \left(3 I^k \right)}^2\,.
\end{equation}
\end{prop}

\begin{proof}
Applying the change of variable $t \to t+\alpha$, we write
\begin{equation} \label{20230203eq02}
F(\alpha)=\int_{I_p^{0, k+w}} \left| \int_{I_q^{m, k}-\alpha} f_1(x-t)f_2\left(x+(t+\alpha)^2 \right)e^{3i 2^{\frac{m}{2}+k} \cdot \frac{q^2}{2^m} \cdot \lambda(x) t} dt \right|^2 dx.
\end{equation}
Since $\alpha \in I_0^{0,k+w+\del_w m}$ we have $|\alpha|^2 \le 2^{-2k-2w-2\del_w m}$ and thus
$(t+\alpha)^2=t^2+2\alpha t_0+ O \left(2^{-\frac{m}{2}-2k-w-\del_w m} \right)$ where $t_0$ is the center of the interval\footnote{Note that $I_q^{m, k}-\alpha \subseteq 3 I_q^{m, k}$} $I_q^{m, k}$. Recalling now that $f_2$ is morally constant on intervals of length $2^{-m-2k}$ (and $2^{-\frac{m}{2}-2k-w-\del_w m}<2^{-m-2k}$), we have\footnote{Recall Observation \ref{ETerm}.}
\begin{equation} \label{20200203eq01}
f_2\left(x+(t+\alpha)^2 \right) \approx f_2 \left(x+t^2+2\alpha t_0 \right).
\end{equation}
Plugging \eqref{20200203eq01} back to \eqref{20230203eq02} and applying another change variable $x \to x-2\alpha t_0$, we can write
\begin{eqnarray} \label{20230203eq11}
F(\alpha) &\approx& \int_{I_p^{0, k+w}+2\alpha t_0} \left| \int_{I_q^{m, k}-\alpha} f_1(x-t-2\alpha t_0) f_2(x+t^2) e^{3i 2^{\frac{m}{2}+k} \cdot \frac{q^2}{2^m} \cdot \lambda \left(x-2\alpha t_0 \right) t} dt \right|^2 dx \nonumber \\
&\approx& \int_{I_p^{0, k+w}+2\alpha t_0} \left| \int_{I_q^{m, k}-\alpha} f_1(x-t) f_2(x+t^2) e^{3i 2^{\frac{m}{2}+k} \cdot \frac{q^2}{2^m} \cdot \lambda \left(x-2\alpha t_0 \right) t} dt \right|^2 dx,
\end{eqnarray}
where in the last equation above, we used that $f_1$ is morally constant on intervals of length $2^{-m-k}$, while
\begin{equation}\label{vars}
|\alpha t_0| \sim 2^{-2k-w-\del_w m}\,.
\end{equation}
Next, appealing to the same \eqref{vars}, we have\footnote{Throughout this paper, given $A,B$ two measurable sets on the real line, we will use the notation $A\Delta B:=(A\setminus B)\cup (B\setminus A)$ for the symmetric difference of the pair $(A,B)$.}
\begin{enumerate}
    \item [(1)] $\lambda(x-2\alpha t_0)\approx\lambda(x)$ since from our hypothesis $\lambda$ is constant on intervals of length $2^{-2k-w}$;

    \item [(2)] $\left| \left(I_p^{0, k+w}+2\alpha t_0  \right)\:\:\Delta \:\:I_p^{0, k+w} \right| \lesssim 2^{-\del_w m-k} \left|I_p^{0, k+w} \right| $;

    \item [(3)] $ \left| \left(I_q^{m, k}-\alpha \right)\:\:\Delta\:\:I_q^{m, k}  \right| \lesssim 2^{-\del_w m} \left|I_q^{m, k} \right| $.
\end{enumerate}
Using the above itemization the desired conclusion follows via standard reasonings. For reader's convenience we will exemplify only the step corresponding to the first item, that is we sketch the proof of
\begin{equation} \label{20230203eq12}
|F_1(\alpha)-F(\alpha)| \lesssim 2^{-\left(\del_w-2\mu_0 \right) m} \cdot 2^{-m-k-w} \cdot \left\|f_1\right\|^2_{L^2 \left(3 I^k\right)} \left\|f_2 \right\|^2_{L^2 \left(3 I^k \right)}\,,
\end{equation}
where
$$F_1(\alpha):=\int_{I_p^{0,k+w}+2\alpha t_0} \left| \int_{I_q^{m, k}-\alpha} f_1(x-t) f_2(x+t^2) e^{3i 2^{\frac{m}{2}+k} \cdot \frac{q^2}{2^m} \cdot \lambda \left(x \right) t} dt \right|^2 dx\,.$$

For proving \eqref{20230203eq12} it is enough to note that the functions $\lambda(x)$ and $\lambda(x-2\alpha t_0)$ may only differ on a subset $\frakJ \subset I_p^{0, k+w}+2\alpha t_0$ whose size obeys\footnote{More precisely, $\frakJ$ is a union of at most $2^k$ many intervals with length $\lesssim 2^{-\del_w m-2k-w}$.}
\begin{equation}\label{setI}
|\frakJ| \lesssim 2^{-\del_w m} |I_p^{0, k+w}|\,.
\end{equation}
Therefore, via a Cauchy--Schwarz argument we immediately deduce
\begin{equation}\label{20230203eq13}
\left|F_1(\alpha)-F(\alpha) \right| \lesssim \int_{\frakJ} \left(\int_{I_q^{m, k}-\alpha} |f_1(x-t)|^2 dt \right) \left(\int_{I_q^{m, k}-\alpha} \left|f_2(x+t^2)\right|^2 dt \right) dx.
\end{equation}
Since both $f_1$ and $f_2$ are uniform, based on Observation \ref{CZtype}, the desired \eqref{20230203eq12} follows now from
\begin{equation*}
\int_{I_q^{m, k}-\alpha} |f_1(x-t)|^2 dt \lesssim 2^{(\mu_0-\frac{1}{2})m}\, \left\|f_1 \right\|_{L^2 \left(3 I^k \right)}^2
\end{equation*}
and
\begin{equation*}
\int_{I_q^{m, k}-\alpha}  \left| f_2\left(x+t^2 \right) \right|^2 dt\lesssim 2^{(\mu_0-\frac{1}{2})m}\,\left\|f_2 \right\|^2_{L^2 \left(3 I^k \right)}.
\end{equation*}
We leave the other remaining details to the interested reader.
\end{proof}

It is now useful to record the following straightforward corollary of Proposition \ref{constantprop} above:
\begin{cor} \label{20230203cor01}
With the previous notations and conventions, for each $p \sim 2^w$ and $q \sim 2^{\frac{m}{2}}$, one has
\begin{eqnarray} \label{20230203eq31}
&& \qquad\qquad\qquad\int_{I_p^{0, k+w}} \left| \int_{I_q^{m, k}} f_1(x-t)f_2(x+t^2)e^{3i 2^{\frac{m}{2}+k} \cdot \frac{q^2}{2^m} \cdot \lambda(x) t} dt \right|^2 dx \nonumber \\
&& \quad  \quad \lesssim \frac{2^{\del_w m}}{\left|I_0^{0, k+w} \right|} \int_{3I_0^{0, k+w}} \int_{I_0^{0, k+w}} \bigg| \int_{I_q^{m, k}} f_1 \left(\alpha-t+\frac{p}{2^{k+w}} \right)\,f_2 \left(x+t^2+\frac{p}{2^{k+w}} \right) e^{3i 2^{\frac{m}{2}+k} t \cdot \frac{q^2}{2^m} \cdot \lambda \left(x+\frac{p}{2^{k+w}} \right)}\,dt \bigg|^2 dx d\alpha \nonumber\\
&& \quad\quad+\,O\left(2^{-(\del_w-2\mu_0) m} \cdot 2^{-m-k-w} \cdot \left\|f_1 \right\|_{L^2 \left(3 I^k \right)}^2 \left\|f_2 \right\|_{L^2 \left(3 I^k \right)}^2\right)\nonumber.
\end{eqnarray}
with all the implicit constants involved above being independent of  the choice of $p, q, m$, and $k$.
\end{cor}

With these completed, we are now ready to state the essential step in addressing Part 1:

$\newline$ \noindent\textbf{Key Heuristic 4.} \textit{Our strategy relies on the following \textbf{dichotomy}:
$\newline$
Fix $p \sim 2^w$ and recall \eqref{vmk}. Then, the following holds:
\begin{enumerate}
    \item [$\bullet$] either $\lambda (\cdot)$ is morally constant on $I_p^{0, k+w}$, i.e.
    \begin{equation} \label{20230205eq01}
    \left| \lambda(x)-\lambda(y) \right| \lesssim 1 \quad \textrm{for ``most" of } \ x, y \in I_p^{0, k+w};
    \end{equation}
    \item [$\bullet$] or there exists some $\epsilon_3^w>0$, such that
    \begin{equation} \label{20230205eq02}
    \left|V_{m, k}^{p,w} \left(f_1, f_2 \right) \right| \lesssim 2^{-2\epsilon_3^w m} \left\|f_1 \right\|_{L^2 \left(3 I^k \right)}^2 \left\| f_2 \right\|_{L^2 \left(3I^{k} \right)}^2\,.
    \end{equation}
\end{enumerate}}
\medskip

In view of \eqref{Lvmk} the above heuristic together with Proposition \ref{mainprop02} below may be interpreted as creating a correspondence between the \emph{structure} of the linearizing function $\lambda$ and the ability to control the \emph{pointwise} or the \emph{average} behavior of $V_{m, k}^{p,w}(f_1, f_2)$.

The remainder of this section will deal with the rigorous proof of Key Heuristic 4. Since such an enterprise requires in particular a precise quantification of \eqref{20230205eq01} our first step will be to gain more information on the \emph{structure} of the linearizing function $\lambda$:
$\newline$

Depending on the value of $w\in\{k,\frac{m}{2}\}$, we fix $\epsilon_{0}^{w}>0$ small (to be chosen later--see Observation \ref{bkI}) and divide the range of $\lambda$ into smaller intervals of length $2^{\epsilon_{0}^{w} m}$, \emph{i.e.}
$$
\textrm{Range of} \ \lambda(\cdot)=\bigcup_{l \sim 2^{\frac{m}{2}-\epsilon_{0}^{w} m}} \left[l2^{\epsilon_{0}^{w} m}, \left(l+1 \right) 2^{\epsilon_{0}^{w} m}  \right].
$$
Next, we partition $I_p^{0, k+w}$ in sets corresponding to the pre-images of $\lambda$ of the above subintervals, \textit{i.e.}, for each $l \sim 2^{\frac{m}{2}-\epsilon_{0}^{w} m}$, we consider\footnote{By the definition of $\lambda$, $\frakM_l \left(I_p^{0, k+w} \right)$ consists of a union of intervals of the form $I_{\ell'}^{0, 2k+w}$ with $\ell'\sim 2^{k+w}$ (see, \eqref{constancy666}).}
$$
\frakM_l \left(I_p^{0, k+w} \right):= \left\{ x \in I_p^{0, k+w}: \lambda(x) \in \left[l 2^{\epsilon_{0}^{w} m}, \left(l+1 \right) 2^{\epsilon_{0}^{w} m}  \right] \right\}.
$$
We now define
\begin{itemize}
\item the \emph{heavy} set of $l$-indices:
$$
\calH_p^w:= \left\{l \sim 2^{\frac{m}{2}-\epsilon_{0}^{w} m}: \left| \frakM_l \left(I_p^{0, k+w} \right) \right| \gtrsim \left| I_p^{0, k+w} \right| 2^{-\epsilon_{0}^{w} m} \right\}\,;
$$
\item the \emph{light} set of $l$-indices:
$$
\calL_p^w:=\left\{l \sim 2^{\frac{m}{2}-\epsilon_{0}^{w} m}: \left| \frakM_l \left(I_p^{0, k+w} \right) \right| \lesssim \left| I_p^{0, k+w} \right| 2^{-\epsilon_{0}^{w} m} \right\}.
$$
\end{itemize}
Moreover, we let
\begin{equation} \label{20230205eq11}
\widetilde{\calH}_p^{w}:=\bigcup_{l \in \calH_p^w} \frakM_l \left(I_p^{0, k+w} \right) \quad \textrm{and} \quad  \widetilde{\calL}_p^{w}:=\bigcup_{l \in \calL_p^w} \frakM_l \left(I_p^{0, k+w} \right).
\end{equation}
It is now immediate to notice that
\begin{equation} \label{partIp}
I_p^{0, k+w}=\widetilde{\calH}_p^{w} \cup \widetilde{\calL}_p^{w}\quad\textrm{and}\quad \#\calH_p^w \lesssim 2^{ \epsilon_{0}^{w} m}.
\end{equation}

Next, using \eqref{20230205eq11}, we decompose
\begin{equation} \label{20230205eq12}
V_{m, k}^{p,w}(f_1, f_2)=V_{m, k}^{p, \calH_p^w}(f_1, f_2)+V_{m, k}^{p, \calL_p^w}(f_1, f_2),
\end{equation}
where
\begin{equation*}
V_{m, k}^{p, \calH_p^w}  \left(f_1, f_2 \right):=2^{\frac{m}{2}+k+w} \sum_{q \sim 2^{\frac{m}{2}}} \int_{\widetilde{\calH}_p^{w}} \left| \int_{I_q^{m, k}} f_1(x-t) f_2(x+t^2) e^{3i 2^{\frac{m}{2}+k} \cdot \frac{q^2}{2^m} \lambda(x) t} dt \right|^2 dx
\end{equation*}
and
$$
V_{m, k}^{p, \calL_p^w}  \left(f_1, f_2 \right):=2^{\frac{m}{2}+k+w} \sum_{q \sim 2^{\frac{m}{2}}} \int_{\widetilde{\calL}_p^{w}} \left| \int_{I_q^{m, k}} f_1(x-t) f_2(x+t^2) e^{3i 2^{\frac{m}{2}+k} \cdot \frac{q^2}{2^m} \lambda(x) t} dt \right|^2 dx.
$$

Once at this point, let us notice that in our Key Heuristic 4, \eqref{20230205eq01} corresponds to the heavy term $V_{m, k}^{p, \calH_p^w}$ for which $\lambda (\cdot)$ is essentially constant over large(r) intervals, while \eqref{20230205eq02} corresponds to the light term $V_{m, k}^{p, \calL_p^w}$. Indeed, this latter situation is the subject of the following:

\begin{prop} \label{mainprop01}
Let $w\in\{k,\frac{m}{2}\}$. Then, there exists some $\epsilon_{3}^{w}>0$, such that
\begin{equation} \label{lighttermest}
\left|V_{m, k}^{p, \calL_p^w} \left(f_1, f_2 \right) \right| \lesssim 2^{-2\epsilon_{3}^{w} m} \left\|f_1 \right\|_{L^2 \left(3 I^k \right)}^2 \left\| f_2 \right\|_{L^2 \left(3 I^{k} \right)}^2.
\end{equation}
\end{prop}

\begin{proof}
We divide our proof into several steps.

\medskip

\noindent\textsf{Step I: Extension of $V_{m, k}^{p, \calL_p^w} \left(f_1, f_2 \right)$.} While we may asssume that $\widetilde{\calL}_p^{w} \subset I_p^{0, k+w}$ is ``large" \footnote{As otherwise we can apply a trivial argument in order to obtain the desired control on $V_{m, k}^{p, \calL_p^w}$.}, it turns out that it is more convenient to extend the set $\widetilde{\calL}_p^{w}$ to \emph{fully} cover the interval $I_p^{0, k+w}$. For this purpose, we define a new measurable function $\widetilde{\lambda}$ that is
\begin{enumerate}
    \item [$\bullet$] constant on intervals $\left\{I_\ell^{0, 2k+w}\right\}_{\ell \sim 2^{k+w}}$;
    \item [$\bullet$] takes the same value as $\lambda$ for $x \in \widetilde{\calL}_p^{w}$;
    \item [$\bullet$] the set of heavy indices corresponding to $\widetilde{\lambda}$, denoted here by $\underline{\calH}_p^w$,  is empty; this can be achieved for example as follows: since the original $\widetilde{\calH}_p^{w}$ is a finite union of intervals of the form $I_\ell^{0, 2k+w}$, we arrange them in the increasing order of $\ell$, relabel this resulting collection of intervals as $\left\{\frakJ_1, \frakJ_2, \dots \right\}$ and for each $x\in\frakJ_{\ell}$ we let\footnote{Here and throughout the paper $\lfloor x\rfloor$ stands for the entire part of $x\in\R$.} $\widetilde{\lambda}(x):=2^{\frac{m}{2}}+\ell-\left\lfloor\frac{\ell}{2^{\frac{m}{2}}}\right\rfloor\,2^{\frac{m}{2}}$;
   \item [$\bullet$] finally, as a technical convenience to be used later in the proof, we repeat the above process for the adjacent $8$ left and $8$ right brothers of  $I_p^{0, k+w}$ so that $\widetilde{\lambda}(x)$ becomes ``completely light" on $17 I_p^{0, k+w}$.
\end{enumerate}
From the above, we  deduce that
$$
V_{m, k}^{p, \calL_p^w}(f_1, f_2) \le \widetilde{V_{m, k}^{p, \calL_p^w}}(f_1, f_2),
$$
where
\begin{equation} \label{20230206eq01}
\widetilde{V_{m, k}^{p, \calL_p^w}}(f_1, f_2):=2^{\frac{m}{2}+k+w} \sum_{q \sim 2^{\frac{m}{2}}} \int_{I_p^{0, k+w}} \left| \int_{I_q^{m, k}} f_1(x-t) f_2(x+t^2) e^{3i 2^{\frac{m}{2}+k} \cdot \frac{q^2}{2^m} \widetilde{\lambda}(x) t} dt \right|^2 dx.
\end{equation}

\medskip

\noindent\textsf{Step II: A reduction via Corollary \ref{20230203cor01} and $TT^*$.}  Via Corollary \ref{20230203cor01}, we see that up to an error term of at most $O\left(2^{-(\del_w-2\mu_0) m} \, \left\|f_1 \right\|_{L^2 \left(3 I^k \right)}^2 \left\|f_2 \right\|_{L^2 \left(3 I^k \right)}^2\right)$, the expression in \eqref{20230206eq01} is bounded from above by
\begin{equation} \label{20230206eq02}
\sum_{q \sim 2^{\frac{m}{2}}} \frac{2^{\frac{m}{2}+k+w+\del_w m}}{\left|I_0^{0, k+w} \right|} \int_{3I_0^{0, k+w}} \int_{I_0^{0, k+w}} \bigg| \int_{I_q^{m, k}} f_1 \left(\alpha-t+\frac{p}{2^{k+w}} \right)\,f_2 \left(x+t^2+\frac{p}{2^{k+w}} \right) e^{3i 2^{\frac{m}{2}+k} t \cdot \frac{q^2}{2^m} \cdot \widetilde{\lambda} \left(x+\frac{p}{2^{k+w}} \right)}\,dt \bigg|^2 dx d\alpha.
\end{equation}
Now, without loss of generality, we may
\begin{enumerate}
    \item [$\bullet$] assume $\alpha=0$, since the $\alpha$-integral in \eqref{20230206eq02} represents an average value relative to the interval $3I_0^{0, k+w}$ and our estimates below will not depend on the specific location of $\alpha$;
    \item [$\bullet$] relabel $f_1 \left(\frac{p}{2^{k+w}}-\cdot \right)$ by $f_1 (\cdot)$, $f_2 \left(\frac{p}{2^{k+w}}+\cdot \right)$ by $f_2 (\cdot)$, and $\widetilde{\lambda} \left(\frac{p}{2^{k+w}}+\cdot \right)$ by $\widetilde{\lambda}  (\cdot)$, respectively, since $p\sim 2^{w}$ is fixed throughout our proof.
\end{enumerate}

Therefore, it suffices to consider the term
\begin{equation} \label{20230206eq03}
\calV_{m, k}^{w}:=2^{\frac{m}{2}+k+w+\del_w m} \sum_{q \sim 2^{\frac{m}{2}}} \int_{I_0^{0, k+w}} \left| \int_{I_q^{m, k}} f_1 \left(t \right)  f_2 \left(x+t^2 \right) e^{3i 2^{\frac{m}{2}+k} t \cdot \frac{q^2}{2^m} \cdot \widetilde{\lambda}  \left(x\right)} dt \right|^2 dx.
\end{equation}
Applying the change of variable $t \to  t+\frac{q}{2^{\frac{m}{2}+k}}$ and then using the fact that
$$
\left(t+\frac{q}{2^{\frac{m}{2}+k}} \right)^2=\frac{2tq}{2^{\frac{m}{2}+k}}+\frac{q^2}{2^{m+2k}}+O \left(\frac{1}{2^{m+2k}} \right),
$$
we have
\begin{equation}
\calV_{m, k}^{w} \approx 2^{\frac{m}{2}+k+w+\del_w m} \sum_{q \sim 2^{\frac{m}{2}}} \int_{I_0^{0, k+w}} \bigg| \int_{I_0^{m, k}} f_1 \left(t+\frac{q}{2^{\frac{m}{2}+k}} \right)\,f_2 \left(x+\frac{2tq}{2^{\frac{m}{2}+k}}+\frac{q^2}{2^{m+2k}} \right)  e^{3i 2^{\frac{m}{2}+k} t \cdot \frac{q^2}{2^m} \cdot \widetilde{\lambda} \left(x\right)} dt \bigg|^2 dx.
\end{equation}
Applying another change of variable $x \to x-\frac{q^2}{2^{m+2k}}$ and Fubini we further have
\begin{eqnarray*}
\calV_{m, k}^{w} &\lesssim& 2^{\frac{m}{2}+k+w+\del_w m} \sum_{q \sim 2^{\frac{m}{2}}} \int_{9I_0^{0, k+w}} \bigg| \int_{I_0^{m, k}} f_1 \left(t+\frac{q}{2^{\frac{m}{2}+k}} \right)\,f_2 \left(x+\frac{2tq}{2^{\frac{m}{2}+k}} \right)  e^{3i 2^{\frac{m}{2}+k} t \cdot \frac{q^2}{2^m} \cdot \widetilde{\lambda}  \left(x-\frac{q^2}{2^{m+2k}}\right)} dt \bigg|^2 dx \\
&=&2^{\frac{m}{2}+k+w+\del_w m} \sum_{q \sim 2^{\frac{m}{2}}} \iint_{\left(I_0^{m, k} \right)^2} f_1 \left(t+\frac{q}{2^{\frac{m}{2}+k}} \right) f_1 \left(s+\frac{q}{2^{\frac{m}{2}+k}} \right) \\
&\cdot& \bigg(\int_{9I_0^{0, k+w}} f_2 \left(x+\frac{2tq}{2^{\frac{m}{2}+k}}  \right) f_2 \left(x+\frac{2sq}{2^{\frac{m}{2}+k}}  \right)\,e^{3i 2^{\frac{m}{2}+k} (t-s) \cdot \frac{q^2}{2^m} \cdot \widetilde{\lambda} \left(x-\frac{q^2}{2^{m+2k}}\right)} dx \bigg) dtds.
\end{eqnarray*}
Applying Cauchy-Schwarz twice, we notice
\begin{equation} \label{20230206eq30}
\calV_{m, k}^{w} \lesssim 2^{\frac{m}{2}+k+w+\del_w m} {\calA}\: {\calB},
\end{equation}
where
$$
{\calA}^2:=\iint_{\left(I_0^{m, k} \right)^2} \sum_{q \sim 2^{\frac{m}{2}}} \left|f_1 \left(t+\frac{q}{2^{\frac{m}{2}+k}} \right) \right|^2\left|f_1 \left(s+\frac{q}{2^{\frac{m}{2}+k}} \right) \right|^2 dtds
$$
and
\begin{equation} \label{20230206eq32}
{\calB}^2:=\iint_{\left(I_0^{m, k} \right)^2} \sum_{q \sim 2^{\frac{m}{2}}}  \bigg| \int_{9I_0^{0, k+w}} f_2 \left(x+\frac{2tq}{2^{\frac{m}{2}+k}}  \right) f_2 \left(x+\frac{2sq}{2^{\frac{m}{2}+k}}  \right) e^{3i 2^{\frac{m}{2}+k} (t-s) \cdot \frac{q^2}{2^m} \cdot \widetilde{\lambda}  \left(x-\frac{q^2}{2^{m+2k}}\right)} dx \bigg|^2 dtds.
\end{equation}

The estimate of $\calA$ follows easily from the uniformity of $f_1$:
\begin{equation*}
\calA^2=\sum_{q \sim 2^{\frac{m}{2}}} \left( \int_{I_0^{m, k}} \left|f_1 \left(t+\frac{q}{2^{\frac{m}{2}+k}} \right) \right|^2 dt \right) \left( \int_{I_0^{m, k}} \left|f_1 \left(s+\frac{q}{2^{\frac{m}{2}+k}} \right) \right|^2 ds \right) \lesssim 2^{(2\mu_0-\frac{1}{2})m}\,\left\|f_1\right\|_{L^2 \left(3 I^k \right)}^4.
\end{equation*}
Plugging the above estimate back to \eqref{20230206eq30}, we have
\begin{equation} \label{20230206eq31}
\calV_{m, k}^{w} \lesssim 2^{\frac{m}{4}+k+w+(\del_w+\mu_0)m} \|f_1\|_{L^2 \left(3 I^k \right)}^2 \calB.
\end{equation}

\medskip

\noindent\textsf{Step III: Treatment of $\calB$.}

By applying in \eqref{20230206eq32} the change of variables $t \to \frac{2^{\frac{m}{2}}}{q}t, s \to \frac{2^{\frac{m}{2}}}{q} s$ followed by $s \to t-v$ and Fubini we get
\begin{eqnarray} \label{20230207eq01}
{\calB}^2 &\lesssim& \sum_{q \sim 2^{\frac{m}{2}}} \iint_{\left(9I_0^{0, k+w} \right)^2} \iint_{3I_0^{m, k}\times 7I_0^{m, k}} f_2 \left(x+\frac{2t}{2^k} \right) f_2 \left(x+\frac{2(t-v)}{2^k} \right)  \\
&\cdot&  f_2 \left(y+\frac{2t}{2^k} \right) f_2 \left(y+\frac{2 (t-v)}{2^k} \right) e^{3i 2^{\frac{m}{2}+k} v \cdot \frac{q}{2^{\frac{m}{2}}} \cdot \left[\widetilde{\lambda}  \left(x-\frac{q^2}{2^{m+2k}}\right)-\widetilde{\lambda} \left(y-\frac{q^2}{2^{m+2k}} \right) \right] } dt dv dx dy. \nonumber
\end{eqnarray}
We next apply another change of variable $x \to x-\frac{2t}{2^k}$ and  $y \to y-\frac{2t}{2^k}$ which gives,
\begin{equation}\label{admiserr}
\textrm{up to an admissible global error term}\:O\left(2^{-(\frac{1}{4}-\del_w-2\mu_0) m} \, \left\|f_1 \right\|_{L^2 \left(3 I^k \right)}^2 \left\|f_2 \right\|_{L^2 \left(3 I^k \right)}^2\right)\,,
\end{equation}
\begin{eqnarray*}
{\calB}^2&\lesssim&\iint_{\left(9I_0^{0, k+w} \right)^2}\left|\int_{7I_0^{m, k}} f_2 \left(x\right) f_2 \left(y\right) f_2 \left(x-\frac{2v}{2^k}\right) f_2 \left(y-\frac{2v}{2^k} \right) \nonumber \right.\\
&\cdot& \left.\bigg(\int_{3 I_0^{m, k}}\bigg( \sum_{q \sim 2^{\frac{m}{2}}} e^{3i 2^{\frac{m}{2}+k} v \cdot \frac{q}{2^{\frac{m}{2}}} \cdot \left[\widetilde{\lambda}  \left(x-\frac{2t}{2^k}-\frac{q^2}{2^{m+2k}}\right)-\widetilde{\lambda}  \left(y-\frac{2t}{2^k}-\frac{q^2}{2^{m+2k}}  \right) \right]} \bigg) dt\bigg)dv\,\right|dx\,dy.
\end{eqnarray*}

By Cauchy-Schwarz, we have
\begin{equation} \label{20230207eq05}
\frakB^2 \le \frakC \frakD,
\end{equation}
where
\begin{equation}
\frakC^2:=\iiint_{\left(9I_0^{0, k+w} \right)^2 \times 7I_0^{m, k}} \bigg| f_2 \left(x\right) f_2 \left(y \right)\,f_2 \left(x-\frac{2v}{2^k} \right) f_2 \left(y-\frac{2v}{2^k} \right) \bigg|^2 dxdy dv,
\end{equation}
and
\begin{equation}
\frakD^2:=\iiint_{\left(9I_0^{0, k+w} \right)^2 \times 7I_0^{m, k}} \bigg|\int_{3I_0^{m, k}}\bigg( \sum_{q \sim 2^{\frac{m}{2}}} e^{3i 2^{\frac{m}{2}+k} v \cdot \frac{q}{2^{\frac{m}{2}}} \cdot \left[\widetilde{\lambda}  \left(x-\frac{2t}{2^k}-\frac{q^2}{2^{m+2k}}\right)-\widetilde{\lambda}  \left(y-\frac{2t}{2^k}-\frac{q^2}{2^{m+2k}} \right) \right]}\bigg) dt \bigg|^2 dv dx dy.
\end{equation}
As usual, the treatment of $\frakC$ is straightforward due to the uniformity of $f_2$:
\begin{equation}
\frakC^2\lesssim 2^{-\frac{m}{2}+k-2w} \cdot 2^{4\mu_0m} \left\|f_2\right\|_{L^2 \left(3 I^k \right)}^8,
\end{equation}
which combined with \eqref{20230206eq31} and \eqref{20230207eq05} gives:
\begin{equation} \label{20230207eq06}
\calV_{m, k}^{w} \lesssim 2^{\frac{m}{8}+\frac{5k}{4}+\frac{w}{2}+\left(\del_w+2\mu_0\right)m} \left\|f_1 \right\|_{L^2 \left(3 I^k \right)}^2 \left\|f_2 \right\|_{L^2 \left(3 I^k \right)}^2 \frakD^{\frac{1}{2}}.
\end{equation}

\medskip

\noindent\textsf{Step IV: Treatment of $\frakD$.} Developing the square, and applying Fubini and triangle inequality we have
\begin{eqnarray*}
&&\frakD^2\lesssim  \sum_{q\sim 2^{\frac{m}{2}} } \sum_{q_1\sim 2^{\frac{m}{2}}} \iint_{\left(9I_0^{0, k+w} \right)^2}  \iint_{\left(3I_0^{m, k} \right)^2} dt dt_1 dx dy\\
&&\bigg|\int_{7 I_0^{m, k}} \quad e^{3i 2^{\frac{m}{2}+k} v \cdot \left(\frac{q}{2^{\frac{m}{2}}} \cdot \left[\widetilde{\lambda} \left(x-\frac{2t}{2^k}-\frac{q^2}{2^{m+2k}}\right)-\widetilde{\lambda}  \left(y-\frac{2t}{2^k}-\frac{q^2}{2^{m+2k}}  \right) \right]- \frac{q_1}{2^{\frac{m}{2}}} \cdot \left[\widetilde{\lambda} \left(x-\frac{2t_1}{2^k}-\frac{q_1^2}{2^{m+2k}}\right)-\widetilde{\lambda}  \left(y-\frac{2t_1}{2^k}-\frac{q_1^2}{2^{m+2k}} \right) \right] \right)}\,dv\bigg|.
 \end{eqnarray*}
Next, applying another change of variables $x \to x+\frac{2t}{2^k}+\frac{q^2}{2^{m+2k}}$ and $y \to y+\frac{2t}{2^k}+\frac{q^2}{2^{m+2k}}$ followed by $t\to t_1+s$ we deduce\footnote{Throughout the paper, for $x\in\R$ we use the notation $\lceil x \rceil:=\frac{1}{1+|x|}$.}:
\begin{equation} \label{20230208eq22}
 \frakD^2 \le \left|I_0^{m, k} \right|^2 \sum_{q\sim 2^{\frac{m}{2}}} \sum_{q_1\sim 2^{\frac{m}{2}}} \iiint_{\left(17 I_0^{0, k+w} \right)^2\times 6I_0^{m,k}}  \left\lceil\Xi_{x, y, s}(q,q_1)\right\rceil dsdxdy\,,
\end{equation}
where
$$
\Xi_{x, y, s}(q,q_1):= \left[\widetilde{\lambda} \left(x\right)-\widetilde{\lambda}  \left(y \right) \right]- \frac{q_1}{q} \cdot \left[\widetilde{\lambda} \left(x+\frac{2 s}{2^k}-\frac{q_1^2-q^2}{2^{m+2k}}\right)-\widetilde{\lambda}  \left(y+\frac{2 s}{2^k}-\frac{q_1^2-q^2}{2^{m+2k}}\right) \right]\,.
$$

Recall now that by the construction of the light set for  $\widetilde{\lambda}$--see \eqref{20230206eq03}--we have
\begin{equation} \label{20230208eq21}
17I_0^{0, k+w}=\bigcup_{l \in \widetilde{\calL}^w} \widetilde{\frakM_l}\left(17I_0^{0, k+w} \right),
\end{equation}
where
$$
\widetilde{\frakM_l}\left(17 I_0^{0, k+w} \right):=\left\{x \in 17 I_0^{0, k+w}: \widetilde{\lambda}(x) \in \left[l2^{\epsilon_{0}^{w} m}, (l+1)2^{\epsilon_{0}^{w} m} \right] \right\},
$$
and
\begin{equation} \label{20230208eq33}
\widetilde{\calL}^{w}:=\left\{l \sim 2^{\frac{m}{2}-\epsilon_{0}^{w} m}: \left| \widetilde{\frakM_l}\left(17 I_0^{0, k+w} \right) \right| \lesssim \left|I_0^{0, k+w} \right| 2^{-\epsilon_{0}^{w} m} \right\}.
\end{equation}
Substituting  \eqref{20230208eq21} back to \eqref{20230208eq22} we let
\begin{equation} \label{20230208eq24}
\frakD^2=\frakD_{\textnormal{diag}}^2+\frakD_{\textnormal{off}}^2,
\end{equation}
where
\begin{equation} \label{20230505eq02}
\frakD_{\textnormal{diag}}^2=\left|I_0^{m, k} \right|^2 \sum_{q\sim 2^{\frac{m}{2}}} \sum_{q_1\sim 2^{\frac{m}{2}}} \sum_{ \substack{l, l_1 \in \widetilde{\calL}^{w} \\ |l-l_1| \le 10} }\iiint_{\widetilde{\frakM_l}\left(17 I_0^{0, k+w} \right) \times \widetilde{\frakM_{l_1}}\left(17 I_0^{0, k+w} \right)\times 6I_0^{m,k}}  \left\lceil\Xi_{x, y, s}(q,q_1)\right\rceil dsdxdy,
\end{equation}
and
\begin{equation*}
\frakD_{\textnormal{off}}^2=\left|I_0^{m, k} \right|^2 \sum_{q\sim 2^{\frac{m}{2}}} \sum_{q_1\sim 2^{\frac{m}{2}}} \sum_{ \substack{l, l_1 \in \widetilde{\calL}^{w} \\ |l-l_1|>10} }\iiint_{\widetilde{\frakM_l}\left(17 I_0^{0, k+w} \right) \times \widetilde{\frakM_{l_1}}\left(17 I_0^{0, k+w} \right)\times 6I_0^{m,k}} \left\lceil\Xi_{x, y, s}(q,q_1)\right\rceil dsdxdy\,.
\end{equation*}
The contribution of the diagonal term $\frakD_{\textnormal{diag}}^2$ is straightforward:
\begin{eqnarray*}
\frakD_{\textnormal{diag}}^2  &\le&  \frac{2^{m}}{2^{\frac{3m}{2}+3k}} \cdot \sum_{l \in \widetilde{\calL}^w} \left(\int_{\widetilde{\frakM_l}\left(17 I_0^{0, k+w} \right)} dx \right) \cdot \left( \sum_{|l-l_1| \le 10} \int_{\widetilde{\frakM_{l_1}}\left(17 I_0^{0, k+w} \right)} dy \right)  \\
&\lesssim& \frac{2^{m}}{2^{\frac{3m}{2}+3k}} \cdot \left|I_0^{0, k+w} \right|^2 \cdot 2^{-\epsilon_{0}^{w} m}=2^{ -\frac{m}{2}-5k-2w-\epsilon_{0}^{w}m}.
\end{eqnarray*}
Substituting this in \eqref{20230207eq06}, we see that the contribution from the diagonal term is dominated by
\begin{equation}\label{contrib}
 2^{\left(\del_w+2\mu_0-\frac{\epsilon_{0}^{w}}{4} \right)m} \left\|f_1 \right\|_{L^2 \left(3 I^k \right)}^2 \left\|f_2 \right\|_{L^2 \left(3 I^k \right)}^2
\end{equation}
which gives the desired decay as long as we choose $0<2\epsilon_{3}^{w}\leq \frac{\epsilon_{0}^{w}}{4}-\del_w-2\mu_0$ with $\epsilon_{0}^{w}>4\del_w+8\mu_0$.

\medskip

\noindent\textsf{Step V: Treatment of $\frakD_{\textnormal{off}}^2$.} Applying the changes of variable $\frac{2 s}{2^k}\to \frac{2z}{2^k}+\frac{q_1^2-q^2}{2^{m+2k}}-\frac{1}{2^{\frac{m}{2}+2k}}\left\lfloor\frac{q_1^2-q^2}{2^{\frac{m}{2}}}\right\rfloor$ followed by $u=\left\lfloor\frac{q_1^2-q^2}{2^{\frac{m}{2}}}\right\rfloor$ and using that $\frac{q_1}{q}=\sqrt{1+\frac{2^{\frac{m}{2}}u}{q^2}+O(\frac{1}{2^{\frac{m}{2}}})}$, we have
\begin{equation} \label{20230208eq40}
\frakD_{\textnormal{off}}^2 \lesssim\left|I_0^{m, k} \right|^2  \sum_{ \substack{l, l_1 \in \widetilde{\calL}^w \\ |l-l_1|>10} }\iiint_{\widetilde{\frakM_l}\left(17 I_0^{0, k+w} \right) \times \widetilde{\frakM_{l_1}}\left(17 I_0^{0, k+w} \right)\times 9I_0^{m,k}} \sum_{q, u \sim 2^{\frac{m}{2}}}  \left\lceil\Xi_{x, y, z,u}(q)\right\rceil  dzdxdy\,,
\end{equation}
where
$$
\Xi_{x, y, z, u}(q):=\left[\widetilde{\lambda} \left(x\right)-\widetilde{\lambda}  \left(y \right) \right]- \frac{\sqrt{2^{\frac{m}{2}}u+q^2}}{q} \cdot \left[\widetilde{\lambda} \left(x+\frac{2z}{2^k}-\frac{u}{2^{\frac{m}{2}+2k}}\right)-\widetilde{\lambda}  \left(y+\frac{2z}{2^k}-\frac{u}{2^{\frac{m}{2}+2k}} \right) \right]\,.
$$

Observe that since $|l-l_1|>10$, $x \in \widetilde{\frakM_l}\left(17 I_0^{0, k+w} \right)$, and $y \in \widetilde{\frakM_{l_1}}\left(17 I_0^{0, k+w} \right)$, one has
$$
\left|\widetilde{\lambda}(x)-\widetilde{\lambda}(y) \right| \geq 2^{\epsilon_{0}^{w} m}\,.
$$
Fix now $x,\,y$ above; it is then sufficient to consider the situation when the pairs $(z,u)$ satisfy the condition
\begin{equation} \label{20230208eq35}
\left| \widetilde{\lambda} \left(x+\frac{2z}{2^k}-\frac{u}{2^{\frac{m}{2}+2k}}\right)-\widetilde{\lambda}  \left(y+\frac{2z}{2^k}-\frac{u}{2^{\frac{m}{2}+2k}} \right) \right|>\frac{2^{\epsilon_{0}^{w} m}}{100}.
\end{equation}
Indeed, if this is not the case, then the magnitude of $|\Xi_{x, y, z, u}(q)|$ is $\gtrsim 2^{\epsilon_{0}^{w} m}$ since $\frac{\sqrt{2^{\frac{m}{2}}u+q^2}}{q}<10$ and therefore, by \eqref{20230208eq40}, the corresponding term is less than
\begin{equation}  \label{20230208eq50}
\lesssim 2^{-\frac{3m}{2}-3k} \cdot 2^m \cdot 2^{-\epsilon_{0}^{w} m} \cdot \left|I_0^{0, k+w} \right|^2\lesssim 2^{-\frac{m}{2}-5k-2w-\epsilon_{0}^{w} m}
\end{equation}
which, combined with \eqref{20230207eq06}, gives the same contribution as in \eqref{contrib}.

Thus, we may assume now that \eqref{20230208eq35} holds and fix $l, l_1 \in \widetilde{\calL}^w$, $u \sim 2^{\frac{m}{2}}$ and $(x,y)\in  \widetilde{\frakM_l}\left( 17 I_0^{0, k+w} \right) \times \widetilde{\frakM_{l_1}}\left( 17 I_0^{0, k+w} \right)$. We then note that the mapping
$$
[2^{\frac{m}{2}},\,2^{\frac{m}{2}+1}]\ni q \mapsto \frac{\sqrt{2^{\frac{m}{2}}u+q^2}}{q}\in \left[\frac{1}{10},10\right]
$$
is strictly decreasing. Moreover,
\begin{equation} \label{20230208eq51}
\left| \frac{d}{dq} \left( \frac{\sqrt{2^{\frac{m}{2}}u+q^2}}{q} \right) \right|= \left| \left(\frac{2^{\frac{m}{2}}u}{q^2}
+1 \right)^{-\frac{1}{2}} \cdot \frac{2^{\frac{m}{2}}u}{q^3} \right|\approx\frac{1}{2^{\frac{m}{2}}}.
\end{equation}
Define now the \emph{singular set} associated to $x, y, z$ and $u$ by
$$
\calS_{x, y, z, u}:=\left\{q \sim 2^{\frac{m}{2}}: |\Xi_{x, y, z, u}(q)| \le 2^{\frac{\epsilon_{0}^{w} m}{2}} \right\},
$$
and the \emph{non-singular set} associated to $x, y, z$ and $u$ by
$$
\calN_{x, y, z, u}:=\left\{q \sim 2^{\frac{m}{2}}: |\Xi_{x, y, z, u}(q) | >2^{\frac{\epsilon_{0}^{w} m}{2}} \right\}.
$$
Note that by \eqref{20230208eq35} and \eqref{20230208eq51}, one has for fixed $x, y, z$ and $u$,
$$
\# \calS_{x, y, z, u} \lesssim 2^{\frac{m}{2}-\frac{\epsilon_{0}^{w} m}{2}}.
$$
Therefore, using \eqref{20230208eq40}, one can write
$$
\frakD_{\textnormal{off}}^2 \le \frakD_{\textnormal{off}}^{2, \calS}+\frakD_{\textnormal{off}}^{2, \calN},
$$
where
\begin{equation}
\frakD_{\textnormal{off}}^{2, \calS} :=\left|I_0^{m, k} \right|^2  \sum_{ \substack{l, l_1 \in \widetilde{\calL}^w \\ |l-l_1|>10} }\iiint_{\widetilde{\frakM_l}\left( 17 I_0^{0, k+w} \right) \times \widetilde{\frakM_{l_1}}\left( 17 I_0^{0, k+w} \right)\times 9 I_0^{m,k}} \sum_{u \sim 2^{\frac{m}{2}}} \sum_{q \in \calS_{x, y, z, u}}  \lceil\Xi_{x, y, z, u}(q)\rceil dzdxdy
\end{equation}
and
\begin{equation}
\frakD_{\textnormal{off}}^{2, \calN} :=\left|I_0^{m, k} \right|^2  \sum_{ \substack{l, l_1 \in \widetilde{\calL}^w \\ |l-l_1|>10} }\iiint_{\widetilde{\frakM_l}\left( 17 I_0^{0, k+w} \right) \times \widetilde{\frakM_{l_1}}\left( 17 I_0^{0, k+w} \right) \times 9 I_{0}^{m,k}} \sum_{u \sim 2^{\frac{m}{2}}} \sum_{q \in \calN_{x, y,z, u}}  \lceil\Xi_{x, y, z, u}(q)\rceil dzdxdy.
\end{equation}
The treatment of the term $\frakD_{\textnormal{off}}^{2, \calN}$ follows exactly the same argument as that in \eqref{20230208eq50} with $\epsilon_{0}^{w}$ there replaced now by $\frac{\epsilon_{0}^{w}}{2}$. Thus, in the current situation we obtain the desired control by choosing $0<\epsilon_{3}^{w}\leq \frac{\epsilon_{0}^{w}}{8}-\del_w-2\mu_0$ and $\epsilon_{0}^{w}>8\del_w+16\mu_0$.

Finally, for the term $\frakD_{\textnormal{off}}^{2, \calS}$, we have
\begin{eqnarray*}
\left| \frakD_{\textnormal{off}}^{2, \calS}\right| &\lesssim& \left|I_0^{m, k} \right|^2  \sum_{ \substack{l, l_1 \in \widetilde{\calL}^w \\ |l-l_1|>10} } \iiint_{\widetilde{\frakM_l}\left( 17 I_0^{0, k+w} \right) \times \widetilde{\frakM_{l_1}}\left( 17 I_0^{0, k+w} \right) \times 9 I_{0}^{m,k}} \sum_{u \sim 2^{\frac{m}{2}}} \sum_{q \in \calS_{x, y, z,u}} 1\,dzdxdy \\
&\lesssim& 2^{\frac{m}{2}-\frac{\epsilon_{0}^{w} m}{2}} \left|I_0^{m, k} \right|^2  \sum_{ \substack{l, l_1 \in \widetilde{\calL}^w \\ |l-l_1|>10} } \iiint_{\widetilde{\frakM_l}\left( 17 I_0^{0, k+w} \right) \times \widetilde{\frakM_{l_1}}\left( 17 I_0^{0, k+w} \right) \times 9 I_{0}^{m,k}} \sum_{u \sim 2^{\frac{m}{2}}} 1\,dzdxdy.
\end{eqnarray*}
From this point on, one can again mimic the proof of \eqref{20230208eq50}.

Conclude thus, that \eqref{lighttermest} holds with the final choice
\begin{equation}\label{ep3}
2\epsilon_{3}^{w}=\frac{\epsilon_{0}^{w}}{8}-\del_w-2\mu_0\,,
\end{equation}
under the assumptions $\epsilon_{0}^{w}>8\del_w+16\mu_0$. This completes the proof of Proposition \ref{mainprop01}.
\end{proof}

We end the proof of this first part with the following

\begin{obs}\label{boots2}[\textsf{Completing the bootstrap algorithm}]

The conclusion of the above allows us to reduce matters to the treatment of the heavy component derived from \eqref{20230205eq12} that is meant to address the situation summarized by \eqref{20230205eq01}.

Recalling now that
$$
\widetilde{\calH}_p^{w}=\bigcup_{l \in \calH_p^w} \frakM_l \left(I_p^{0, k+w} \right) \quad \textrm{and} \quad \# \calH_p^w \lesssim 2^{\epsilon_{0}^{w} m},
$$
we further deduce that up to an admissible loss of at most $2^{2\epsilon_{0}^{w} m}$ we may assume without loss of generality that $\widetilde{\calH_p^{w}}=\frakM_{l} \left(I_p^{0, k+w} \right)=I_p^{0, k+w}$ for some heavy index $l=l(p)$ and that
\begin{equation} \label{20230211eq01}
\lambda(x)=\lambda(y)\quad\quad \forall\:x,\,y \in I_p^{0, k+w}.
\end{equation}

Applying now all the above reasonings first for $w=k$ and then for $w=\frac{m}{2}$ (recall Observation \ref{constancy}) we are able to place ourselves in the situation when $\lambda$ is constant on intervals of length $2^{-\frac{m}{2}-k}$.
\end{obs}

$\newline$
\noindent\textsf{Part 2}. \textsf{The global behavior of $\lambda$: treating the low variational component}
$\newline$

Our entire focus in this second part will be on proving the following

\begin{prop} \label{mainprop02}
Assume that for each $p \sim 2^{\frac{m}{2}}$ relation \eqref{20230211eq01} holds for $w=\frac{m}{2}$. Then, recalling \eqref{defL}, we have that there exists $\epsilon_4>0$ absolute constant such that
\begin{equation} \label{20230210eq01}
\textbf{L}_{m,k}^2 \left(f_1, f_2 \right) \lesssim 2^{-2\epsilon_4 m} \left\|f_1 \right\|^2_{L^2 \left(3 I^k \right)}\left\|f_2 \right\|^2_{L^2 \left(3 I^k \right)}.
\end{equation}
\end{prop}

\begin{proof}
We proceed as before by dividing our proof into several steps:
\medskip

\noindent\textsf{Step I: Preliminary reductions.} For technical reasons we appeal to the following two step procedure:
\begin{itemize}
\item first we refine the present partition $\{I_q^{m, k}\}_{q\sim 2^{\frac{m}{2}}}$ of the interval $I^k$, that is, we will work with a new partition of the form $\{I_{\mathfrak{q}}^{m, k+\tau m}\}_{\mathfrak{q}\sim 2^{\frac{m}{2}+\tau m}}$ for a suitable small parameter $\tau>0$ that will be chosen later. To implement this we appeal to Cauchy-Schwarz inequality:
\begin{eqnarray*}
\textbf{L}_{m,k}^2 (f_1, f_2)&=&\frac{1}{\left|I_0^{m, k} \right|} \sum_{ q \sim 2^{\frac{m}{2}}} \int_{I^k} \left| \int_{I_q^{m, k}} f_1(x-t)f_2(x+t^2)e^{3i 2^{\frac{m}{2}+k} \cdot \frac{q^2}{2^m} \cdot \lambda(x) t} dt \right|^2 dx \\
&\leq & \frac{1}{\left|I^{m,k+\tau m}_0 \right|}  \sum_{\mathfrak{q}\sim 2^{\frac{m}{2}+\tau m}} \int_{I^k} \left| \int_{I_\mathfrak{q}^{m, k+\tau m}} f_1(x-t)f_2(x+t^2)e^{3i 2^{\frac{m}{2}+k} \cdot \frac{\left\lfloor\frac{\mathfrak{q}}{2^{\tau m}}\right\rfloor^2}{2^m} \cdot \lambda(x) t} dt \right|^2 dx  \,.
\end{eqnarray*}

\item secondly, we subdivide the range of $\mathfrak{q}$ as follows: for some small $\frac{1}{10}>\sigma>\tau$ (to be specified later--see Observation \ref{bkI}) we write
\begin{equation} \label{qpart}
\left\{\mathfrak{q} \sim 2^{\frac{m}{2}+\tau m} \right\}=\bigcup_{\frakj \sim 2^{\sigma m}} \mathfrak{I}_{\frakj}=:\bigcup_{\frakj \sim 2^{\sigma m}}\left[2^{\frac{m}{2}+\tau m-\sigma\,m}\frakj,  2^{\frac{m}{2}+\tau m-\sigma\,m}(\frakj+1)\right]\,.
\end{equation}
\end{itemize}

With this done, we proceed with the proof of \eqref{20230210eq01}:  we start by performing the change of variable $t \to t+\frac{\mathfrak{q}}{2^{\frac{m}{2}+\tau m+k}}$ in order to get
\begin{eqnarray*}
&&\qquad\qquad\qquad\qquad\qquad\textbf{L}_{m,k}^2 (f_1, f_2)\lesssim\frac{1}{\left|I^{m,k+\tau m}_0 \right|}\sum_{\frakj \sim 2^{\sigma m}}\sum_{\mathfrak{q}\in \mathfrak{I}_{\frakj}}\\
&&\int_{I^k} \bigg| \int_{I_0^{m, k+\tau m}} f_1 \left(x-t-\frac{\mathfrak{q}}{2^{\frac{m}{2}+k+\tau m}} \right)\,f_2\left(x+\frac{2\mathfrak{q}t}{2^{\frac{m}{2}+k+\tau m}}+\frac{\mathfrak{q}^2}{2^{m+2k+2\tau m}} \right) e^{3i2^{\frac{m}{2}+k}t \cdot \frac{\left\lfloor\frac{\mathfrak{q}}{2^{\tau m}}\right\rfloor^2}{2^m} \lambda(x) }dt \bigg|^2 dx.
\end{eqnarray*}

Next, we set $x_{\frakj}:=\left(\frac{\frakj}{2^{k+\sigma m}}\right)^2$ and applying another change of variable $x \to x-\frac{\mathfrak{q}^2}{2^{m+2\tau m+2k}}$ and using \eqref{20230211eq01} and an approximation argument similar with the one provided in the proof of Proposition \ref{constantprop}, we have
\begin{eqnarray} \label{20230211eq03}
&&\textbf{L}_{m,k}^2 \left(f_1, f_2 \right) \lesssim O( 2^{-(\sigma-2\mu_0)m}\,\left\|f_1 \right\|^2_{L^2 \left(3 I^k \right)}\left\|f_2 \right\|^2_{L^2 \left(3 I^k \right)})\,+\,\frac{1}{\left|I^{m,k+\tau m}_0 \right|}\sum_{\frakj \sim 2^{\sigma m}}\sum_{\mathfrak{q}\in \mathfrak{I}_{\frakj}}\\
&&\int_{I^k} \bigg| \int_{I_0^{m, k+\tau m}} f_1 \left(x-t-\frac{\mathfrak{q}}{2^{\frac{m}{2}+k+\tau m}}-\frac{\mathfrak{q}^2}{2^{m+2k+2\tau m}}  \right)\,f_2\left(x+\frac{2\mathfrak{q}t}{2^{\frac{m}{2}+k+\tau m}}\right) e^{3i2^{\frac{m}{2}+k}t \cdot \frac{\left\lfloor\frac{\mathfrak{q}}{2^{\tau m}}\right\rfloor^2}{2^m} \lambda\left(x-x_{\frakj}\right) }dt \bigg|^2 dx.\nonumber
\end{eqnarray}

Further on, for each $\frakj$, $\mathfrak{q}\in \mathfrak{I}_{\frakj}$ we apply the change of variable $t \to t-\frac{\mathfrak{q}^2}{2^{m+2\tau m+2k}}+x_{\frakj}$ in order to get (putting aside the first error term above)

\begin{eqnarray*}
&&\substack{\textrm{RHS in} \\\\ \eqref{20230211eq03}}\lesssim \frac{1}{\left|I^{m,k+\tau m}_0 \right|}\sum_{\frakj \sim 2^{\sigma m}}\sum_{\mathfrak{q}\in \mathfrak{I}_{\frakj}}
\int_{I^k} \bigg| \int_{I_0^{m, k+\tau m}+\frac{q^2}{2^{m+2\tau m+2k}}-x_{\frakj}} f_1 \left(x-x_{\frakj}-t-\frac{\mathfrak{q}}{2^{\frac{m}{2}+k+\tau m}}\right)\\
&&\qquad\qquad\qquad f_2\left(x+\frac{2\mathfrak{q}t}{2^{\frac{m}{2}+k+\tau m}}-\frac{2\mathfrak{q}^3}{2^{\frac{3m}{2}+3\tau m +3k}}+\frac{2\mathfrak{q}x_{\frakj}}{2^{\frac{m}{2}+k+\tau m}} \right) e^{3i2^{\frac{m}{2}+k}t \cdot \frac{\left\lfloor\frac{\mathfrak{q}}{2^{\tau m}}\right\rfloor^2}{2^m} \lambda(x-x_{\frakj}) }dt \bigg|^2 dx,
\end{eqnarray*}
which, by yet another change variable $x \to x+\frac{2\mathfrak{q}^3}{2^{\frac{3m}{2}+3\tau m+3k}}-\frac{2\mathfrak{q}x_{\frakj}}{2^{\frac{m}{2}+k+\tau m}}$, becomes
\begin{eqnarray} \label{20230212eq04}
&&=\frac{1}{\left|I^{m,k+\tau m}_0 \right|}\sum_{\frakj \sim 2^{\sigma m}}\sum_{\mathfrak{q}\in \mathfrak{I}_{\frakj}}
\int_{I^k-\frac{2\mathfrak{q}^3}{2^{\frac{3m}{2}+3\tau m+3k}}+\frac{2\mathfrak{q}x_{\frakj}}{2^{\frac{m}{2}+k+\tau m}}} \bigg| \int_{I_0^{m, k+\tau m}+\frac{\mathfrak{q}^2}{2^{m+2\tau m+2k}}-x_{\frakj}}\nonumber\\
&&f_1 \left(x-x_{\frakj}-t-\frac{\mathfrak{q}}{2^{\frac{m}{2}+k+\tau m}}+\frac{2\mathfrak{q}^3}{2^{\frac{3m}{2}+3\tau m+3k}}-\frac{2\mathfrak{q}x_{\frakj}}{2^{\frac{m}{2}+k+\tau m}}\right)\qquad\qquad\nonumber\\
&& f_2\left(x+\frac{2\mathfrak{q}t}{2^{\frac{m}{2}+k+\tau m}}\right) e^{3i2^{\frac{m}{2}+k}t \cdot \frac{\left\lfloor\frac{\mathfrak{q}}{2^{\tau m}}\right\rfloor^2}{2^m} \lambda\left(x-x_{\frakj}+\frac{2\mathfrak{q}^3}{2^{\frac{3m}{2}+3\tau m+3k}}-\frac{2\mathfrak{q}x_{\frakj}}{2^{\frac{m}{2}+k+\tau m}}\right) }dt \bigg|^2 dx
\end{eqnarray}

Since $\left|\frac{\mathfrak{q}^2}{2^{m+2\tau m+2k}}-x_{\frakj}\right|\lesssim 2^{-2k-\sigma m}$ (and hence $\left|\frac{2\mathfrak{q}^3}{2^{\frac{3m}{2}+3\tau m+3k}} -\frac{2\mathfrak{q}x_{\frakj}}{2^{\frac{m}{2}+k+\tau m}} \right| \lesssim 2^{-3k-\sigma m}$) we deduce that
$$\left|I^k \Delta \left(I^k-\frac{2\mathfrak{q}^3}{2^{\frac{3m}{2}+3\tau m+3k}} + \frac{2\mathfrak{q}x_{\frakj}}{2^{\frac{m}{2}+k+\tau m}} \right)\right| \lesssim 2^{-2k} \left|I^k \right|$$ and
$$\left|I_0^{m, k+\tau m}\Delta \left(I_0^{m, k+\tau m}+\frac{\mathfrak{q}^2}{2^{m+2\tau m+2k}} - x_{\frakj}\right) \right| \lesssim 2^{-(\sigma-\tau) m} \left|I_0^{m, k+\tau m}\right|$$ for any $\frakj$ and $\mathfrak{q}\in \mathfrak{I}_{\frakj}$. These estimates together with \eqref{20230211eq01} and the fact that $f_1$ is morally a constant on intervals of length $2^{-m-k}$ implies that \eqref{20230212eq04} reduces up to an admissible error of at most
\begin{equation} \label{Err0}
O( 2^{-2(\sigma-\tau-\mu_0)m}\,\left\|f_1 \right\|^2_{L^2 \left(3 I^k \right)}\left\|f_2 \right\|^2_{L^2 \left(3 I^k \right)}
\end{equation}
to
\begin{equation} \label{20230212eq05}
\frac{1}{\left|I^{m,k+\tau m}_0 \right|}\sum_{\frakj \sim 2^{\sigma m}}\sum_{\mathfrak{q}\in \mathfrak{I}_{\frakj}} \int_{I^k}\bigg|\int_{I^{m,k+\tau m}_0} f_1 \left(x-x_{\frakj}-t-\frac{\mathfrak{q}}{2^{\frac{m}{2}+k+\tau m}}\right)
f_2\left(x+\frac{2\mathfrak{q}t}{2^{\frac{m}{2}+k+\tau m}}\right) e^{3i2^{\frac{m}{2}+k}t \cdot \frac{\left\lfloor\frac{\mathfrak{q}}{2^{\tau m}}\right\rfloor^2}{2^m} \lambda\left(x-x_{\frakj}\right) }dt \bigg|^2 dx.
\end{equation}
We now choose another parameter $\varsigma$ with $0<5 \varsigma<\tau(<\sigma)$ and further subdivide
$$\bigcup_{\frakj \sim 2^{\sigma m}} \bigcup_{\mathfrak{q}\in \mathfrak{I}_{\frakj}}=\bigcup_{\frakl\sim 2^{\varsigma m}} \bigcup_{\frakj= \frakl 2^{(\sigma-\varsigma) m}}^{(\frakl+1) 2^{(\sigma-\varsigma) m}} \bigcup_{\mathfrak{q}\in \mathfrak{I}_{\frakj}}=:\bigcup_{\frakl\sim 2^{\varsigma m}} \bigcup_{\frakj\in A_{\frakl}} \bigcup_{\mathfrak{q}\in \mathfrak{I}_{\frakj}}\,.$$
Letting $A:=A_{2^{\varsigma m}}$ and using an $l^{\infty}$ argument in $\frakl$ we may assume wlog that \eqref{20230212eq05} is controlled by
\begin{equation} \label{supl}
\frac{2^{\varsigma m}}{\left|I^{m,k+\tau m}_0 \right|}\sum_{\frakj\in A}\sum_{\mathfrak{q}\in \mathfrak{I}_{\frakj}}\int_{I^k} \bigg| \int_{I^{m,k+\tau m}_0} f_1 \left(x-x_{\frakj}-t-\frac{\mathfrak{q}}{2^{\frac{m}{2}+k+\tau m}}\right)
f_2\left(x+\frac{2\mathfrak{q}t}{2^{\frac{m}{2}+k+\tau m}}\right) e^{3i2^{\frac{m}{2}+k}t \cdot \frac{\left\lfloor\frac{\mathfrak{q}}{2^{\tau m}}\right\rfloor^2}{2^m} \lambda\left(x-x_{\frakj}\right) }dt \bigg|^2 dx.
\end{equation}

Next, by applying the change of variable $t \to \frac{2^{\frac{m}{2}+\tau m}}{\mathfrak{q}} t$, we have that \eqref{supl} is equivalent with
\begin{equation} \label{20230212eq060}
\frac{2^{\varsigma m}}{\left|I^{m,k+\tau m}_0 \right|}\sum_{\frakj\in A}\sum_{\mathfrak{q}\in \mathfrak{I}_{\frakj}}
\int_{I^k} \bigg| \int_{\frac{\mathfrak{q}}{2^{\frac{m}{2}+\tau m}}I^{m,k+\tau m}_0} f_1 \left(x-x_{\frakj}-\frac{2^{\frac{m}{2}+\tau m}}{\mathfrak{q}} t-\frac{\mathfrak{q}}{2^{\frac{m}{2}+k+\tau m}}\right)
f_2\left(x+\frac{2t}{2^{k}}\right) e^{3i2^{k}t \cdot \frac{2^{\tau m}}{\mathfrak{q}}\left\lfloor\frac{\mathfrak{q}}{2^{\tau m}}\right\rfloor^2 \lambda\left(x-x_{\frakj}\right) }dt \bigg|^2 dx.
\end{equation}

Using now that for any $j, j'\in A$ and $\mathfrak{q}\in \mathfrak{I}_{\frakj}$ we have that
$\left|I^{m,k+\tau m}_0\Delta \left(\frac{\mathfrak{q}}{2^{\frac{m}{2}+\tau m}}I^{m,k+\tau m}_0\right) \right| \lesssim 2^{-\varsigma m} \left|I^{m,k+\tau m}_0\right|$ and that $|x_j-x_{j'}|\lesssim 2^{-2k-\varsigma m}$, via a similar approximation argument with the one above, we have that up to an admissible global error term of at most
\begin{equation} \label{Err01}
O\left( 2^{-(\varsigma-2\mu_0)m}\,\left\|f_1 \right\|^2_{L^2 \left(3 I^k \right)}\left\|f_2 \right\|^2_{L^2 \left(3 I^k \right)}\right)\,,
\end{equation}
the expression in \eqref{20230212eq060} reduces to
\begin{equation} \label{20230212eq06}
\frac{2^{\varsigma m}}{\left|I^{m,k+\tau m}_0 \right|}\sum_{\frakj\in A}\sum_{\mathfrak{q}\in \mathfrak{I}_{\frakj}}
\int_{I^k} \bigg| \int_{I^{m,k+\tau m}_0} f_1 \left(x-x_{\frakj}-\frac{2^{\frac{m}{2}+\tau m}}{\mathfrak{q}} t-\frac{\mathfrak{q}}{2^{\frac{m}{2}+k+\tau m}}\right)
f_2\left(x+\frac{2t}{2^{k}}\right) e^{3i2^{k}t \cdot \frac{2^{\tau m}}{\mathfrak{q}}\left\lfloor\frac{\mathfrak{q}}{2^{\tau m}}\right\rfloor^2 \lambda\left(x\right) }dt \bigg|^2 dx\,,
\end{equation}
wehere in the above we made a slight notational abuse and redenoted by $\lambda(x)$ the expression $\lambda(x-x_{j_0})$ where here $j_0$ stands for a fixed representative within the set $A$.

Now, by Cauchy-Schwarz, we have
\begin{equation*}
\eqref{20230212eq06}\lesssim 2^{\frac{m}{2}+k+\tau m +\varsigma m} \frakE \frakF,
\end{equation*}
where
$$
\frakE^2:= \int_{I^k} \iint_{\left(I^{m,k+\tau m}_0 \right)^2} \left| f_2\left(x+\frac{2t}{2^k} \right)f_2\left(x+\frac{2s}{2^k} \right) \right|^2 dtdsdx,
$$
and
\begin{eqnarray*}
&&\frakF^2:= \int_{I^k} \iint_{\left(I^{m,k+\tau m}_0 \right)^2} \bigg |\sum_{\frakj\in A}\sum_{\mathfrak{q}\in \mathfrak{I}_{\frakj}} f_1 \left(x-x_{\frakj}-\frac{2^{\frac{m}{2}+\tau m}t}{\mathfrak{q}}-\frac{\mathfrak{q}}{2^{\frac{m}{2}+\tau m+k}}\right)\\
&& \qquad\qquad\qquad f_1 \left(x-x_{\frakj}-\frac{2^{\frac{m}{2}+\tau m}s}{\mathfrak{q}}-\frac{\mathfrak{q}}{2^{\frac{m}{2}+\tau m+k}}\right) e^{3i2^k(t-s) \cdot \frac{2^{\tau m}}{\mathfrak{q}}\left\lfloor\frac{\mathfrak{q}}{2^{\tau m}}\right\rfloor^2 \lambda (x) }\bigg|^2\,dtdsdx.
\end{eqnarray*}
As usual, $\frakE$ is easy to treat since $f_2$ is uniform:
\begin{equation}
 \frakE^2=\int_{I^k} \left( \int_{I^{m,k+\tau m}_0} \left|f_2 \left(x+\frac{2t}{2^k} \right) \right|^2 dt \right)^2 dx \lesssim 2^{-k} \cdot \left(2^{\mu_0 m} \cdot 2^{-\frac{m}{2}-\tau m} \right)^2 \cdot \left\|f_2\right\|_{L^2 \left(3 I^k \right)}^4 = \frac{2^{2(\mu_0-\tau) m}}{2^{m+k}} \left\|f_2 \right\|_{L^2\left(3 I^k\right)}^4.
\end{equation}
Therefore, the main estimate now becomes
\begin{equation} \label{20230212eq10}
\textbf{L}_{m,k}^2 \left(f_1, f_2 \right)%
\lesssim 2^{\frac{k}{2}} \cdot 2^{\mu_0 m+ \varsigma m} \left\|f_2 \right\|_{L^2 \left(3 I^k \right)}^2 \frakF.
\end{equation}

\noindent\textsf{Step II: Treatment of $\frakF$.}  In what follows, we focus on the $\frakF$-term:
\begin{eqnarray*}
&&\frakF^2=\int_{I^k} \iint_{\left(I^{m,k+\tau m}_0 \right)^2} \sum_{\frakj\in A}\sum_{\mathfrak{q}\in \mathfrak{I}_{\frakj}} \sum_{\frakj_1\in A}\sum_{\mathfrak{q}_1\in \mathfrak{I}_{\frakj_1}}
 f_1 \left(x-x_{\frakj}-\frac{2^{\frac{m}{2}+\tau m}t}{\mathfrak{q}}-\frac{\mathfrak{q}}{2^{\frac{m}{2}+\tau m+k}}\right)\\
&&f_1 \left(x-x_{\frakj_1}-\frac{2^{\frac{m}{2}+\tau m}t}{\mathfrak{q}_1}-\frac{\mathfrak{q}_1}{2^{\frac{m}{2}+\tau m+k}}\right) f_1\left(x-x_{\frakj}-\frac{2^{\frac{m}{2}+\tau m}s}{\mathfrak{q}}-\frac{\mathfrak{q}}{2^{\frac{m}{2}+\tau m+k}}\right) \\
&&f_1 \left(x-x_{\frakj_1}-\frac{2^{\frac{m}{2}+\tau m}s}{\mathfrak{q}_1}-\frac{\mathfrak{q}_1}{2^{\frac{m}{2}+\tau m+k}}\right) e^{3i2^k(t-s)  \left(\frac{2^{\tau m}}{\mathfrak{q}}\left\lfloor\frac{\mathfrak{q}}{2^{\tau m}}\right\rfloor^2 -\frac{2^{\tau m}}{\mathfrak{q}_1}\left\lfloor\frac{\mathfrak{q}_1}{2^{\tau m}}\right\rfloor^2 \right) \lambda (x) } dtdsdx.\qquad
\end{eqnarray*}
Applying now the change variables $x \to x+\frac{\mathfrak{q}}{2^{\frac{m}{2}+\tau m+k}}$, $t \to \frac{\mathfrak{q}_1}{2^{\frac{m}{2}+\tau m}} t$, and $s \to \frac{\mathfrak{q}_1}{2^{\frac{m}{2}+\tau m}}s$, together with another approximation argument, we have that
\begin{equation} \label{mainerrorterm}
\textrm{up to an admissible global error term}\: O( 2^{-(\frac{\varsigma}{2}-2\mu_0)m}\,\left\|f_1 \right\|^2_{L^2 \left(3 I^k \right)}\left\|f_2 \right\|^2_{L^2 \left(3 I^k \right)})
\end{equation}
one can reduce matters to
\begin{eqnarray}\label{initF}
&& \int_{I_{-1}^{0, k}} \iint_{\left(I^{m,k+\tau m}_0 \right)^2} \sum_{{\frakj,\,\frakj_1\in A}\atop{|\frakj-\frakj_1|\geq 2^{(\sigma-2\varsigma) m}}}\sum_{{\mathfrak{q}\in \mathfrak{I}_{\frakj}}\atop{\mathfrak{q}_1\in \mathfrak{I}_{\frakj_1}}}
 f_1 \left(x-x_{\frakj}-\frac{\mathfrak{q}_1t}{\mathfrak{q}}\right) f_1 \left(x-x_{\frakj_1}-t+\frac{\mathfrak{q}-\mathfrak{q}_1}{2^{\frac{m}{2}+\tau m+k}} \right)\\
&& \quad \cdot  f_1 \left(x-x_{\frakj}-\frac{\mathfrak{q}_1 s}{\mathfrak{q}} \right)  f_1 \left(x-x_{\frakj_1}-s+\frac{\mathfrak{q}-\mathfrak{q}_1}{2^{\frac{m}{2}+\tau m+k}} \right) e^{3i2^{k-\frac{m}{2}}(t-s) \left(\frac{\mathfrak{q}_1}{\mathfrak{q}}\left\lfloor\frac{\mathfrak{q}}{2^{\tau m}}\right\rfloor^2 -\left\lfloor\frac{\mathfrak{q}_1}{2^{\tau m}}\right\rfloor^2 \right)\lambda \left(x+\frac{\mathfrak{q}}{2^{\frac{m}{2}+\tau m+k}} \right) } dtdsdx.\nonumber
\end{eqnarray}
Now, for a fixed $\frakj, \frakj_1$ with $2^{(\sigma-2\varsigma) m}\leq |\frakj-\frakj_1|\leq 2^{(\sigma-\varsigma) m}$ and under the restriction $p,r\in\Z$, we apply the change of variable
$$
\begin{cases}
p=\mathfrak{q}-\mathfrak{q}_1 \\
\\
\frac{r}{2^{\frac{m}{2}+\tau m}}=\frac{\mathfrak{q}_1}{\mathfrak{q}}+O(2^{-\frac{m}{2}-\tau m})
\end{cases}
$$
Notice that $|p|\sim 2^{\frac{m}{2}+\tau m-\sigma\,m}|\frakj-\frakj_1|\geq  2^{\frac{m}{2}+\tau m-2\varsigma\,m}$. This in turn implies
$$
\begin{cases}
\mathfrak{q}=\frac{p}{1-\frac{r}{2^{\frac{m}{2}+\tau m}}}+O\left(\frac{2^{\sigma m}}{|j-j_1|}\right) \\
\\
\mathfrak{q}_1=\frac{p \cdot \frac{r}{2^{\frac{m}{2}+\tau m}}}{1-\frac{r}{2^{\frac{m}{2}+\tau m}}}+O\left(\frac{2^{\sigma m}}{|j-j_1|}\right)
\end{cases}\,.
$$
Using that $\frac{\mathfrak{q}_1}{\mathfrak{q}}\,t=\frac{r}{2^{\frac{m}{2}+\tau m}}\,t\,+\,O(\frac{1}{2^{m+k+2\tau m}})$ and the morally $2^{-m-k}$--constancy of $f_1$, and setting
$$R:=2^{\frac{m}{2}+\tau m}\,\frac{j_1}{j}+\left[-2^{\frac{m}{2}+\tau m-\sigma\,m},\,2^{\frac{m}{2}+\tau m-\sigma\,m}\right]$$
and
$$P(r):=\left(1-\frac{r}{2^{\frac{m}{2}+\tau m}}\right)\,2\mathfrak{I}_{\frakj} \cap\left(\frac{2^{\frac{m}{2}+\tau m}}{r}-1\right)\,2\mathfrak{I}_{\frakj_1}$$
we have that\footnote{It is here where we exploit the crucial constancy hypothesis at scale $2^{-\frac{m}{2}-k}$ imposed on $\lambda$.}
\begin{equation} \label{mainerrorterm2}
\textrm{up to an admissible global error term}\: O( 2^{-(\frac{\tau-2\varsigma}{2}-2\mu_0)m}\,\left\|f_1 \right\|^2_{L^2 \left(3 I^k \right)}\left\|f_2 \right\|^2_{L^2 \left(3 I^k \right)})
\end{equation}
\eqref{initF} is bounded from above by at most $O(2^{2(\sigma-\varsigma) m})$ terms of the form
\begin{eqnarray} \label{20230213eq10}
&&\frac{2^{\sigma m}}{|j-j_1|}\,\int_{I_{-1}^{0, k}} \iint_{\left(I^{m,k+\tau m}_0 \right)^2} \sum_{r\in R} \left|\sum _{p\in P(r)} f_1 \left(x-x_{\frakj}-\frac{rt}{2^{\frac{m}{2}+\tau m}}\right) f_1 \left(x-x_{\frakj_1}-t+\frac{p}{2^{\frac{m}{2}+k+\tau m}} \right)\right.\,\nonumber\\
&&f_1 \left(x-x_{\frakj}-\frac{rs}{2^{\frac{m}{2}+\tau m}} \right) f_1 \left(x-x_{\frakj_1}-s+\frac{p}{2^{\frac{m}{2}+k+\tau m}} \right) e^{3i2^{\frac{m}{2}+k+\tau m} (t-s)  \Phi_{r, p}(x)}\Bigg|\,dtdsdx,
\end{eqnarray}
where from now on $\frakj, \frakj_1$ are fixed with $2^{(\sigma-2\varsigma) m}\leq |\frakj-\frakj_1|\leq 2^{(\sigma-\varsigma) m}$ and
\begin{equation} \label{20230215eq01}
\Phi_{r, p}(x):=\frac{\frac{r}{2^{\frac{m}{2}+\tau m}} \cdot p^2}{2^{m+3\tau m} \left(1-\frac{r}{2^{\frac{m}{2}+\tau m}} \right)} \lambda \left(x+\frac{\frac{p}{1-\frac{r}{2^{\frac{m}{2}+\tau m}}}}{2^{\frac{m}{2}+k+\tau m}} \right)\,.
\end{equation}

We then apply Cauchy-Schwarz twice, first in the variable $r$ and next in the $3$-tuple $(t, s, x)$ to deduce
\begin{equation*}
\eqref{20230213eq10} \lesssim \frac{2^{\sigma m}}{|j-j_1|}\, \frakG \frakH,
\end{equation*}
where, for notational simplicity letting $f_{1,a}(x):=f_1(x-a)$, $a\in\R$, we have
$$
\frakG^2:=\int_{I_{-1}^{0, k}}\iint_{\left(I^{m,k+\tau m}_0 \right)^2} \sum_{r\in R} \left|f_{1,x_{\frakj}} \left(x-\frac{rt}{2^{\frac{m}{2}+\tau m}} \right) f_{1,x_{\frakj}} \left(x-\frac{rs}{2^{\frac{m}{2}+\tau m}} \right)  \right|^2 dtdsdx,
$$
and
\begin{equation}
\frakH^2:=\int_{I_{-1}^{0, k}} \iint_{\left(I^{m,k+\tau m}_0 \right)^2} \sum_{r\in R} \bigg|  \sum_{p \in P(r)} f_{1,x_{\frakj_1}} \left(x-t+\frac{p}{2^{\frac{m}{2}+k+\tau m}} \right) f_{1,x_{\frakj_1}} \left(x-s+\frac{p}{2^{\frac{m}{2}+k+\tau m}} \right)\,e^{3i2^{\frac{m}{2}+k+\tau m} (t-s)  \Phi_{r, p}(x) } \bigg|^2\,dtdsdx\,.
 \end{equation}

The term $\frakG$ can be handled by using the uniformity assumption on $f_1$ and Observation \ref{CZtype}:
\begin{equation}
\frakG^2=\sum_{r\in R} \int_{I_{-1}^{0, k}} \left(\int_{I^{m,k+\tau m}_0} \left|f_{1,x_{\frakj}} \left(x-\frac{rt}{2^{\frac{m}{2}+\tau m}}  \right) \right|^2 dt \right)^2 dx \lesssim 2^{-\frac{m}{2}-k} \cdot 2^{2\mu_0 m} \cdot 2^{-\sigma m}\cdot2^{-\tau m}\cdot \left\|f_1\right\|_{L^2(3I^k)}^4.
\end{equation}
Therefore, the main estimate \eqref{20230212eq10} becomes now
\begin{equation} \label{20230213eq20}
\textbf{L}_{m,k}^2 \left(f_1, f_2 \right)\lesssim  2^{\frac{k}{4}-\frac{m}{8}} \cdot 2^{\frac{3\mu_0m}{2}+\frac{3}{4}\sigma m+\varsigma m-\frac{\tau}{4}m} \left\|f_1 \right\|_{L^2\left(3 I^k \right)}  \left\|f_2\right\|^2_{L^2\left(3 I^k \right)}   \frakH^{\frac{1}{2}}.
\end{equation}
Next, we estimate the term $\frakH$. Note that
\begin{eqnarray*}
\frakH^2&&=\sum_{r\in R}\sum_{ p_1, p_2 \in P(r)} \int_{I_{-1}^{0, k}} \iint_{\left(I^{m,k+\tau m}_0 \right)^2} f_{1,x_{\frakj_1}} \left(x-t+\frac{p_1}{2^{\frac{m}{2}+k+\tau m}} \right) f_{1,x_{\frakj_1}} \left(x-s+\frac{p_1}{2^{\frac{m}{2}+k+\tau m}} \right) \\
&& \cdot f_{1,x_{\frakj_1}}\left(x-t+\frac{p_2}{2^{\frac{m}{2}+k+\tau m}} \right) f_{1,x_{\frakj_1}} \left(x-s+\frac{p_2}{2^{\frac{m}{2}+k+\tau m}} \right)e^{3i2^{\frac{m}{2}+k+\tau m} (t-s) \left(\Phi_{r, p_1}(x)-\Phi_{r, p_2}(x) \right) }  dtdsdx.
\end{eqnarray*}
Next, we set $P:=[(\frakj-\frakj_1-1)\,2^{\frac{m}{2}+\tau m-\sigma m},\,(\frakj-\frakj_1+1)\,2^{\frac{m}{2}+\tau m-\sigma m}]$ and define the interval $R(p)$ by the condition $p\in P(r)$ iff $r\in R(p)$. Then, by applying first Fubini followed by a Cauchy-Schwarz, we have
\begin{equation} \label{20230214eq11}
\frakH^2 \le \calI \calJ
\end{equation}
where
\begin{eqnarray*}
&& \calI^2:=\sum_{p_1, p_2 \in P} \int_{I_{-1}^{0, k}} \iint_{\left(I^{m,k+\tau m}_0 \right)^2} \bigg|f_{1,x_{\frakj_1}} \left(x-t+\frac{p_1}{2^{\frac{m}{2}+k+\tau m}} \right) f_{1,x_{\frakj_1}}\left(x-t+\frac{p_2}{2^{\frac{m}{2}+k+\tau m}} \right)\\
&& \quad \quad \quad \quad \quad \quad \quad \quad \quad \quad \quad \quad \cdot  f_{1,x_{\frakj_1}} \left(x-s+\frac{p_1}{2^{\frac{m}{2}+k+\tau m}} \right)f_{1,x_{\frakj_1}} \left(x-s+\frac{p_2}{2^{\frac{m}{2}+k+\tau m}} \right)\bigg|^2 dtdsdx
\end{eqnarray*}
and, for $R(p_1,p_2):=R(p_1)\cap R(p_2)$
$$
\calJ^2:=\sum_{p_1, p_2 \in P} \int_{I_{-1}^{0, k}} \iint_{\left(I^{m,k+\tau m}_0 \right)^2} \left| \sum_{r\in R(p_1,p_2)} e^{3i 2^{\frac{m}{2}+k+\tau m} (t-s) \left(\Phi_{r, p_1}(x)-\Phi_{r, p_2}(x) \right)} \right|^2 dt ds dx.
$$
The former term $\calI^2$ can be handled easily due to the uniformity of $f_1$:
\begin{equation}
\calI^2\lesssim 2^k \cdot 2^{4\mu_0 m}\cdot 2^{-2\sigma m} \cdot \left\|f_1 \right\|^8_{L^2 \left(3 I^k \right)},
\end{equation}
which, together with \eqref{20230213eq20} and \eqref{20230214eq11}, gives
\begin{equation} \label{20230214eq12}
\textbf{L}_{m,k}^2 \left(f_1, f_2 \right)\lesssim 2^{\frac{3k}{8}-\frac{m}{8}} \cdot 2^{2\mu_0 m+\frac{1}{2}\sigma m-\frac{\tau}{4}m+\varsigma m} \left\|f_1 \right\|_{L^2(3 I^k)}^2 \left\|f_2 \right\|_{L^2(3 I^k)}^2 \calJ^{\frac{1}{4}}.
\end{equation}

\medskip

\noindent\textsf{Step III: Treatment of $\calJ$.} By definition, we have
$$
\calJ^2=\sum_{\substack{p_1, p_2 \in P\\ r_1, r_2\in R(p_1,p_2)}} \int_{I_{-1}^{0, k}} \iint_{\left(I^{m,k+\tau m}_0 \right)^2} e^{3i2^{\frac{m}{2}+k+\tau m}(t-s)\,V_{p_1, p_2, r_1, r_2}(x)} dtdsdx,
$$
where
\begin{equation} \label{phaseV}
V_{p_1, p_2, r_1, r_2}(x):=\Phi_{r_1, p_1}(x)-\Phi_{r_1, p_2}(x)-\Phi_{r_2, p_1}(x)+\Phi_{r_2, p_2}(x).
\end{equation}

Deduce that up to an admissible global error term of the form $O( 2^{-(\frac{\sigma}{8}-\varsigma -2\mu_0)m}\,\left\|f_1 \right\|^2_{L^2 \left(3 I^k \right)}\left\|f_2 \right\|^2_{L^2 \left(3 I^k \right)})$, we have that
\begin{equation} \label{20230215eq10}
\calJ^2\lesssim 2^{-m-2k-2\tau m}\sum_{p_1, p_2 \in P}\:\:\sum_{{r_1, r_2 \in R(p_1,p_2)}\atop{|r_1-r_2|\geq 2^{\frac{m}{2}+\tau m-2\sigma m}}} \int_{I_{-1}^{0, k}} \lceil V_{p_1, p_2, r_1, r_2}(x)\rceil\,dx\,,
\end{equation}
where, from the definition of $\Phi_{r, p}$ (see, \eqref{20230215eq01}), we have
\begin{eqnarray*}
&&\left|V_{p_1, p_2, r_1, r_2}(x) \right|\\
&=&\vast| \frac{r_1}{2^{m+3\tau m}}(2^{\frac{m}{2}+\tau m}-r_1) \left(\frac{p_1}{2^{\frac{m}{2}+\tau m}-r_1}\right)^2 \lambda \left(x+\frac{1}{2^{k}} \cdot \frac{p_1}{2^{\frac{m}{2}+\tau m}-r_1} \right) \\
&-&\frac{r_1}{2^{m+3\tau m}}(2^{\frac{m}{2}+\tau m}-r_1) \left(\frac{p_2}{2^{\frac{m}{2}+\tau m}-r_1}\right)^2  \lambda \left(x+\frac{1}{2^k} \cdot \frac{p_2}{2^{\frac{m}{2}+\tau m}-r_1} \right) \\
&-& \frac{r_2}{2^{m+3\tau m}}(2^{\frac{m}{2}+\tau m}-r_2) \left(\frac{p_1}{2^{\frac{m}{2}+\tau m}-r_2}\right)^2  \lambda \left(x+\frac{1}{2^k} \cdot \frac{p_1}{2^{\frac{m}{2}+\tau m}-r_2} \right)\\
&+& \frac{r_2}{2^{m+3\tau m}}(2^{\frac{m}{2}+\tau m}-r_2) \left(\frac{p_2}{2^{\frac{m}{2}+\tau m}-r_2}\right)^2 \lambda \left(x+\frac{1}{2^k} \cdot \frac{p_2}{2^{\frac{m}{2}+\tau m}-r_2}\right) \vast |.
\end{eqnarray*}

Our goal here is to improve over the trivial estimate
$$\calJ^2\lesssim 2^{m-3k-4\sigma m+2\tau m-2\varsigma m}\,.$$
This will be done by exploiting the curvature that is present in the mixing of the parameters $p_1, p_2, r_1, r_2$ that are part from the definition of \eqref{phaseV}, spirit that has been seen present taking different but related forms first in \cite{Lie19} (see the non-degeneracy condition in Definition 11 and the proof of Proposition 49) and more recently in \cite{BBLV21} (see the proof of Proposition 5.5 therein).

In order to fulfill our goal, we first apply the change variable
$$
u_1=\left\lfloor\frac{2^{\frac{m}{2}+\tau m}p_1}{2^{\frac{m}{2}+\tau m}-r_1}\right\rfloor, \quad u_2=\left\lfloor\frac{2^{\frac{m}{2}+\tau m}p_2}{2^{\frac{m}{2}+\tau m}-r_1}\right\rfloor, \quad \textrm{and}\quad v_1=\left\lfloor\frac{2^{\frac{m}{2}+\tau m}p_1}{2^{\frac{m}{2}+\tau m}-r_2}\right\rfloor,
$$
that is,
$$
\frac{p_1}{2^{\frac{m}{2}+\tau m}-r_1}=\frac{u_1}{2^{\frac{m}{2}+\tau m}}+O(2^{-\frac{m}{2}-\tau m}), \quad \frac{p_2}{2^{\frac{m}{2}+\tau m}-r_1}=\frac{u_2}{2^{\frac{m}{2}+\tau m}}+O(2^{-\frac{m}{2}-\tau m})$$
and  $\frac{p_1}{2^{\frac{m}{2}+\tau m}-r_2}=\frac{v_1}{2^{\frac{m}{2}+\tau m}}+O(2^{-\frac{m}{2}-\tau m})$.
Notice that this further implies
$$\frac{p_2}{2^{\frac{m}{2}+\tau m}-r_2}=\frac{u_2 v_1}{2^{\frac{m}{2}+\tau m} u_1}+O(2^{-\frac{m}{2}-\tau m}) \quad \textrm{and} \quad
\frac{2^{\frac{m}{2}+\tau m}-r_2}{2^{\frac{m}{2}+\tau m}-r_1}=\frac{u_1}{v_1}+O(2^{-\frac{m}{2}-\tau m})\,.$$
Then, using the fact that $\lambda$ is constant on intervals of length $2^{-\frac{m}{2}-k}$ with $\lambda\sim 2^{\frac{m}{2}}$, and the observation that due to the restrictions in the ranges of $p_1,p_2,r_1,r_2$ we have $|p_1|\, |p_2|,\,|2^{\frac{m}{2}+\tau m}-r_1|,\,|2^{\frac{m}{2}+\tau m}-r_2|\sim |\frakj-\frakj_1|\,2^{\frac{m}{2}+\tau m-\sigma m}$, $u_1,u_2,v_1\in 2\mathfrak{I}_{\frakj}$ and $|u_1-v_1|\gtrsim 2^{\frac{m}{2}+\tau m-2\sigma m}$, we now have that
\begin{equation} \label{mainerrorterm3}
\textrm{up to an admissible global error term}\: O( 2^{-(\frac{\tau}{8}-2\mu_0-\varsigma)m}\,\left\|f_1 \right\|^2_{L^2 \left(3 I^k \right)}\left\|f_2 \right\|^2_{L^2 \left(3 I^k \right)})
\end{equation}

\begin{eqnarray*}
&&|\tilde{V}_{r_1, u_1,u_2,v_1}(x)|:=\frac{|\frakj-\frakj_1|\,2^{\frac{m}{2}+\tau m-\sigma m}}{|r_1-2^{\frac{m}{2}+\tau m}|}\,\left|V_{p_1, p_2, r_1, r_2}(x) \right|\\
&=& |\frakj-\frakj_1|\,2^{\frac{m}{2}-2\tau m-\sigma m}\vast | r_1\left(\frac{u_1}{2^{m+\tau m}} \right)^2 \lambda \left(x+\frac{u_1}{2^{\frac{m}{2}+k+\tau m}} \right)- r_1\left(\frac{u_2}{2^{m+\tau m}} \right)^2 \lambda \left(x+\frac{u_2}{2^{\frac{m}{2}+k+\tau m}} \right) \\
&-& \left(\frac{u_1}{v_1}(r_1-2^{\frac{m}{2}+\tau m})+2^{\frac{m}{2}+\tau m}\right)\frac{u_1}{v_1}\left(\frac{v_1}{2^{m+\tau m}} \right)^2\lambda \left(x+\frac{v_1}{2^{\frac{m}{2}+k+\tau m}} \right)\\
&+&\left(\frac{u_1}{v_1}(r_1-2^{\frac{m}{2}+\tau m})+2^{\frac{m}{2}+\tau m}\right)\frac{u_1}{v_1}\left(\frac{u_2v_1}{u_1 2^{m+\tau m}} \right)^2 \lambda \left(x+\frac{\frac{u_2v_1}{u_1}}{2^{\frac{m}{2}+k+\tau m}} \right) \vast|.
\end{eqnarray*}

Separating the terms in the $r_1$ variable we further have
\begin{equation}
|\tilde{V}_{r_1, u_1,u_2,v_1}(x)| \simeq |\frakj-\frakj_1|\,2^{-\tau m-\sigma m}\, \left| A_{u_1,u_2,v_1}(x)+\frac{r_1}{2^{\frac{m}{2}+\tau m}} B_{u_1,u_2,v_1}(x) \right|,
\end{equation}
where
\begin{equation*}
A_{u_1,u_2,v_1}(x):=-\frac{(v_1-u_1)u_1}{2^{m+2\tau m}} \lambda \left(x+\frac{v_1}{2^{\frac{m}{2}+k+\tau m}} \right)+\frac{(v_1-u_1)u_2^2}{2^{m+2\tau m} u_1} \lambda \left(x+\frac{u_2v_1}{2^{\frac{m}{2}+k+2\tau m} u_1} \right)
\end{equation*}
and
\begin{eqnarray*}
&&B_{u_1,u_2,v_1}(x):=\frac{u_1^2}{2^{m+2\tau m}} \lambda \left(x+\frac{u_1}{2^{\frac{m}{2}+k+\tau m}} \right)-\frac{u_2^2}{2^{m+2\tau m}} \lambda \left(x+\frac{u_2}{2^{\frac{m}{2}+k+\tau m}} \right)\\
&&\qquad\qquad\qquad-\frac{u_1^2}{2^{m+2 \tau m}} \lambda \left(x+\frac{v_1}{2^{\frac{m}{2}+k+\tau m}} \right)+\frac{u_2^2}{2^{m+2\tau m}} \lambda \left(x+\frac{u_2v_1}{2^{\frac{m}{2}+k+\tau m} u_1} \right).
\end{eqnarray*}
Putting everything together, we have that \eqref{20230215eq10} implies

\begin{equation} \label{Jkey}
\calJ^2\lesssim 2^{-m-2k-2\tau m}\sum_{r_1\in R}\:\sum_{{u_1,u_2,v_1\in 2\mathfrak{I}_{\frakj}}\atop{|u_1-v_1|\gtrsim 2^{\frac{m}{2}+\tau m-2\sigma m}}} \int_{I_{-1}^{0, k}} \left\lceil 2^{-2\varsigma m-\tau m}\, \left| A_{u_1,u_2,v_1}(x)+\frac{r_1}{2^{\frac{m}{2}+\tau m}} B_{u_1,u_2,v_1}(x) \right|\right\rceil\,dx\,.
\end{equation}
Set now:
\begin{equation*}
\tilde{A}_{u_1,u_2,v_1}(x):=-\frac{2^{m+2\tau m}}{(v_1-u_1)u_1}\,A_{u_1,u_2,v_1}\left(x-\frac{v_1}{2^{\frac{m}{2}+k+\tau m}}\right)=\lambda\left(x\right)-\frac{u_2^2}{u_1^2}  \lambda \left(x+\frac{v_1}{2^{\frac{m}{2}+k+\tau m}}\left(\frac{u_2}{u_1}-1\right)\right)
\end{equation*}
and
\begin{equation*}
\tilde{B}_{u_1,u_2,v_1}(x):=-\frac{2^{m+2\tau m}}{(v_1-u_1)u_1}\,B_{u_1,u_2,v_1}\left(x-\frac{v_1}{2^{\frac{m}{2}+k+\tau m}}\right)\,.
\end{equation*}
Then \eqref{Jkey} becomes
\begin{equation} \label{Jkey1}
\calJ^2\lesssim 2^{-m-2k-2\tau m}\sum_{r_1\in R}\:\sum_{{u_1,u_2,v_1\in 2\mathfrak{I}_{\frakj}}\atop{|u_1-v_1|\gtrsim 2^{\frac{m}{2}+\tau m-2\sigma m}}} \int_{5I^{k}} \left\lceil 2^{-4\sigma m}\, \left| \tilde{A}_{u_1,u_2,v_1}(x)\frac{2^{\frac{m}{2}+\tau m}}{r_1} \,+\,\tilde{B}_{u_1,u_2,v_1}(x) \right|\right\rceil\,dx\,.
\end{equation}

Set now $\frac{u_2}{u_1}=\frac{q}{2^{\frac{m}{2}+\tau m}}+O(2^{-\frac{m}{2}-\tau m})$ and, for $l\in\{2,\ldots,\,2^{(\sigma-\varsigma)m}\}$, define
$$\mathcal{C}_l:=\left\{q\in 10\mathfrak{I}_{2^{\sigma m}}\,\Big|\,\left|\frac{q}{2^{\frac{m}{2}+\tau m}}-1\right|\in [(l-1) 2^{\varsigma m-2\sigma m},\, l 2^{\varsigma m-2\sigma m}]\right\}\,.$$
For each such $l$ (fixed) and $q\in \mathcal{C}_l $ we further make the change of variable
$v=\left\lfloor \frac{2^{(-\varsigma+2\sigma) m}}{l}\,(\frac{q}{2^{\frac{m}{2}+\tau m}}-1)\,v_1\right\rfloor$.

Deduce now, that  up to the same admissible global error term as in \eqref{mainerrorterm3}, we have

\begin{eqnarray}\label{Jkey2}
&&\qquad\qquad\qquad\calJ^2\lesssim 2^{-m-2k-2\tau m}\,\sum_{l=2}^{2^{(\sigma-\varsigma)m}}\,\sum_{r_1\in R}\:\sum_{{u_1\in 2\mathfrak{I}_{\frakj},\: v\in 4 \mathfrak{I}_{\frakj}}\atop{q\in 10\mathfrak{I}_{2^{\sigma m}}}}\nonumber\\
&& \int_{5I^{k}} \left\lceil \left|2^{-4\sigma m}\,
\left(\lambda\left(x\right)-\frac{q^2}{2^{m+2\tau m}}  \lambda \left(x+\frac{v\,l\,2^{(\varsigma-2\sigma) m}}{2^{\frac{m}{2}+k+\tau m}}\right)\right)\frac{2^{\frac{m}{2}+\tau m}}{r_1}\,
+\,C_{q,l,u_1,v}(x)\right|\right\rceil\,dx\,,
\end{eqnarray}
for some suitable measurable functions $|C_{q,l,u_1,v}(x)|\lesssim 2^{\frac{m}{2}}$.

Next, for each fixed $x$, $l$, $u_1$ and $v$, we excise the set of $q$'s for which
$\left|\lambda\left(x\right)-\frac{q^2}{2^{m+2\tau m}}  \lambda \left(x+\frac{v\,l\,2^{(\varsigma-2\sigma) m}}{2^{\frac{m}{2}+k+\tau m}}\right)\right|<2^{\frac{m}{2}-3\sigma m}$ obtaining thus an admissible global contribution of at most
$O( 2^{-(\frac{\sigma}{8}-2\mu_0-\varsigma)m}\,\left\|f_1 \right\|^2_{L^2 \left(3 I^k \right)}\left\|f_2 \right\|^2_{L^2 \left(3 I^k \right)})$.

Finally, on the remaining component, summing in the free variable $r_1$, we deduce
\begin{equation}\label{Jkey3}
\calJ^2\lesssim 2^{\frac{1}{2}m-3k+5\sigma m+2\tau m}\,.
\end{equation}

Substituting now \eqref{Jkey3} into \eqref{20230214eq12} and recalling \eqref{Err0}, \eqref{mainerrorterm}, \eqref{mainerrorterm2} and \eqref{mainerrorterm3} we conclude that

\begin{equation} \label{finalprop4p8}
\textbf{L}_{m,k}^2 \left(f_1, f_2 \right)\lesssim 2^{2\mu_0 m}\,(2^{-\frac{m}{16}} \cdot 2^{\frac{9}{8}\sigma m+\varsigma m}\,+\,2^{-2(\sigma-\tau)m}\,+\,2^{-\frac{\varsigma}{2}m}\,+\, 2^{-(\frac{\tau}{8}-\varsigma)m}) \left\|f_1 \right\|_{L^2(3I^k)}^2 \left\|f_2 \right\|_{L^2(3I^k)}^2\,,
\end{equation}
finishing thus the proof of Proposition \ref{mainprop02} with the choice
\begin{equation} \label{choiceps4}
2\epsilon_4=\min\left\{\frac{1}{16}-\frac{9}{8}\sigma-\varsigma,\:
2(\sigma-\tau),\:\frac{\varsigma}{2},\:\frac{\tau}{8}-\varsigma\right\}-2\mu_0\,.
\end{equation}
\end{proof}

We conclude this section with the following

\begin{obs} \label{bkI}[\textsf{Bookkeeping (I)}]  Here we put together the conditions that our many parameters obey:
\begin{itemize}
\item From \eqref{sparse-I}, \eqref{sparse-II}, \eqref{unifc} and \eqref{20230202eq01} we have that the parameter $\epsilon_1$ in Theorem \ref{mainthIb} can be taken as
\begin{equation} \label{ep1}
\epsilon_1=\min\left\{\frac{\mu}{2},\:\:\:\frac{\mu_0-\mu}{2},\:\:\:\epsilon_2-\frac{\mu}{2}\right\}\,.
\end{equation}

\item  Running the bootstrap algorithm first for the values $w=k$ and then for $w=\frac{m}{2}$, from Proposition \ref{constantprop}, Corollary \ref{20230203cor01}, Proposition \ref{mainprop01}, Observation \ref{boots2} and Proposition \ref{mainprop02}  the resulting value for the parameter $\epsilon_2$ in Theorem \ref{20230202thm01} is
\begin{equation} \label{ep2}
\epsilon_2=\min\left\{\frac{\delta_k}{2}-\mu_0,\:\:\:\epsilon_{3}^{k},\:\:\:\frac{\delta_{\frac{m}{2}}}{2}-\mu_0-\epsilon_0^k,\:\:\: \epsilon_{3}^{\frac{m}{2}}-\epsilon_0^k,\:\:\: \epsilon_4-\epsilon_{0}^{\frac{m}{2}}-\epsilon_{0}^{k} \right\}\,.
\end{equation}

\item The parameter $\epsilon_{3}^{w}$ in Proposition \ref{mainprop01}--recalling \eqref{ep3}--is given by
\begin{equation}\label{eps3}
\epsilon_{3}^{w}=\frac{\epsilon_{0}^{w}}{16}-\frac{\del_w}{2}-\mu_0\,.
\end{equation}

\item The parameter $\epsilon_4$ in Proposition \ref{mainprop02}--see \eqref{choiceps4}--is given by
\begin{equation}\label{eps4}
\epsilon_4=\min\left\{\frac{1}{32}-\frac{9}{16}\sigma-\frac{\varsigma}{2},\:
\sigma-\tau,\:\frac{\varsigma}{4},\:\frac{\tau}{16}-\frac{\varsigma}{2}\right\}\,-\mu_0\,.
\end{equation}
\end{itemize}
It is now easy to notice that one can make the choice $\mu_0=2\mu$, $\varsigma=\frac{\tau}{12}$, $\sigma=\frac{49}{48}\tau$,  $\tau=\frac{1}{21}$, $\delta_k=\frac{\epsilon_0^k}{16}$,
$\delta_{\frac{m}{2}}=\frac{\epsilon_0^{\frac{m}{2}}}{16}$, $\epsilon_{0}^{\frac{m}{2}}=33 \epsilon_{0}^{k}$, $\mu=\frac{\epsilon_{0}^{k}}{96}$ and $\epsilon_0^k\approx\frac{1}{50000}$. Then, Theorem \ref{mainthIb} holds with $\epsilon_1\approx 10^{-7}$.
\end{obs}

\section{The diagonal term $\Lambda^{D}$: The case $0 \le k \le \frac{m}{4}$} \label{Sec05}

With the setting and notations specified at the beginning of Section \ref{SDTmotiv} we restate the goal for the present section as

\begin{thm} \label{mainthmI}
There exists some $\bar{\epsilon}>0$, such that for any $0 \le k \le \frac{m}{4}$ and $\vec{p}\in {\bf H}^{+}$
\begin{equation}\label{keyestim}
\left|\Lambda_m^{k, I_0^{0, k}}(\vec{f})\right| \lesssim 2^{-\bar{\epsilon}\,k} \,\|\vec{f} \|_{L^{\vec{p}}(3\, I_0^{0, k})}\,.
\end{equation}
\end{thm}

\subsection{A time-frequency sparse--uniform dichotomy}\label{SU1}

 We start by reconsidering the foliation of the time-frequency space as follows: for $j\in\{1,\ldots, 4\}$, define
$$
\Omega_j^{L_j}:=\left[L_j 2^{m+(\underline{j}-1)k}, \left(L_j+1 \right) 2^{m+(\underline{j}-1)k} \right], \quad L_j \sim 2^k.
$$
where we recall that $\underline{j}=\min \{j, 3\}$, and set
$$
\left\{\xi_j: \xi_j \sim 2^{m+\underline{j} k} \right\}= \bigcup_{L_j \sim 2^k} \Omega_j^{L_j}\,.
$$
Notice that $\left|\Omega_j^{L_j} \right| \ge \left|\omega_j^{l_j} \right|$.

Next, we perform the sparse/uniform decomposition adapted to the above foliation. More precisely, for $\mu \in (0, 1)$ suitably small, we define the \emph{sparse set} of indices as follows: for $j\in\{1, 2\}$,
$$
\bbS_\mu \left(f_j \right):=\left\{ (q_j, L_j) : \int_{3 I_{q_j}^{0, (j+1)k}} \left|f_j^{\Omega_j^{L_j}} \right|^2  \gtrsim 2^{-3\mu k} \int_{3 I_{q_j}^{0, (j+1)k}} \left|f_j \right|^2, \ q_j \sim 2^{jk}, L_j \sim 2^k \right\},
$$
and for $j\in\{3, 4\}$
\begin{equation} \label{20230225eq23}
\bbS_\mu \left(f_j \right):=\left\{ (q_j, L_j) : \int_{3 I_{q_j}^{0, 3k}} \left|f_j^{3 \Omega_j^{L_j}} \right|^2  \gtrsim 2^{-\mu k} \int_{3 I_{q_j}^{0, 3k}} \left|f_j \right|^2, \ q_j \sim 2^{2k}, L_j \sim 2^k \right\}.
\end{equation}
As usual, for $\chi$ a smooth cut-off with $\supp \ \chi \subset \left[\frac{1}{2}, 2\right]$, we denote by $f_j^{\Omega_j^{L_j}}$ the frequency projection onto $\Omega_j^{L_j}$
$$f_j^{\Omega_j^{L_j}}(x):=\int_{\R} \widehat{f_j}(\xi) \chi \left(\frac{\xi-L_j2^{m+(\underline{j}-1)k}}{\left|\Omega_j^{L_j} \right|} \right) e^{i\xi x} d\xi, \quad j\in\{1,\ldots, 4\}\,.$$
 We also define the \emph{uniform sets} of indices as the complement of the corresponding sparse sets.

Now, we decompose our functions into the sparse and uniform components as follows: for $j\in\{1,\ldots,4\}$ and  $\iota:= \min\{j+1,3\}$ we set
\begin{enumerate}
    \item [$\bullet$] the \emph{uniform component} of $f_j$ is given by
    \begin{equation}\label{unifj}
    f_j^{\bbU_\mu}:=\sum_{(q_j, L_j) \in \bbU_\mu(f_j)} \one_{I_{q_j}^{0, \iota k}} f_j^{\Omega_j^{L_j}};
    \end{equation}
    \item [$\bullet$] the \emph{sparse component} of $f_j$ is given by
    \begin{equation}\label{sparsj}
    f_j^{\bbS_\mu}:=\sum_{(q_j, L_j) \in \bbS_\mu(f_j)} \one_{I_{q_j}^{0, \iota k}} f_j^{\Omega_j^{L_j}}.
    \end{equation}
\end{enumerate}

With these done, we decompose
\begin{equation}\label{Ldecksmall}
\Lambda_m^k(\vec{f})=:\Lambda_m^{k, \bbU_\mu}(\vec{f})\,+\,\Lambda_m^{k, \bbS_\mu}(\vec{f}),
\end{equation}
with $\Lambda_m^{k, \bbS_\mu}(\vec{f}):=\Lambda_m^k \left(f_1^{\bbS_\mu}, f_2^{\bbS_\mu}, f_3^{\bbS_\mu}, f_4^{\bbS_\mu} \right)$.

\begin{obs}\textsf{[Error terms II]}\label{ETerm2} In agreement with the view expressed in Observation \ref{ETerm}, in what follows we will only focus on the main terms arising from the time-frequency localization operated in \eqref{unifj} and \eqref{sparsj} thus pretending that the information carried by these functions is (essentially) supported both in space and in frequency corresponding to the given cutoffs. Of course, one can run an approximation argument for the tails or simply work directly with adapted wave-packet decompositions instead of spatial/frequency projections.
\end{obs}

\subsection{Treatment of the uniform component $\Lambda_m^{k, \bbU_\mu}(\vec{f})$}

 Applying Cauchy-Schwarz in \eqref{20230224main01}, we have
\begin{equation} \label{20230224eq01}
|\Lambda_m^{k, \bbU_\mu}(\vec{f})|\lesssim 2^k \sum_{p \sim 2^{\frac{m}{2}+2k}} \bbA_p \bbB_p,
\end{equation}
where
\begin{equation} \label{20230305eq01}
\bbA_p^2:=\sum_{\substack{n_3 \sim 2^k \\ q \sim 2^{\frac{m}{2}}}} \sum_{\left(\ell_1, \ell_2, \ell_3 \right) \in \TFC_q} \left| \left\langle f_3^{\omega_3^{\ell_3}}, \phi_{p+\frac{q^3}{2^m}, \ell_3 2^k+n_3}^{3k} \right \rangle \right|^2  \left| \left \langle f_4 e^{i 2^{\frac{m}{2}+3k} \left(\ell_1 2^{-k}+\ell_2 \right) \left( \cdot \right) }, \phi_{p, \ell_3 2^k+n_3}^{3k} \right \rangle \right|^2,
\end{equation}
and
\begin{eqnarray} \label{20230305eq02}
\bbB_p^2:=\sum_{\substack{n_3 \sim 2^k \\ q \sim 2^{\frac{m}{2}}}} \sum_{\left(\ell_1, \ell_2, \ell_3 \right) \in \TFC_q}  &&\frac{1}{\left|I_p^{m, 3k} \right|}\,\int_{I_p^{m, 3k}} \bigg| \int_{I_q^{m, k}}  \underline{f}_1^{\omega_1^{\ell_1}}(x-t) \underline{f}_2^{\omega_2^{\ell_2}}(x+t^2)\\
&&\cdot e^{i 2^{\frac{m}{2}+k}t \cdot \frac{3q^2}{2^m}n_3}  e^{i 2^{\frac{m}{2}+k}t \cdot 2^k \left(-\ell_1 +\ell_2 \frac{2q}{2^{\frac{m}{2}}}+ \ell_3 \frac{3q^2}{2^m}\right)} dt  \bigg|^2 dx. \nonumber
\end{eqnarray}

\medskip

\noindent\textit{Treatment of $\bbA_p$}. First, applying the change of variable $l_2 \to l_2-l_1 2^{-k}$, we deduce
\begin{equation}\label{ap}
\bbA_p^2\lesssim \sum_{\substack{n_3 \sim 2^k \\ q \sim 2^{\frac{m}{2}}}} \sum_{\ell_2, \ell_3 \sim 2^{\frac{m}{2}-k}}  \left| \left\langle f_3^{\omega_3^{\ell_3}}, \phi_{p+\frac{q^3}{2^m}, \ell_3 2^k+n_3}^{3k} \right \rangle \right|^2  \left| \left \langle f_4, \phi_{p, \ell_3 2^k+n_3+\ell_2}^{3k} \right \rangle \right|^2.
\end{equation}
Next, on each spatial fiber, summing up with respect to $\ell_2$ followed by $n_3$ and using then the fact that $2^k \le 2^{\frac{m}{2}-k}$, we see that the above term is further bounded from above by\footnote{Here we use the convention $\left| \left\langle f_3^{\omega_3^{\ell_3}}, \phi_{p+\frac{q^3}{2^m}, \ell_3 2^k+\left[0, 2^{k+1} \right]}^{3k} \right \rangle \right|^2:= \sum\limits_{n \in [0, 2^{k+1}] \cap \Z} \left| \left\langle f_3^{\omega_3^{\ell_3}}, \phi_{p+\frac{q^3}{2^m}, \ell_3 2^k+n}^{3k} \right \rangle \right|^2.$}
$$
\lesssim \sum_{\substack{\ell_3 \sim 2^{\frac{m}{2}-k} \\ q \sim 2^{\frac{m}{2}}}} \left| \left\langle f_3^{\omega_3^{\ell_3}}, \phi_{p+\frac{q^3}{2^m}, \ell_3 2^k+\left[0, 2^{k+1} \right]}^{3k} \right \rangle \right|^2  \left| \left \langle f_4, \phi_{p, \ell_3 2^k+\left[0, 2^{\frac{m}{2}-k+1} \right]}^{3k} \right \rangle \right|^2\,.
$$

By Parseval, we further have
\begin{eqnarray} \label{20230225eq01}
\bbA_p^2 &\lesssim& \sum_{\substack{\ell_3 \sim 2^{\frac{m}{2}-k} \\ q \sim 2^{\frac{m}{2}}}} \left\|f_3^{3\omega_3^{\ell_3}} \right\|^2_{L^2 \left(I_{p+\frac{q^3}{2^m}}^{m, 3k} \right)} \left\|f_4^{2^{\frac{m}{2}+4k}\ell_3+\left[0, 2^{m+2k+1} \right]} \right\|_{L^2 \left(I_p^{m, 3k} \right)}^2 \nonumber  \\
&\lesssim& \sum_{\ell_3 \sim 2^{\frac{m}{2}-k}} \left\|f_3^{3\omega_3^{\ell_3}} \right\|^2_{L^2 \left(3 I_{\frac{p}{2^{\frac{m}{2}}}}^{0, 3k} \right)} \left\|f_4^{2^{\frac{m}{2}+4k}\ell_3+\left[0, 2^{m+2k+1} \right]} \right\|_{L^2 \left(I_p^{m, 3k} \right)}^2 \nonumber  \\
&\lesssim& \sum_{L_3 \sim 2^k} \left\|f_3^{3 \Omega_3^{L_3}} \right\|_{L^2\left(3I_{\frac{p}{2^{\frac{m}{2}}}}^{0, 3k} \right)}^2  \left\|f_4^{3\Omega_4^{L_3}} \right\|_{L^2 \left(I_p^{m, 3k} \right)}^2.
\end{eqnarray}

\medskip

\noindent\textit{Treatment of $\bbB_p$.} Observe that for any $h\in L^2(I^k)$, by Parseval, one has
\begin{equation} \label{20230225eq02}
\sum_{n_3 \sim 2^k} \left|\int_{I_q^{m, k}} h(t) e^{i2^{\frac{m}{2}+k}t \cdot \frac{3q^2}{2^m} n_3} dt \right|^2 \lesssim  \left|I_q^{m, k} \right| \int_{I_q^{m, k}} |h(s)|^2 ds.
\end{equation}
Applying \eqref{20230225eq02} for
$$
h(t)=\underline{f}_1^{\omega_1^{\ell_1}}(x-t) \underline{f}_2^{\omega_2^{\ell_2}}(x+t^2) e^{i 2^{\frac{m}{2}+k}t \cdot 2^k \left(-\ell_1 +\ell_2 \frac{2q}{2^{\frac{m}{2}}}+ \ell_3 \frac{3q^2}{2^m}\right)},
$$
we have
\begin{equation} \label{20230225eq03}
\bbB_p^2 \lesssim 2^{2k}\,\sum_{\substack{\ell_1, \ell_2 \sim 2^{\frac{m}{2}-k} \\q \sim 2^{\frac{m}{2}}}}  \int_{I_p^{m, 3k}} \int_{I_q^{m, k}} \left|\underline{f}_1^{\omega_1^{\ell_1}}(x-t) \underline{f}_2^{\omega_2^{\ell_2}}(x+t^2) \right|^2 dtdx.
\end{equation}

\medskip

Combining \eqref{20230224eq01}, \eqref{20230225eq01}, and \eqref{20230225eq03}, we conclude that
\begin{eqnarray} \label{20230305eq05}
|\Lambda_m^{k, \bbU_\mu}(\vec{f})| &\lesssim& 2^{2k} \sum_{p \sim 2^{\frac{m}{2}+2k}} \left( \sum_{L_3 \sim 2^k} \left\|f_3^{3 \Omega_3^{L_3}} \right\|_{L^2\left(3 I_{\frac{p}{2^{\frac{m}{2}}}}^{0, 3k} \right)}^2  \left\|f_4^{3\Omega_3^{L_3}} \right\|_{L^2 \left(I_p^{m, 3k} \right)}^2  \right)^{\frac{1}{2}}  \\
&\cdot& \left(\sum_{\substack{\ell_1, \ell_2 \sim 2^{\frac{m}{2}-k} \\q \sim 2^{\frac{m}{2}}}}  \int_{I_p^{m, 3k}} \int_{I_q^{m, k}} \left|\underline{f}_1^{\omega_1^{\ell_1}}(x-t) \underline{f}_2^{\omega_2^{\ell_2}}(x+t^2) \right|^2 dtdx \right)^{\frac{1}{2}}\,. \nonumber
\end{eqnarray}

Applying a partial Cauchy-Schwarz in $p$ and then grouping the intervals $\{I_q^{m, k}\}_{q \sim 2^{\frac{m}{2}}}$ into a collection of larger intervals $\left\{I_{\widetilde{q}}^{0, 2k}\right\}_{\widetilde{q} \sim 2^k}$, we deduce that
\begin{eqnarray} \label{20230225eq04}
|\Lambda_m^{k, \bbU_\mu}(\vec{f})|&\lesssim& 2^{2k} \sum_{\widetilde{p} \sim  2^{2k}}  \left( \sum_{L_3 \sim 2^k} \left\|f_3^{3 \Omega_3^{L_3}} \right\|^2_{L^2 \left(3 I_{\widetilde{p}}^{0, 3k} \right)}\left\|f_4^{3\Omega_3^{L_3}} \right\|^2_{L^2 \left(3 I_{\widetilde{p}}^{0, 3k} \right)} \right)^{\frac{1}{2}} \\
&\cdot& \left(\sum_{\substack{\ell_1, \ell_2 \sim 2^{\frac{m}{2}-k} \\ \widetilde{q} \sim 2^k}}   \int_{I_{\widetilde{p}}^{0, 3k}} \int_{I_{\widetilde{q}}^{0, 2k}} \left|\underline{f}_1^{\omega_1^{\ell_1}}(x-t) \underline{f}_2^{\omega_2^{\ell_2}}(x+t^2) \right|^2 dtdx \right)^{\frac{1}{2}}.  \nonumber
\end{eqnarray}
Taking the double integral in \eqref{20230225eq04} and applying the change of variable $\begin{cases}
u=x-t \\
v=x+t^2
\end{cases}$, we have
\begin{equation} \label{20230225eq38}
\int_{I_{\widetilde{p}}^{0, 3k}} \int_{I_{\widetilde{q}}^{0, 2k}} \left|\underline{f}_1^{\omega_1^{\ell_1}}(x-t) \underline{f}_2^{\omega_2^{\ell_2}}(x+t^2) \right|^2 dtdx \lesssim \left(\int_{3 I^{0, 2k}_{\frac{\widetilde{p}}{2^k}-\widetilde{q}}} \left|\underline{f}_1^{\omega_1^{\ell_1}}(u) \right|^2 du \right)  \left( \int_{3 I^{0, 3k}_{\widetilde{p}+\frac{\widetilde{q}^2}{2^k}}} \left| \underline{f}_2^{\omega_2^{\ell_2}}(v) \right|^2 dv \right)\,,
\end{equation}
where in the above we used that the Jacobian for this change variable is $\sim 1$.

Now substituting \eqref{20230225eq38} back to \eqref{20230225eq04} and then applying Parseval, we conclude that
\begin{eqnarray} \label{20230225main01}
\left|\Lambda_m^{k, \bbU_\mu}(\vec{f})\right| &\lesssim& 2^{2k} \sum_{\widetilde{p} \sim  2^{2k}}  \left( \sum_{L_3 \sim 2^k} \left\|f_3^{3 \Omega_3^{L_3}} \right\|^2_{L^2 \left(3I_{\widetilde{p}}^{0, 3k} \right)}\left\|f_4^{3\Omega_3^{L_3}} \right\|^2_{L^2 \left(3 I_{\widetilde{p}}^{0, 3k} \right)} \right)^{\frac{1}{2}} \\
&\cdot& \left[\sum_{\substack{L_1, L_2 \sim 2^k \\ \widetilde{q} \sim 2^k}}
\left(\int_{3 I^{0, 2k}_{\frac{\widetilde{p}}{2^k}-\widetilde{q}}} \left|\underline{f}_1^{\Omega_1^{L_1}}(u) \right|^2 du \right)  \left( \int_{3 I^{0, 3k}_{\widetilde{p}+\frac{\widetilde{q}^2}{2^k}}} \left| \underline{f}_2^{\Omega_2^{L_2}}(v) \right|^2 dv \right) \right]^{\frac{1}{2}}.  \nonumber
\end{eqnarray}
Now we consider several different cases.

\subsubsection{Either $f_3$ or $f_4$ is uniform} \label{20230305subsec01} In this case, without loss of generality, we assume $f_3=f_3^{\bbU_\mu}$. Then from \eqref{20230225main01}, Parseval and Cauchy--Schwarz one has
 \begin{eqnarray} \label{Uniform-I-Small}
&&\left|\Lambda_m^{k, \bbU_\mu}(\vec{f})\right|\lesssim 2^{2k} \sum_{\widetilde{p} \sim  2^{2k}}  \left( \sum_{L_3 \sim 2^k} 2^{-\mu k} \left\|f_3 \right\|^2_{L^2 \left(3I_{\widetilde{p}}^{0, 3k} \right)}\left\|f_4^{3\Omega_3^{L_3}} \right\|^2_{L^2 \left(3 I_{\widetilde{p}}^{0, 3k} \right)} \right)^{\frac{1}{2}}\nonumber \\
&&\qquad\qquad\cdot \left[\sum_{\substack{L_1, L_2 \sim 2^k \\ \widetilde{q} \sim 2^k}}
\left(\int_{3 I^{0, 2k}_{\frac{\widetilde{p}}{2^k}-\widetilde{q}}} \left|\underline{f}_1^{\Omega_1^{L_1}}(u) \right|^2 du \right)  \left( \int_{3 I^{0, 3k}_{\widetilde{p}+\frac{\widetilde{q}^2}{2^k}}} \left| \underline{f}_2^{\Omega_2^{L_2}}(v) \right|^2 dv \right) \right]^{\frac{1}{2}}\nonumber \\
&&\lesssim 2^{2k} \cdot 2^{-\frac{\mu k}{2}} \left( \sum_{\widetilde{p} \sim  2^{2k}}  \left\|f_3 \right\|^2_{L^2 \left(3I_{\widetilde{p}}^{0, 3k} \right)}\left\|f_4 \right\|^2_{L^2 \left(3 I_{\widetilde{p}}^{0, 3k} \right)} \right)^{\frac{1}{2}}\,\left[ \sum_{\substack{ \doublewidetilde{p} \sim 2^{k} \\ \widetilde{q} \sim 2^k}}
\left(\int_{3 I^{0, 2k}_{\doublewidetilde{p}-\widetilde{q}}} \left|f_1(u) \right|^2 du \right)  \left( \int_{3 I^{0, 2k}_{\doublewidetilde{p}}} \left| f_2(v) \right|^2 dv \right) \right]^{\frac{1}{2}} \nonumber \\
&&\lesssim 2^{2k} \cdot 2^{-\frac{\mu k}{2}} \left\|f_1 \right\|_{L^2 \left(3 I^k \right)} \left\|f_2 \right\|_{L^2 \left(3 I^k \right)}\left( \sum_{\widetilde{p} \sim  2^{2k}}  \left\|f_3 \right\|^2_{L^2 \left(3 I_{\widetilde{p}}^{0, 3k} \right)}\left\|f_4 \right\|^2_{L^2 \left(3 I_{\widetilde{p}}^{0, 3k} \right)} \right)^{\frac{1}{2}}\nonumber\,,
\end{eqnarray}
which implies
\begin{equation}\label{unif1ksmall}
\left|\Lambda_m^{k, \bbU_\mu}(f_1, f_2,f_3^{\bbU_\mu},f_4)\right| \lesssim 2^{-\frac{\mu k}{2}} \,\|\vec{f} \|_{L^{\vec{p}}(3 I^{k})}\qquad \forall\:\vec{p}\in {\bf H}^{+}\,.
\end{equation}

\subsubsection{Both $f_3$ and $f_4$ are sparse, and at least one of the functions $f_1$ or $f_2$ is uniform} \label{subsec4.2.2}

In this case, by \eqref{20230225eq23}, we deduce that up to a factor of at most $2^{\mu k}$, there exists a unique measurable function $L_3: \left[2^{2k}, 2^{2k+1}\right]\cap \Z \to \left[2^k, 2^{k+1} \right]\cap\Z$ such that
\begin{equation} \label{20230225eq54}
 2^{-\mu k} \int_{3 I_{\widetilde{p}}^{0, 3k}} \left|f_j \right|^2\lesssim\int_{3 I_{\widetilde{p}}^{0, 3k}} \left| f_j^{\Omega_j^{3 L_3 \left(\widetilde{p} \right)}} \right|^2  \lesssim \int_{3 I_{\widetilde{p}}^{0, 3k}} |f_j|^2, \quad j\in\{3, 4\}.
\end{equation}
Recall that $\omega_3^{\ell_3}=\left[\ell_3 2^{\frac{m}{2}+4k}, \left(\ell_3+1 \right) 2^{\frac{m}{2}+4k} \right]$ and $\Omega_3^{L_3}=\left[L_3 2^{m+2k}, \left(L_3+1 \right) 2^{m+2k} \right]$, and hence if $\omega_3^{\ell_3} \subseteq \Omega_3^{L_3}$, one has
\begin{equation} \label{20230225eq32}
\ell_3 \in \left[L_3 2^{\frac{m}{2}-2k}, \left(L_3+1 \right) 2^{\frac{m}{2}-2k} \right].
\end{equation}
Now for each $p \sim 2^{\frac{m}{2}+2k}$ and $q \sim 2^{\frac{m}{2}}$, we set
\begin{equation*}
\TFC_{p, q}:=\left\{ \left(\ell_1, \ell_2, \ell_3 \right) \in \TFC_q: \ell_3 \in \left[L_3 \left(  \left\lfloor \frac{p}{2^{\frac{m}{2}}} \right\rfloor\right) 2^{\frac{m}{2}-2k}, \left(L_3 \left(\left\lfloor \frac{p}{2^{\frac{m}{2}}} \right\rfloor\right)+1 \right) 2^{\frac{m}{2}-2k} \right] \right\}.
\end{equation*}
Hence, by \eqref{20230224main01}, we have
\begin{eqnarray*}
|\Lambda_m^{k, \bbU_\mu}(\vec{f})|\lesssim 2^k \cdot 2^{\mu k} \sum_{\substack{p \sim 2^{\frac{m}{2}+2k}\\ q \sim 2^{\frac{m}{2}} \\ n_3 \sim 2^k}} \quad\sum_{(\ell_1, \ell_2, \ell_3) \in \TFC_{p, q}}\left| \left\langle f_3^{\omega_3^{\ell_3}}, \phi_{p+\frac{q^3}{2^m}, \ell_3 2^k+n_3}^{3k} \right\rangle \right|  \left| \left\langle f_4 e^{i 2^{\frac{m}{2}+3k} \left(\ell_1 2^{-k}+\ell_2 \right) \left( \cdot \right) }, \phi_{p, \ell_3 2^k+n_3}^{3k} \right\rangle \right| \nonumber \\
\cdot \frac{1}{\left|I_p^{m, 3k} \right|} \bigg| \int_{I_p^{m, 3k}} \int_{I_q^{m, k}}  \underline{f}_1^{\omega_1^{\ell_1}}(x-t) \underline{f}_2^{\omega_2^{\ell_2}}(x+t^2) e^{i 2^{\frac{m}{2}+k}t \cdot \frac{3q^2}{2^m}n_3}\, e^{i 2^{\frac{m}{2}+k}t \cdot 2^k \left(-\ell_1 +\ell_2 \frac{2q}{2^{\frac{m}{2}}}+ \ell_3 \frac{3q^2}{2^m}\right)} dt dx \bigg|.
\end{eqnarray*}
Arguing as in \eqref{20230224eq01}, \eqref{20230225eq01}, \eqref{20230225eq03}, and \eqref{20230225eq04}, we have
\begin{equation} \label{20250225eq30}
|\Lambda_m^{k, \bbU_\mu}(\vec{f})| \lesssim 2^{2k} \cdot 2^{\mu k} \sum_{\widetilde{p} \sim 2^{2k}} \left\|f_3^{3 \Omega_3^{L_3(\widetilde{p})}} \right\|_{L^2 \left(3 I_{\widetilde{p}}^{0, 3k} \right)} \left\|f_4^{3 \Omega_4^{L_3(\widetilde{p})}} \right\|_{L^2 \left(3 I_{\widetilde{p}}^{0, 3k} \right)} \calR_{\widetilde{p}},
\end{equation}
where
$$
\calR_{\widetilde{p}}^2:=\sum_{q \sim 2^{\frac{m}{2}}} \sum_{\left(\ell_1, \ell_2, \ell_3 \right) \in \TFC_{2^{\frac{m}{2}} \widetilde{p}, q}}  \int_{I_{\widetilde{p}}^{0, 3k}} \int_{I_q^{m, k}} \left| \underline{f}_1^{\omega_1^{\ell_1}}(x-t) \underline{f}_2^{\omega_2^{\ell_2}}(x+t^2) \right|^2 dxdt.
$$
Now we estimate the term $\calR_{\widetilde{p}}$. Recall that $\left(\ell_1, \ell_2, \ell_3 \right) \in \TFC_{2^{\frac{m}{2}} \widetilde{p}, q}$ means
\begin{equation} \label{20230225eq31}
\begin{cases}
\ell_1-\frac{2q}{2^{\frac{m}{2}}} \ell_2-\frac{3q^2}{2^m} \ell_3=O(1); \\
\\
\ell_3 \in \left[L_3\left(\widetilde{p} \right)2^{\frac{m}{2}-2k}, \left(L_3\left(\widetilde{p} \right)+1 \right) 2^{\frac{m}{2}-2k} \right].
\end{cases}
\end{equation}
Now if $\omega_1^{\ell_1} \subset \Omega_1^{L_1}$ and $\omega_2^{\ell_2} \subset \Omega_2^{L_2}$, that is, by \eqref{20230225eq32}, if $\ell_1 \in \left[L_12^{\frac{m}{2}-2k}, \left(L_1+1 \right)2^{\frac{m}{2}-2k} \right]$
and $\ell_2 \in \left[L_22^{\frac{m}{2}-2k}, \left(L_2+1 \right)2^{\frac{m}{2}-2k} \right]$, then by \eqref{20230225eq31}, we have
$$
L_1 -\frac{2q}{2^{\frac{m}{2}}} L_2 -\frac{3q^2}{2^m} L_3 \left(\widetilde{p} \right)=O\left(1 \right).
$$
This motivates the introduction of the time-frequency correlations relative to $q \sim 2^{\frac{m}{2}}$ and $\widetilde{p} \sim 2^{2k}$ as
$$
\overline{\TFC}_{\widetilde{p}, q}:=\left\{ \substack{\left(L_1, L_2 \right)\\ L_1, L_2 \sim 2^k} : L_1 -\frac{2q}{2^{\frac{m}{2}}} L_2 -\frac{3q^2}{2^m} L_3 \left(\widetilde{p} \right)=O\left(1 \right) \right\}.
$$
Then, by Parseval, we have that
\begin{eqnarray*}
&& \calR_{\widetilde{p}}^2  \lesssim \sum_{q \sim 2^{\frac{m}{2}}} \sum_{(L_1, L_2) \in \overline{\TFC}_{\widetilde{p}, q}} \sum_{\substack{ \ell_1 \in \left[L_12^{\frac{m}{2}-2k}, \left(L_1+1 \right)2^{\frac{m}{2}-2k} \right] \\ \ell_2 \in \left[L_22^{\frac{m}{2}-2k}, \left(L_2+1 \right)2^{\frac{m}{2}-2k} \right] \\  \ell_3 \in \left[L_3\left(\widetilde{p} \right)2^{\frac{m}{2}-2k}, \left(L_3\left(\widetilde{p} \right)+1 \right) 2^{\frac{m}{2}-2k} \right] \\ \ell_1-\frac{2q}{2^{\frac{m}{2}}} \ell_2-\frac{3q^2}{2^m} \ell_3=O(1)}} \int_{I_{\widetilde{p}}^{0, 3k}} \int_{I_q^{m, k}} \left| \underline{f}_1^{\omega_1^{\ell_1}}(x-t) \underline{f}_2^{\omega_2^{\ell_2}}(x+t^2) \right|^2 dtdx \\
&& \quad \lesssim \sum_{q \sim 2^{\frac{m}{2}}} \sum_{(L_1, L_2) \in \overline{\TFC}_{\widetilde{p}, q}} \int_{I_{\widetilde{p}}^{0, 3k}}\int_{I_q^{m, k}} \left|f_1^{\Omega_1^{L_1}}(x-t) f_2^{\Omega_2^{L_2}} \left(x+t^2 \right) \right|^2 dtdx,
\end{eqnarray*}
where in the above estimate, we used the fact that if $q, \ell_1$ and $\ell_2$ are chosen, then $\ell_3$ is (up to $O(1)$) uniquely determined.

Next, by splitting the regime of $q$ as usual as
$\left\{q \sim 2^{\frac{m}{2}} \right\}=\bigcup_{\widetilde{q} \sim 2^k} \left[ \widetilde{q}2^{\frac{m}{2}-k}, \left(\widetilde{q}+1 \right) 2^{\frac{m}{2}-k} \right]$ we notice that whenever $q \in  \left[ \widetilde{q}2^{\frac{m}{2}-k}, \left( \widetilde{q}+1 \right) 2^{\frac{m}{2}-k} \right]$ one has that
\begin{equation} \label{20230225eq39}
(L_1, L_2) \in \overline{\TFC}_{\widetilde{p}, q}\quad\Longleftrightarrow\quad (L_1, L_2) \in \overline{\TFC}_{\widetilde{p}, 2^{\frac{m}{2}-k}\widetilde{q}}.
\end{equation}
As a consequence, we have
\begin{eqnarray*}
\calR_{\widetilde{p}}^2 &\lesssim& \sum_{\widetilde{q} \sim 2^k} \sum_{q= \widetilde{q}2^{\frac{m}{2}-k}}^{\left(\widetilde{q}+1 \right) 2^{\frac{m}{2}-k}} \sum_{(L_1, L_2) \in \overline{\TFC}_{\widetilde{p}, 2^{\frac{m}{2}-k} \widetilde{q}}}
 \int_{I_{\widetilde{p}}^{0, 3k}} \int_{I_q^{m, k}} \left|f_1^{\Omega_1^{\frac{2\widetilde{q}}{2^k}L_2+ \frac{3\widetilde{q}^2}{2^{2k}} L_3\left(\widetilde{p} \right)}}(x-t) f_2^{\Omega_2^{L_2}} \left(x+t^2 \right) \right|^2 dtdx \\
 &\lesssim& \sum_{\substack{ \widetilde{q} \sim 2^k \\ L_2 \sim 2^k}}  \int_{I_{\widetilde{p}}^{0, 3k}} \int_{I_{\widetilde{q}}^{0, 2k}}
 \left|f_1^{\Omega_1^{\frac{2\widetilde{q}}{2^k}L_2+ \frac{3\widetilde{q}^2}{2^{2k}} L_3\left(\widetilde{p}  \right)}}(x-t) f_2^{\Omega_2^{L_2}} \left(x+t^2 \right) \right|^2 dtdx.
\end{eqnarray*}
Applying now a similar argument with that in \eqref{20230225eq38}, we can rewrite the above relation as
\begin{equation} \label{20230225eq47}
\calR_{\widetilde{p}}^2 \lesssim \sum_{\substack{\widetilde{q} \sim 2^k \\ L_2 \sim 2^k}} \left( \int_{3 I_{\frac{\widetilde{p}}{2^k}-\widetilde{q}}^{0, 2k}} \left|f_1^{\Omega_1^{\frac{2\widetilde{q}}{2^k}L_2+ \frac{3\widetilde{q}^2}{2^{2k}} L_3\left(\widetilde{p}  \right)}}\right|^2  \right) \left(\int_{3 I_{\widetilde{p}+\frac{\widetilde{q}^2}{2^k}}^{0, 3k}} \left| f_2^{\Omega_2^{L_2}} \right|^2 \right).
\end{equation}
Recalling now that we are in the regime in which at least one of the functions $f_1$, $f_2$ is uniform. without loss of generality, we may assume $f_1=f_1^{\bbU_\mu}$. Then, by \eqref{20230225eq47} and Parseval, we have
\begin{equation}
\calR_{\widetilde{p}}^2 \lesssim 2^{-3\mu k} \sum_{\substack{\widetilde{q} \sim 2^k \\ L_2 \sim 2^k}} \left( \int_{3 I_{\frac{\widetilde{p}}{2^k}-\widetilde{q}}^{0, 2k}} \left |f_1 \right|^2 \right) \left(\int_{3 I_{\widetilde{p}+\frac{\widetilde{q}^2}{2^k}}^{0, 3k}} \left| f_2^{\Omega_2^{L_2}} \right|^2 \right)\lesssim 2^{-3\mu k} \sum_{\widetilde{q} \sim 2^k } \left( \int_{3 I_{\frac{\widetilde{p}}{2^k}-\widetilde{q}}^{0, 2k}} \left |f_1 \right|^2 \right) \left(\int_{3 I_{\widetilde{p}+\frac{\widetilde{q}^2}{2^k}}^{0, 3k}} \left| f_2 \right|^2 \right).
\end{equation}
This together with \eqref{20250225eq30} gives
\begin{equation} \label{20230305eq40}
|\Lambda_m^{k, \bbU_\mu}(\vec{f})| \lesssim 2^{2k} \cdot 2^{-\frac{\mu k}{2}} \sum_{\widetilde{p} \sim 2^{2k}} \left\|f_3^{3 \Omega_3^{L_3(\widetilde{p})}} \right\|_{L^2 \left(3 I_{\widetilde{p}}^{0, 3k} \right)} \left\|f_4^{3 \Omega_4^{L_3(\widetilde{p})}} \right\|_{L^2 \left(3 I_{\widetilde{p}}^{0, 3k} \right)}\, \left[  \sum_{\widetilde{q} \sim 2^k } \left( \int_{3 I_{\frac{\widetilde{p}}{2^k}-\widetilde{q}}^{0, 2k}} \left |f_1 \right|^2 \right) \left(\int_{3 I_{\widetilde{p}+\frac{\widetilde{q}^2}{2^k}}^{0, 3k}} \left| f_2 \right|^2 \right)  \right]^{\frac{1}{2}}.
\end{equation}
Applying Cauchy-Schwarz in $\widetilde{p}$ and then the change variable $\widetilde{q} \to \widetilde{q}+\frac{\widetilde{p}}{2^k}$, we have
\begin{eqnarray} \label{Uniform-II-Small}
\left|\Lambda_m^{k, \bbU_\mu}(\vec{f})\right| &\lesssim& 2^{2k} \cdot 2^{-\frac{\mu k}{2}} \cdot \left( \sum_{\widetilde{p} \sim 2^{2k}} \left\|f_3^{3 \Omega_3^{L_3(\widetilde{p})}} \right\|^2_{L^2 \left(3 I_{\widetilde{p}}^{0, 3k} \right)} \left\|f_4^{3 \Omega_4^{L_3(\widetilde{p})}} \right\|^2_{L^2 \left(3 I_{\widetilde{p}}^{0, 3k} \right)} \right)^{\frac{1}{2}} \nonumber \\
&\cdot& \left[ \sum_{\widetilde{p} \sim 2^{2k}} \sum_{\widetilde{q} \sim 2^k } \left( \int_{3 I_{-\widetilde{q}}^{0, 2k}} \left |f_1 \right|^2 \right) \left(\int_{3 I_{\widetilde{p}+\frac{(\widetilde{q}+\frac{\widetilde{p}}{2^k})^2}{2^k}}^{0, 3k}} \left| f_2 \right|^2 \right)  \right]^{\frac{1}{2}} \nonumber  \\
&\lesssim& 2^{2k} \cdot 2^{-\frac{\mu k}{2}}  \left\|f_1 \right\|_{L^2 \left(3 I^k \right)} \left\|f_2 \right\|_{L^2 \left(3 I^k \right)} \left( \sum_{\widetilde{p} \sim 2^{2k}} \left\|f_3 \right\|^2_{L^2 \left(3 I_{\widetilde{p}}^{0, 3k} \right)} \left\|f_4 \right\|^2_{L^2 \left(3 I_{\widetilde{p}}^{0, 3k} \right)} \right)^{\frac{1}{2}}\nonumber,
\end{eqnarray}
where gives
\begin{equation}\label{unif2ksmall}
\left|\Lambda_m^{k, \bbU_\mu}(f_1^{\bbU_\mu}, f_2,f_3^{\bbS_\mu},f_4^{\bbS_\mu})\right| \lesssim 2^{-\frac{\mu k}{2}} \,\|\vec{f} \|_{L^{\vec{p}}(3 I^{k})}\qquad \forall\:\vec{p}\in {\bf H}^{+}\,.
\end{equation}

\subsection{Treatment of the sparse component $\Lambda_m^{k, \bbS_\mu}$}

Recall that in this case, we have that all $\left\{f_1, f_2, f_3, f_4\right\}$ are sparse. Therefore, we see that
\begin{enumerate}
    \item [$\bullet$] up to a $2^{\mu k}$ factor, there is a unique measurable function $L_3: \left[2^{2k}, 2^{2k+1} \right]\cap\Z \mapsto \left[2^k, 2^{k+1}\right]\cap\Z$ such that\footnote{Note that the frequency localization of $f_3$ and $f_4$ are essentially determined by the same frequency interval $\left[L_3 2^{m+2k}, (L_3+1)2^{m+2k} \right]$--see \eqref{20230225eq01}.} for each $p \sim 2^{2k}$,
\begin{equation} \label{20230226eq01}
2^{-\mu k}\,\int_{3 I_{\widetilde{p}}^{0, 3k}} |f_j|^2\lesssim\int_{3 I_{\widetilde{p}}^{0, 3k}} \left| f_j^{3 \Omega_j^{L_3 \left(\widetilde{p} \right)}} \right|^2  \lesssim \int_{3 I_{\widetilde{p}}^{0, 3k}} |f_j|^2, \quad j\in\{3, 4\}.
\end{equation}

\item [$\bullet$] up to a $2^{3\mu k}$ factor and for given $j \in \{1, 2\}$, there exists a unique measurable function $L_j: \left[2^{jk}, 2^{jk+1} \right]\cap\Z \mapsto \left[2^k, 2^{k+1}\right]\cap\Z$ such that for each $q_j \sim 2^{jk}$,
 \begin{equation} \label{20230226eq02}
2^{-3\mu k}\,\int_{3 I_{q_j}^{0, (j+1)k}} |f_j|^2\lesssim \int_{3 I_{q_j}^{0, (j+1)k}} \left|f_j^{\Omega_j^{L_j(q_j)}} \right|^2 \lesssim \int_{3 I_{q_j}^{0, (j+1)k}} |f_j|^2, \quad j\in\{1, 2\}.
 \end{equation}
\end{enumerate}

Notice from \eqref{20230226eq01} and \eqref{20230226eq02} that the frequency information of the functions $\{f_1, f_2, f_3, f_4\}$ is localized in a ``compressed" interval within each spatial fiber. This will be later exploited via time-frequency correlations (that are encoding the curvature of our problem) in order to obtain the desired $k-$decaying estimate.

\subsubsection{Reductions via frequency sparsity of the spatial fibers} Starting from \eqref{20230224main01} and in the spirit provided by the approach in Subsection \ref{subsec4.2.2}, we first exploit the frequency localization as follows:
\begin{equation} \label{20230226eq10}
|\Lambda^{k, \bbS_\mu}_m(\vec{f})| \lesssim 2^k \cdot 2^{7\mu k} \sum_{\substack{ \widetilde{p} \sim 2^{2k}\,\widetilde{q} \sim 2^k}} \bbA_{\widetilde{p}, \widetilde{q}}\:\bbB_{\widetilde{p}, \widetilde{q}},
\end{equation}
where
\begin{equation}
\bbA_{\widetilde{p}, \widetilde{q}}^2 :=\sum_{\substack{p \in \left[\widetilde{p}2^{\frac{m}{2}}, \left(\widetilde{p}+1 \right)2^{\frac{m}{2}} \right] \\ q \in \left[ \widetilde{q} 2^{\frac{m}{2}-k}, \left( \widetilde{q}+1 \right)2^{\frac{m}{2}-k} \right] \\ n_3 \sim 2^k}} \sum_{\substack{(\ell_1, \ell_2, \ell_3) \in \TFC_{p, q} \\\eqref{20230226eq01}\,\&\, \eqref{20230226eq02}}} \left| \left\langle f_3^{\omega_3^{\ell_3}}, \phi_{p+\frac{q^3}{2^m}, \ell_3 2^k+n_3}^{3k} \right \rangle \right|^2  \left| \left \langle f_4\, e^{i 2^{\frac{m}{2}+3k} \left(\ell_1 2^{-k}+\ell_2 \right) \left( \cdot \right) }, \phi_{p, \ell_3 2^k+n_3}^{3k} \right \rangle \right|^2,
\end{equation}
and
\begin{eqnarray*}
&& \bbB_{\widetilde{p}, \widetilde{q}}^2:=\sum_{\substack{p \in \left[\widetilde{p}2^{\frac{m}{2}}, \left(\widetilde{p}+1 \right)2^{\frac{m}{2}} \right] \\ q \in \left[ \widetilde{q} 2^{\frac{m}{2}-k}, \left( \widetilde{q}+1 \right)2^{\frac{m}{2}-k} \right] \\ n_3 \sim 2^k}} \sum_{\substack{(\ell_1, \ell_2, \ell_3) \in \TFC_{p, q} \\\eqref{20230226eq01}\,\&\, \eqref{20230226eq02}}}\\
&& \quad \quad \quad  \quad \frac{1}{\left|I_p^{m, 3k} \right|^2} \bigg| \int_{I_p^{m, 3k}} \int_{I_q^{m, k}}  \underline{f}_1^{\omega_1^{\ell_1}}(x-t) \underline{f}_2^{\omega_2^{\ell_2}}(x+t^2) e^{i 2^{\frac{m}{2}+k}t \cdot \frac{3q^2}{2^m}n_3} \, e^{i 2^{\frac{m}{2}+k}t \cdot 2^k \left(-\ell_1 +\ell_2 \frac{2q}{2^{\frac{m}{2}}}+ \ell_3 \frac{3q^2}{2^m}\right)} dt dx \bigg|^2.
\end{eqnarray*}

\noindent\textit{Treatment of $\bbA_{\widetilde{p}, \widetilde{q}}$.} The estimate of this part is similar to the treatment of the term $\bbA_p$ in \eqref{20230225eq01}:
$$
\bbA_{\widetilde{p}, \widetilde{q}}^2 \lesssim \sum_{\substack{p \in \left[\widetilde{p}2^{\frac{m}{2}}, \left(\widetilde{p}+1 \right)2^{\frac{m}{2}} \right] \\ q \in \left[ \widetilde{q} 2^{\frac{m}{2}-k}, \left( \widetilde{q}+1 \right)2^{\frac{m}{2}-k} \right] \\ \ell_3 \in \left[L_3 \left(\widetilde{p}\right)2^{m-2k}, \left(L_3\left(\widetilde{p} \right)+1 \right)2^{m-2k}\right]}}\left\|f_3^{3\omega_3^{\ell_3}} \right\|^2_{L^2 \left(I_{p+\frac{q^3}{2^m}}^{m, 3k} \right)} \left\|f_4^{2^{\frac{m}{2}+4k}\ell_3+\left[0, 2^{m+2k+1} \right]} \right\|_{L^2 \left(I_p^{m, 3k} \right)}^2.
$$
Applying Parseval in $\ell_3$ first, then in $q$, followed by the partition of $\left[\widetilde{p}2^{\frac{m}{2}}, \left(\widetilde{p}+1 \right)2^{\frac{m}{2}} \right]$ into smaller intervals of length $2^{\frac{m}{2}-k}$, we have
\begin{eqnarray}  \label{20230306eq01}
\bbA_{\widetilde{p}, \widetilde{q}}^2 &\lesssim&   \sum_{\substack{p \in \left[\widetilde{p}2^{\frac{m}{2}}, \left(\widetilde{p}+1 \right)2^{\frac{m}{2}} \right] \\ q \in \left[ \widetilde{q} 2^{\frac{m}{2}-k}, \left( \widetilde{q}+1 \right)2^{\frac{m}{2}-k} \right]}} \left\|f_3^{3 \Omega_3^{L_3 (\widetilde{p})}} \right\|^2_{L^2 \left(I_{p+\frac{q^3}{2^m}}^{m, 3k} \right)} \left\|f_4^{3\Omega_4^{L_3(\widetilde{p})}} \right\|_{L^2 \left(I_p^{m, 3k} \right)}^2 \\
&\lesssim& \sum_{s=0}^{2^k-1} \left\|f_3^{3 \Omega_3^{L_3(\widetilde{p})}} \right\|^2_{L^2 \left(3 I_{2^k \widetilde{p}+s+\frac{\widetilde{q}^3}{2^{2k}}}^{0, 4k} \right)} \left\|f_4^{3\Omega_4^{L_3(\widetilde{p})}} \right\|_{L^2 \left(3 I_{2^k \widetilde{p}+s}^{0, 4k} \right)}^2. \nonumber
\end{eqnarray}

\noindent\textit{Treatment of $\bbB_{\widetilde{p}, \widetilde{q}}$.} The estimate for $\bbB_{\widetilde{p}, \widetilde{q}}$ follows from similar arguments as in \eqref{20230225eq03}, \eqref{20230225eq04}, \eqref{20250225eq30} and \eqref{20230225eq47}:
$$
 \bbB_{\widetilde{p}, \widetilde{q}}^2 \lesssim 2^{2k} \sum_{
 \substack{L_2 \sim 2^k \\ \eqref{20230226eq02}}}  \left( \int_{3 I_{\frac{\widetilde{p}}{2^k}-\widetilde{q}}^{0, 2k}} \left|f_1^{\Omega_1^{\frac{2\widetilde{q}}{2^k}L_2+ \frac{3\widetilde{q}^2}{2^{2k}} L_3\left(\widetilde{p}  \right)}}\right|^2  \right) \left(\int_{3 I_{\widetilde{p}+\frac{\widetilde{q}^2}{2^k}}^{0, 3k}} \left| f_2^{\Omega_2^{L_2}} \right|^2 \right).
$$

\medskip

Combining now the estimates of $\bbA_{\widetilde{p}, \widetilde{q}}$ and $\bbB_{\widetilde{p}, \widetilde{q}}$ with \eqref{20230226eq10}, we have
\begin{eqnarray} \label{20230226eq11}
&& |\Lambda^{k, \bbS_\mu}_m(\vec{f})|  \lesssim 2^{2k} \cdot 2^{7\mu k} \sum_{\substack{\widetilde{p} \sim 2^{2k} \\ \widetilde{q} \sim 2^k}}\left[  \sum_{s=0}^{2^k-1} \left\|f_3^{3 \Omega_3^{L_3(\widetilde{p})}} \right\|^2_{L^2 \left(3 I_{2^k \widetilde{p}+s+\frac{\widetilde{q}^3}{2^{2k}}}^{0, 4k} \right)} \left\|f_4^{3\Omega_4^{L_3(\widetilde{p})}} \right\|_{L^2 \left(3 I_{2^k \widetilde{p}+s}^{0, 4k} \right)}^2 \right]^{\frac{1}{2}}    \\
&& \quad \quad \quad \quad \cdot \left[ \sum_{
 \substack{L_2 \sim 2^k \\ \eqref{20230226eq02}}}   \left( \int_{3 I_{\frac{\widetilde{p}}{2^k}-\widetilde{q}}^{0, 2k}} \left|f_1^{\Omega_1^{\frac{2\widetilde{q}}{2^k}L_2+ \frac{3\widetilde{q}^2}{2^{2k}} L_3\left(\widetilde{p}  \right)}}\right|^2  \right) \left(\int_{3 I_{\widetilde{p}+\frac{\widetilde{q}^2}{2^k}}^{0, 3k}} \left| f_2^{\Omega_2^{L_2}} \right|^2 \right) \right]^{\frac{1}{2}}.  \nonumber
\end{eqnarray}
Motivated by the time-frequency localization induced by the above estimate, we define the time-frequency correlation set
$$
\sTFC:=\left\{\substack{\left( \widetilde{p}, \widetilde{q} \right)\\ \widetilde{p} \sim 2^{2k}, \widetilde{q} \sim 2^k}: \frac{2\widetilde{q}}{2^k}L_2 \left(\widetilde{p}+\frac{\widetilde{q}^2}{2^k} \right)+\frac{3\widetilde{q}^2}{2^{2k}} L_3 \left( \widetilde{p} \right)=L_1 \left(\frac{\widetilde{p}}{2^k}-\widetilde{q} \right)+O(1) \right\}.
$$
Therefore, by Cauchy-Schwarz, \eqref{20230226eq11} is further bounded from above by
\begin{eqnarray}  \label{Uniform-I}
&&\lesssim 2^{2k} \cdot 2^{7\mu k} \sum_{\left( \widetilde{p}, \widetilde{q} \right) \in \sTFC} \left\|f_1 \right\|_{L^2 \left(3 I_{\frac{\widetilde{p}}{2^k}-\widetilde{q}}^{0, 2k} \right)} \left\|f_2 \right\|_{L^2 \left(3 I_{\widetilde{p}+\frac{\widetilde{q}^2}{2^k}}^{0, 3k}\right)}\,\left[  \sum_{s=0}^{2^k-1} \left\|f_3 \right\|^2_{L^2 \left(3 I_{2^k \widetilde{p}+s+\frac{\widetilde{q}^3}{2^{2k}}}^{0, 4k} \right)} \left\|f_4 \right\|_{L^2 \left(3 I_{2^k \widetilde{p}+s}^{0, 4k} \right)}^2 \right]^{\frac{1}{2}} \nonumber\\
&& \lesssim 2^{(2+7\mu) k} \cdot \left( \# \sTFC \right)^{\frac{1}{2}} \vast[ \sum_{\substack{ \widetilde{p} \sim 2^{2k} \\ \widetilde{q} \sim 2^k }}\left\|f_1 \right\|^2_{L^2 \left(3 I_{\frac{\widetilde{p}}{2^k}-\widetilde{q}}^{0, 2k} \right)} \left\|f_2 \right\|^2_{L^2 \left(3 I_{\widetilde{p}+\frac{\widetilde{q}^2}{2^k}}^{0, 3k}\right)}\left(\sum_{s=0}^{2^k-1} \left\|f_3 \right\|^2_{L^2 \left(3 I_{2^k \widetilde{p}+s+\frac{\widetilde{q}^3}{2^{2k}}}^{0, 4k} \right)} \left\|f_4 \right\|_{L^2 \left(3 I_{2^k \widetilde{p}+s}^{0, 4k} \right)}^2 \right) \vast]^{\frac{1}{2}}\nonumber,
\end{eqnarray}
which implies
\begin{equation}\label{sparseksmall}
|\Lambda^{k, \bbS_\mu}_m(\vec{f})| \lesssim  2^{7\mu k} \cdot \left(2^{-3k}\,\# \sTFC \right)^{\frac{1}{2}}\,\|\vec{f} \|_{L^{\vec{p}}(3 I^{k})}\qquad \forall\:\vec{p}\in {\bf H}^{+}\,.
\end{equation}

Our desired estimate follows now if we can improve over the trivial estimate $\# \sTFC\lesssim 2^{3k}$.

\subsubsection{Improved control over the time-frequency correlation set: a level set analysis}  \label{20230306subsec01}

As suggested by the title, this section is devoted to prove the key estimate \eqref{TFCcontrolkey} below.\footnote{Simpler (two term) but related level set estimates were a also a key ingredient in \cite{Lie19}. More general level set estimates were recently studied by M. Christ, see \emph{e.g.} \cite{Chr22}.} For notational simplicity, we set $L_1=c$, $L_2=a$, $L_3=b$, $p=\widetilde{p}$ and $q=\widetilde{q}$; with these the time-frequency correlation set can be rewritten as follows\footnote{Recall that the equality is understood up to $O(1)$.}:
$$
\sTFC:=\left\{ \substack{\left( p, q\right)\\p \sim 2^{2k}, q \sim 2^k}\,:\,\frac{2q}{2^k} a \left(p+\frac{q^2}{2^k} \right)+\frac{3q^2}{2^{2k}} b \left(p \right)=c \left(\frac{p}{2^k}-q \right) \right\},
$$
where $a(\cdot)$, $b(\cdot)$, $c(\cdot)$ are some generic measurable functions with
$$
a, b: \left[2^{2k}, 2^{2k+1} \right] \cap \Z \longmapsto \left[2^k, 2^{k+1} \right]  \cap \Z,
$$
and
$$
c: \left[2^k, 2^{k+1} \right] \cap \Z \longmapsto \left[2^k, 2^{k+1} \right]  \cap \Z.
$$
Our goal is to show that there exists some absolute $\epsilon_5>0$, such that
\begin{equation} \label{TFCcontrolkey}
\# \sTFC \lesssim_{\epsilon_5} 2^{(3-\epsilon_5)k}.
\end{equation}
We prove \eqref{TFCcontrolkey} by contradiction. We thus assume that for any $\epsilon_5>0$
\begin{equation} \label{20230226eq30}
\# \sTFC \gtrsim_{\epsilon_5} 2^{(3-\epsilon_5)k}.
\end{equation}
Write now
$$
\sTFC=\bigcup_{r \sim 2^k} \sTFC_r,
$$
where
$$
\sTFC_r:=\left\{ (p, q) \in \sTFC: p \in \left[r2^k, (r+1)2^k \right] \right\}.
$$
For $r \sim 2^k$ and $s \sim 2^k$, set
$$
a_r(s)=a\left(\left(r-\frac{1}{2} \right)2^k+s \right) \quad \textrm{and} \quad b_r(s)=b\left(\left(r-\frac{1}{2} \right)2^k+s \right)\,.
$$
With these, \eqref{20230226eq30} can be rephrased as
$$
\sum_{r, \textsf p, q \sim 2^k} \left\lceil \frac{2q}{2^k}a_r\left(\sfp+\frac{q^2}{2^k} \right)+\frac{3q^2}{2^{2k}} b_r (\sfp)-c(r-q)   \right\rceil \gtrsim 2^{(3-\epsilon_5)k}.
$$
Our approach will exploit the curvature present in the definition of the time-frequency correlation set $\sTFC$ above via the short-long range interaction between the local parameters $p,\,q$ (short range) and the global parameter $r$ (long range).

Concretely, our argument develops as follows: applying first a Cauchy-Schwarz, we deduce that
$$
\sum_{r, q \sim 2^k} \left(\sum_{\sfp \sim 2^k} \left\lceil \frac{2q}{2^k}a_r\left(\sfp+\frac{q^2}{2^k} \right)+\frac{3q^2}{2^{2k}} b_r (\sfp)-c(r-q)   \right\rceil  \right)^2\gtrsim 2^{(4-2\epsilon_5)k},
$$
which implies
\begin{eqnarray} \label{20230301eq02}
&& \sum_{r, q \sim 2^k} \sum_{\sfp_1, \sfp_2 \sim 2^k}\left\lceil 2a_r\left(\sfp_1+\frac{q^2}{2^k} \right)+\frac{3q}{2^{k}} b_r (\sfp_1)- \frac{3q}{2^{k}} b_r (\sfp_2)-2a_r\left(\sfp_2+\frac{q^2}{2^k} \right)  \right\rceil \\
&& \quad \quad \quad\quad \quad\cdot \left\lceil \frac{2q}{2^k}a_r\left(\sfp_1+\frac{q^2}{2^k} \right)+\frac{3q^2}{2^{2k}} b_r (\sfp_1)-c(r-q)   \right\rceil \gtrsim 2^{(4-2\epsilon_5)k}. \nonumber
\end{eqnarray}
Applying a further change variable $\frac{q^2}{2^k}+\sfp_1 \to u$ and $\sfp_2 \to \sfp_2+\sfp_1$, we then have\footnote{Note here that we have $u-\sfp_1 \sim 2^k$. Also, by convention, given $x\in\R$ we set $c(x):=c(\lfloor x\rfloor)$ and similarly for $a(\cdot)$ and $b(\cdot)$.}
\begin{eqnarray*}
&& \sum_{r, u \sim 2^k} \sum_{\sfp_1, \sfp_2 \sim 2^k} \left\lceil 2a_r\left(u \right)-2a_r\left(\sfp_2+u \right) +\frac{3\sqrt{u-\sfp_1}}{2^{\frac{k}{2}}} \left(b_r (\sfp_1)-b_r (\sfp_1+\sfp_2) \right) \right\rceil \\
&& \quad  \cdot \left\lceil \frac{2\sqrt{u-\sfp_1}}{2^{\frac{k}{2}}}a_r\left(u \right)+\frac{3(u-\sfp_1)}{2^k} b_r (\sfp_1)-c \left(r- \sqrt{2^k(u-\sfp_1)} \right)   \right\rceil \gtrsim 2^{(4-2\epsilon_5)k}.
\end{eqnarray*}
Applying now a second Cauchy-Schwarz with respect to $r, u, \sfp_2$, we get
\begin{eqnarray*}
&& \sum_{r, u, \sfp_2 \sim 2^k} \Bigg( \sum_{\sfp_1 \sim 2^k} \left\lceil 2a_r\left(u \right)-2a_r\left(\sfp_2+u \right) +\frac{3\sqrt{u-\sfp_1}}{2^{\frac{k}{2}}} \left(b_r (\sfp_1)-b_r (\sfp_1+\sfp_2) \right) \right\rceil \\
&& \quad \quad  \cdot \left\lceil \frac{2\sqrt{u-\sfp_1}}{2^{\frac{k}{2}}}a_r\left(u \right)+\frac{3(u-p_1)}{2^k} b_r (\sfp_1)-c \left(r- \sqrt{2^k(u-\sfp_1)} \right)   \right\rceil \Bigg)^2\gtrsim 2^{(5-4\epsilon_5)k},
\end{eqnarray*}
which implies
\begin{eqnarray} \label{20230227eq40}
 && \sum_{\substack{r, u, \sfp_2 \sim 2^k \\ \sfp_1, \sfp_3 \sim 2^k}}  \left\lceil \sqrt{\frac{u-\sfp_1}{u-\sfp_3}} \left(b_r(\sfp_1)-b_r(\sfp_1+\sfp_2) \right)-   \left(b_r(\sfp_3)-b_r(\sfp_3+\sfp_2) \right)\right\rceil \\
 && \quad \quad \cdot \left\lceil 2a_r\left(u \right)-2a_r\left(\sfp_2+u \right) +\frac{3\sqrt{u-\sfp_1}}{2^{\frac{k}{2}}} \left(b_r (\sfp_1)-b_r (\sfp_1+\sfp_2) \right) \right\rceil \nonumber \\
 && \quad \quad \cdot \left\lceil \frac{2\sqrt{u-\sfp_1}}{2^{\frac{k}{2}}}a_r\left(u \right)+\frac{3(u-\sfp_1)}{2^k} b_r (\sfp_1)-c \left(r- \sqrt{2^k(u-\sfp_1)} \right)   \right\rceil \nonumber \\
 && \quad \quad \cdot \left\lceil \frac{2\sqrt{u-\sfp_3}}{2^{\frac{k}{2}}}a_r\left(u \right)+\frac{3(u-\sfp_3)}{2^k} b_r (\sfp_3)-c \left(r- \sqrt{2^k(u-\sfp_3)} \right)   \right\rceil \gtrsim 2^{(5-4\epsilon_5)k}. \nonumber
\end{eqnarray}
Let us now observe that the main contribution on the left-hand side of \eqref{20230227eq40} comes from the off-diagonal term, \textit{i.e.} $\left|\sfp_1-\sfp_3 \right| \gtrsim 2^{k-4\epsilon_5 k}$. This is trivial since the diagonal component of the left-hand side of \eqref{20230227eq40} is bounded from above by
$$
\lesssim 2^{4k} \cdot 2^{k-4\epsilon_5 k}\leq 2^{(5-4\epsilon_5)k}.
$$
Next, we claim that one may also assume $(r, \sfp_1, \sfp_2) \in \bbS$, where
\begin{equation} \label{defS}
\bbS:=\left\{(r, \sfp_1, \sfp_2): \left|b_r(\sfp_1)-b_r(\sfp_1+\sfp_2) \right| \lesssim 2^{8\epsilon_5 k} \right\}.
\end{equation}
Indeed, for any fixed choice of $r, \sfp_1, \sfp_2$, and $\sfp_3$ with $(r, \sfp_1, \sfp_2) \in \bbS^c$ and $\left|\sfp_1-\sfp_3 \right| \gtrsim 2^{k-4\epsilon_5 k}$, we consider the mapping
$$
[2^k, 2^{k+1}] \ni u \longmapsto \sqrt{ \frac{u-\sfp_1}{u-\sfp_3}} \sim 1.
$$
It is clear that such a mapping is monotone. Moreover, its derivative is of magnitude $\gtrsim \frac{1}{2^{(1+4\epsilon_5)k}}$. Therefore, (up to polynomial factors that we never account for in our estimates since we are after exponential decay)
\begin{equation} \label{20230227eq50}
\sum_{\substack{(r, \sfp_1, \sfp_2) \in \bbS^c \\ \sfp_3, u \sim 2^k}}  \left\lceil \sqrt{\frac{u-\sfp_1}{u-\sfp_3}} \left(b_r(\sfp_1)-b_r(\sfp_1+\sfp_2) \right)-   \left(b_r(\sfp_3)-b_r(\sfp_3+\sfp_2) \right)\right\rceil   \lesssim  2^{(5-4\epsilon_5)k}.
\end{equation}

Thus, we must have that
\begin{eqnarray} \label{20230228eq01}
 && \sum_{\substack{r, u, \sfp_2 \sim 2^k \\ \sfp_1, \sfp_3 \sim 2^k \\ |\sfp_1-\sfp_3| \gtrsim 2^{k-4\epsilon_5 k} \\ (r, \sfp_1, \sfp_2) \in \bbS}}  \left\lceil \sqrt{\frac{u-\sfp_1}{u-\sfp_3}} \left(b_r(\sfp_1)-b_r(\sfp_1+\sfp_2) \right)-   \left(b_r(\sfp_3)-b_r(\sfp_3+\sfp_2) \right)\right\rceil \\
 && \quad \quad \cdot \left\lceil 2a_r\left(u \right)-2a_r\left(\sfp_2+u \right) +\frac{3\sqrt{u-\sfp_1}}{2^{\frac{k}{2}}} \left(b_r (\sfp_1)-b_r (\sfp_1+\sfp_2) \right) \right\rceil \nonumber \\
 && \quad \quad \cdot \left\lceil \frac{2\sqrt{u-\sfp_1}}{2^{\frac{k}{2}}}a_r\left(u \right)+\frac{3(u-\sfp_1)}{2^k} b_r (\sfp_1)-c \left(r- \sqrt{2^k(u-\sfp_1)} \right)   \right\rceil \nonumber \\
 && \quad \quad \cdot \left\lceil \frac{2\sqrt{u-\sfp_3}}{2^{\frac{k}{2}}}a_r\left(u \right)+\frac{3(u-\sfp_3)}{2^k} b_r (\sfp_3)-c \left(r- \sqrt{2^k(u-\sfp_3)} \right)   \right\rceil\gtrsim 2^{(5-4\epsilon_5)k}. \nonumber
\end{eqnarray}
Now from \eqref{defS}, \eqref{20230228eq01}, triangle's inequality and the change of variable $\sfp_2 \to \sfp_2-\sfp_1$, we deduce
\begin{equation} \label{20230227eq42}
\sum_{\substack{r, u \\ \sfp_1, \sfp_2} \sim 2^k}  \left\lceil \frac{2\sqrt{u-\sfp_1}}{2^{\frac{k}{2}}}a_r\left(u \right)+\frac{3(u-\sfp_1)}{2^k} b_r (\sfp_2)-c \left(r- \sqrt{2^k(u-\sfp_1)} \right)   \right\rceil \gtrsim 2^{(4-12\epsilon_5)k}\,.
\end{equation}
Applying another change variable $\frac{q^2}{2^k}=u-\sfp_1$, relation \eqref{20230227eq42} can be further reduced to
$$
\sum_{r, u, q, \sfp_2 \sim 2^k} \left\lceil \frac{2q}{2^k}a_r\left(u \right)+\frac{3q^2}{2^{2k}} b_r (\sfp_2)-c \left(r- q \right)   \right\rceil \gtrsim 2^{(4-12\epsilon_5)k}.
$$
Therefore, there exists some $u_0, p_0 \sim 2^k$  such that
$$
\sum_{r, q \sim 2^k} \left \lceil \frac{2q}{2^k} a_r(u_0)+\frac{3q^2}{2^{2k}} b_r(p_0)-c(r-q) \right\rceil \gtrsim 2^{(2-12\epsilon_5) k}.
$$
Denote $a_0(r):=a_r(u_0)$ and $b_0(r):=b_r(u_0)$. Then by applying the change variable $r-q \to q$, we have
$$
\sum_{r, q \sim 2^k} \left\lceil \frac{2(r-q)}{2^k}a_0(r)+\frac{3(r-q)^2}{2^{2k}}b_0(r)-c(q) \right\rceil \gtrsim 2^{(2-12\epsilon_5)k}.
$$
Applying again Cauchy-Schwarz, we see that
\begin{equation} \label{20230301eq01}
\sum_{q, r_1, r_2 \sim 2^k} \bigg\lceil \frac{3q^2}{2^{2k}} \sfa(r_1, r_2)+\frac{2q}{2^k} \sfb(r_1, r_2)+\sfc(r_1, r_2) \bigg\rceil \gtrsim 2^{(3-24\epsilon_5)k}.
\end{equation}
where, for notational simplicity, we set
$$
\sfa(r_1, r_2):=b_0(r_1)-b_0(r_2),
$$
$$
\sfb(r_1, r_2):=a_0(r_2)-a_0(r_1)  +\frac{3r_2}{2^k}b_0(r_2)-\frac{3r_1}{2^k}b_0(r_1),
$$
and
$$
\sfc(r_1, r_2):=\frac{2r_1}{2^k}a_0(r_1)-\frac{2r_2}{2^k}a_0(r_2)+\frac{3r_1^2}{2^{2k}}b_0(r_1)-\frac{3r_2^2}{2^{2k}}b_0(r_2).
$$
Applying Cauchy-Schwarz in $q$ in \eqref{20230301eq01}, we deduce
\begin{equation*}
\sum_{{r_1, r_2 \sim 2^k}\atop{q_1, q_2 \sim 2^k}}\left\lceil \frac{3(q_1^2-q_2^2)}{2^{2k}} \sfa(r_1, r_2)+\frac{2(q_1-q_2)}{2^k} \sfb(r_1, r_2) \right\rceil \left\lceil \frac{3q_2^2}{2^{2k}} \sfa(r_1, r_2)+\frac{2q_2}{2^k} \sfb(r_1, r_2)+\sfc(r_1, r_2) \right\rceil  \gtrsim 2^{(4-48\epsilon_5)k},
\end{equation*}
which further implies
\begin{equation*}
\sum_{r_1, r_2 \sim 2^k}  \sum_{q_1, q_2 \sim 2^k} \left\lceil \frac{3 q_1}{2^{k}} \sfa(r_1, r_2)+ 2 \sfb(r_1, r_2) \right\rceil\,\left\lceil \frac{3q_2^2}{2^{2k}} \sfa(r_1, r_2)+\frac{2q_2}{2^k} \sfb(r_1, r_2)+\sfc(r_1, r_2) \right\rceil  \gtrsim 2^{(4-96\epsilon_5)k}.
\end{equation*}
From here we deduce that
\begin{equation}\label{abc}
\sum_{r_1, r_2 \sim 2^k} \left\lceil \sfa(r_1, r_2) \right\rceil  \left\lceil \sfb(r_1, r_2) \right\rceil  \left\lceil \sfc(r_1, r_2) \right\rceil \gtrsim 2^{2k-1536\epsilon_5 k}\,.
\end{equation}
Now a moment of thought gives the following quadratic dependence:
\begin{equation}\label{abc1}
c(r_1,r_2)+\frac{2 r_2}{2^k} b(r_1,r_2)+\frac{3 r_2^2}{2^{2k}} a(r_1,r_2)=\frac{2(r_1-r_2)}{2^k} a_0(r_1)+\frac{3(r_1-r_2)^2}{2^{2k}} b_0(r_1)\,.
\end{equation}
Combining \eqref{abc} with \eqref{abc1} we have
\begin{equation}\label{abl}
\sum_{r_1, r_2 \sim 2^k} \left\lceil 2 a_0(r_1)+\frac{3r_2}{2^k} b_0(r_1) \right\rceil \gtrsim 2^{2k-6144\epsilon_5 k}\,.
\end{equation}
However, since $a_0(r_1), b_0(r_1) \sim 2^k$, we must have that
\begin{equation} \label{abu}
\sum_{r_1, r_2 \sim 2^k} \left\lceil 2 a_0(r_1)+\frac{3r_2}{2^k} b_0(r_1) \right\rceil \lesssim k\,2^{k}\,.
\end{equation}
From \eqref{abl}and \eqref{abu} we deduce that
\begin{equation} \label{20230306eq03}
k\,2^{k} \gtrsim 2^{2k-6144\epsilon_5 k},
\end{equation}
which gives the desired contradiction if we pick $\epsilon_5=\frac{1}{10000}$.

Conclude from \eqref{sparseksmall} and the above that
\begin{equation}\label{sparseksmall1}
|\Lambda^{k, \bbS_\mu}_m(\vec{f})| \lesssim  2^{(7\mu-\frac{1}{20000})k}\,\|\vec{f} \|_{L^{\vec{p}}(I^{k})}\qquad \forall\:\vec{p}\in {\bf H}^{+}\,.
\end{equation}
thus finishing the sparse case.
\medskip

\begin{obs} \label{bkII}[\textsf{Bookkeeping (II)}]  Putting now together \eqref{Ldecksmall}, \eqref{unif1ksmall}, \eqref{unif2ksmall} and \eqref{sparseksmall1} we deduce that Theorem \ref{mainthmI} holds with $\bar{\epsilon}=\min\{\frac{\mu}{2},\,\frac{1}{20000}-7\mu\}\approx \frac{1}{3}\cdot 10^{-5}$.
\end{obs}

The proof of the case $0 \le k < \frac{m}{4}$ is now complete.

\section{The diagonal term $\Lambda^{D}$: The case $\frac{m}{4} \le k<\frac{m}{2}$}\label{Sec06}

In this section we prove a slightly weaker version than the one needed for concluding Theorem \ref{mainthm-2022}. Indeed, instead of presenting the direct analogue of Theorem \ref{mainthmI} we will prove the following

\begin{thm} \label{mainthmII}
There exists some $\underline{\epsilon}>0$, such that for any $\frac{m}{4} \le k \le \frac{m}{2}$ and $\vec{p}\in {\bf H}^{+}$
\begin{equation}\label{keyestim}
\left|\Lambda_m^{k, I_0^{0, k}}(\vec{f})\right| \lesssim \min\{1,\,k\,2^{-\underline{\epsilon}\,(\frac{m}{2}-k)}\} \,\|\vec{f} \|_{L^{\vec{p}}(3\, I_0^{0, k})}\,.
\end{equation}
\end{thm}

For covering the desired exponential decay also in the limiting case when $k$ approaches $\frac{m}{2}$ one has to apply/adapt the reasonings developed in Section \ref{MaindiagKpositive}. A very brief discussion of the latter is presented in Section \ref{AppendixC} of the Appendix.

Focussing now exclusively on Theorem \ref{mainthmII}, its proof is closely related to the proof of Theorem \ref{mainthmI}, and therefore, we will only provide a sketch of our approach outlining along the way the main differences/modifications and leaving the details to the interested reader.

\subsection{A time-frequency sparse--uniform dichotomy}\label{SU2}

Recall from the previous case that for $j\in\{1,\ldots,4\}$,
$$
\omega_j^{\ell_j}=\left[\ell_j 2^{\frac{m}{2}+\left(\underline{j}+1 \right)k}, \left(\ell_j+1 \right) 2^{\frac{m}{2}+\left(\underline{j}+1 \right)k}  \right], \quad \ell_j \sim 2^{\frac{m}{2}-k},
$$
and
$$
\Omega_j^{L_j}= \left[L_j 2^{m+(\underline{j}-1)k}, \left(L_j+1 \right) 2^{m+(\underline{j}-1)k}\right], \quad L_j \sim 2^k\,,
$$
and notice that in the current context we have the reverse inequality $|\omega_j^{\ell_j}|\geq |\Omega_j^{L_j}|$.

Adapting ourselves to the present situation we define the \emph{sparse set} of indices as follows: for $j\in\{1, 2\}$,
$$
\bbS_{\mu}(f_j):= \left\{(q_j, \ell_j): \int_{3 I_{q_j}^{0, (j+1)k}} \left|f_j^{\omega_j^{\ell_j}} \right|^2  \gtrsim 2^{-3\mu \left(\frac{m}{2}-k \right)} \int_{3 I_{q_j}^{0, (j+1)k}} |f_j|^2, \ q_j \sim 2^{jk}, \ell_j \sim 2^{\frac{m}{2}-k}\right\},
$$
and for $j\in\{3, 4\}$,
$$
\bbS_{\mu}(f_j):= \left\{(q_j, L_j): \int_{3 I_{q_j}^{0, 3k}} \left|f_j^{3\Omega_j^{L_j}} \right|^2  \gtrsim 2^{-\mu \left(\frac{m}{2}-k \right)} \int_{3 I_{q_j}^{0, 3k}} |f_j|^2, \  q_j \sim 2^{2 k},   L_j \sim 2^k \right\}.
$$
We define the \emph{uniform set} of indices, that is, $\bbU_\mu(f_j), j\in\{1, \ldots, 4\}$, and decompose the main term \eqref{20230224main01} into \emph{sparse/uniform components}, that is, $\Lambda_m^{k, \bbS_\mu}$ and $\Lambda_m^{k, \bbU_\mu}$ as usual within the regime $k$ small. We start with the usual pattern by treating first the uniform component.

\subsection{Treatment of the uniform component $\Lambda_m^{k, \bbU_\mu}$}

Following \eqref{20230224eq01}, we apply Cauchy-Schwarz and write
\begin{equation} \label{20230305eq10}
\Lambda_m^{k, \bbU_\mu} \lesssim 2^k \sum_{p \sim 2^{\frac{m}{2}+2k}} \bbA_p \bbB_p,
\end{equation}
where $\bbA_p$ and $\bbB_p$ are defined by \eqref{20230305eq01} and \eqref{20230305eq02}, respectively. The estimate of $\bbB_p$ follows the same steps as in \eqref{20230225eq03} in order to conclude
$$
\bbB_p^2 \lesssim  2^{2k}  \sum_{\substack{\ell_1, \ell_2 \sim 2^{\frac{m}{2}-k} \\q \sim 2^{\frac{m}{2}}}}  \int_{I_p^{m, 3k}} \int_{I_q^{m, k}} \left|\underline{f}_1^{\omega_1^{\ell_1}}(x-t) \underline{f}_2^{\omega_2^{\ell_2}}(x+t^2) \right|^2 dtdx.
$$
Now we turn to the estimate of the term $\bbA_p$. Recalling \eqref{ap}, and using the fact that $2^k \ge 2^{\frac{m}{2}-k}$, we split the regime of $n_3$ into smaller intervals of length $2^{\frac{m}{2}-k}$, in order to get
\begin{equation}
\bbA_p^2 \lesssim \sum_{\substack{\widetilde{n_3} \sim 2^{2k-\frac{m}{2}} \\ q \sim 2^{\frac{m}{2}}}} \quad\sum_{n_3 \in \left[ \widetilde{n_3}2^{\frac{m}{2}-k}, \left(\widetilde{n_3}+1 \right) 2^{\frac{m}{2}-k} \right]}\quad \sum_{\ell_2, \ell_3 \sim 2^{\frac{m}{2}-k}} \left| \left\langle f_3^{\omega_3^{\ell_3}}, \phi_{p+\frac{q^3}{2^m}, \ell_3 2^k+n_3}^{3k} \right \rangle \right|^2  \left| \left \langle f_4, \phi_{p, \ell_3 2^k+n_3+\ell_2}^{3k} \right \rangle \right|^2.
\end{equation}
Taking the above summation with respect to $\ell_2$ first and then $n_3$, we see that the above term is bounded above by\footnote{Notice here that within each spatial fiber we zoom in into the original frequency interval $\omega_3^{\ell_3}$ in order to obtain a more precise localization at the level of an interval of length $\approx 2^{m+2k}$.}
\begin{equation} \label{20230305eq04}
\sum_{\widetilde{n_3} \sim 2^{2k-\frac{m}{2}}} \sum_{ \substack{ \ell_3 \sim 2^{\frac{m}{2}-k} \\ q \sim 2^{\frac{m}{2}}}} \left|  \left \langle f_3^{\omega_3^{\ell_3}}, \phi_{p+\frac{q^3}{2^m}, \ell_3 2^k+\widetilde{n_3} 2^{\frac{m}{2}-k}+ 3 \left[0, 2^{\frac{m}{2}-k} \right]}^{3k} \right\rangle \right|^2\,\left|  \left \langle f_4, \phi_{p, \ell_3 2^k+\widetilde{n_3} 2^{\frac{m}{2}-k}+ 3 \left[0, 2^{\frac{m}{2}-k} \right]}^{3k} \right\rangle \right|^2.
\end{equation}
Deduce now from \eqref{20230305eq04} that
\begin{eqnarray} \label{20230205eq13}
\bbA_p^2 &\lesssim& \sum_{\widetilde{n_3} \sim 2^{2k-\frac{m}{2}}} \sum_{\substack{ \ell_3 \sim 2^{\frac{m}{2}-k} \\ q \sim 2^{\frac{m}{2}}}} \left\|f_3^{3 \Omega_3^{2^{2k-\frac{m}{2}}\ell_3+\widetilde{n_3}}} \right\|^2_{L^2 \left(I_{p+\frac{q^3}{2^m}}^{m, 3k} \right)} \left\|f_4^{3 \Omega_4^{2^{2k-\frac{m}{2}}\ell_3+\widetilde{n_3}}} \right\|^2_{L^2 \left(I_p^{m, 3k} \right)} \nonumber  \\
&\lesssim& \sum_{\substack{\widetilde{n_3} \sim 2^{2k-\frac{m}{2}} \\ \ell_3 \sim 2^{\frac{m}{2}-k}}} \left\|f_3^{3 \Omega_3^{2^{2k-\frac{m}{2}}\ell_3+\widetilde{n_3}}} \right\|^2_{L^2 \left(3 I_{\frac{p}{2^{\frac{m}{2}}}}^{0, 3k} \right)} \left\|f_4^{3 \Omega_4^{2^{2k-\frac{m}{2}}\ell_3+\widetilde{n_3}}} \right\|^2_{L^2 \left(I_p^{m, 3k} \right)} \nonumber \\
&\lesssim& \sum_{L_3 \sim 2^k} \left\|f_3^{3 \Omega_3^{L_3}} \right\|_{L^2\left(3 I_{\frac{p}{2^{\frac{m}{2}}}}^{0, 3k} \right)}^2  \left\|f_4^{3\Omega_4^{L_3}} \right\|_{L^2 \left(I_p^{m, 3k} \right)}^2.
\end{eqnarray}
Substituting the estimates for $\bbA_p$ and $\bbB_p$ within \eqref{20230305eq10}, we obtain precisely \eqref{20230305eq05}. Therefore, in the case when \emph{either $f_3$ or $f_4$ is uniform}, one can apply similar arguments with the ones in Subsection \ref{20230305subsec01} in order to obtain the analogue of \eqref{unif1ksmall} with $\mu k$ there replaced by $\mu \left(\frac{m}{2}-k \right)$.

\medskip

We move now to the case when \emph{both $f_3$ and $f_4$ are sparse, and at least one of the functions $f_1$ or $f_2$ is uniform}. Again, as in the case $0 \le k \le \frac{m}{4}$, we can pick---up to a $2^{\mu \left(\frac{m}{2}-k \right)}$-factor---a unique measurable function $L_3: \left[ 2^{2k}, 2^{2k+1} \right] \to \left[2^k, 2^{k+1} \right]$, such that, for $\widetilde{p} \sim 2^{2k}$,
$$
2^{-\mu \left(\frac{m}{2}-k \right)}\int_{3 I_{\widetilde{p}}^{0, 3k}} |f_j|^2\lesssim \int_{3 I_{\widetilde{p}}^{0, 3k}} \left| f_j^{3 \Omega_j^{L_3 \left(\widetilde{p} \right)}} \right|^2  \lesssim \int_{3 I_{\widetilde{p}}^{0, 3k}} |f_j|^2, \quad j\in\{3, 4\}\,.
$$
Note that in the regime $\frac{m}{4} \le k <\frac{m}{2}$, if $\Omega_3^{L_3} \subseteq \omega_3^{\ell_3}$, then  $\ell_3=\left\lfloor \frac{L_3}{2^{2k-\frac{m}{2}}} \right\rfloor+ O(1)$. Consequently, it makes sense to define the map $\ell_3:\,[2^{2k}, 2^{2k+1}] \to [2^{\frac{m}{2}-k}, 2^{\frac{m}{2}-k+1}]$ by
$\ell_3 \left(\widetilde{p} \right):= \left\lfloor \frac{L_3(\widetilde{p})}{2^{2k-\frac{m}{2}}} \right\rfloor$.

Arguing as in \eqref{20230225eq03}--\eqref{20230225eq04} and making use of \eqref{20230205eq13}, we deduce that
\begin{equation} \label{20230305eq20}
\Lambda_m^{k, \bbU_\mu} \lesssim 2^{2k} \cdot 2^{\mu \left(\frac{m}{2}-k \right)} \sum_{\widetilde{p} \sim 2^{2k}} \left\|f_3^{3 \Omega_3^{L_3(\widetilde{p})}} \right\|_{L^2 \left(3 I_{\widetilde{p}}^{0, 3k} \right)} \left\|f_4^{3 \Omega_4^{L_3(\widetilde{p})}} \right\|_{L^2 \left(I_{\widetilde{p}}^{0, 3k} \right)} \cdot \calR_{\widetilde{p}},
\end{equation}
where
$$
\calR_{\widetilde{p}}^2:=\sum_{q \sim 2^{\frac{m}{2}}} \sum_{\left(\ell_1, \ell_2, \ell_3 \left(\widetilde{p} \right) \right) \in \TFC_q} \int_{I_{\widetilde{p}}^{0, 3k}} \int_{I_q^{m, k}} \left| \underline{f}_1^{\omega_1^{\ell_1}}(x-t) \underline{f}_2^{\omega_2^{\ell_2}}(x+t^2) \right|^2 dtdx.
$$
Similar to the case $0 \le k \le \frac{m}{4}$, in order to estimate $\calR_{\widetilde{p}}$ we first split the regime of $q$ as follows $\left\{q \sim 2^{\frac{m}{2}} \right\}=\bigcup_{\widetilde{q} \sim 2^k} \left[ \widetilde{q} 2^{\frac{m}{2}-k}, \left( \widetilde{q}+1 \right) 2^{\frac{m}{2}-k}\right]\,,$
and observe that if $q \in \left[ \widetilde{q} 2^{\frac{m}{2}-k}, \left( \widetilde{q}+1 \right) 2^{\frac{m}{2}-k}\right]$ and $\left(\ell_1, \ell_2, \ell_3( \widetilde{p}) \right) \in \TFC_q$, then $\left(\ell_1, \ell_2, \ell_3 (\widetilde{p}) \right) \in \TFC_{2^{\frac{m}{2}-k} \widetilde{q}}$.  Indeed, this is an immediate consequence of
$$
\left| \left(  \ell_1-\frac{2q}{2^{\frac{m}{2}}} \ell_2-\frac{3q^2}{2^m} \ell_3 (\widetilde{p}) \right)- \left( \ell_1-\frac{2  \widetilde{q} \cdot 2^{\frac{m}{2}-k}}{2^{\frac{m}{2}}} \ell_2-\frac{3 \widetilde{q}^2  \cdot 2^{m-2k} }{2^m} \ell_3 (\widetilde{p}) \right)  \right| \lesssim 2^{\frac{m}{2}-2k} \le 1.
$$

Hence,  via Parseval and a similar argument with the one in \eqref{20230225eq38}, one has
\begin{eqnarray} \label{20230306eq02}
\calR_{\widetilde{p}}^2&\lesssim &\sum_{\widetilde{q} \sim 2^k} \sum_{\substack{ q \in \left[ \widetilde{q}2^{\frac{m}{2}-k}, \left(\widetilde{q}+1 \right)2^{\frac{m}{2}-k} \right] \\ \left(\ell_1, \ell_2, \ell_3 \left(\widetilde{p} \right) \right) \in \TFC_{2^{\frac{m}{2}-k}\widetilde{q}}}} \int_{I_{\widetilde{p}}^{0, 3k}} \int_{I_q^{m, k}} \left| \underline{f}_1^{\omega_1^{\ell_1}}(x-t) \underline{f}_2^{\omega_2^{\ell_2}}(x+t^2) \right|^2 dtdx  \nonumber \\
&\lesssim& \sum_{\substack{\widetilde{q} \sim 2^k \\ \ell_2 \sim 2^{\frac{m}{2}-k}}} \int_{I_{\widetilde{p}}^{0, 3k}} \int_{I_{\widetilde{q}}^{0, 2k}}\left|\underline{f}_1^{\omega_1^{\frac{2 \widetilde{q}}{2^k} \ell_2+\frac{3\widetilde{q}^2}{2^{2k}} \ell_3(\widetilde{p})}}(x-t)  \underline{f}_2^{\omega_2^{\ell_2}}(x+t^2) \right|^2 dtdx \nonumber  \\
&\lesssim& \sum_{\substack{\widetilde{q} \sim 2^k \\ \ell_2 \sim 2^{\frac{m}{2}-k}}}  \left(\int_{3 I^{0, 2k}_{\frac{\widetilde{p}}{2^k}-\widetilde{q}}} \left| f_1^{\omega_1^{\frac{2 \widetilde{q}}{2^k} \ell_2+\frac{3\widetilde{q}^2}{2^{2k}} \ell_3(\widetilde{p})}}  \right|^2 \right) \left(\int_{3 I_{\widetilde{p}+\frac{\widetilde{q}^2}{2^k}}^{0, 3k}} \left|f_2^{\omega_2^{\ell_2}} \right|^2 \right).
\end{eqnarray}
Recall that from our assumption, at least one of the functions $f_1$, $f_2$ is uniform. Without loss of generality, we may assume $f_1=f_1^{\bbU_\mu}$. Therefore, we have
\begin{equation*}
\calR_{\widetilde{p}}^2  \lesssim 2^{-3\mu \left(\frac{m}{2}-k \right)} \cdot  \sum_{\widetilde{q} \sim 2^k}   \left(\int_{3 I^{0, 2k}_{\frac{\widetilde{p}}{2^k}-\widetilde{q}}} \left| f_1 \right|^2 \right) \left(\int_{3 I_{\widetilde{p}+\frac{\widetilde{q}^2}{2^k}}^{0, 3k}} \left|f_2 \right|^2 \right).
\end{equation*}
Inserting the above estimate back to \eqref{20230305eq20}, we see that the resulting estimate is precisely \eqref{20230305eq40} with $\mu k$ replaced by $\mu \left(\frac{m}{2}-k \right)$. The rest of the proof, therefore, follows now the same treatment as in \eqref{Uniform-II-Small}.

\subsection{Treatment of the sparse component $\Lambda_m^{k, \bbS_\mu}$}\label{SparseII}

In this case, since all the functions $f_1, f_2, f_3, f_4$ are sparse, we have
\begin{enumerate}
    \item [$\bullet$] up to a $2^{\mu \left(\frac{m}{2}-k \right)}$ factor, there is a unique measurable $L_3: \left[2^{2k}, 2^{2k+1} \right] \to \left[2^k, 2^{k+1} \right]$ such that for $j\in\{3, 4\}$ and $\widetilde{p} \sim 2^{2k}$
   \begin{equation} \label{ctl34}
   2^{-\mu \left(\frac{m}{2}-k \right)}\,\int_{3 I_{\widetilde{p}}^{0, 3k}} |f_j|^2\lesssim \int_{3 I_{\widetilde{p}}^{0, 3k}} \left| f_j^{3 \Omega_j^{L_3 \left(\widetilde{p} \right)}} \right|^2 \lesssim \int_{3 I_{\widetilde{p}}^{0, 3k}} |f_j|^2\,.
   \end{equation}
    Moreover, as in the previous case, we denote $\ell_3 \left(\widetilde{p} \right):= \left\lfloor \frac{L_3(\widetilde{p})}{2^{2k-\frac{m}{2}}} \right\rfloor$ for $\widetilde{p} \sim 2^{2k}$.

    \item [$\bullet$] up to a $2^{3\mu \left(\frac{m}{2}-k\right)}$ factor, there exists unique measurable functions $\ell_j: \left[2^{jk}, 2^{jk+1} \right] \to \left[2^{\frac{m}{2}-k}, 2^{\frac{m}{2}-k+1} \right]$, $j\in\{1, 2\}$, such that for each $q_j \sim 2^{jk}$
    \begin{equation} \label{20230305eq60}
   2^{-3\mu \left(\frac{m}{2}-k\right)}\,\int_{3I_{q_j}^{0, (j+1)k}} \left|f_j \right|^2\lesssim \int_{3I_{q_j}^{0, (j+1)k}} \left|f_j^{\omega_j^{{\ell_j(q_j)}}} \right|^2 \lesssim  \int_{3I_{q_j}^{0, (j+1)k}} \left|f_j \right|^2\,.
    \end{equation}
\end{enumerate}
Using these together with \eqref{20230224main01} and Cauchy--Schwarz, we deduce
\begin{equation} \label{20230305eq61}
\Lambda^{k, \bbS_\mu}_m  \lesssim  2^k \cdot 2^{7\mu  \left(\frac{m}{2}-k \right)} \sum_{\substack{ \widetilde{p} \sim 2^{2k} \\ \widetilde{q} \sim 2^k}} \bbA_{\widetilde{p}, \widetilde{q}}\bbB_{\widetilde{p}, \widetilde{q}},
\end{equation}
where
\begin{eqnarray*}
&& \bbA_{\widetilde{p}, \widetilde{q}}^2 :=\sum_{\substack{p \in \left[\widetilde{p}2^{\frac{m}{2}}, \left(\widetilde{p}+1 \right)2^{\frac{m}{2}} \right] \\ q \in \left[ \widetilde{q} 2^{\frac{m}{2}-k}, \left( \widetilde{q}+1 \right)2^{\frac{m}{2}-k} \right] \\ n_3 \sim 2^k}} \sum_{\substack{\left(\ell_1, \ell_2, \ell_3\left( \frac{p}{2^{\frac{m}{2}}} \right) \right) \in \TFC_{p, q} \\ \eqref{ctl34}\&\eqref{20230305eq60}}} \\
&& \qquad \qquad\qquad\qquad   \left| \left\langle f_3^{\omega_3^{\ell_3\left( \frac{p}{2^{\frac{m}{2}}} \right) }}, \phi_{p+\frac{q^3}{2^m}, \ell_3\left( \frac{p}{2^{\frac{m}{2}}} \right)  2^k+n_3}^{3k} \right\rangle \right|^2 \left| \left\langle f_4, \phi_{p, \ell_3\left( \frac{p}{2^{\frac{m}{2}}} \right)  2^k+n_3+\ell_1 2^{-k}+\ell_2 }^{3k} \right\rangle \right|^2,
\end{eqnarray*}
and
\begin{eqnarray*}
&& \bbB_{\widetilde{p}, \widetilde{q}}^2:=\sum_{\substack{p \in \left[\widetilde{p}2^{\frac{m}{2}}, \left(\widetilde{p}+1 \right)2^{\frac{m}{2}} \right] \\ q \in \left[ \widetilde{q} 2^{\frac{m}{2}-k}, \left( \widetilde{q}+1 \right)2^{\frac{m}{2}-k} \right] \\ n_3 \sim 2^k}} \sum_{\substack{\left(\ell_1, \ell_2, \ell_3\left( \frac{p}{2^{\frac{m}{2}}} \right) \right) \in \TFC_{p, q} \\ \eqref{ctl34}\&\eqref{20230305eq60}}} \\
&&\frac{1}{\left|I_p^{m, 3k} \right|^2} \bigg| \int_{I_p^{m, 3k}} \int_{I_q^{m, k}}  \underline{f}_1^{\omega_1^{\ell_1}}(x-t) \underline{f}_2^{\omega_2^{\ell_2}}(x+t^2) e^{i 2^{\frac{m}{2}+k}t \cdot \frac{3q^2}{2^m}n_3}\,e^{i 2^{\frac{m}{2}+k}t \cdot 2^k \left(-\ell_1 +\ell_2 \frac{2q}{2^{\frac{m}{2}}}+ \ell_3\left( \frac{p}{2^{\frac{m}{2}}} \right) \frac{3q^2}{2^m}\right)} dt dx \bigg|^2.
\end{eqnarray*}
We first estimate $\bbA_{\widetilde{p}, \widetilde{q}}$ by following the arguments in \eqref{20230305eq04} and \eqref{20230205eq13}:
$$
\bbA_{\widetilde{p}, \widetilde{q}}^2  \lesssim \sum_{\substack{p \in \left[\widetilde{p}2^{\frac{m}{2}}, \left(\widetilde{p}+1 \right)2^{\frac{m}{2}} \right] \\ q \in \left[ \widetilde{q} 2^{\frac{m}{2}-k}, \left( \widetilde{q}+1 \right)2^{\frac{m}{2}-k} \right]}} \left\|f_3^{3\Omega_3^{L_3 (\widetilde{p})}} \right\|^2_{L^2 \left(3 I_{p+\frac{q^3}{2^m}}^{m, 3k} \right)} \left\|f_4^{3\Omega_4^{L_3(\widetilde{p})}} \right\|_{L^2 \left(3 I_p^{m, 3k} \right)}^2
$$
which matches the estimate \eqref{20230306eq01}, and therefore
$$
\bbA_{\widetilde{p}, \widetilde{q}}^2  \lesssim \sum_{s=0}^{2^k-1} \left\|f_3^{3\Omega_3^{L_3(\widetilde{p})}} \right\|^2_{L^2 \left(3 I_{2^k \widetilde{p}+s+\frac{\widetilde{q}^3}{2^{2k}}}^{0, 4k} \right)} \left\|f_4^{3\Omega_4^{L_3(\widetilde{p})}} \right\|_{L^2 \left(3 I_{2^k \widetilde{p}+s}^{0, 4k} \right)}^2.
$$
Next, we estimate $\bbB_{\widetilde{p}, \widetilde{q}}$. Arguing as in \eqref{20230225eq03} and \eqref{20230306eq02}, we have
$$
\bbB_{\widetilde{p}, \widetilde{q}}^2 \lesssim 2^{2k} \sum_{\substack{\ell_2 \sim 2^{\frac{m}{2}-k} \\ \eqref{ctl34}\&\eqref{20230305eq60}}} \left(\int_{3 I^{0, 2k}_{\frac{\widetilde{p}}{2^k}-\widetilde{q}}} \left| f_1^{\omega_1^{\frac{2 \widetilde{q}}{2^k} \ell_2+\frac{3\widetilde{q}^2}{2^{2k}} \ell_3(\widetilde{p})}}  \right|^2 \right) \left(\int_{3 I_{\widetilde{p}+\frac{\widetilde{q}^2}{2^k}}^{0, 3k}} \left|f_2^{\omega_2^{\ell_2}} \right|^2 \right).
$$
Combining the both estimates of $\bbA_{\widetilde{p}, \widetilde{q}}$ and $\bbB_{\widetilde{p}, \widetilde{q}}$ with \eqref{20230305eq61}, we have
\begin{eqnarray*}
\Lambda_m^{k, \bbS_\mu} &\lesssim& 2^{2k} \cdot 2^{7 \mu \left(\frac{m}{2}-k \right)} \sum_{\widetilde{p} \sim 2^{2k}, \ \widetilde{q} \sim 2^k} \left[\sum_{s=0}^{2^k-1} \left\|f_3^{3 \Omega_3^{L_3(\widetilde{p})}} \right\|^2_{L^2 \left(3 I_{2^k \widetilde{p}+s+\frac{\widetilde{q}^3}{2^{2k}}}^{0, 4k} \right)} \left\|f_4^{3\Omega_4^{L_3(\widetilde{p})}} \right\|_{L^2 \left(3 I_{2^k \widetilde{p}+s}^{0, 4k} \right)}^2 \right]^{\frac{1}{2}} \\
&\cdot& \left[ \sum_{\substack{\ell_2 \sim 2^{\frac{m}{2}-k} \\ \eqref{ctl34}\&\eqref{20230305eq60}}} \left(\int_{3 I^{0, 2k}_{\frac{\widetilde{p}}{2^k}-\widetilde{q}}} \left| f_1^{\omega_1^{\frac{2 \widetilde{q}}{2^k} \ell_2+\frac{3\widetilde{q}^2}{2^{2k}} \ell_3(\widetilde{p})}}  \right|^2 \right) \left(\int_{3 I_{\widetilde{p}+\frac{\widetilde{q}^2}{2^k}}^{0, 3k}} \left|f_2^{\omega_2^{\ell_2}} \right|^2 \right) \right]^{\frac{1}{2}}.
\end{eqnarray*}
Again, using the inherited time-frequency localization in the above estimate, we define
$$
\sTFC:=\left\{ \substack{\left( \widetilde{p}, \widetilde{q} \right)\\\widetilde{p} \sim 2^{2k}, \widetilde{q} \sim 2^k }\::\:  \frac{2\widetilde{q}}{2^k}\ell_2 \left(\widetilde{p}+\frac{\widetilde{q}^2}{2^k} \right)+\frac{3\widetilde{q}^2}{2^{2k}} \ell_3 \left( \widetilde{p} \right)=\ell_1 \left(\frac{\widetilde{p}}{2^k}-\widetilde{q} \right)+O(1) \right\}.
$$
Following the argument in \eqref{Uniform-I}, we see that it suffices to show that there exists some $\epsilon_6>0$, such that
$$
\# \sTFC \lesssim k\, 2^{3k-\epsilon_6 \left(\frac{m}{2}-k \right)}.
$$
The proof of the above estimate follows essentially the argument in Subsection \ref{20230306subsec01} with only some simple modifications/adaptations among which we mention:
\begin{enumerate}
    \item [(1)] The three generic measurable functions $a(\cdot), b(\cdot)$, and $c(\cdot)$ have now the range
    $\newline$ $\left[2^{\frac{m}{2}-k}, 2^{\frac{m}{2}-k+1}\right] \cap \Z$, rather than the previous $[2^k, 2^{k+1} ]\cap \Z$;

    \item [(2)] In the current argument one replaces $\epsilon k$ by $\epsilon \left(\frac{m}{2}-k \right)$;

     \item [(3)] Finally, from the analogue of \eqref{abu}, that in this new setting reads:
 \begin{equation} \label{abu1}
\sum_{r_1, r_2 \sim 2^k} \left\lceil 2 a_0(r_1)+\frac{3r_2}{2^k} b_0(r_1) \right\rceil \lesssim k\,2^{2k}\, 2^{-(\frac{m}{2}-k)}\,,
\end{equation}
and some straightforward reasonings one gets $\epsilon_6\approx\frac{1}{10000}$.
\end{enumerate}

This completes the proof of Theorem \ref{mainthmII} with $\left(\frac{\mu}{2}=\right)\underline{\epsilon}\approx \frac{1}{3}\cdot 10^{-5}$.

\medskip

\section{The stationary off-diagonal term $\Lambda^{\not{D},S}$: The case $j=l>m+100$ and $k \ge 0$} \label{20230321sec01}

\begin{obs}\textsf{[Hierarchy]}\label{lgcsufficient}  The approach of $\Lambda^{\not{D},S}$ is based entirely on the Rank I LGC methodology and is consistent with the fact that this term represents the stationary off-diagonal component that is ``less singular" than the main diagonal component. Even so, it is worth noticing that the treatment of  $\Lambda^{\not{D},S}$ encompasses as a particular case the bilinear Hilbert transform along (some suitable) non-flat curves providing thus yet another simpler approach of the latter topic originally treated in \cite{Li13} and \cite{Lie15}. Finally, the complete treatment of $\Lambda^{\not{D},S}$ will be covered in the next three sections: the present one, Section \ref{stationaryoff2} and Section \ref{stationaryoff3}.
\end{obs}

We now start by noticing that in the case addressed in this section, we have
\begin{align*}
 & \supp \widehat{f_1} \subseteq \left\{|\xi| \sim 2^{j+k} \right\},  \quad  \supp \widehat{f_2} \subseteq \left\{| \eta| \sim 2^{j+2k} \right\}, \\
& \supp \widehat{f_3} \subseteq \left\{|\tau| \sim 2^{m+3k}  \right\}\  \textrm{and} \quad \supp \widehat{f_4} \subseteq \left\{|\eta| \sim 2^{j+2k} \right\}+ \left\{|\tau| \sim 2^{m+3k} \right\}.
\end{align*}
In what follows, We will split our analysis into the following cases\footnote{The itemization that follows assumes $\xi\in \supp \widehat{f_1}$, $\eta\in \supp \widehat{f_2}$ and $\tau\in \supp \widehat{f_3}$.}:
\begin{enumerate}
    \item [I.] $0 \le m \le \frac{j}{2}$ with the subdivisions
    \begin{enumerate}
        \item [I.1] $0 \le k \le \frac{j-m}{2}$: this corresponds to the case when $|\tau| \lesssim |\xi|$;
        \item [I.2] $\frac{j-m}{2} <k < j-m$: this corresponds to the case when $|\xi| \lesssim |\tau| \lesssim |\eta|$;
        \item [I.3] $k \ge j-m$: this corresponds to the case when $|\eta| \lesssim |\tau|$.
    \end{enumerate}
    \item [II.] $\frac{j}{2} \le m \le j-100$ with the subdivisions
     \begin{enumerate}
        \item [II.1] $0 \le k < j-m$: this corresponds to the case when $|\xi|,\,|\tau| \lesssim |\eta|$;
        \item [II.2] $k \ge j-m$: this corresponds to the case when $|\eta| \lesssim |\tau|$.
    \end{enumerate}
\end{enumerate}

\subsection{Treatment of Case I.1:  $0 \le m \le \frac{j}{2}$ and $0 \le k \le \frac{j-m}{2}$.} \label{20230321subsec01}

 Recall that the main term we wish to understand is given by
\begin{equation}
\Lambda_{j, j, m}^k(\vec{f}):=\Lambda^k_{j, m}(\vec{f})\approx2^k \sum_{\substack{p \sim 2^{\frac{j}{2}+2k} \\ q \sim 2^{\frac{j}{2}}}} \int_{I_p^{j, 3k}} \int_{I_q^{j, k}} f_{1,j+k}(x-t)f_{2,j+2k}(x+t^2)f_{3,m+3k}(x+t^3)f_4(x)dtdx.
\end{equation}
Now, for $x \in I_p^{j, 3k}$ and $t \in I_q^{j, k}$, we have
\begin{equation}
 f_{3,m+3k}(x+t^3)=\int_{\R} \widehat{f_3}(\tau) \phi \left(\frac{\tau}{2^{m+3k}} \right) e^{i\tau(x_p+t_q^3)}  \cdot e^{i\tau(x-x_p)} e^{i\tau(t^3-t_q^3)} d\tau,
\end{equation}
where\footnote{Here for any interval $I$, $c(I)$ refers to the center of such an interval.} $x_p=c \left(I_p^{j, 3k} \right)$ and $t_q=c \left(I_q^{j, k} \right)$. We now notice that
$$
\left| \tau \left(x-x_p \right) \right| \lesssim  2^{m+3k}\cdot\left|I_p^{j, 3k} \right|=2^{m-\frac{j}{2}}=O(1),
$$
and
$$
\left| \tau \left(t^3-t^3_q \right) \right| \lesssim  2^{m+3k}\cdot\left| I_q^{j, k} \right|  \cdot 2^{-2k} =2^{m-\frac{j}{2}}=O(1).
$$
Therefore, we conclude that $f_3(x+t^3) \approx f_3(x_p+t_q^3)$ for any $x \in I_p^{j, 3k}, t \in I_q^{j, k}$ and consequently, one expects\footnote{For notational convenience we now drop the dependence on $j,m,k$ in the formulas involving $f_i$'s.}
\begin{equation} \label{20230310eq10}
\Lambda_{j, m}^k(\vec{f}) \approx 2^k \sum_{\substack{p \sim 2^{\frac{j}{2}+2k} \\ q \sim 2^{\frac{j}{2}}}} \left(\frac{1}{\left|I_0^{j, 3k} \right|} \int_{I_{p+\frac{q^3}{2^j}}^{j, 3k}} f_3 \right) \int_{I_p^{j, 3k}} \int_{I_q^{j, k}} f_1(x-t)f_2(x+t^2)f_4(x) dtdx.
\end{equation}
Once at this point, as already alluded, we apply the Rank I LGC methodology, as follows:

Thus, at \textsf{Step 1}, we isolate the appropriate physical scale that has the \emph{linearizating} effect on the $t-$parameter:
\begin{itemize}
\item for $f_1(x-t)$ we use a spatial discretization in intervals of length $2^{-\frac{j}{2}-k}$;
\item for $f_2(x+t^2)$ we use a spatial discretization in intervals of length $2^{-\frac{j}{2}-2k}$.
\end{itemize}

Next, at \textsf{Step 2}, we apply an \emph{adapted Gabor decomposition of the input functions} $f_1$ and $f_2$:
\begin{equation} \label{20230313eq01}
f_r=\sum_{{p_r \sim 2^{\frac{j}{2}+(r-1)k}}\atop{ n_r \sim 2^{\frac{j}{2}}}} \left\langle f_r, \phi_{p_r, n_r}^{j, rk} \right\rangle \phi_{p_r, n_r}^{j, rk}, \qquad r\in\{1, 2\},
\end{equation}
where we recall that
$$
\phi_{p_r, n_r}^{j, rk}(x):=\left|I_0^{j, rk} \right|^{-\frac{1}{2}} \phi \left(2^{\frac{j}{2}+rk}x-p_r \right) e^{i2^{\frac{j}{2}+rk}xn_r}, \quad r\in\{1, 2\}.
$$
Therefore, we have
\begin{eqnarray*}
&& \int_{I_p^{j, 3k}} \int_{I_q^{j, k}} f_1(x-t)f_2(x+t^2)f_4(x) dtdx  \\
=&&\sum_{\substack{p_1 \sim 2^{\frac{j}{2}} \\ n_1 \sim 2^{\frac{j}{2}}}} \sum_{\substack{p_2 \sim 2^{\frac{j}{2}+k} \\ n_2 \sim 2^{\frac{j}{2}}}} \left\langle f_1, \phi_{p_1, n_1}^{j, k} \right\rangle
\left\langle f_2, \phi_{p_2, n_2}^{j, 2k} \right\rangle\int_{I_p^{j, 3k}}\int_{I_q^{j, k}} \phi_{p_1, n_1}^{j, k}(x-t) \phi_{p_2, n_2}^{j, 2k}(x+t^2) f_4(x)dtdx.
\end{eqnarray*}
Finally, at \textsf{Step 3}, we extract the \emph{time-frequency correlations}:

Using the definition of the terms $\phi_{p_1, n_1}^{j, k}(x-t)$  and $\phi_{p_2, n_2}^{j, 2k}(x+t^2)$, together with the facts that $x \in I_p^{j, 3k}$ and $t \in I_q^{j, k}$, we deduce the first spatial correlations:
\begin{equation} \label{TFCaseI.a-I}
p_1=\frac{p}{2^{2k}}-q+O(1), \quad \textrm{and} \quad
p_2=\frac{p}{2^k}+\frac{q^2}{2^{\frac{j}{2}}}+O(1).
\end{equation}
Next, making use  of the linearizing effect of the scale chosen in \textsf{Step 1} that reflects into the condition $t \in I_q^{j, k}$, we have that $t^2=\frac{2tq}{2^{\frac{j}{2}+k}}-\frac{q^2}{2^{j+2k}}+O\left(\frac{1}{2^{j+2k}} \right)$ and hence, via \eqref{TFCaseI.a-I}, we deduce
\begin{eqnarray} \label{20230310eq1011}
&\left|\int_{I_p^{j, 3k}}\int_{I_q^{j, k}} \phi_{p_1, n_1}^{j, k}(x-t) \phi_{p_2, n_2}^{j, 2k}(x+t^2) f_4(x)dtdx \right| \approx \nonumber \\
& 2^{\frac{j}{2}+\frac{3k}{2}} \left| \int_{I_p^{j, 3k}} \left( \int_{I_q^{j, k}} e^{i2^{\frac{j}{2}+k}t \left(-n_1+\frac{2q}{2^{\frac{j}{2}}} n_2\right)} dt \right)
e^{i2^{\frac{j}{2}+2k} x \left(\frac{n_1}{2^k}+n_2 \right)} f_4(x) dx \right| \lesssim 2^{-k-\frac{j}{4}} \left| \left \langle f_4, \phi_{p, \frac{n_1}{2^{2k}}+\frac{n_2}{2^k}}^{j, 3k} \right \rangle \right|
\end{eqnarray}
together with the \emph{time-frequency correlation} condition
\begin{equation} \label{TFCaseI.a-II}
n_1=\frac{2q}{2^{\frac{j}{2}}} n_2+O(1).
\end{equation}
Therefore, from \eqref{20230310eq10} and \eqref{20230310eq1011}, one concludes the discretized model
\begin{equation} \label{20230311eq01}
\left|\Lambda_{j, m}^k(\vec{f})\right|\lesssim 2^{-\frac{j}{4}} \sum_{\substack{p \sim 2^{\frac{j}{2}+2k} \\ q, n_2 \sim 2^{\frac{j}{2}}}} \left(\frac{1}{\left|I_0^{j, 3k} \right|} \int_{I_{p+\frac{q^3}{2^j}}^{j, 3k}} f_3 \right) \left| \left \langle f_1, \phi_{\frac{p}{2^{2k}}-q, \frac{2q}{2^{\frac{j}{2}}} n_2}^{j, k} \right \rangle \right|\,\left| \left \langle f_2, \phi_{\frac{p}{2^k}+\frac{q^2}{2^{\frac{j}{2}}}, n_2}^{j, 2k} \right \rangle \right| \left| \left \langle f_4, \phi_{p, \frac{n_2}{2^k}\left(1+\frac{2q}{2^{\frac{j}{2}}} \cdot \frac{1}{2^k} \right)}^{j, 3k} \right \rangle \right|.
\end{equation}
We next decompose both spatial and frequency indices as follows. For spatial indices, we write
$$\left\{p \sim 2^{\frac{j}{2}+2k} \right\}=\bigcup_{P \sim 2^{\frac{j}{2}+k}} J_{P}^k, \quad \textrm{where} \quad J_P^k:=[P2^k, (P+1)2^k]\:\:\:\:\textrm{and},$$
$$\left\{q \sim 2^{\frac{j}{2}} \right\}=\bigcup_{Q \sim 2^{\frac{j}{2}-k}} J_Q^k, \quad \textrm{where} \quad J_Q^k:=[Q2^k, (Q+1)2^k].$$
For frequency indices, we have
$$\supp \widehat{f_1} \subset \left\{|\xi| \sim 2^{j+k} \right\}=\bigcup_{L_1 \sim 2^{\frac{j}{2}-k}} \Omega_1^{L_1}, \quad \textrm{where} \quad \Omega_1^{L_1}:=\left[L_1 2^{\frac{j}{2}+2k},  \left(L_1+1 \right)2^{\frac{j}{2}+2k} \right],$$
$$\supp \widehat{f_2} \subset \left\{|\eta| \sim 2^{j+2k} \right\}=\bigcup_{L_2 \sim 2^{\frac{j}{2}-k}} \Omega_2^{L_2}, \quad \textrm{where} \quad \Omega_2^{L_2}:=\left[L_2 2^{\frac{j}{2}+3k},  \left(L_2+1 \right)2^{\frac{j}{2}+3k} \right]\quad\textrm{and}$$
$$\supp \widehat{f_4} \subset \left\{|\eta| \sim 2^{j+2k} \right\}=\bigcup_{L_4 \sim 2^{\frac{j}{2}-k}} \Omega_4^{L_4}, \quad \textrm{where} \quad \Omega_4^{L_4}:=\left[L_4 2^{\frac{j}{2}+3k},  \left(L_4+1 \right)2^{\frac{j}{2}+3k} \right].$$
Note that for $n_2 2^{\frac{j}{2}+2k} \in \Omega_2^{L_2}$ we have $n_2 \in \widetilde{\Omega_2^{L_2}}:=\left[L_2 2^k, (L_2+1)2^k \right]$. Substituting all these decomposition back to \eqref{20230311eq01}, we see that
\begin{equation}
\left|\Lambda_{j, m}^k(\vec{f})\right|  \lesssim \sum_{\substack{P \sim 2^{\frac{j}{2}+k} \\ Q, L_2 \sim 2^{\frac{j}{2}-k}}} \sum_{\substack{p \in J_P^k \\ q \in J_Q^k \\ n_2 \in \widetilde{\Omega_2^{L_2}}}} \left(\frac{2^{-\frac{j}{4}}}{\left|I_0^{j, 3k} \right|} \int_{I_{p+\frac{q^3}{2^j}}^{j, 3k}} f_3 \right) \left| \left \langle f_4, \phi^{j, 3k}_{p, L_2 \left(1+\frac{2Q}{2^{\frac{j}{2}}} \right)} \right \rangle \right|\left| \left \langle f_1, \phi_{\frac{p}{2^{2k}}-q, \frac{2q}{2^{\frac{j}{2}}} n_2}^{j, k} \right \rangle \right| \left| \left \langle f_2, \phi_{\frac{p}{2^k}+\frac{q^2}{2^{\frac{j}{2}}}, n_2}^{j, 2k} \right \rangle \right|.
\end{equation}
Applying Cauchy-Schwarz first in $n_2$ and then in $p, q$, we further have
\begin{eqnarray*}
&&\left|\Lambda_{j, m}^k(\vec{f})\right|  \lesssim 2^{-\frac{j}{4}} \sum_{\substack{P \sim 2^{\frac{j}{2}+k} \\ Q, L_2 \sim 2^{\frac{j}{2}-k}}} \sum_{\substack{p \in J_P^k \\ q \in J_Q^k }} \left(\frac{1}{\left|I_0^{j, 3k} \right|} \int_{I_{p+\frac{q^3}{2^j}}^{j, 3k}} f_3 \right) \left| \left \langle f_4, \phi^{j, 3k}_{p, L_2 \left(1+\frac{2Q}{2^{\frac{j}{2}}} \right)} \right \rangle \right| \\
&& \quad \quad \quad \quad \quad \quad \quad \quad \cdot \left\|f_1^{\Omega_1^{\frac{2QL_2}{2^{\frac{j}{2}-k}}}} \right\|_{L^2 \left(I_{\frac{p}{2^{2k}}-q}^{j, k} \right)} \left\| f_2^{\Omega_2^{L_2}} \right\|_{L^2 \left( I_{\frac{p}{2^k}+\frac{q^2}{2^{\frac{j}{2}}} }^{j, 2k} \right)} \\
&& \quad \lesssim 2^{-\frac{j}{4}} \sum_{\substack{P \sim 2^{\frac{j}{2}+k} \\ Q, L_2 \sim 2^{\frac{j}{2}-k}}}  \left( \sum_{\substack{p \in J_P^k \\ q \in J_Q^k }} \left(\frac{1}{\left|I_0^{j, 3k} \right|} \int_{I_{p+\frac{q^3}{2^j}}^{j, 3k}} f_3 \right)^2 \left\|f_1^{\Omega_1^{\frac{2QL_2}{2^{\frac{j}{2}-k}}}} \right\|^2_{L^2 \left(I_{\frac{p}{2^{2k}}-q}^{j, k} \right)} \right)^{\frac{1}{2}} \\
&& \quad \quad \quad \quad \quad \quad \quad \quad\cdot \left( \sum_{\substack{p \in J_P^k \\ q \in J_Q^k }} \left\| f_4^{\Omega_4^{L_2 \left(1+\frac{2Q}{2^{\frac{j}{2}}} \right)}} \right\|^2_{L^2 \left(I_p^{j, 3k} \right)} \left\| f_2^{\Omega_2^{L_2}} \right\|^2_{L^2 \left( I_{\frac{p}{2^k}+\frac{q^2}{2^{\frac{j}{2}}} }^{j, 2k} \right)} \right)^{\frac{1}{2}}.
\end{eqnarray*}
By Parseval, we deduce that
\begin{equation} \label{20230311eq20}
\left|\Lambda_{j, m}^k(\vec{f})\right| \lesssim 2^{\frac{3k}{2}} \sum_{\substack{P \sim 2^{\frac{j}{2}+k} \\ Q, L_2 \sim 2^{\frac{j}{2}-k}}} \left\|f_3 \right\|_{L^2 \left(I^{j, 2k}_{P+\frac{Q^3}{2^{j-2k}}} \right)}  \left\|f_1^{\Omega_1^{\frac{2QL_2}{2^{\frac{j}{2}-k}}}} \right\|_{L^2 \left(I_{\frac{P}{2^{2k}}-Q}^{j, 0} \right)}\left\| f_4^{\Omega_4^{L_2 \left(1+\frac{2Q}{2^{\frac{j}{2}}} \right)}} \right\|_{L^2 \left(I_P^{j, 2k} \right)} \left\| f_2^{\Omega_2^{L_2}} \right\|_{L^2 \left( I_{\frac{P}{2^k}+\frac{Q^2}{2^{\frac{j}{2}-k}} }^{j, k}  \right)}.
\end{equation}
Define now the \emph{sparse set} of indices for {\bf Case I.1} as follows\footnote{Recall that $\underline{i}:=\min\{i, 3\}$.}: for $i\in\{1, 2, 4\}$,
$$\sfS_\mu(f_i):= \left\{ \substack{(Q_i, L_i)\\Q_i \sim 2^{\frac{j}{2}+(\underline{i}-2)k}\\ L_i \sim 2^{\frac{j}{2}-k}} : \int_{I_{Q_i}^{j, (\underline{i}-1)k}} \left| f_i^{\Omega_i^{L_i}} \right|^2 \gtrsim 2^{-\mu \left(\frac{j}{2}-k \right)} \int_{I_{Q_i}^{j, (\underline{i}-1)k}} |f_i|^2\right\}\,.$$
Also, for $i\in\{1, 2, 4\}$,  define in the usual way the \emph{uniform set} of indices $\sfU_\mu(f_i)$ as well as the sparse/uniform components $f_i^{\sfS_\mu}$ and $f_i^{\sfU_\mu}$. With these done, we let
$$\Lambda_{j, m}^k=\Lambda_{j, m}^{k, \sfS_\mu}+\Lambda_{j, m}^{k, \sfU_\mu}$$
where $\Lambda_{j, m}^{k, \sfS_\mu}:=\Lambda_{j, m}^k \left(f_1^{\sfS_\mu}, f_2^{\sfS_\mu}, f_3, f_4^{\sfS_\mu} \right)$. As expected, we divide our proof into two cases:

\subsubsection{Treatment of $\Lambda_{j, m}^{k, \sfU_\mu}$} In this case, we have at least one of $f_1, f_2, f_4$ is uniform; without loss of generality, we may assume $f_1=f_1^{\sfU_\mu}$. Then by \eqref{20230311eq20} and applying Cauchy-Schwarz first in $L_2$, and then in $P, Q$, one has
\begin{eqnarray} \label{20230312eq01}
\left|\Lambda_{j, m}^k(\vec{f})\right|  \lesssim 2^{\frac{3k}{2}} \cdot 2^{-\mu \left(\frac{j}{2}-k \right)} \sum\limits_{\substack{P \sim 2^{\frac{j}{2}+k} \\ Q \sim 2^{\frac{j}{2}-k}}} \left\|f_3 \right\|_{L^2 \left(I^{j, 2k}_{P+\frac{Q^3}{2^{j-2k}}} \right)}  \left\|f_1\right\|_{L^2 \left(I_{\frac{P}{2^{2k}}-Q}^{j, 0} \right)} \left\| f_4 \right\|_{L^2 \left(I_P^{j, 2k} \right)} \left\| f_2 \right\|_{L^2 \left( I_{\frac{P}{2^k}+\frac{Q^2}{2^{\frac{j}{2}-k}} }^{j, k}  \right)} \qquad\nonumber  \\
\lesssim  2^{\frac{3k}{2}} \cdot 2^{-\mu \left(\frac{j}{2}-k \right)}  \left( \sum\limits_{\substack{P \sim 2^{\frac{j}{2}+k} \\ Q \sim 2^{\frac{j}{2}-k}}}  \left\|f_1\right\|^2_{L^2 \left(I_{\frac{P}{2^{2k}}-Q}^{j, 0} \right)} \left\| f_4 \right\|^2_{L^2 \left(I_P^{j, 2k} \right)} \right)^{\frac{1}{2}} \left( \sum\limits_{\substack{P \sim 2^{\frac{j}{2}+k} \\ Q \sim 2^{\frac{j}{2}-k}}}  \left\| f_2 \right\|^2_{L^2 \left( I_{\frac{P}{2^k}+\frac{Q^2}{2^{\frac{j}{2}-k}} }^{j, k}  \right)} \left\|f_3 \right\|^2_{L^2 \left(I^{j, 2k}_{P+\frac{Q^3}{2^{j-2k}}} \right)} \right)^{\frac{1}{2}}.
\end{eqnarray}
For the first double summation in \eqref{20230312eq01}, via Parseval, we get
$$
\sum_{\substack{P \sim 2^{\frac{j}{2}+k} \\ Q \sim 2^{\frac{j}{2}-k}}}  \left\|f_1\right\|^2_{L^2 \left(I_{\frac{P}{2^{2k}}-Q}^{j, 0} \right)} \left\| f_4 \right\|^2_{L^2 \left(I_P^{j, 2k} \right)}  \lesssim \left\|f_1 \right\|_{L^2(3I^k)}^2 \left\|f_4 \right\|_{L^2(3I^k)}^2.
$$
while for the second double summation in \eqref{20230312eq01}, via successive Cauchy--Schwarz applications, we get
$$\sum_{P_2 \sim 2^k}\sum_{P_1=P_2 2^{\frac{j}{2}-k}}^{(P_2+1) 2^{\frac{j}{2}-k}}\sum_{Q \sim 2^{\frac{j}{2}-k}} \sum_{P=P_12^k}^{(P_1+1)2^k} \left\| f_2 \right\|^2_{L^2 \left( I_{\frac{P}{2^k}+\frac{Q^2}{2^{\frac{j}{2}-k}} }^{j, k}  \right)} \left\|f_3 \right\|^2_{L^2 \left(I^{j, 2k}_{P+\frac{Q^3}{2^{j-2k}}} \right)}\lesssim \sum_{P_2 \sim 2^k} \left\|f_2 \right\|^2_{L^2 \left(3I_{P_2}^{0, 2k} \right)} \left\|f_3 \right\|_{L^2 \left(3 I_{P_2}^{0, 2k} \right)}^2.$$

Substituting these estimates back to \eqref{20230312eq01}, we conclude that
\begin{equation}\label{20230312eq02}
\left|\Lambda_{j, m}^k(\vec{f}) \right| \lesssim 2^{\frac{3k}{2}} \cdot 2^{-\mu \left(\frac{j}{2}-k \right)} \left\|f_1 \right\|_{L^2(3I^k)} \left\|f_4\right\|_{L^2(3I^k)}\left( \sum_{P_2 \sim 2^k} \left\|f_2 \right\|^2_{L^2 \left(3I_{P_2}^{0, 2k} \right)} \left\|f_3 \right\|_{L^2 \left(3 I_{P_2}^{0, 2k} \right)}^2 \right)^{\frac{1}{2}}.
\end{equation}

\subsubsection{Treatment of $\Lambda_{j, m}^{k, \sfS_\mu}$} Recall that in this case, we have $f_i=f_i^{\sfS_\mu}$ for $i\in\{1, 2, 4\}$. Therefore, up to a factor of $2^{\mu\left(\frac{j}{2}-k \right)}$, there exists a unique measurable function $L_i: \left[2^{\frac{j}{2}+(\underline{i}-2)k}, 2^{\frac{j}{2}+(\underline{i}-2)k+1} \right]\cap\Z \rightarrow \left[2^{\frac{j}{2}-k}, 2^{\frac{j}{2}-k+1} \right]\cap \Z$, such that
$$
2^{-\mu\left(\frac{j}{2}-k \right)}\, \int_{I_{Q_i}^{j, (\underline{i}-1)k}} |f_i|^2\lesssim \int_{I_{Q_i}^{j, (\underline{i}-1)k}} \left| f_1^{\Omega_i^{L_i}} \right|^2 \lesssim \int_{I_{Q_i}^{j, (\underline{i}-1)k}} |f_i|^2.
$$
Combining this with \eqref{20230311eq20}, we see that
\begin{eqnarray} \label{20230312eq23}
\left|\Lambda^{k, \sfS_\mu}_{j, m}(\vec{f}) \right| \lesssim 2^{\frac{3k}{2}} \cdot 2^{3\mu \left(\frac{j}{2}-k \right)} \cdot \sum\limits_{(P, Q) \in \cTFC}  \left\|f_3 \right\|_{L^2 \left(I^{j, 2k}_{P+\frac{Q^3}{2^{j-2k}}} \right)}  \left\|f_1\right\|_{L^2 \left(I_{\frac{P}{2^{2k}}-Q}^{j, 0} \right)}\,\left\| f_4 \right\|_{L^2 \left(I_P^{j, 2k} \right)} \left\| f_2\right\|_{L^2 \left( I_{\frac{P}{2^k}+\frac{Q^2}{2^{\frac{j}{2}-k}} }^{j, k}  \right)} \nonumber \\
\lesssim 2^{\frac{3k}{2}} \cdot 2^{3\mu \left(\frac{j}{2}-k \right)} \cdot \left(\# \cTFC \right)^{\frac{1}{2}} \vast ( \sum\limits_{\substack{P \sim 2^{\frac{j}{2}+k} \\ Q \sim 2^{\frac{j}{2}-k}}} \left\|f_1\right\|^2_{L^2 \left(I_{\frac{P}{2^{2k}}-Q}^{j, 0} \right)} \left\| f_2\right\|^2_{L^2 \left( I_{\frac{P}{2^k}+\frac{Q^2}{2^{\frac{j}{2}-k}} }^{j, k}  \right)} \left\|f_3 \right\|^2_{L^2 \left(I^{j, 2k}_{P+\frac{Q^3}{2^{j-2k}}} \right)}\left\| f_4 \right\|^2_{L^2 \left(I_P^{j, 2k} \right)}  \vast)^{\frac{1}{2}}\nonumber
\end{eqnarray}
where
$$\cTFC:=\left\{ \substack{(P, Q)\\P \sim 2^{\frac{j}{2}+k}, \ Q \sim 2^{\frac{j}{2}-k}\\\exists\:L_2^*\in\N}\,:\, \frac{2QL_2^{*}}{2^{\frac{j}{2}-k}}=L_1 \left(\frac{P}{2^{2k}}-Q \right), \  L_2^*=L_2 \left(\frac{P}{2^k}+\frac{Q^2}{2^{\frac{j}{2}-k}} \right), \ L_2^* \cdot \left(1+\frac{2Q}{2^{{\frac{j}{2}}}} \right)=L_4(P) \right\}.$$
The desired conclusion follows now from the claim that there exists a suitable $\epsilon>0$, such that
\begin{equation} \label{20230312eq24}
\# \cTFC \lesssim 2^j \cdot 2^{-\epsilon \left(\frac{j}{2}-k \right)}.
\end{equation}
The proof of this claim follows now similar reasoning with those presented in the previous sections and thus we leave further details to the interested reader.

In summary, in this case, we proved there exists some $\epsilon>0$, such that
$$
\left|\Lambda_{j, m}^k (\vec{f}) \right| \lesssim 2^{-\epsilon \left(\frac{j}{2}-k \right)} \left\| \vec{f} \right\|_{L^{\vec{p}} \left(3I_0^{0, k}\right)}, \quad \forall \: \vec{p} \in {\bf H}^{+}.
$$

\subsection{Treatment of Case I.2:  $0 \le m \le \frac{j}{2}$ and $\frac{j-m}{2}< k < j-m$.}

In this situation, we combine two features:
\begin{itemize}
\item Rank I LGC methodology adapted to the input functions $f_{1,j+k}$ and $f_{2,j+2k}$, and
\item decomposition of the third input $f_{3,m+3k}$ into Gabor frames adapted to the spatial scale inherited from the second input $f_{2,j+2k}$.
\end{itemize}
    More precisely, at \textsf{Step 1}, we divide the $(x,t)$ domain of integration as
$$
\Lambda_{j, m}^k(\vec{f})\approx 2^k \sum_{\substack{p \sim 2^{\frac{j}{2}+k} \\ q \sim 2^{\frac{j}{2}}}} \int_{I_p^{j, 2k}} \int_{I_q^{j, k}} f_{1,j+k}(x-t)f_{2,j+2k}(x+t^2)f_{3,m+3k}(x+t^3)f_4(x) dtdx\,,
$$
obtaining thus the desired \emph{linearizating} effect on the $t-$parameter.

At \textsf{Step 2}, we apply a \emph{Gabor wave-packet decomposition} as follows: for $f_1$ and $f_2$ we proceed as in \eqref{20230313eq01} while for $f_3$ we use wave packets adapted to $|I_p^{j, 2k}|$, that is,
$$f_3=\sum_{p_3 \sim 2^{\frac{j}{2}+k}, \ n_3 \sim 2^{m-\frac{j}{2}+k}} \left\langle f_3, \phi_{p_3, n_3}^{j, 2k} \right\rangle \phi_{p_3, n_3}^{j, 2k}.$$
Therefore, we obtain
\begin{eqnarray} \label{20230313eq05}
\Lambda_{j, m}^k(\vec{f}) &\approx& 2^k \sum_{\substack{p \sim 2^{\frac{j}{2}+k} \\ q \sim 2^{\frac{j}{2}}}} \sum_{\substack{p_1 \sim 2^{\frac{j}{2}} \\ n_1 \sim 2^{\frac{j}{2}}}} \sum_{\substack{p_2 \sim 2^{\frac{j}{2}+k} \\ n_2 \sim 2^{\frac{j}{2}}}} \sum_{\substack{p_3 \sim 2^{\frac{j}{2}+k} \\ n_3 \sim 2^{m-\frac{j}{2}+k}}}\left\langle f_1, \phi_{p_1, n_1}^{j, k} \right\rangle
\left\langle f_2, \phi_{p_2, n_2}^{j, 2k} \right\rangle \\
&\cdot&
\left\langle f_3, \phi_{p_3, n_3}^{j, 2k} \right\rangle \int_{I_p^{j, 2k}} \int_{I_q^{j, k}} \phi_{p_1, n_1}^{j, k}(x-t) \phi_{p_2, n_2}^{j, 2k}(x+t^2) \phi_{p_3, n_3}^{j, 2k}(x+t^3) f_4(x)dtdx. \nonumber
\end{eqnarray}
At \textsf{Step 3}, we derive the \emph{time-frequency correlations} as follows:
\begin{itemize}
\item Arguing as before and using the double integral in the above expression together with the definitions of $\phi_{p_1, n_1}^{j, k}$, $\phi_{p_2, n_2}^{j, 2k}$ and $\phi_{p_3, n_3}^{j, 2k}$, we deduce the {\bf Case I.2}--spatial correlations:
\begin{equation} \label{20230313eq02}
p_1=\frac{p}{2^k}-q+O(1), \quad  p_2=p+\frac{q^2}{2^{\frac{j}{2}}}+O(1), \quad \textrm{and} \quad p_3=p+\frac{q^3}{2^{j+k}}+O(1).
\end{equation}
\item Next, exploiting the localization $t \in I_q^{j, k}$, one has
$t^2=\frac{2tq}{2^{\frac{j}{2}+k}}-\frac{q^2}{2^{j+2k}}+O\left(\frac{1}{2^{j+2k}} \right)$ and
$t^3=\frac{3q^2t}{2^{j+2k}}-\frac{2q^3}{2^{\frac{3j}{2}+3k}}+O \left(\frac{1}{2^{j+3k}} \right)$ that combined with \eqref{20230313eq02} further gives
\begin{eqnarray*}
&&\left| \int_{I_p^{j, 2k}} \int_{I_q^{j, k}} \phi_{p_1, n_1}^{j, k}(x-t) \phi_{p_2, n_2}^{j, 2k}(x+t^2) \phi_{p_3, n_3}^{j, 2k}(x+t^3) f_4(x)dtdx  \right| \\
&\approx&2^{\frac{3j}{4}+\frac{5k}{2}} \bigg| \int_{I_p^{j, 2k}} \left(\int_{I_q^{j, k}} e^{i2^{\frac{j}{2}+k}(-n_1t+2^kt^2 n_2+2^kt^2n_2+2^kt^3n_3)} dt \right)\,e^{i2^{\frac{j}{2}+2k}x \left(\frac{n_1}{2^k}+n_2+n_3 \right)} f_4(x) dx \bigg|\\
&\approx& 2^{\frac{3j}{4}+\frac{5k}{2}} \left| \int_{I_p^{j, 2k}} \left(\int_{I_q^{j, k}} e^{i2^{\frac{j}{2}+k}t\left(-n_1+\frac{2q}{2^{\frac{j}{2}}}n_2+\frac{3q^2}{2^{j+k}}n_3 \right)} dt \right) e^{i2^{\frac{j}{2}+2k}x \left(\frac{n_1}{2^k}+n_2+n_3 \right)} f_4(x) dx \right|.
\end{eqnarray*}
Now, since $\left|\frac{3q^2}{2^{j+k}}n_3 \right| \sim  2^{m-\frac{j}{2}}=O(1)$ we conclude from the above that
\begin{equation}\label{20230313eq04}
\left|\int_{I_p^{j, 2k}} \int_{I_q^{j, k}} \phi_{p_1, n_1}^{j, k}(x-t) \phi_{p_2, n_2}^{j, 2k}(x+t^2) \phi_{p_3, n_3}^{j, 2k}(x+t^3) f_4(x)dtdx \right|\lesssim 2^{\frac{k}{2}} \left| \left \langle f_4, \phi_{p, n_2 \left(1+\frac{2q}{2^{\frac{j}{2}+k}} \right)+n_3}^{j, 2k} \right \rangle \right|\,,
\end{equation}
together with a second time-frequency correlation for {\bf Case I.2} given by
\begin{equation} \label{20230313eq03}
n_1=\frac{2q}{2^{\frac{j}{2}}}n_2+O(1)\,.
\end{equation}
\end{itemize}

Substituting \eqref{20230313eq02}, \eqref{20230313eq03}, and \eqref{20230313eq04} back to \eqref{20230313eq05}, we obtain the discretized model
\begin{equation} \label{20230313eq10}
\left|\Lambda_{j, m}^k(\vec{f}) \right| \lesssim 2^{\frac{3k}{2}}
 \sum_{\substack{p \sim 2^{\frac{j}{2}+k} \\ q \sim 2^{\frac{j}{2}}}} \sum_{\substack{n_2 \sim 2^{\frac{j}{2}} \\ n_3 \sim 2^{m-\frac{j}{2}+k}}} \left| \left \langle f_1, \phi_{\frac{p}{2^k}-q, \frac{2qn_2}{2^{{\frac{j}{2}}}}}^{j, k} \right \rangle \right| \left | \left \langle f_2, \phi_{p+\frac{q^2}{2^{\frac{j}{2}}}, n_2}^{j, 2k} \right \rangle \right |  \left| \left \langle f_3, \phi_{p+\frac{q^3}{2^{j+k}}, n_3}^{j, 2k} \right \rangle \right| \left| \left \langle f_4, \phi_{p, n_2 \left(1+\frac{2q}{2^{\frac{j}{2}+k}} \right)+n_3}^{j, 2k} \right \rangle \right|.
\end{equation}
Motivated by the fact that $n_3 \sim 2^{m-\frac{j}{2}+k}$, we split both spatial and frequency indices as follows:
\begin{itemize}
\item For spatial indices, we have
$$\left\{p \sim 2^{\frac{j}{2}+k}  \right\}=\bigcup_{P \sim 2^{j-m}} J_P^k, \quad \textrm{where} \quad J_P^k:=\left[P2^{m-\frac{j}{2}+k}, (P+1)2^{m-\frac{j}{2}+k} \right]\qquad\textrm{and}$$
$$\left\{q \sim 2^{\frac{j}{2}}  \right\}=\bigcup_{Q \sim 2^{j-m-k}} J_Q^k, \quad \textrm{where} \quad J_Q^k:=\left[Q2^{m-\frac{j}{2}+k}, (Q+1)2^{m-\frac{j}{2}+k} \right].$$
\item For frequency indices, we have
$$\supp \widehat{f_1} \subseteq \left\{|\xi| \sim 2^{j+k} \right\}=\bigcup_{L_1 \sim 2^{j-m-k}} \Omega_1^{L_1}, \quad \textrm{where} \quad \Omega_1^{L_1}:= \left[L_1 2^{m+2k}, \left(L_1+1 \right) 2^{m+2k} \right],$$
$$\supp \widehat{f_2} \subseteq \left\{|\eta| \sim 2^{j+2k} \right\}=\bigcup_{L_2 \sim 2^{j-m-k}} \Omega_2^{L_2}, \quad \textrm{where} \quad \Omega_2^{L_2}:= \left[L_2 2^{m+3k}, \left(L_2+1 \right) 2^{m+3k} \right] \quad \textrm{and}$$
$$\supp \widehat{f_4} \subseteq \left\{|\eta| \sim 2^{j+2k} \right\}=\bigcup_{L_4 \sim 2^{j-m-k}} \Omega_4^{L_4}, \quad \textrm{where} \quad \Omega_4^{L_4}:= \left[L_4 2^{m+3k}, \left(L_4+1 \right) 2^{m+3k} \right].$$
\end{itemize}

Note that if $n_2 2^{\frac{j}{2}+2k} \in \Omega_2^{L_2}$, then we have $n_2 \in \widetilde{\Omega_2^{L_2}}:=\left[L_2 2^{m-\frac{j}{2}+k}, \left(L_2+1 \right) 2^{m-\frac{j}{2}+k} \right]$. Performing the above decomposition in \eqref{20230313eq10}, we have
\begin{eqnarray*}
&& \left| \Lambda_{j, m}^k(\vec{f})\right| \lesssim 2^{\frac{3k}{2}} \sum_{\substack{P \sim 2^{j-m} \\ Q, L_2 \sim 2^{j-m-k}}} \sum_{\substack{p \in J_P^k, \ q \in J_Q^k \\ n_2 \in \widetilde{\Omega_2^{L_2}},  \ n_3 \sim 2^{m-\frac{j}{2}+k}}} \left| \left \langle f_1, \phi_{\frac{p}{2^k}-q, \frac{2qn_2}{2^{{\frac{j}{2}}}}}^{j, k} \right \rangle \right| \left | \left \langle f_2, \phi_{p+\frac{q^2}{2^{\frac{j}{2}}}, n_2}^{j, 2k} \right \rangle \right | \\
 && \quad \quad \quad \quad \quad \cdot \left| \left \langle f_3, \phi_{p+\frac{q^3}{2^{j+k}}, n_3}^{j, 2k} \right \rangle \right| \left| \left \langle f_4, \phi_{p, n_2 \left(1+\frac{2q}{2^{\frac{j}{2}+k}} \right)+n_3}^{j, 2k} \right \rangle \right|.
\end{eqnarray*}
Applying Cauchy-Schwarz in $p, q, n_2$ and $n_3$, we have
$$
 \left| \Lambda_{j, m}^k(\vec{f})\right| \lesssim 2^{\frac{3k}{2}} \sum_{\substack{P \sim 2^{j-m} \\ Q, L_2 \sim 2^{j-m-k}}} \calA(P, Q, L_2) \:\calB(P, Q, L_2),
$$
where
$$
\calA^2(P, Q, L_2):=\sum_{\substack{p \in J_P^k, q \in J_Q^k \\ n_2 \in \widetilde{\Omega_2^{L_2}},  n_3 \sim 2^{m-\frac{j}{2}+k}}} \left| \left \langle f_1, \phi_{\frac{p}{2^k}-q, \frac{2qn_2}{2^{{\frac{j}{2}}}}}^{j, k} \right \rangle \right|^2  \left| \left \langle f_3, \phi_{p+\frac{q^3}{2^{j+k}}, n_3}^{j, 2k} \right \rangle \right|^2
$$
and
$$
\calB^2(P, Q, L_2):=\sum_{\substack{p \in J_P^k, \ q \in J_Q^k \\ n_2 \in \widetilde{\Omega_2^{L_2}},  \ n_3 \sim 2^{m-\frac{j}{2}+k}}} \left | \left \langle f_2, \phi_{p+\frac{q^2}{2^{\frac{j}{2}}}, n_2}^{j, 2k} \right \rangle \right |^2 \left| \left \langle f_4, \phi_{p, n_2 \left(1+\frac{2q}{2^{\frac{j}{2}+k}} \right)+n_3}^{j, 2k} \right \rangle \right|^2.
$$
For the term $\calA(P, Q, L_2)$, we first apply Parseval on the variables $n_2$ and $n_3$, which gives
$$
\calA^2(P, Q, L_2) \lesssim  \sum_{p \in J_P^k, q \in J_Q^k} \left\|f_1^{\Omega_1^{\frac{2QL_2}{2^{j-m-k}}}} \right\|^2_{L^2 \left(I_{\frac{p}{2^k}-q}^{j, k} \right)}  \left\|f_3 \right\|^2_{L^2 \left(I_{p+\frac{q^3}{2^{j+k}}}^{j, 2k} \right)}.
$$
Note that when $q\in J_Q^k$, the variation of the $\frac{q^3}{2^{j+k}}$ is at most $\lesssim \left| \frac{2^j \cdot 2^{m-\frac{j}{2}+k}}{2^{j+k}} \right| =2^{m-\frac{j}{2}} \le 1$ and, similarly, when $p\in J_P^k$  the variation of the term $\frac{p}{2^{m-\frac{j}{2}+2k}}$ has  magnitude $\sim \frac{1}{2^k} \le 1$; thus, applying Parseval in $q$ we get
$$\calA^2(P, Q, L_2) \lesssim  \sum_{p \in J_P^k} \left\|f_1^{\Omega_1^{\frac{2QL_2}{2^{j-m-k}}}} \right\|^2_{L^2 \left(I_{\frac{P}{2^k}-Q}^{2(j-m), 0} \right)}  \left\|f_3 \right\|^2_{L^2 \left(I_{p+\frac{Q^3}{2^{\frac{5j}{2}-3m-2k}}}^{j, 2k} \right)}.$$
Using now the almost disjointness in $p$ and the fact that $\left|\frac{Q^3}{2^{2(j-m+k)+k}} \right| \sim \frac{2^{3(j-m-k)}}{2^{2(j-m-k)+k}}=2^{j-m-2k} \le 1$, we conclude
$$
\calA(P, Q, L_2) \lesssim  \left\|f_1^{\Omega_1^{\frac{2QL_2}{2^{j-m-k}}}} \right\|_{L^2 \left(I_{\frac{P}{2^k}-Q}^{2(j-m), 0} \right)}  \left\|f_3 \right\|_{L^2 \left(I_p^{2(j-m), k} \right)}.
$$
Next, for the term $\calB(P, Q, L_2)$, by a similar argument as above, we have
$$
\calB(P, Q, L_2) \lesssim \left\|f_2^{\Omega_2^{L_2}} \right\|_{L^2 \left(I^{2(j-m), k}_{P+\frac{Q^2}{2^{j-m-k}}} \right)}\left\|f_4^{\Omega_4^{L_2}} \right\|_{L^2 \left(I_P^{2(j-m), k} \right)}.
$$
Therefore, we have
\begin{equation} \label{20230314eq01}
\left|\Lambda_{j, m}^k(\vec{f})\right| \lesssim 2^{\frac{3k}{2}} \sum_{\substack{P \sim 2^{j-m} \\ Q, L_2 \sim 2^{j-m-k}}} \left\|f_1^{\Omega_1^{\frac{2QL_2}{2^{j-m-k}}}} \right\|_{L^2 \left(I_{\frac{P}{2^k}-Q}^{2(j-m), 0} \right)}  \left\|f_3 \right\|_{L^2 \left(I_p^{2(j-m), k} \right)} \left\|f_2^{\Omega_2^{L_2}} \right\|_{L^2 \left(I^{2(j-m), k}_{P+\frac{Q^2}{2^{j-m-k}}} \right)}\left\|f_4^{\Omega_4^{L_2}} \right\|_{L^2 \left(I_P^{2(j-m), k} \right)}.
\end{equation}
Motivated by the above expression, we define the \emph{sparse set} of indices in {\bf Case I.2} as follows: for $i\in\{1, 2, 4\}$,
\begin{equation}
\sfS_\mu(f_i):= \bigg\{ (Q_i, L_i): \int_{I_{Q_i}^{2(j-m), (\Breve{i}-1)k}} \left| f_1^{\Omega_i^{L_i}} \right|^2\gtrsim 2^{-\mu \left(j-m-k \right)} \int_{I_{Q_i}^{2(j-m), (\Breve{i}-1)k}} |f_i|^2,  \ \substack{Q_i \sim 2^{j-m+(\Breve{i}-2)k} \\ L_i \sim 2^{j-m-k}} \bigg\},
\end{equation}
where $\Breve{i}:=\min\{i, 2\}$. We also define in the obvious fashion the \emph{uniform set} of indices $\sfU_\mu(f_i)$ for $i\in\{1, 2, 4\}$, as well as decompose the functions $f_1, f_2$ and $f_4$ into their sparse/uniform components--denoted by $f_i^{\sfS_\mu}$ and $f_i^{\sfU_\mu}$, $i\in\{1, 2, 4\}$, respectively. With these, we have
$$
\Lambda_{j, m}^k=:\Lambda_{j, m}^{k, \sfS_\mu}+\Lambda_{j, m}^{k, \sfU_\mu},
$$
with $\Lambda_{j, m}^{k, \sfS_\mu}(\vec{f}):=\Lambda_{j, m}^k \left(f_1^{\sfS_\mu}, f_2^{\sfS_\mu}, f_3, f_4^{\sfS_\mu} \right)$,
and, in the usual fashion, divide our proof in two cases:

\subsubsection{Treatment of $\Lambda_{j, m}^{k, \sfU_\mu}$} In this case, we have that at least one of $\{f_1, f_2, f_4\}$ is uniform, and thus, without loss of generality, we may assume $f_4=f_4^{\sfU_\mu}$. Then, by \eqref{20230314eq01} and Cauchy-Schwarz in $L_2$,  one has
\begin{equation}
\left|\Lambda_{j, m}^k(\vec{f}) \right| \lesssim  2^{\frac{3k}{2}} \cdot 2^{-\mu \left(j-m-k \right)} \sum_{\substack{P \sim 2^{j-m} \\ Q \sim 2^{j-m-k}}} \left\|f_3 \right\|_{L^2 \left(I_P^{2(j-m), k} \right)} \left\|f_4 \right\|_{L^2 \left(I_P^{2(j-m), k} \right)}  \left\|f_1 \right\|_{L^2 \left(I_{\frac{P}{2^k}-Q}^{2(j-m), 0} \right)}   \left\|f_2\right\|_{L^2 \left(I^{2(j-m), k}_{P+\frac{Q^2}{2^{j-m-k}}} \right)}.
\end{equation}
Next, by applying Cauchy-Schwarz in $Q$, we conclude
\begin{equation} \label{20230314eq30}
|\Lambda_{j, m}^k(\vec{f})| \lesssim 2^{\frac{3k}{2}} \cdot 2^{-\mu (j-m-k)} \left\|f_1 \right\|_{L^2(I^k)} \sum_{\widetilde{P} \sim 2^k} \left\|f_2 \right\|_{L^2 \left(I_{\widetilde{P}}^{0, 2k} \right)} \left\|f_3\right\|_{L^2 \left(I_{\widetilde{P}}^{0, 2k} \right)} \left\|f_4 \right\|_{L^2 \left(I_{\widetilde{P}}^{0, 2k} \right)}.
\end{equation}

\subsubsection{Treatment of $\Lambda_{j, m}^{k, \sfS_\mu}$} In this case, all the functions $\{f_1, f_2, f_4\}$ are sparse. Therefore, up to a factor $2^{\mu\left(j-m-k \right)}$, and for $i \in \{1, 2, 4\}$, there exists a unique measurable function $L_i: \left[2^{j-m+\left(\Breve{i}-2 \right)k}, 2^{j-m+\left(\Breve{i}-2 \right)k+1}\right] \to \left[2^{j-m-k}, 2^{j-m-k+1} \right]$, such that
$$
\int_{I_{Q_i}^{2(j-m), (\Breve{i}-1)k}} \left| f_1^{\Omega_i^{L_i}} \right|^2 \simeq  \int_{I_{Q_i}^{2(j-m), (\Breve{i}-1)k}} |f_i|^2.
$$
Therefore, arguing as usual, we have
\begin{eqnarray*}
\left| \Lambda_{j, m}^{k, \sfS_\mu}(\vec{f}) \right| &\lesssim& 2^{\frac{3k}{2}} \cdot 2^{3\mu(j-m-k)} \cdot \left(\# \cTFC \right)^{\frac{1}{2}} \\
&\cdot&  \vast(\sum_{\substack{P \sim 2^{j-m} \\ Q \sim 2^{j-m-k}}} \left\|f_1 \right\|^2_{L^2 \left(I_{\frac{P}{2^k}-Q}^{2(j-m), 0} \right)}  \left\|f_2\right\|^2_{L^2 \left(I^{2(j-m), k}_{P+\frac{Q^2}{2^{j-m-k}}} \right)}
 \left\|f_3 \right\|^2_{L^2 \left(I_p^{2(j-m), k} \right)}
 \left\|f_4 \right\|^2_{L^2 \left(I_P^{2(j-m), k} \right)} \vast)^{\frac{1}{2}},
\end{eqnarray*}
where
$$
\cTFC:=\left\{\substack{(P, Q)\\P \sim 2^{j-m}, Q \sim 2^{j-m-k}\\\exists L_2^{*}\in\N}\,:\,\frac{2QL_2^*}{2^{j-m-k}}=L_1 \left(\frac{P}{2^k}-Q \right), \ L_2^*=L_2 \left(P+\frac{Q^2}{2^{j-m-k}} \right), \  L_2^*=L_4(P) \right\}
$$
The proof of this case concludes with the usual statement: there exists some $\epsilon>0$, such that
\begin{equation} \label{20230314eq31}
\# \cTFC \lesssim 2^{(2-\epsilon)(j-m-k)+k}.
\end{equation}

In summary, in this case, we have there exists some $\epsilon>0$, such that
$$
\left|\Lambda_{j, m}^k (\vec{f}) \right| \lesssim 2^{-\epsilon \left(j-m-k \right)} \left\| \vec{f} \right\|_{L^{\vec{p}} \left(3I_0^{0, k}\right)}, \quad \forall\: \vec{p} \in {\bf H}^{+}.
$$

\subsection{Treatment of Case I.3:  $0 \le m \le \frac{j}{2}$ and $k \ge j-m$.}

In this case, we combine the following two features:
\begin{itemize}
\item Rank I LGC methodology adapted to the input functions $f_{1,j+k}$ and $f_{2,j+2k}$, and
\item the dominance of the Fourier localization of $f_{3,m+3k}$ over the corresponding Foruier localizations of $f_{1,j+k}$ and $f_{2,j+2k}$ that in turn implies that
\begin{equation} \label{f4support}
\supp\, \widehat{f_4}\subseteq\{ |\tau| \sim 2^{m+3k} \}\,.
\end{equation}
\end{itemize}
Concretely, based on \eqref{f4support}, Heisenberg's principle suggests the decomposition of the $x$-domain of integration into smaller intervals of length $2^{-m-3k}$ while the first item above imposes--part of \textsf{Step 1} of the LGC method--the linearizing scale for the $t$-domain at the level of $2^{-\frac{j}{2}-k}$; consequently, we write
$$\Lambda_{j, m}^k(\vec{f})\approx 2^k \sum_{\substack{p \sim 2^{m+2k} \\ q \sim 2^{\frac{j}{2}}}} \int_{I_p^{2m, 3k}} \int_{I_q^{j, k}} f_{1,j+k}(x-t)f_{2,j+2k}(x+t^2)f_{3,m+3k}(x+t^3) )f_4(x)dtdx.$$
Exploiting now the information below:
\begin{itemize}
\item $f_3$ is morally a constant on intervals of length $2^{-m-3k}$, and
\item the variation of the argument $t^3$ on any interval $I_q^{j, k}$ is $\lesssim 2^{-\frac{j}{2}-3k}<2^{-m-3k}$,
\end{itemize}
we deduce that
$$
\Lambda_{j, m}^k(\vec{f})\approx 2^k \sum_{\substack{p \sim 2^{m+2k} \\ q \sim 2^{\frac{j}{2}}}} \int_{I_p^{2m, 3k}} \left(\int_{I_q^{j, k}} f_1(x-t)f_2(x+t^2) dt  \right) f_3 \left(x+\frac{q^3}{2^{\frac{3j}{2}+3k}} \right) f_4(x) dx.
$$
Next, since
\begin{itemize}
\item $f_1$ is morally constant on intervals of length $2^{-j-k}>2^{-m-3k}$;
\item $f_2$ is morally constant on intervals of length $2^{-j-2k}>2^{-m-3k}$,
\end{itemize}
we see that the map
$$
x \longmapsto \int_{I_q^{j, k}} f_1(x-t)f_2(x+t^2) dt
$$
is essentially constant on the interval $I_p^{2m, 3k}$.

Putting all these elements together, we are able to \emph{decouple} the information carried by the pair $(f_1,f_2)$ from that carried by $(f_3,f_4)$, deriving the model
\begin{eqnarray}
\left|\Lambda_{j, m}^k(\vec{f})\right| \approx 2^k \vast |\sum_{\substack{p \sim 2^{m+2k} \\ q \sim 2^{\frac{j}{2}}}} \left(\frac{1}{ \left|I_p^{2m, 3k} \right|} \int_{I_p^{2m, 3k}} \int_{I_q^{j, k}}  f_1(x-t)f_2(x+t^2) dtdx\right)\,\left(\int_{I_p^{2m, 3k}} f_3 \left(x+\frac{q^3}{2^{\frac{3j}{2}+3k}} \right) f_4(x) dx \right)\vast|\quad\nonumber \\
\lesssim 2^{\frac{m}{2}+\frac{5k}{2}} \sum_{\substack{p \sim 2^{\frac{j}{2}+k} \\ q \sim 2^{\frac{j}{2}}}} \left( \int_{I_{p}^{j, 2k}} \left| \int_{I_q^{j, k}}  f_1(x-t)f_2(x+t^2)dt \right|^2 dx \right)^{\frac{1}{2}} \left[ \sum_{r=p2^{m-\frac{j}{2}+k}}^{(p+1)2^{m-\frac{j}{2}+k}}  \left(\int_{I_{r+\frac{q^3}{2^{\frac{3j}{2}-m}}}^{2m, 3k}} |f_3|^2 \right)  \left(\int_{I_r^{2m, 3k}} |f_4|^2 \right) \right]^{\frac{1}{2}},\nonumber
\end{eqnarray}
where for the latter inequality we made use of Cauchy--Schwarz.

With these settled, we focus on the first term and continue with \textsf{Step 2} of Rank I LGC by decomposing $f_1$ and $f_2$ in Gabor wave-packets:
$$
f_1=\sum_{\substack{p_1 \sim 2^{\frac{j}{2}} \\ n_1 \sim 2^{\frac{j}{2}}}} \left\langle f_1 , \phi_{p_1, n_1}^{j, k} \right\rangle  \phi_{p_1, n_1}^{j, k}, \quad \textrm{and} \quad
f_2=\sum_{\substack{p_2 \sim 2^{\frac{j}{2}+k} \\ n_2 \sim 2^{\frac{j}{2}}}} \left\langle f_2 , \phi_{p_2, n_2}^{j, 2k} \right\rangle  \phi_{p_2, n_2}^{j, 2k}.
$$
This gives
\begin{equation} \label{20230315eq03}
\int_{I_q^{j, k}} f_1(x-t)f_2(x+t^2)dt=\sum_{\substack{p_1 \sim 2^{\frac{j}{2}} \\ n_1 \sim 2^{\frac{j}{2}}}}\sum_{\substack{p_2 \sim 2^{\frac{j}{2}+k} \\ n_2 \sim 2^{\frac{j}{2}}}}  \left\langle f_1 , \phi_{p_1, n_1}^{j, k} \right\rangle  \left\langle f_2 , \phi_{p_2, n_2}^{j, 2k} \right\rangle \int_{I_q^{j, k}} \phi_{p_1, n_1}^{j, k}(x-t) \phi_{p_2, n_2}^{j, 2k}(x+t^2)dt.
\end{equation}

Finally, at \textsf{Step 3}, we begin by extracting the \emph{time-frequency correlations} as follows:  analyzing the $t$-integral in the above expression and using $x \in I_p^{j, 2k}$, we deduce the time correlations:
\begin{equation} \label{20230315eq01}
p_1=\frac{p}{2^k}-q+O(1) \quad \textrm{and} \quad p_2=p+\frac{q^2}{2^{\frac{j}{2}}}+O(1).
\end{equation}
Therefore, we have
\begin{equation}
\int_{I_q^{j, k}} \phi_{p_1, n_1}^{j, k}(x-t) \phi_{p_2, n_2}^{j, 2k}(x+t^2)dt\approx 2^{\frac{j}{2}+\frac{3k}{2}} \cdot e^{i 2^{\frac{j}{2}+2k} x \left(\frac{n_1}{2^k}+n_2 \right)} \cdot \int_{I_q^{j, k}} e^{i 2^{\frac{j}{2}+k}  t \left(-n_1+\frac{2qn_2}{2^{\frac{j}{2}}} \right)} dt,
\end{equation}
which further gives the time-frequency correlation in {\bf Case I.3} as:
\begin{equation} \label{20230315eq02}
n_1=\frac{2q}{2^{\frac{j}{2}}}n_2+O(1).
\end{equation}
Replacing \eqref{20230315eq01} and \eqref{20230315eq02} back to \eqref{20230315eq03}, and applying Parseval we get
\begin{equation} \label{20230315eq04}
 \int_{I_{p}^{j, 2k}} \left| \int_{I_q^{j, k}}  f_1(x-t)f_2(x+t^2)dt \right|^2 dx \lesssim 2^{-\frac{j}{2}-k}\,\sum_{n_2 \sim 2^{\frac{j}{2}}} \left| \left \langle f_1, \phi_{\frac{p}{2^k}-q, \frac{2qn_2}{2^{\frac{j}{2}}}}^{j, k} \right \rangle \right|^2 \left| \left\langle f_2, \phi_{p+\frac{q^2}{2^{\frac{j}{2}}}, n_2}^{j, 2k} \right \rangle \right|^2.
\end{equation}

Inserting now  \eqref{20230315eq04} into the previous model, we conclude
\begin{eqnarray}\label{f1}
&& |\Lambda_{j, m}^k(\vec{f})| \lesssim 2^{\frac{m}{2}-\frac{j}{4}+2k} \sum_{\substack{p \sim 2^{\frac{j}{2}+k} \\ q \sim 2^{\frac{j}{2}}}} \left( \sum_{n_2 \sim 2^{\frac{j}{2}}} \left| \left \langle f_1, \phi_{\frac{p}{2^k}-q, \frac{2qn_2}{2^{\frac{j}{2}}}}^{j, k} \right \rangle \right|^2 \left| \left\langle f_2, \phi_{p+\frac{q^2}{2^{\frac{j}{2}}}, n_2}^{j, 2k} \right \rangle \right|^2  \right)^{\frac{1}{2}} \\
 && \quad \quad \quad \quad \quad \quad
 \quad \quad  \quad \quad\cdot \left( \sum_{r=p 2^{m-\frac{j}{2}+k}}^{(p+1)2^{m-\frac{j}{2}+k}}  \left(\int_{I_{r+\frac{q^3}{2^{\frac{3j}{2}-m}}}^{2m, 3k}} |f_3|^2 \right)  \left(\int_{I_r^{2m, 3k}} |f_4|^2 \right) \right)^{\frac{1}{2}}\nonumber.
\end{eqnarray}

With the grouping of the terms in the model finalized, we pass to the final stage that focuses on the analysis of the time-frequency correlation set via the sparse-uniform dichotomy.

We thus define the \emph{sparse set} of indices relative to this improved BHT bound case as follows: for $i\in\{1, 2\}$,
\begin{equation}\label{Spars}
\sfS_\mu(f_i):=\left\{(p_i, n_i): \left| \left\langle f_i, \phi_{p_i, n_i}^{j, ik}  \right \rangle \right|^2 \gtrsim 2^{-\mu j} \int_{I_{p_i}^{j, ik}} |f_i|^2,  \quad p_i \sim 2^{\frac{j}{2}+(i-1)k}, n_i \sim 2^{\frac{j}{2}} \right\}.
\end{equation}
Define in the usual fashion the \emph{uniform set} of indices $\sfU_\mu(f_i)$ for $i\in\{1, 2\}$ and decompose the functions $f_1$ and $ f_2$ into their sparse/uniform components--denoted by $f_i^{\sfS_\mu}$ and $f_i^{\sfU_\mu}$.

Next, as expected, we decompose our discussion into two different cases:

In the uniform case, namely, if one of the functions $f_1$ or $f_2$ is uniform--assume without loss of generality that $f_1=f_1^{\sfU_\mu}$--taking \eqref{f1} we first apply Parseval in $n_2$ followed by a Cauchy-Schwarz in $p, q$ to get
\begin{eqnarray}
|\Lambda^{k, \sfU_\mu}_{j, m}(\vec{f})|&\lesssim& 2^{\frac{m}{2}-\frac{j}{4}+2k}\,2^{-\frac{\mu j}{2}}  \sum_{\substack{p \sim 2^{\frac{j}{2}+k} \\ q \sim 2^{\frac{j}{2}}}} \left\|f_1 \right\|_{L^2(I_{\frac{p}{2^k}-q}^{j, k})}  \left\|f_2 \right\|_{L^2(I_{p+\frac{q^2}{2^{\frac{j}{2}}}}^{j, 2k})}\nonumber\\
&\cdot&\left( \sum_{r=p 2^{m-\frac{j}{2}+k}}^{(p+1)2^{m-\frac{j}{2}+k}}  \left(\int_{I_{r+\frac{q^3}{2^{\frac{3j}{2}-m}}}^{2m, 3k}} |f_3|^2 \right)  \left(\int_{I_r^{2m, 3k}} |f_4|^2 \right) \right)^{\frac{1}{2}}\nonumber\\
&\lesssim& 2^{-\frac{\mu j}{2}}\,2^{2k}\,\left\|f_1 \right\|_{L^2(I^k)} \left\|f_2 \right\|_{L^2(I^k)}
\left(\sum_{s\sim2^{2k}} \left\|f_3 \right\|^2_{L^2(I^{0,3k}_s)} \left\|f_4 \right\|^2_{L^2(I^{0,3k}_s)}\right)^\frac{1}{2}\,.
\end{eqnarray}

In the sparse case, that is when both $f_1$ and $f_2$ are sparse, we use the fact that up to a factor of at most $2^{2\mu j}$ there are unique measurable functions $L_1(\cdot)$ and $L_2(\cdot)$ that achieve the maximum value for the local Fourier coefficients appearing in \eqref{Spars}. Now define in the usual manner the time-frequency correlations set
$$\cTFC:= \left\{\substack{(p_1, q) \\p_1 \sim 2^{\frac{j}{2}+k}, q \sim 2^{\frac{j}{2}}}\,:\,L_1 \left(\frac{p_1}{2^k}-q \right)=\frac{2q}{2^{\frac{j}{2}}} L_2 \left(p_1+\frac{q^2}{2^{\frac{j}{2}}} \right) \right\}.
$$
With these, recalling \eqref{f1}, we have
\begin{eqnarray*}
|\Lambda^{k, \sfS_\mu}_{j, m}(\vec{f})| &\lesssim& 2^{\frac{m}{2}-\frac{j}{4}+2k}\,2^{\mu j}\,\left(\# \cTFC \right)^{\frac{1}{2}} \\
&\cdot&\vast[ \sum_{\substack{p \sim 2^{\frac{j}{2}+k} \\ q \sim 2^{\frac{j}{2}}}} \left\|f_1 \right\|^2_{L^2(I_{\frac{p}{2^k}-q}^{j, k})}\,\left\|f_2 \right\|^2_{L^2(I_{p+\frac{q^2}{2^{\frac{j}{2}}}}^{j, 2k})}\left( \sum_{r=p 2^{m-\frac{j}{2}+k}}^{(p+1)2^{m-\frac{j}{2}+k}}  \left(\int_{I_{r+\frac{q^3}{2^{\frac{3j}{2}-m}}}^{2m, 3k}} |f_3|^2 \right)  \left(\int_{I_r^{2m, 3k}} |f_4|^2 \right) \right) \vast]^{\frac{1}{2}}\,.
\end{eqnarray*}

The desired conclusion follows from the fact that there exists some $\epsilon>0$, such that
$$\# \cTFC \lesssim 2^{j+k-\epsilon j}.$$

In summary, in this case, we have that there exists some $\epsilon>0$ such that
$$\left|\Lambda_{j, m}^k (\vec{f}) \right| \lesssim 2^{-\epsilon j} \left\| \vec{f} \right\|_{L^{\vec{p}} \left(3I_0^{0, k}\right)} \quad \forall\: \vec{p} \in {\bf H}^{+}.$$

\subsection{Treatment of Case II.1:  $\frac{j}{2} \le m \le j-100$ and $k < j-m$.}

In this case we apply Rank I LGC methodology adapted to (all) the three input functions $f_{1,j+k}$, $f_{2,j+2k}$, and $f_{3,m+3k}$. Indeed, at \textsf{Step 1} we divide the $(x,t)$ domain of integration as

$$\Lambda_{j, m}^k(\vec{f})\approx 2^k \sum_{\substack{p \sim 2^{\frac{j}{2}+2k} \\ q \sim 2^{\frac{j}{2}}}} \int_{I_p^{j, 3k}} \int_{I_q^{j, k}} f_{1,j+k}(x-t)f_{2,j+2k}(x+t^2)f_{3,m+3k}(x+t^3)f_4(x) dtdx\,,$$
obtaining the linearizing effect in the $t-$variable.

Next, at \textsf{Step 2}, we decompose the input functions in adapted Gabor wave-packets:
$$f_r \simeq \sum_{p_r \sim 2^{\frac{j}{2}+(r-1)k} , \ n_r \sim 2^{m-\frac{j}{2}}} \left \langle f_r, \phi_{p_r, n_r}^{j, rk} \right \rangle \phi_{p_r, n_r}^{j, rk}\,,\qquad r\in\{1,2,3\}\,,$$
in order to obtain
\begin{eqnarray} \label{20230315eq51}
&& \left|\Lambda_{j, m}^k(\vec{f}) \right| \simeq  2^k \sum_{\substack{p \sim 2^{\frac{j}{2}+2k} \\ q \sim 2^{\frac{j}{2}}}} \sum_{\substack{p_1 \sim 2^{\frac{j}{2}} \\ n_1 \sim 2^{\frac{j}{2}}}} \sum_{\substack{p_2 \sim 2^{\frac{j}{2}+k} \\ n_2 \sim 2^{\frac{j}{2}}}} \sum_{\substack{p_3 \sim 2^{\frac{j}{2}+2k} \\ n_3 \sim 2^{m-\frac{j}{2}}}} \left| \left \langle f_1, \phi_{p_1, n_1}^{j, k} \right \rangle \right| \left| \left \langle f_2, \phi_{p_2, n_2}^{j, 2k} \right \rangle \right| \left| \left \langle f_3, \phi_{p_3, n_3}^{j, 3k} \right \rangle \right| \\
&& \quad \quad \cdot \left| \int_{I_p^{j, 3k}} \int_{I_q^{j, k}} \phi_{p_1, n_1}^{j, k}(x-t) \phi_{p_2, n_2}^{j, 2k}(x+t^2) \phi_{p_3, n_3}^{j, 3k} (x+t^3) f_4(x)dtdx \right|. \nonumber
\end{eqnarray}
 Finally, at \textsf{Step 3}, in the first stage we identify the  time-frequency correlations: from the double integral above and  the restrictions $x \in I_p^{j, 3k}$ and $t \in I_q^{j, k}$, we get
\begin{equation} \label{20230315eq50}
p_1=\frac{p}{2^{2k}}-q+O(1), \quad  p_2=\frac{p}{2^k}+\frac{q^2}{2^{\frac{j}{2}}}+O(1), \quad \textrm{and} \quad p_3=p+\frac{q^3}{2^j}+O(1)\,.
\end{equation}
Inserting \eqref{20230315eq50} back to the double integral in \eqref{20230315eq51} and linearizing the phase function as usual, we deduce
\begin{eqnarray*}
\left|\int_{I_p^{j, 3k}} \int_{I_q^{j, k}} \phi_{p_1, n_1}^{j, k}(x-t) \phi_{p_2, n_2}^{j, 2k}(x+t^2) \phi_{p_3, n_3}^{j, 3k} (x+t^3) f_4(x)dtdx\right|\qquad\qquad\qquad\qquad\qquad\qquad\\
\approx 2^{\frac{3j}{4}+3k}\left| \int_{I_p^{j, 3k}} \left(\int_{I_q^{j, k}} e^{i2^{\frac{j}{2}+k}t \left(-n_1+\frac{2q n_2}{2^{\frac{j}{2}}}+\frac{3q^2n_3}{2^j} \right)} dt \right) e^{i2^{\frac{j}{2}+3k} x \left(n_3+\frac{n_2}{2^k}+\frac{n_1}{2^{2k}} \right)} f_4(x)dx\right|\lesssim 2^{\frac{k}{2}} \left| \left \langle f_4, \phi^{j, 3k}_{p, n_3+\frac{n_2}{2^k}+\frac{n_1}{2^{2k}}} \right \rangle \right|\,,
\end{eqnarray*}
together with the time-frequency correlation for {\bf Case II.1}:
\begin{equation} \label{20230315eq511}
n_1-\frac{2qn_2}{2^{\frac{j}{2}}}-\frac{3q^2n_3}{2^j}=0+O(1).
\end{equation}
Substituting \eqref{20230315eq50} and \eqref{20230315eq511} back to \eqref{20230315eq51}, we obtain the discretized model operator
\begin{eqnarray} \label{20230316eq80}
&& \left|\Lambda_{j, m}^k(\vec{f})\right| \lesssim 2^{\frac{3k}{2}} \sum_{\substack{p \sim 2^{\frac{j}{2}+2k} \\ q \sim 2^{\frac{j}{2}}}} \sum_{\substack{n_2 \sim 2^{\frac{j}{2}} \\ n_3 \sim 2^{m-\frac{j}{2}}}} \left| \left \langle f_1, \phi_{\frac{p}{2^{2k}}-q, \frac{2qn_2}{2^{\frac{j}{2}}}+\frac{3q^2n_3}{2^j}}^{j, k} \right \rangle \right| \left| \left \langle f_2, \phi_{\frac{p}{2^k}+\frac{q^2}{2^{\frac{j}{2}}}, n_2}^{j, 2k} \right \rangle \right| \nonumber \\
&& \quad \quad \quad \quad \quad \quad \quad \quad \quad \quad \cdot \left| \left \langle f_3, \phi_{p+\frac{q^3}{2^j}, n_3}^{j, 3k} \right \rangle \right| \left| \left \langle f_4, \phi^{j, 3k}_{p, n_3+\frac{n_2}{2^k}+\frac{1}{2^{2k}} \cdot \left(\frac{2qn_2}{2^{\frac{j}{2}}}+\frac{3q^2n_3}{2^j} \right)} \right \rangle \right|.
\end{eqnarray}
Observe that since $0 \le k \le j-m$, we have $\supp \,\widehat{f_4} \subseteq \{|\eta| \sim 2^{j+2k}\}$. Again, motivated by the restriction on the range of $n_3 \sim 2^{m-\frac{j}{2}}$, we decompose our spatial and frequency indices as follows:
\begin{itemize}
\item for spatial indices, we have
$$\left\{p \sim 2^{\frac{j}{2}+2k} \right\}=\bigcup_{P \sim 2^{j-m+2k}} J_P^k, \quad \textrm{where} \quad J_P^k:=\left[P2^{m-\frac{j}{2}}, (P+1)2^{m-\frac{j}{2}} \right]\qquad\textrm{and}$$
$$\left\{q \sim 2^{\frac{j}{2}} \right\}=\bigcup_{Q \sim 2^{j-m}} J_Q^k, \quad \textrm{where} \quad J_Q^k:=\left[Q2^{m-\frac{j}{2}}, (Q+1)2^{m-\frac{j}{2}} \right].$$
\item for frequency indices, we will need two types of decompositions:
\begin{itemize}
\item the first type is  given by
$$\supp \widehat{f_1} \subseteq \left\{|\xi| \sim 2^{j+k} \right\}=\bigcup_{L_1 \sim 2^{j-m}} \Omega_1^{L_1} \quad \textrm{where} \quad \Omega_1^{L_1}:=\left[L_12^{m+k}, (L_1+1)2^{m+k} \right]\quad\textrm{and}$$
$$\supp \widehat{f_2} \subseteq \left\{|\eta| \sim 2^{j+2k} \right\}=\bigcup_{L_2 \sim 2^{j-m}} \Omega_2^{L_2} \quad \textrm{where} \quad \Omega_2^{L_2}:=\left[L_2 2^{m+2k}, (L_2+1)2^{m+2k} \right]\,.$$
Notice here that if $n_2 2^{\frac{j}{2}+2k} \in \Omega_2^{L_2}$ then $n_2 \in \widetilde{\Omega_2^{L_2}}:=\left[L_2 2^{m-\frac{j}{2}}, (L_2+1) 2^{m-\frac{j}{2}} \right]$.
\item the second type is given by
$$\supp \widehat{f_1} \subseteq \left\{|\xi| \sim 2^{j+k} \right\}=\bigcup_{\ell_1 \sim 2^{j-m-k}} \omega_1^{\ell_1} \quad \textrm{where} \quad \omega_1^{\ell_1}:=\left[\ell_12^{m+2k}, (\ell_1+1)2^{m+2k} \right],$$
$$\supp \widehat{f_2} \subseteq \left\{|\xi| \sim 2^{j+2k} \right\}=\bigcup_{\ell_2 \sim 2^{j-m-k}} \omega_2^{\ell_2} \quad \textrm{where} \quad \omega_2^{\ell_2}:=\left[\ell_22^{m+3k}, (\ell_2+1)2^{m+3k} \right]\quad\textrm{and}$$
$$\supp \widehat{f_4} \subseteq \left\{|\eta| \sim 2^{j+2k} \right\}=\bigcup_{\ell_4 \sim 2^{j-m-k}} \omega_4^{\ell_4} \quad \textrm{where} \quad \omega_4^{\ell_4}:=\left[\ell_4 2^{m+3k}, (\ell_4+1)2^{m+3k} \right].$$
\end{itemize}
\end{itemize}
With these settled, using Cauchy-Schwarz, we write
\begin{eqnarray} \label{202303163170}
&& \left|\Lambda_{j, m}^k(\vec{f})\right| \lesssim 2^{\frac{3k}{2}}  \sum_{\substack{P \sim 2^{j-m+2k} \\ L_2, Q \sim 2^{j-m}}} \calA(P, Q , L_2)\: \calB(P, Q, L_2),
\end{eqnarray}
where
\begin{equation} \label{20230316eq90}
\calA^2(P, Q, L_2):=\sum_{\substack{p \in J_P^k \\ q \in J_Q^k}} \sum_{\substack{n_2 \in \widetilde{\Omega_2^{L_2}} \\ n_3 \sim 2^{m-\frac{j}{2}}}} \left| \left \langle f_1, \phi_{\frac{p}{2^{2k}}-q, \frac{2qn_2}{2^{\frac{j}{2}}}+\frac{3q^2n_3}{2^j}}^{j, k} \right \rangle \right|^2 \left| \left \langle f_3, \phi_{p+\frac{q^3}{2^j}, n_3}^{j, 3k} \right \rangle \right|^2,
\end{equation}
and
\begin{equation} \label{20230316eq91}
\calB^2(P, Q, L_2):=\sum_{\substack{p \in J_P^k \\ q \in J_Q^k}} \sum_{\substack{n_2 \in \widetilde{\Omega_2^{L_2}} \\ n_3 \sim 2^{m-\frac{j}{2}}}} \left| \left \langle f_2, \phi_{\frac{p}{2^k}+\frac{q^2}{2^{\frac{j}{2}}}, n_2}^{j, 2k} \right \rangle \right|^2 \left| \left \langle f_4, \phi^{j, 3k}_{p, n_3+\frac{n_2}{2^k}+\frac{1}{2^{2k}} \cdot \left(\frac{2qn_2}{2^{\frac{j}{2}}}+\frac{3q^2n_3}{2^j} \right)} \right \rangle \right|^2.
\end{equation}
By Parseval, we have
$$
\calA \left(P, Q, L_2 \right) \lesssim \left\| f_1^{\Omega_1^{\frac{2L_2Q}{2^{j-m}}}} \right\|_{L^2 \left(I_{\frac{P}{2^{2k}}-Q}^{2(j-m), k} \right)} \left\|f_3 \right\|_{L^2 \left(I_{P+\frac{Q^3}{2^{2(j-m)}}}^{2(j-m), 3k} \right)},
$$
and
$$
\calB(P, Q, L_2) \lesssim \left\| f_2^{\Omega_2^{L_2}} \right\|_{L^2 \left(I^{2(j-m), 2k}_{\frac{P}{2^k}+\frac{Q^2}{2^{j-m}}} \right)} \left\|f_4^{\omega_4^{\frac{L_2}{2^k}+\frac{2QL_2}{2^{j-m+2k}} }} \right\|_{L^2 \left(I_P^{2(j-m), 3k} \right)} ,
$$
where in the last estimate above, we have used the fact that $\left|\frac{2QL_2}{2^{\frac{3j}{2}-m+2k}} \right| \lesssim 2^{\frac{j}{2}-m}=O(1)$. Plugging the estimates of $\calA(P, Q, L_2)$ and $\calB(P, Q, L_2)$ back to \eqref{202303163170}, we deduce that
\begin{eqnarray} \label{20230316eq71}
&& \left|\Lambda_{j, m}^k(\vec{f})\right| \lesssim  2^{\frac{3k}{2}} \sum_{\substack{P \sim 2^{j-m+2k} \\ Q \sim 2^{j-m}}} \sum_{\ell_2 \sim 2^{j-m-k}} \left\|f_1^{\omega_1^\frac{2\ell_2Q}{2^{j-m}}} \right\|_{L^2 \left(I_{\frac{P}{2^{2k}}-Q}^{2(j-m), k} \right)}  \left\| f_2^{\omega_2^{\ell_2}} \right\|_{L^2 \left(I^{2(j-m), 2k}_{\frac{P}{2^k}+\frac{Q^2}{2^{j-m}}} \right)} \\
&& \quad \quad \quad \quad \quad \quad \cdot \left\|f_3 \right\|_{L^2 \left(I_{P+\frac{Q^3}{2^{2(j-m)}}}^{2(j-m), 3k} \right)}   \left\|f_4^{\omega_4^{\left(1+\frac{2Q}{2^{j-m+k}}\right)\ell_2}} \right\|_{L^2 \left(I_P^{2(j-m), 3k} \right)}. \nonumber
\end{eqnarray}
Motivated by the above expression, at the second stage of \textsf{Step 3} we define the \emph{sparse set} of indices
\begin{equation}
\sfS_{\mu}(f_i):= \bigg\{ (P_i, \ell_i): \int_{I_{P_i}^{2(j-m), \underline{i}k}} \left|f_i^{\omega_i^{\ell_i}} \right|^2 \gtrsim 2^{-\mu(j-m-k)} \int_{I_{P_i}^{2(j-m), \underline{i}k}} \left|f_i \right|^2, \ \substack{P_i \sim 2^{j-m+(\underline{i}-1)k} \\ \ell_i \sim 2^{j-m-k}} \bigg\},
\end{equation}
where $i\in\{1, 2, 4\}$ and $\underline{i}:=\min\{i, 3\}$. Define as usual the \emph{uniform set} of indices $\sfU_\mu(f_i)$ and decompose the functions $f_1, f_2$ and $f_4$ into their sparse/uniform components denoted by $f_i^{\sfS_\mu}$ and $f_i^{\sfU_\mu}$, $i=1, 2, 4$, respectively. Proceed similarly for $\Lambda_{m, j}^{k, \sfS_\mu}$ and $\Lambda_{m, j}^{k, \sfU_\mu}$.

\subsubsection{Treatment of $\Lambda_{j, m}^{k, \sfU_\mu}$} In this case, we have that at least one of the functions $f_1, f_2, f_4$ is uniform, and thus, without loss of generality, we may assume $f_4=f_4^{\sfU_\mu}$. Therefore, from \eqref{20230316eq71}, we have
\begin{equation}
|\Lambda_{j, m}^k(\vec{f})|\lesssim 2^{\frac{3k}{2}} \cdot 2^{-\mu(j-m-k)}  \sum_{\substack{P \sim 2^{j-m+2k} \\ Q \sim 2^{j-m}}} \left\|f_1  \right\|_{L^2 \left(I_{\frac{P}{2^{2k}}-Q}^{2(j-m), k} \right)} \left\| f_2 \right\|_{L^2 \left(I^{2(j-m), 2k}_{\frac{P}{2^k}+\frac{Q^2}{2^{j-m}}} \right)} \left\|f_3 \right\|_{L^2 \left(I_{P+\frac{Q^3}{2^{2(j-m)}}}^{2(j-m), 3k} \right)}   \left\|f_4 \right\|_{L^2 \left(I_P^{2(j-m), 3k} \right)}.
\end{equation}

\subsubsection{Treatment of $\Lambda_{j, m}^{k, \sfS_\mu}$} In this case, all the functions $f_1, f_2$ and $f_4$ are sparse and hence, from \eqref{20230316eq71}
\begin{eqnarray*}
&& \left|\Lambda_{j, m}^k(\vec{f}) \right| \lesssim 2^{\frac{3k}{2}} \cdot 2^{3\mu(j-m-k)} \cdot \left(\# \cTFC \right)^{\frac{1}{2}} \cdot \vast( \sum_{\substack{P \sim 2^{j-m+2k} \\ Q \sim 2^{j-m}}} \left\|f_1\right\|^2_{L^2 \left(I_{\frac{P}{2^{2k}}-Q}^{2(j-m), k} \right)}   \\
&& \quad \quad \quad \cdot \left\| f_2 \right\|^2_{L^2 \left(I^{2(j-m), 2k}_{\frac{P}{2^k}+\frac{Q^2}{2^{j-m}}} \right)}\left\|f_3 \right\|^2_{L^2 \left(I_{P+\frac{Q^3}{2^{2(j-m)}}}^{2(j-m), 3k} \right)}   \left\|f_4\right\|^2_{L^2 \left(I_P^{2(j-m), 3k} \right)} \vast)^{\frac{1}{2}},
\end{eqnarray*}
where
$$
\cTFC:= \left\{\substack{(P, Q)\\ P \sim 2^{j-m+2k}, Q \sim 2^{j-m}}\,:\,\frac{2\ell_2Q}{2^{j-m}}=L_1 \left(\frac{P}{2^{2k}}-Q \right), \ \ell_2=L_2 \left(\frac{P}{2^k}+\frac{Q^2}{2^{j-m}} \right), \ \left(1+\frac{2Q}{2^{j-m+k}}\right)\ell_2=L_4(P) \right\},
$$
and for $i\in\{1, 2, 4\}$, $L_i: \left[2^{j-m+(\underline{i}-1)k}, 2^{j-m+(\underline{i}-1)k+1} \right] \mapsto \left[2^{j-m-k}, 2^{j-m-k+1} \right]$ measurable functions.

Our desired decay bound follows from the control over the time-frequency correlation set $\cTFC $ in the usual form: there exists some $\epsilon>0$, such that
$$
\# \cTFC \lesssim 2^{(2-\epsilon)(j-m)+2k}.
$$
Again, the proof of the last statement above follows a standard approach now, and thus we omit the details here.

In summary, in this case, we have that there exists some $\epsilon>0$, such that
$$
\left|\Lambda_{j, m}^k (\vec{f}) \right| \lesssim 2^{-\epsilon \left(j-m-k \right)} \left\| \vec{f} \right\|_{L^{\vec{p}} \left(3I_0^{0, k}\right)} \quad \forall\: \vec{p} \in {\bf H}^{+}.
$$

\subsection{Treatment of Case II.2:  $\frac{j}{2} \le m \le j-100$ and $k \ge j-m$.}

 The treatment of this case follows a similar argument with the one presented at {\bf Case I.3}. Thus we will only outline the required adaptations/modifications on the  Rank I LGC methodology leaving the details to the interested reader:
\begin{enumerate}
\item [(1)] Firstly, we choose a decomposition of the $(x,t)-$domain adapted to the interval $I_q^{2m, k}$ rather than $I_q^{j, k}$ for the $t$-variable with the $x$-discretization preserved, that is
$$
\Lambda_{j, m}^k(\vec{f})\approx 2^k \sum_{ \substack{p \sim 2^{m+2k} \\ q \sim 2^m}} \int_{I_p^{2m, 3k}} \int_{I_q^{2m, k}} f_{1,j+k}(x-t)f_{2,j+2k}(x+t^2)f_{3,m+3k}(x+t^3)f_4(x)dtdx.
$$
\item [(2)] Secondly, we decompose $f_1$ and $f_2$ using Garbor frames adapted to the $m$ parameter instead of $j$:
$$
f_1=\sum_{\substack{p_1 \sim 2^m \\ n_1 \sim 2^{j-m}}} \left \langle f_1, \phi_{p_1, n_1}^{2m, k} \right \rangle \phi_{p_1, n_1}^{2m, k} \quad \textrm{and} \quad  f_2=\sum_{\substack{p_2 \sim 2^{m+k} \\ n_2 \sim 2^{j-m}}} \left \langle f_2, \phi_{p_2, n_2}^{2m, 2k} \right \rangle \phi_{p_2, n_2}^{2m, 2k}.
$$
\item[(3)] Thirdly, one is only left--essentially--with replacing in the proof of {\bf Case I.3} the scale $\frac{j}{2}$ by $m$.
\end{enumerate}

We summarize the final estimate in this case as follows: there exists some $\epsilon>0$ such that
$$
\left|\Lambda_{j, m}^k (\vec{f}) \right| \lesssim 2^{-\epsilon \left(j-m \right)} \left\| \vec{f} \right\|_{L^{\vec{p}} \left(3I_0^{0, k}\right)} \quad \forall\:p \in {\bf H}^{+}.
$$

\section{The stationary off-diagonal term $\Lambda^{\not{D},S}$: The case $l=m>j+100$ and $k \ge 0$}\label{stationaryoff2}

In this case, one has
$$
\frac{|\tau|}{2^{3k}} \simeq \frac{|\eta|}{2^{2k}} \gg \frac{|\xi|}{2^k},
$$
and hence\footnote{Note that in this case, $\supp \widehat{f_4}\subseteq\left\{|\tau| \sim 2^{m+3k} \right\}$, which makes the analysis of this case simpler than the one in Section \ref{20230321sec01} in which the location/size of $\supp \widehat{f_4}$ changes depending on the relative position of the parameters $k,j$ and $m$.}
$$
\supp\,\widehat{f_1} \subseteq \left\{|\xi| \sim 2^{j+k} \right\}, \ \supp\,\widehat{f_2} \subseteq \left\{|\eta| \sim 2^{m+2k} \right\}, \ \supp\,\widehat{f_3} \subseteq \left\{|\tau| \sim 2^{m+3k} \right\}, \ \textrm{and} \ \supp\,\widehat{f_4} \subseteq \left\{|\tau| \sim 2^{m+3k} \right\}.
$$
We consider two different cases.
\begin{enumerate}
\item [I.] $0 \le j \le \frac{m}{2}$;
\item [II.] $\frac{m}{2} \le j \le m-100$.
\end{enumerate}

\subsection{Treatment of Case I: $0 \le j \le \frac{m}{2}$.}

The proof of this case follows from a straightforward modification of the corresponding proof in Section \ref{20230321subsec01}, and hence, in what follows, we only provide a sketch.  To begin with, we decompose our main term as
\begin{equation} \label{20230321eq02}
\Lambda_{j, m, m}^k(\vec{f}):=\Lambda^k_{m, j}(\vec{f})\approx2^k \sum_{\substack{p \sim 2^{\frac{m}{2}+2k} \\ q \sim 2^{\frac{m}{2}}}} \int_{I_p^{m, 3k}} \int_{I_q^{m, k}} f_{1,j+k}(x-t)f_{2,m+2k}(x+t^2)f_{3,m+3k}(x+t^3)f_4(x)dtdx.
\end{equation}
Observe that on each $I_q^{m, k}$, the function $f_{1, j+k}$ is morally a constant, and hence we can rewrite \eqref{20230321eq02} as
\begin{equation} \label{20230321eq01}
\Lambda^k_{m, j}(\vec{f}) \approx 2^k \sum_{\substack{p \sim 2^{\frac{m}{2}+2k} \\ q \sim 2^{\frac{m}{2}}}} \left(\frac{1}{\left|I_0^{m, k} \right|} \int_{I^{m, k}_{\frac{p}{2^{2k}}-q}} f_1 \right) \int_{I_p^{m, 3k}} \int_{I_q^{m, k}} f_2(x+t^2)f_3(x+t^3)f_4(x)dtdx.
\end{equation}
Next, we decompose $f_2$ and $f_3$ in Garbor wave-packets adapted to the linearizing scale level:
\begin{equation} \label{20230321eq03}
f_i \simeq \sum_{\substack{p_i \sim 2^{\frac{m}{2}+(i-1)k} \\ n_i \sim 2^{\frac{m}{2}}}} \left\langle f_i, \phi_{p_i, n_i}^{m, ik}   \right \rangle \phi_{p_i, n_i}^{m, ik}, \quad i\in\left\{2, 3 \right\},
\end{equation}
and analyze the double integral in \eqref{20230321eq01} as usual by extracting the corresponding time-frequency correlations. This gives
\begin{equation}
\left| \Lambda^k_{m, j} (\vec{f}) \right| \lesssim 2^{\frac{3k}{2}} \sum_{\substack{p \sim 2^{\frac{m}{2}+2k} \\ q, n_3 \sim 2^{\frac{m}{2}}}} \left(\frac{1}{\left|I_0^{m, k} \right|^{\frac{1}{2}}} \int_{I_{\frac{p}{2^{2k}}-q}} |f_1| \right) \left| \left \langle f_2, \phi_{\frac{p}{2^k}+\frac{q^2}{2^{\frac{m}{2}}}, -\frac{3q n_3}{2 \cdot 2^{\frac{m}{2}}}}^{m, 2k} \right \rangle \right | \left| \left \langle f_3, \phi_{p+\frac{q^3}{2^m},n_3}^{m, 3k} \right \rangle \right| \left | \left \langle f_4, \phi^{m, 3k}_{p, \left(1-\frac{3q}{2^k \cdot 2^{\frac{m}{2}+1}} \right) n_3} \right \rangle \right|
\end{equation}
We are now ready to apply a time-frequency sparse-uniform dichotomy: we define the \emph{sparse set} of indices in this case by
$$
\sfS_\mu(f_i):= \left\{(p_i, n_i): \left| \left \langle f_i, \phi_{p_i, n_i}^{m, \underline{i}k} \right \rangle \right|^2 \gtrsim 2^{-\mu m} \int_{I_{p_i}^{m, \underline{i}k}} |f_i|^2, \quad p_i \sim 2^{\frac{m}{2}+(\underline{i}-1) k}, n_i \sim 2^{\frac{m}{2}} \right\}, \quad i\in\{2, 3, 4\}.
$$
We also define the \emph{uniform set} of indices $\sfU_\mu (f_i), i\in\{2, 3, 4\}$, and decompose the term $\Lambda^k_{m, j} (\vec{f})$ into its uniform and sparse components $\Lambda^{\sfU_\mu}_{m, j, k} (\vec{f})$ and $\Lambda^{\sfS_\mu}_{m, j, k} (\vec{f})$ as usual. Following now the standard reasonings (see for example Section \ref{20230321subsec01}), we have that the \emph{time-frequency correlation set} in this case is given by
$$
\cTFC:= \left\{ \substack{(p, q) \\ p \sim 2^{\frac{m}{2}+2k}, q \sim 2^{\frac{m}{2}}}: L_2 \left(\frac{p}{2^k}+\frac{q^2}{2^{\frac{m}{2}}} \right)=-\frac{3qn_3}{2^{\frac{m}{2}+1}}, \ L_3 \left(p+\frac{q^3}{2^m} \right)=n_3, \ L_4(p)= \left(1- \frac{3q}{2^k\cdot 2^{\frac{m}{2}+1}} \right) n_3 \right\}.
$$
where $L_i: \left[2^{\frac{m}{2}+(\underline{i}-1) k}, 2^{\frac{m}{2}+(\underline{i}-1) k+1} \right] \cap \Z \mapsto \left[2^{\frac{m}{2}}, 2^{\frac{m}{2}+1} \right] \cap \Z, \ i\in\{2, 3, 4\}$. In this case, the estimates for the sparse/uniform components follow the usual pattern and are given by
$$\left| \Lambda^{\sfU_\mu}_{m, j, k} (\vec{f}) \right| \lesssim 2^{-\frac{\mu m}{2}} \left\| \vec{f} \right\|_{L^{\vec{p}} \left(3I_0^{0, k}\right)}  \quad \textrm{and} \quad \left| \Lambda^{\sfS_\mu}_{m, j, k} (\vec{f}) \right| \lesssim 2^{(\frac{3\mu}{2}-\widetilde{\epsilon}) m} \left\| \vec{f} \right\|_{L^{\vec{p}} \left(3I_0^{0, k}\right)} \quad \forall p \in {\bf H}^{+},$$
for some $\widetilde{\epsilon}>0$ being sufficiently small. As a consequence, we conclude that in this case there exists $\epsilon>0$ such that
$$
\left| \Lambda^k_{m, j} (\vec{f}) \right| \lesssim 2^{-\epsilon m} \left\| \vec{f} \right\|_{L^{\vec{p}} \left(3I_0^{0, k}\right)} \quad \forall \vec{p} \in {\bf H}^{+}.
$$

\subsection{Treatment of Case II: $\frac{m}{2} \le j \le m-100$.} \label{20230321subsec10}
In this situation, we apply Rank I LGC methodology adapted to (all) the three input functions $f_{1,j+k}$, $f_{2,m+2k}$ and $f_{3,m+3k}$ with $m$ playing the role of the dominant parameter.

Indeed, following the usual steps, we apply the Garbor frame decomposition presented in \eqref{20230321eq03} for all the functions $f_1, f_2$ and $f_3$ and substitute them back into \eqref{20230321eq02} in order to deduce that
\begin{eqnarray*}
&& \Lambda^k_{m, j}(\vec{f}) \approx 2^{\frac{3k}{2}} \sum_{\substack{p \sim 2^{\frac{m}{2}+2k} \\ q \sim 2^{\frac{m}{2}}}}  \sum_{\substack{n_3 \sim 2^{\frac{m}{2}} \\ n_1 \sim 2^{j-\frac{m}{2}}}}  \left\langle f_1, \phi^{m, k}_{\frac{p}{2^{2k}}-q, n_1} \right\rangle  \left \langle f_2, \phi_{\frac{p}{2^k}+\frac{q^2}{2^{\frac{m}{2}}}, \frac{2^{\frac{m}{2}}}{2q} n_1-\frac{3q}{2^{\frac{m}{2}+1}}n_3}^{m, 2k} \right\rangle \\
&& \quad \quad \quad \quad \cdot \left\langle f_3, \phi_{p+\frac{q^3}{2^m}, n_3}^{m, 3k} \right \rangle \left \langle f_4, \phi_{p, n_3+\frac{1}{2^k} \cdot \left(\frac{2^{\frac{m}{2}}}{2q} n_1-\frac{3q}{2^{\frac{m}{2}+1}}n_3 \right) +\frac{n_1}{2^{2k}}}^{m, 3k} \right \rangle.
\end{eqnarray*}
Now for each $L_i \sim 2^{m-j}$, define
$$
\Omega_i^{L_i}:= \left[L_i2^{j+\underline{i}k}, (L_i+1)2^{j+\underline{i}k} \right], \quad i\in\{2, 3, 4\}.
$$
where we recall that $\underline{i}=\min \left\{i, 3 \right\}$. In particular, if $n_3 2^{\frac{m}{2}+3k} \in \Omega_3^{L_3}$, then $n_3 \in \widetilde{\Omega^{L_3}_3}:= \left[L_3 2^{j-\frac{m}{2}}, \left(L_3+1 \right) 2^{j-\frac{m}{2}} \right]$. Moreover, we decompose the spatial indices as follows:
$$\left\{ p \sim 2^{\frac{m}{2}+2k} \right\}=\bigcup_{P \sim 2^{m-j+2k}} J_P^k, \quad \textrm{where} \quad J_P^k:=\left[P2^{j-\frac{m}{2}}, (P+1)2^{j-\frac{m}{2}} \right]\qquad\textrm{and}$$
$$\left\{q \sim 2^{\frac{m}{2}} \right\}=\bigcup_{Q \sim 2^{m-j}} J_P^k, \quad \textrm{where} \quad J_Q^k:=\left[Q2^{j-\frac{m}{2}}, (Q+1)2^{j-\frac{m}{2}} \right].$$
Applying the above decomposition and arguing by standard Cauchy-Schwarz and Parseval, we have
\begin{eqnarray*}
&& \left|\Lambda^k_{m, j}(\vec{f}) \right| \lesssim 2^{\frac{3k}{2}} \sum_{\substack{P \sim 2^{m-j+2k} \\ L_3, Q \sim 2^{m-j}}} \left\|f_1 \right\|_{L^2 \left(I_{\frac{P}{2^{2k}}-Q}^{2(m-j), k} \right)} \left\|f_2^{\Omega_2^{-\frac{3Q}{2 \cdot 2^{m-j}} L_3}} \right\|_{L^2 \left(I_{\frac{P}{2^k}+\frac{Q^2}{2^{m-j}}}^{2(m-j), 2k} \right)} \\
&& \quad \quad \quad \quad \cdot \left\|f_3^{\Omega_3^{L_3}} \right\|_{L^2 \left(I_{P+\frac{Q^3}{2^{2(m-j)}}}^{2(m-j), 3k}\right)}  \left\| f_4^{\Omega_4^{ \left(1-\frac{1}{2^k} \cdot \frac{3Q}{2 \cdot 2^{m-j}}  \right) L_3 }}\right\|_{L^2 \left(I_P^{2(m-j), 3k} \right)}\,.
\end{eqnarray*}
Once at this point, we define the \emph{sparse indices} as follows: for $i\in\{2, 3, 4\}$,
$$
\sfS_\mu(f_i):= \left\{ (P_i, L_i): \int_{I_{P_i}^{2(m-j),\underline{i} k}} \left|f_i^{\Omega_i^{L_i}} \right|^2 \gtrsim 2^{-\mu (m-j)} \int_{I_{P_i}^{2(m-j), \underline{i}k }} |f_i|^2, \ \substack{ P_i \sim 2^{m-j+(\underline{i}-1)k} \\ L_i \sim 2^{m-j} } \right\},
$$
Finally, we define the \emph{uniform indices} $\sfU_\mu(f_i)$ in an obvious way and  decompose the term $\left|\Lambda^k_{m, j} \left(\vec{f}\right) \right|$ as usual into its uniform and sparse components, $\Lambda^{\sfU_\mu}_{m, j, k} (\vec{f})$ and $\Lambda^{\sfS_\mu}_{m, j, k} (\vec{f})$, respectively. Standard reasonings by now led us to the definition of the \emph{time-frequency set} as
$$
\cTFC:= \left\{ \substack{(P, Q) \\ P \sim 2^{m-j+2k}, Q \sim 2^{m-j}}: \ \substack{L_2 \left(\frac{P}{2^k}+\frac{Q}{2^{m-j}} \right)=-\frac{3Q}{2 \cdot 2^{m-j}} L_3, \  L_3 \left(P+\frac{Q^3}{2^{2(m-j)}} \right)=L_3 \\ L_4(P)= \left(1-\frac{1}{2^k} \cdot \frac{3Q}{2 \cdot 2^{m-j}}  \right) L_3 } \right\},
$$
where $L_i: \left[2^{\frac{m}{2}+\underline{i} k}, 2^{\frac{m}{2}+\underline{i} k+1} \right] \cap \Z \mapsto \left[2^{m-j}, 2^{m-j+1} \right] \cap \Z, \ i\in\{2, 3, 4\}$. The estimates for both uniform and sparse components are now
$$
\left| \Lambda^{k, \sfU_\mu}_{m, j} (\vec{f}) \right| \lesssim 2^{-\mu(m-j)} \left\| \vec{f} \right\|_{L^{\vec{p}} \left(3I_0^{0, k}\right)}  \quad \textrm{and} \quad \left| \Lambda^{k, \sfS_\mu}_{m, j} (\vec{f}) \right| \lesssim 2^{(3\mu-\widetilde{\epsilon}) (m-j)} \left\| \vec{f} \right\|_{L^{\vec{p}} \left(3I_0^{0, k}\right)}
$$
for any $\vec{p} \in {\bf H}^{+}$ and some absolute $\widetilde{\epsilon}>0$ sufficiently small. As a consequence, we conclude that
$$\left| \Lambda^k_{m, j} (\vec{f}) \right| \lesssim 2^{-\epsilon (m-j)} \left\| \vec{f} \right\|_{L^{\vec{p}} \left(3I_0^{0, k}\right)} \quad \forall \vec{p} \in {\bf H}^{+}.$$

\section{The stationary off-diagonal term $\Lambda^{\not{D},S}$: The case $j=m>l+100$ and $k \ge 0$}\label{stationaryoff3}

In this case, one has
$$
\frac{|\xi|}{2^k} \simeq \frac{|\tau|}{2^{3k}} \gg \frac{|\eta|}{2^{2k}},
$$
and hence
$$
\supp \widehat{f_1} \subseteq \left\{|\xi| \sim 2^{m+k} \right\}, \ \supp \widehat{f_2} \subseteq \left\{|\eta| \sim 2^{l+2k} \right\}, \ \supp \widehat{f_3} \subseteq \left\{|\tau| \sim 2^{m+3k} \right\} \ \textrm{and} \ \supp \widehat{f_4} \subseteq \left\{|\tau| \sim 2^{m+3k} \right\}.
$$
We consider two different cases.
\begin{enumerate}
\item [I.] $0 \le l \le \frac{m}{2}$;
\item [II.] $\frac{m}{2} \le l \le m-100$.
\end{enumerate}

\subsection{Treatment of Case I: $0 \le l \le \frac{m}{2}$.}

The proof of this case follows again from a straightforward modification of the argument in Section \ref{20230321subsec01}. We initiate the sketch of the proof with the decomposition of the main term:
\begin{equation} \label{20230321eq02a}
\Lambda_{m, l, m}^k (\vec{f}):=\Lambda^k_{m, l}(\vec{f})\approx2^k \sum_{\substack{p \sim 2^{\frac{m}{2}+2k} \\ q \sim 2^{\frac{m}{2}}}} \int_{I_p^{m, 3k}} \int_{I_q^{m, k}} f_{1,m+k}(x-t)f_{2,l+2k}(x+t^2)f_{3,m+3k}(x+t^3)f_4(x)dtdx.
\end{equation}
Observe that on each $I_{r}^{m, 2k}$ with $r\sim 2^{\frac{m}{2}+k}$, the function $f_{2, l+2k}$ is morally constant, and hence \eqref{20230321eq02a} can be rewritten as
\begin{equation} \label{20230321eq01a}
\Lambda^k_{m, l}(\vec{f}) \approx 2^k \sum_{\substack{p \sim 2^{\frac{m}{2}+2k} \\ q \sim 2^{\frac{m}{2}}}} \left(\frac{1}{\left|I_0^{m, 2k} \right|} \int_{I^{m, 2k}_{\frac{p}{2^{k}}+\frac{q^2}{2^{\frac{m}{2}}}}} f_2 \right) \int_{I_p^{m, 3k}} \int_{I_q^{m, k}} f_1(x-t)f_3(x+t^3)f_4(x)dtdx.
\end{equation}
Next, we decompose $f_1$ and $f_3$ in Garbor wave-packets adapted to the linearizing scale level:
\begin{equation} \label{20230321eq03a}
f_i \simeq \sum_{\substack{p_i \sim 2^{\frac{m}{2}+(i-1)k} \\ n_i \sim 2^{\frac{m}{2}}}} \left\langle f_i, \phi_{p_i, n_i}^{m, ik}   \right \rangle \phi_{p_i, n_i}^{m, ik},  \quad i= \left\{1, 3 \right\},
\end{equation}
and analyze the double integral in \eqref{20230321eq01a} as usual by extracting the corresponding time-frequency correlations. This gives
\begin{equation}
\left| \Lambda^k_{m, l} (\vec{f}) \right| \lesssim 2^{\frac{3k}{2}} \sum_{\substack{p \sim 2^{\frac{m}{2}+2k} \\ q, n_3 \sim 2^{\frac{m}{2}}}} \left| \left \langle f_1, \phi^{m, k}_{\frac{p}{2^{2k}}-q, \frac{3q^2n_3}{2^m}} \right \rangle \right|  \left\|f_2 \right\|_{L^2 \left(I_{\frac{p}{2^k}+\frac{q^2}{2^{\frac{m}{2}}}}^{m, 2k} \right)} \left| \left \langle f_3, \phi_{p+\frac{q^3}{2^m}, n_3}^{m, 3k} \right \rangle \right| \left | \left \langle f_4, \phi_{p, n_3 \left(1+\frac{3q^2}{2^m} \cdot \frac{1}{2^{2k}} \right)}^{m, 3k} \right \rangle \right|.
\end{equation}
We define the \emph{sparse set} of indices by
$$
\sfS_\mu(f_i):= \left\{(p_i, n_i): \left| \left \langle f_i, \phi_{p_i, n_i}^{m, \underline{i}k} \right \rangle \right|^2 \gtrsim 2^{-\mu m} \int_{I_{p_i}^{m, \underline{i}k}} |f_i|^2, \quad p_i \sim 2^{\frac{m}{2}+(\underline{i}-1) k}, n_i \sim 2^{\frac{m}{2}} \right\}, \quad i\in\{1, 3, 4\}\,,
$$
and as usual the \emph{uniform set} of indices $\sfU_\mu (f_i)$, $\Lambda^{\sfU_\mu}_{m, l, k} (\vec{f})$ and $\Lambda^{\sfS_\mu}_{m, l, k} (\vec{f})$. Following now the standard estimate (for example, as in Section \ref{20230321subsec01}), we see that the \emph{time-frequency correlation set} in this case is given by
$$
\cTFC:= \left\{ \substack{(p, q) \\ p \sim 2^{\frac{m}{2}+2k}, q \sim 2^{\frac{m}{2}}}: L_1 \left(\frac{p}{2^{2k}}-q \right)=\frac{3q^2n_3}{2^m}, \ L_3 \left(p+\frac{q^3}{2^m} \right)=n_3, \ L_4(p)= \left(1+\frac{3q^2}{2^m} \cdot \frac{1}{2^{2k}} \right) n_3 \right\},
$$
where $L_i: \left[2^{\frac{m}{2}+(\underline{i}-1) k}, 2^{\frac{m}{2}+(\underline{i}-1) k+1} \right] \cap \Z \mapsto \left[2^{\frac{m}{2}}, 2^{\frac{m}{2}+1} \right] \cap \Z, \ i\in\{1, 3, 4\}$. The estimates for the sparse/uniform components, in this case, are given by
$$
\left| \Lambda^{k, \sfU_\mu}_{m, l} (\vec{f}) \right| \lesssim 2^{-\frac{\mu m}{2}} \left\| \vec{f} \right\|_{L^{\vec{p}} \left(3I_0^{0, k}\right)}  \quad \textrm{and} \quad \left| \Lambda^{k, \sfS_\mu}_{m, l} (\vec{f}) \right| \lesssim 2^{(\frac{3\mu}{2}-\widetilde{\epsilon}) m} \left\| \vec{f} \right\|_{L^{\vec{p}} \left(3I_0^{0, k}\right)}
$$
for any $\vec{p} \in {\bf H}^{+}$ and some absolute $\widetilde{\epsilon}>0$ sufficiently small. As a consequence, we conclude that in this case
$$
\left| \Lambda^k_{m, l} (\vec{f}) \right| \lesssim 2^{-\epsilon m} \left\| \vec{f} \right\|_{L^{\vec{p}} \left(3I_0^{0, k}\right)} \quad \forall \vec{p} \in {\bf H}^{+}.
$$

\subsection{Treatment of Case II: $\frac{m}{2} \le l \le m-100$.}
In this situation, we apply Rank I LGC methodology adapted to (all) the three input functions $f_{1,m+k}$, $f_{2,l+2k}$ and $f_{3,m+3k}$ again with $m$ playing the role of the dominant parameter.

Thus, applying the Gabor frame decomposition presented in \eqref{20230321eq03a} but now adapted to all the three functions $f_1, f_2$ and $f_3$ and inserting them back into \eqref{20230321eq01a}, standard LGC arguments yield
\begin{eqnarray*}
&& \Lambda^k_{m, l}(\vec{f}) \approx 2^{\frac{3k}{2}} \sum_{\substack{p \sim 2^{\frac{m}{2}+2k} \\ q \sim 2^{\frac{m}{2}}}}  \sum_{\substack{n_3 \sim 2^{\frac{m}{2}} \\ n_2 \sim 2^{l-\frac{m}{2}}}}  \left\langle f_1, \phi^{m, k}_{\frac{p}{2^{2k}}-q, \frac{2qn_2}{2^{\frac{m}{2}}}+\frac{3q^2n_3}{2^m}} \right\rangle  \left \langle f_2, \phi_{\frac{p}{2^k}+\frac{q^2}{2^{\frac{m}{2}}}, n_2}^{m, 2k} \right\rangle \\
&& \quad \quad \quad \quad \cdot \left\langle f_3, \phi_{p+\frac{q^3}{2^m}, n_3}^{m, 3k} \right \rangle \left \langle f_4, \phi_{p, n_3+\frac{n_2}{2^k}+\frac{1}{2^{2k}} \cdot \left(\frac{2qn_2}{2^{\frac{m}{2}}}+\frac{3q^2n_3}{2^m}   \right) }^{m, 3k} \right \rangle.
\end{eqnarray*}
To treat the above term, we first decompose both spatial and frequency indices into small intervals of length $2^{l-\frac{m}{2}}$ (as in Section \ref{20230321subsec10}) and apply Cauchy-Schwarz and Parseval in order to get
\begin{eqnarray*}
&& \left|\Lambda^k_{m, l}(\vec{f}) \right| \lesssim 2^{\frac{3k}{2}} \sum_{\substack{P \sim 2^{m-l+2k} \\ L_3, Q \sim 2^{m-l}}} \left\|f^{\Omega_1^{\frac{3Q^2}{2^{m-l}}L_3}}_1 \right\|_{L^2 \left(I_{\frac{P}{2^{2k}}-Q}^{2(m-l), k} \right)} \left\|f_2 \right\|_{L^2 \left(I_{\frac{P}{2^k}+\frac{Q^2}{2^{m-l}}}^{2(m-l), 2k} \right)} \\
&& \quad \quad \quad \quad \cdot \left\|f_3^{\Omega_3^{L_3}} \right\|_{L^2 \left(I_{P+\frac{Q^3}{2^{2(m-l)}}}^{2(m-l), 3k}\right)}  \left\| f_4^{\Omega_4^{ \left(1+\frac{1}{2^k} \cdot \frac{3Q}{2 \cdot 2^{m-l}}  \right) L_3 }}\right\|_{L^2 \left(I_P^{2(m-l), 3k} \right)},
\end{eqnarray*}
where $\Omega_i^{L_i}:=\left[L_i2^{j+\underline{i}k}, (L_i+1)2^{j+\underline{i}k} \right], \quad i\in\{1, 3, 4\}$. Therefore, we define the \emph{sparse indices} in this case as follows: for $i\in\{1, 3, 4\}$,
$$
\sfS_\mu(f_i):= \left\{ (P_i, L_i): \int_{I_{P_i}^{2(m-l),\underline{i} k}} \left|f_i^{\Omega_i^{L_i}} \right|^2 \gtrsim 2^{-\mu (m-l)} \int_{I_{P_i}^{2(m-l), \underline{i}k }} |f_i|^2, \ \substack{ P_i \sim 2^{m-l+(\underline{i}-1)k} \\ L_i \sim 2^{m-l} } \right\},
$$
as well as the \emph{uniform indices} $\sfU_\mu(f_i)$ in an obvious way. Finally, we decompose the term $\left|\Lambda^k_{m, j} \left(\vec{f}\right) \right|$ as usual into a uniform, $\Lambda^{\sfU_\mu}_{m, l, k} (\vec{f})$, and a sparse component $\Lambda^{\sfS_\mu}_{m, l, k} (\vec{f})$. Standard reasonings convey that in this case the \emph{time-frequency correlation set} is given by
$$
\cTFC:= \left\{ \substack{(P, Q) \\ P \sim 2^{m-l+2k}, Q \sim 2^{m-l}}: \ \substack{L_1 \left(\frac{P}{2^{2k}}+Q\right)=\frac{3Q}{2^{m-l}} L_3, \  L_3 \left(P+\frac{Q^3}{2^{m-l}} \right)=L_3 \\ L_4(P)= \left(1+\frac{1}{2^k} \cdot \frac{3Q}{2 \cdot 2^{m-l}}  \right) L_3 } \right\},
$$
where $L_i: \left[2^{m-l+(\underline{i}-1) k}, 2^{m-l+(\underline{i}-1) k+1} \right] \cap \Z \mapsto \left[2^{m-l}, 2^{m-l+1} \right] \cap \Z, \ i\in\{1, 3, 4\}$.

The estimates for both uniform and sparse components can be summarized as: $\exists\:\widetilde{\epsilon}>0$ such that
$$
\left| \Lambda^{k, \sfU_\mu}_{m, l} (\vec{f}) \right| \lesssim 2^{-\mu(m-l)} \left\| \vec{f} \right\|_{L^{\vec{p}} \left(3I_0^{0, k}\right)}  \quad \textrm{and} \quad \left| \Lambda^{k, \sfS_\mu}_{m, l} (\vec{f}) \right| \lesssim 2^{(3\mu-\widetilde{\epsilon}) (m-l)} \left\| \vec{f} \right\|_{L^{\vec{p}} \left(3I_0^{0, k}\right)} \quad \forall\: \vec{p} \in {\bf H}^{+}\,.
$$
 As a consequence, we conclude that in this case
$$
\left| \Lambda^k_{m, l} (\vec{f}) \right| \lesssim 2^{-\epsilon (m-l)} \left\| \vec{f} \right\|_{L^{\vec{p}} \left(3I_0^{0, k}\right)} \quad \forall\: \vec{p} \in {\bf H}^{+}.
$$

\section{Treatment of the low oscillatory component $\Lambda^{Lo}$} \label{20230412Sec01}

Recall that
$$
\Lambda_{j, l, m}^k (\vec{f}):=\iint_{\R^2} f_{1, j}^k(x-t) f_{2, l}^k(x+t^2) f_{3, m}^k(x+t^3)f_4(x) 2^k \rho(2^k t)dt,
$$
where $\varrho \in C_0^\infty(\R)$ is an odd function with $\supp\,\varrho \subset \left[-4, 4\right] \backslash \left[-\frac{1}{4}, \frac{1}{4} \right]$, and
$$
f^k_{1, j}=\calF^{-1} \left( \widehat{f_1} \phi \left(\frac{\cdot}{2^{j+k}} \right) \right), \quad f^k_{2, l}=\calF^{-1} \left( \widehat{f_2} \phi \left(\frac{\cdot}{2^{l+2k}} \right) \right), \quad \textrm{and} \quad f^k_{3, m}=\calF^{-1} \left( \widehat{f_3} \phi \left(\frac{\cdot}{2^{m+3k}} \right) \right).
$$
One can further rewrite the above expression as
$$
\Lambda_{j, l, m}^k (\vec{f})=\int_{\R} \int_{\R^3} \widehat{f_1}(\xi)  \widehat{f_2}(\eta) \widehat{f_3}(\tau) \frakm_{j, l, m}^k(\xi, \eta, \tau) e^{ix(\xi+\eta+\tau)} f_4(x) d\xi d\eta d\tau dx,
$$
where
$$
\frakm_{j, l, m}^k \left(\xi, \eta, \tau \right)= \left(\int_{\R} e^{i \left(-\frac{\xi}{2^k}t+\frac{\eta}{2^{2k}} t^2+\frac{\tau}{2^{3k}} t^3 \right)} \rho(t) dt \right) \phi \left(\frac{\xi}{2^{j+k}} \right) \phi \left(\frac{\eta}{2^{l+2k}} \right) \phi \left(\frac{\tau}{2^{m+3k}} \right)\,,
$$
and $\phi$ is an even smooth function with $\supp \ \phi \subset \left[-4, 4 \right] \backslash \left[-\frac{1}{4}, \frac{1}{4} \right]$.

Our goal in this section is to prove the following result:

\begin{thm} \label{20230326lem01}

With the above notations, define the low oscillatory component by
$$
\Lambda^{Lo}:= \sum_{k \in \Z} \sum_{(j, l, m) \in \Z^3 \backslash \N^3} \Lambda_{j, l, m}^k.
$$
Then, for any $\vec{p}\in \textbf{B}$, we have
$$
\left|\Lambda^{Lo}(\vec{f})\right| \lesssim \left\|\vec{f}\, \right\|_{L^{\vec{p}}}\,.
$$
\end{thm}

\begin{rem} In what follows, for the simplicity of the exposition we will only provide the proof for the slightly more restrictive range
$\textbf{\underline{B}}:=\left\{(p_1, p_2, p_3, p_4)\:|\:\frac{1}{p_1}+\frac{1}{p_2}+\frac{1}{p_3}+\frac{1}{p_4}=1\:\textrm{with}\:1<p_1, p_2, p_3, p_4<\infty\right\}$.
However, one can extend the above to cover the full range ${\bf B}$ by using more elaborate paraproducts arguments in the spirit of the Coifman-Meyer Theorem (see \cite{CM78}, \cite{CM97} and also \cite{MS13}).
\end{rem}

\begin{proof}

We consider several different cases.

\subsection{Case 1: All the three indices $j,l,m$ are negative.}\label{case1low}

In this case, we denote
\begin{equation} \label{20230325eq02}
\Lambda_{-, -, -} (\vec{f})=\sum_{k \in \Z} \Lambda_{-, -, -}^k (\vec{f})= \sum_{k \in \Z} \int_{\R} \int_{\R^3} \widehat{f_1}(\xi)  \widehat{f_2}(\eta) \widehat{f_3} (\tau) \frakm_{-. -, -}^k(\xi, \eta, \tau) e^{ix(\xi+\eta+\tau)} f_4(x) d\xi d\eta d\tau dx,
\end{equation}
where
$$
\frakm_{-,-,-}^k \left(\xi, \eta, \tau \right):= \left(\int_{\R} e^{i \left(-\frac{\xi}{2^k}t+\frac{\eta}{2^{2k}} t^2+\frac{\tau}{2^{3k}} t^3 \right)} \rho(t) dt \right) \phi_{-} \left(\frac{\xi}{2^k} \right) \phi_{-} \left(\frac{\eta}{2^{2k}} \right) \phi_{-} \left(\frac{\tau}{2^{3k}} \right),
$$
and $\phi_{-}$ is an even smooth function with $\supp \ \phi_{-} \subseteq [-5, 5]$.

We now apply a Taylor expansion argument to write
\begin{equation*}
\frakm_{-, -, -}^k \left(\xi, \eta, \tau \right) =\sum_{\ell_1, \ell_2, \ell_3 \ge 0} \frac{C_{\ell_1, \ell_2, \ell_3}}{\ell_1! \ell_2! \ell_3!}  \left(\frac{\xi}{2^k}  \right)^{\ell_1} \phi_{-} \left(\frac{\xi}{2^k} \right) \left(\frac{\eta}{2^{2k}}  \right)^{\ell_2}\phi_{-} \left(\frac{\eta}{2^{2k}} \right) \left(\frac{\tau}{2^{3k}}  \right)^{\ell_3} \phi_{-} \left(\frac{\tau}{2^{3k}} \right) \int_{\R} t^{\ell_1+2\ell_2+3\ell_3} \rho(t)dt\,,
\end{equation*}
where here $|C_{\ell_1, \ell_2, \ell_3}|\lesssim 1$ is allowed to change from line to line.

For simplicity, we denote by $\Phi_{-}^{\ell_1}(\xi):=(\xi)^{\ell_1} \phi_{-}(\xi)$ and $\Psi_{-}^{\ell_i}:=\calF^{-1} \left(\Phi_{-}^{\ell_i} \right),\: i\in\{1, 2, 3\}$. Therefore, letting $f_{i, \ell_i, ik}:=f_i * \left( \Psi_{-}^{\ell_i} \right)_{ik}$ and $\Psi_k(x):=2^{k} \Psi (2^{k}x), k \in \Z$, we rewrite \eqref{20230325eq02} as
\begin{eqnarray} \label{20230325eq11}
&&\Lambda_{-, -, -} (\vec{f})= \sum_{\ell_1, \ell_2, \ell_3 \ge 0} \Lambda_{-, -, -}^{\ell_1, \ell_2, \ell_3} (\vec{f}) \nonumber \\
&& \quad :=  \sum_{\ell_1, \ell_2, \ell_3 \ge 0} \frac{C_{\ell_1, \ell_2, \ell_3}}{\ell_1! \ell_2! \ell_3!} \sum_{k \in \Z} \int_{\R} f_{1, \ell_1, k}(x)
f_{2, \ell_2, 2k}(x)f_{3, \ell_3, 3k}(x)f_4(x) \left(  \int_{\R} t^{\ell_1+2\ell_2+3\ell_3} \rho(t)dt \right) dx\,.
\end{eqnarray}

\begin{obs}
It suffices to consider the case when $l_1+l_3$ is odd, as otherwise, the function $t^{\ell_1+2\ell_2+3\ell_3}  \rho(t)$ is odd, which makes the $t$-integral in \eqref{20230325eq11} be $0$ due to the mean zero condition satisfied by $\rho$.
\end{obs}

We first treat the case when $(\ell_1, \ell_2, \ell_3)$ is different from both $(1, 0, 0)$ and $(0, 0, 1)$. Observe that in this situation, there are at least two $\ell$'s in $\{\ell_1, \ell_2, \ell_2\}$ taking strictly positive values. Without loss of generality, we may assume $\ell_1, \ell_2>0$ which further implies $\int_{\R} \Phi_{-}^{\ell_i}(x)dx=0,\: i\in\{1, 2\}$. Therefore, via Cauchy--Schwarz
\begin{eqnarray*}
\left| \sum_{\substack{\ell_1, \ell_2, \ell_3 \ge 0 \\ (\ell_1, \ell_2, \ell_3) \neq (1, 0, 0)\\ (\ell_1, \ell_2, \ell_3) \neq (0, 0, 1)}} \Lambda^{\ell_1, \ell_2, \ell_3}_{-, -, -} (\vec{f})  \right|\lesssim \sum_{\ell_1, \ell_2, \ell_3 \ge 0} \frac{|C_{\ell_1, \ell_2, \ell_3}|}{\ell_1! \ell_2! \ell_3! } \int_{\R} \left( \sum_{k \in \Z} \left|f_{1, \ell_1, k} (x) \right|^2 \right)^{\frac{1}{2}} \left( \sum_{k \in \Z} \left|f_{2, \ell_2, 2k} (x) \right|^2 \right)^{\frac{1}{2}} M f_3(x)|f_4(x)| dx \\
\lesssim  \sum_{\ell_1, \ell_2, \ell_3 \ge 0} \frac{|C_{\ell_1, \ell_2, \ell_3}|}{\ell_1! \ell_2! \ell_3! } \int_{\R} Sf_1(x) Sf_2(x) M f_3(x) |f_4(x)| dx \lesssim \|\vec{f}\|_{L^{\vec{p}}},\qquad\qquad\qquad
\end{eqnarray*}
where in the first inequality of the second line in the above display, for expository reasons, we ignore the $l_j-$dependencies and thus allow $Sf$ and $Mf$ stand for the classical square function and Hardy-Littlewood maximal function, respectively (this slight notational abuse is safe due to the  $l_j !$ term decay).
\medskip

We are now left with the case $(\ell_1, \ell_2, \ell_3)=(1, 0, 0)$ or $(0, 0, 1)$. Due to the symmetry it suffices to consider the case $(\ell_1, \ell_2, \ell_3)=(1, 0, 0)$. To begin with, we have
$$
\left|\Lambda_{-, -, -}^{1, 0, 0} (\vec{f}) \right| \le  \left| \Lambda_{-, -, -;+}^{1, 0, 0} (\vec{f}) \right|\,+\, \left| \Lambda_{-, -, -;-}^{1, 0, 0} (\vec{f}) \right|,
$$
where in the RHS above the first term corresponds to the summation $\sum_{k\in\N}$ while the second term to the summation $\sum_{k\in \Z_{-}}$.

The second term is easy since then one can follow the earlier standard arguments in order to deduce that
\begin{equation*}
\left| \Lambda_{-, -, -;-}^{1, 0, 0} (\vec{f}) \right|\lesssim \int_{\R} Sf_1(x) Mf_2(x) M f_3(x) S f_4(x) dx \lesssim \|\vec{f}\|_{L^{\vec{p}}}\,.
\end{equation*}

The treatment of the first term is more delicate: we start with the following notational simplification: in what follows we will refer to $\phi$ and $\psi$ as two generic functions in $C_0^{\R}$ with, say $\textrm{supp}\,\phi,\,\textrm{supp}\,\psi\subset\{|\xi|<1\}$, and $0\notin \textrm{supp}\,\phi$ (one could allow the more general condition $\phi(0)=0$). With these notations the multiplier for the first term is given by
\begin{equation*}
\sum_{k=0}^{\infty} \phi(\frac{\xi}{2^k})\,\psi(\frac{\eta}{2^{2k}})\,\psi(\frac{\tau}{2^{3k}})\,.
\end{equation*}
Now, via an Abel summation argument, the above expression can be reduced up to terms that can be treated with the standard tools above to the situation
\begin{equation*}
\sum_{k=0}^{\infty} \psi(\frac{\xi}{2^k})\,\phi(\frac{\eta}{2^{2k}})\,\psi(\frac{\tau}{2^{3k}})\simeq \sum_{k=0}^{\infty} \psi(\frac{\xi}{2^k})\,\phi(\frac{\eta}{2^{2k}})\,\sum_{n=0}^{\infty}\phi(\frac{\tau}{2^{3k-n}})\,.
\end{equation*}
Once again the latter expression can be reduced up to easily controlled terms to\footnote{Strictly speaking in the RHS below we should have an expression of the form $\sum_{n=0}^{\infty}c_n\phi(\frac{\tau}{2^n})
\sum_{k=\lfloor\frac{n}{3}\rfloor}^{\lfloor\frac{n}{2}\rfloor} \psi(\frac{\xi}{2^k})\,\phi(\frac{\eta}{2^{2k}})$ for some suitable choice of $c_n\in\{0,1\}$ but this has no relevance for our later treatment so we will ignore the presence of $\{c_n\}_{n\geq 0}$.}
\begin{equation*}
\sum_{k=0}^{\infty} \psi(\frac{\xi}{2^k})\,\phi(\frac{\eta}{2^{2k}})\,\sum_{n=0}^{k}\phi(\frac{\tau}{2^{3k-n}})\approx\sum_{n=0}^{\infty}\phi(\frac{\tau}{2^n})
\sum_{k=\lfloor\frac{n}{3}\rfloor}^{\lfloor\frac{n}{2}\rfloor} \psi(\frac{\xi}{2^k})\,\phi(\frac{\eta}{2^{2k}})\,.
\end{equation*}

With this, we can now deduce that up to simple(r) expressions the term  $\left| \Lambda_{-, -, -;+}^{1, 0, 0} (\vec{f}) \right|$
is controlled by
\begin{equation}\label{Domlow}
\int_{\R}\sum_{n=0}^{\infty} (f_3*\check{\phi}_n)(x)\, (f_4*\check{\phi}_n)(x)\,
\left(\sum_{k=\lfloor\frac{n}{3}\rfloor}^{\lfloor\frac{n}{2}\rfloor} (f_1*\check{\psi}_k)(x)\, (f_2*\check{\phi}_{2k})(x)\right)\,dx\,.
\end{equation}
Define now the maximal (anisotropic) paraproduct operator
\begin{equation}\label{maxpara}
\mathcal{P}_{\mathfrak{M}}(f_1,f_2)(x):=\sup_{N\in\N}
\left|\sum_{k=0}^{N} (f_1*\check{\psi}_k)(x)\, (f_2*\check{\phi}_{2k})(x)\right|\,.
\end{equation}
Then, from \eqref{Domlow} and \eqref{maxpara}, we deduce that for any $\vec{p}\in\textbf{\underline{B}}$
\begin{equation}\label{lowtermcont}
\left| \Lambda_{-, -, -;+}^{1, 0, 0} (\vec{f}) \right|\lesssim_{\vec{p}} \|\vec{f}\|_{L^{\vec{p}}}\,+\, \int_{\R} \mathcal{P}_{\mathfrak{M}}(f_1,f_2)(x)\, Sf_3(x)\,Sf_4(x)\,dx\,.
\end{equation}
The desired conclusion follows now from the claim that $\mathcal{P}_{\mathfrak{M}}$ is bounded from $L^{p_1}(\R)\times L^{p_2}(\R)\,\rightarrow\,L^r(\R)$ for any $\frac{1}{p_1}+\frac{1}{p_2}=\frac{1}{r}$ with $1<p_1,p_2\leq \infty$ and $\frac{1}{2}<r<\infty$. Indeed, if one is willing to sacrifice the end-points corresponding to $p_2=\infty$ one can come up with a very simple and elegant argument for this last claim whose sketch we present now\footnote{This nice argument was generously communicated to us by C. Muscalu and it likely goes back to (at least) the body of work of Coifman and Meyer. In order to cover the full boundedness range of $\mathcal{P}_{\mathfrak{M}}$ one needs to use more sophisticated tools, as for example in the spirit of the work in \cite{MTT06}.}: we first write
\begin{equation}\label{maxpar1}
\|\mathcal{P}_{\mathfrak{M}}(f_1,f_2)\|_{L^r}\lesssim \|\mathcal{M}(\mathcal{P}(f_1,f_2))\|_{L^r}=\|\mathcal{P}(f_1,f_2)\|_{H^r}
\end{equation}
where by definition we set $\mathcal{M}(f):=\sup_{k\in\Z} |f*\check{\psi}_k|$ assuming here wlog that $\psi(0)=1$, $\mathcal{P}(f_1,f_2)(x):=\sum_{k=0}^{\infty}(f_1*\check{\psi}_k)(x)\, (f_2*\check{\phi}_{2k})(x)$ and $H^r$ stands for the Hardy space of index $r$ over $\R$. We next have
\begin{equation}\label{maxpar2}
\|\mathcal{P}(f_1,f_2\|_{H^r}\approx_r \|S(\mathcal{P}(f_1,f_2))\|_{L^r}\lesssim \|M(f_1)\, S(f_2)\|_{L^r}\lesssim
\|M(f_1)\|_{L^{p_1}}\,\|S(f_2)\|_{L^{p_2}}\lesssim_{p_1,p_2}\|f_1\|_{L^{p_1}}\,\|f_2\|_{L^{p_2}}\,,
\end{equation}
ending thus the proof of our claim.

\subsection{Case 2: Two of the indices are negative and one is positive.}

Without loss of generality we may assume that $j, l<0$ and $m \ge 0$ as the other two cases can be treated in a similar fashion. Then, one has
\begin{equation*}
\Lambda_{-, -, +} (\vec{f})=\sum_{m \ge 0} \sum_{k \in \Z} \Lambda_{-, -, m}^k (\vec{f})
=\sum_{m \ge 0} \sum_{k \in \Z} \int_{\R} \int_{\R^3} \widehat{f_1}(\xi)  \widehat{f_2}(\eta) \widehat{f_3} (\tau) \frakm_{-. -, m}^k(\xi, \eta, \tau) e^{ix(\xi+\eta+\tau)} f_4(x) d\xi d\eta d\tau dx,
\end{equation*}
where
$$
\frakm_{-, -, m}^k(\xi, \eta, \tau):=\left(\int_{\R} e^{i \left(-\frac{\xi}{2^k}t+\frac{\eta}{2^{2k}} t^2+\frac{\tau}{2^{3k}} t^3 \right)} \rho(t) dt \right) \phi_{-} \left(\frac{\xi}{2^k} \right) \phi_{-} \left(\frac{\eta}{2^{2k}} \right) \phi\left(\frac{\tau}{2^{m+3k}} \right).
$$
Applying a standard Taylor series expansion together with an integration by parts argument, we have
\begin{eqnarray*}
&& \frakm_{-, -, m}^k \left(\xi, \eta, \tau \right)=\sum_{\ell_1, \ell_2 \ge 0} \frac{C_{\ell_1, \ell_2}}{\ell_1! \ell_2!} \left(\frac{\xi}{2^k} \right)^{\ell_1} \phi_{-} \left(\frac{\xi}{2^k} \right) \left(\frac{\eta}{2^{2k}} \right)^{\ell_2} \phi_{-} \left(\frac{\eta}{2^{2k}} \right)  \phi \left(\frac{\tau}{2^{m+3k}} \right) \int_{\R} e^{i \frac{\tau}{2^{3k}} t^3} t^{\ell_1+2\ell_2} \rho(t)dt \\
&& \quad  \quad  \quad \approx \frac{1}{2^m} \sum_{\ell_1, \ell_2 \ge 0} \frac{C_{\ell_1, \ell_2}}{\ell_1! \ell_2!} \left(\frac{\xi}{2^k} \right)^{\ell_1} \phi_{-} \left(\frac{\xi}{2^k} \right) \left(\frac{\eta}{2^{2k}} \right)^{\ell_2} \phi_{-} \left(\frac{\eta}{2^{2k}} \right)  \phi \left(\frac{\tau}{2^{m+3k}} \right).
\end{eqnarray*}
Once here, the proof follows a similar path with the one provided for the case $j, l, m<0$, and thus we will provide below only a sketch. Recalling that $\calF^{-1} \phi$ has mean zero, we consider two subcases:
\begin{enumerate}
    \item [$\bullet$] If $(\ell_1, \ell_2) \neq (0, 0)$ we may assume without loss of generality that $\ell_1>0$ and hence $\Psi_{-}^{\ell_1}$ has mean zero; therefore, for any $\vec{p}\in \textbf{\underline{B}}$ and with the same notational abuse as before, we have
    $$
    \left|\Lambda_{-, -, +} (\vec{f}) \right| \lesssim \sum_{m \ge 0} \frac{1}{2^m} \sum_{\ell_1, \ell_2>0} \frac{|C_{\ell_1, \ell_2}|}{\ell_1! \ell_2!} \int_{\R} Sf_1(x)\,Mf_2(x)\,Sf_3(x)\, |f_4(x)| dx \lesssim \|\vec{f}\|_{L^{\vec{p}}}\,.
    $$
    \item [$\bullet$] If $(\ell_1, \ell_2)=(0, 0)$ then one can mimic the arguments provided in the second half of Section \ref{case1low}.
\end{enumerate}

\subsection{Case 3: One index is negative and two are positive.}\label{20230326subsec01}

Again, without loss of generality we may assume that $j, l \ge 0$ and $m<0$ as the other two cases can be treated in a similar fashion. Then, one has

\begin{equation} \label{20230330eq01}
\Lambda_{+, +, -} (\vec{f})=\sum_{j, l \ge 0} \sum_{k \in \Z} \Lambda_{j, l, -}^k (\vec{f})
=\sum_{j, l \ge 0} \sum_{k \in \Z} \int_{\R} \int_{\R^3} \widehat{f_1}(\xi)  \widehat{f_2}(\eta) \widehat{f_3} (\tau) \frakm_{j, l, -}^k(\xi, \eta, \tau) e^{ix(\xi+\eta+\tau)} f_4(x) d\xi d\eta d\tau dx,
\end{equation}
where
$$
\frakm_{j, l, -}^k(\xi, \eta, \tau):=\left(\int_{\R} e^{i \left(-\frac{\xi}{2^k}t+\frac{\eta}{2^{2k}} t^2+\frac{\tau}{2^{3k}} t^3 \right)} \rho(t) dt \right) \phi \left(\frac{\xi}{2^{j+k}} \right) \phi \left(\frac{\eta}{2^{l+2k}} \right) \phi_{-}\left(\frac{\tau}{2^{3k}} \right),
$$
which by a standard Taylor series argument can be written as
$$
\sum_{\ell_3 \ge 0} \frac{C_{\ell_3}}{\ell_3!} \left( \frac{\tau}{2^{3k}} \right)^{\ell_3} \phi_{-}\left(\frac{\tau}{2^{3k}} \right)   \phi \left(\frac{\xi}{2^{j+k}} \right) \phi \left(\frac{\eta}{2^{l+2k}} \right) \int_{\R} e^{i \left(-\frac{\xi}{2^k}t+\frac{\eta}{2^{2k}} t^2\right)} \rho_{\ell_3}(t) dt,
$$
where $\rho_{\ell_3}(t):=t^{3\ell_3} \rho(t)$.

Plugging the above expression back to \eqref{20230330eq01}, we see that
\begin{eqnarray*}
&& \left| \Lambda_{+, +, -} (\vec{f}) \right| \lesssim \sum_{\ell_3 \ge 0} \frac{|C_{\ell_3}|}{\ell_3!} \int_{\R} \left| \int_{\R} \widehat{f_3}(\tau) \left(\frac{\tau}{2^{3k}} \right)^{\ell_3} \phi_{-} \left(\frac{\tau}{2^{3k}} \right) e^{ix \tau} d\tau \right| |f_4(x)|\\
&&\left| \sum_{\substack{j, l  \ge 0 \\ k \in \Z}} \int_{\R^2} \widehat{f_1}(\xi) \widehat{f_2}(\eta) \phi \left(\frac{\xi}{2^{j+k}} \right) \phi \left(\frac{\eta}{2^{l+2k}} \right) \left(\int_{\R} e^{i \left(-\frac{\xi}{2^k}t+\frac{\eta}{2^{2k}} t^2 \right)} \rho_{\ell_3}(t) dt \right) e^{ix(\xi+\eta)} d\xi d\tau \right|\, dx \\
&& \quad \lesssim \sum_{\ell_3 \ge 0} \frac{|C_{\ell_3}|}{\ell_3!} \int_{\R} \left|B_{\ell_3}(f_1, f_2)(x) \right|\,M f_3(x)\, |f_4(x)|\,dx \lesssim \|\vec{f}\|_{L^{\vec{p}}},
\end{eqnarray*}
where
$$
B_{\ell_3}(f_1, f_2)(x):=\sum_{\substack{j, l  \ge 0 \\ k \in \Z}} \int_{\R^2} \widehat{f_1}(\xi) \widehat{f_2}(\eta) \phi \left(\frac{\xi}{2^{j+k}} \right) \phi \left(\frac{\eta}{2^{l+2k}} \right) \left(\int_{\R} e^{i \left(-\frac{\xi}{2^k}t+\frac{\eta}{2^{2k}} t^2 \right)} \rho_{\ell_3}(t) dt \right) e^{ix(\xi+\eta)} d\xi d\tau
$$
is the stationary component of the \emph{curved bilinear Hilbert transform along the curve} $(-t, \gamma(t))$ with $\gamma(t)=t^2$ whose complete $L^p$ boundedness range is known and has been investigated in a number of papers, see \textit{e.g.} \cite{Li13}, \cite{Lie15}, \cite{LX16}, \cite{Lie18}, \cite{GL20}, \cite{GL22}.
\end{proof}

\section{The diagonal term $\Lambda^{D}$: An outline of the case $k<0$}\label{Negative}

In this section we focus on the treatment of the main diagonal term for the situation when $k$ is negative. In contrast with what we have seen for the off diagonal case in the previous section and consistent with the story line from Section \ref{MaindiagKpositive} treating the main diagonal term for $k\geq 0$, the crux of the approach presented in this section relies on the newly developed Rank II LGC methodology.

That being said, as expected, the case $k<0$ mirrors in all the key points the case $k\geq 0$ discussed in great detail in Section \ref{MaindiagKpositive}. For this reason and due to space limitation concerns we will only insist on the most important modifications for treating the current negative case. In this context, with the usual notations and conventions, taking $m\in\N,\,k\in\Z_{-}$ and
\begin{equation} \label{20230330eq02}
\Lambda_m^k (\vec{f})=\int_{\R} \int_\R f_{1,m+k}(x-t)\, f_{2,m+2k}(x+t^2) \,f_{3,m+3k}(x+t^3)\, f_4(x)\, 2^k \rho(2^{k} t)\, dt\,dx\,,
\end{equation}
we notice the following interrelated key differences relative to the $k\geq 0$ case:
\begin{itemize}
\item the dominant frequency regime is imposed by $f_1$ and hence we now have $\supp \ \widehat{f_4} \subseteq \left\{|\xi| \sim 2^{m+k} \right\}$;

\item the dominant spatial interval in the $x-$variable becomes $I^{3k}:=I_0^{0, 3k}$ instead of  $I^k:=I_0^{0, k}$ for $k\geq 0$;

\item the roles of $f_1$ and $f_3$ are to be interchanged.
\end{itemize}

With these being said, in what follows we will provide just an outline for the proof of the following result that is the natural analogue of Theorem \ref{mainthm-2022}:

\begin{thm} \label{20230330thm01}
There exists some $\underline{\epsilon_1}>0$, such that for any $k<0$, $m \in\N$ and
$$\forall\:\:\vec{p}\in {\bf H}^{-}:=\left\{\vec{p} \in \overline{{\bf B}} \ | \ p_3=2 \ \& \ (p_1, p_2, p_4) \in \{2, \infty\} \right\}$$
the following holds:
\begin{equation}\label{keyestimneg}
\left|\Lambda_m^{k, I^{3k}}(\vec{f})\right| \lesssim 2^{-\underline{\epsilon_1} \min\{-k,\frac{m}{2}\}} \,\|\vec{f}\,\|_{L^{\vec{p}}(3\, I^{3k})}\,.
\end{equation}
\end{thm}

\subsection{The case $-k \ge \frac{m}{2}$.}

\subsubsection{The correlative time-frequency model} We initiate our approach by applying the change of variable $t^3 \to t$ in the $t$-integral above--see \eqref{20230330eq02}, which, after the standard partitioning of the $x-$domain, gives\footnote{With a slight notational abuse, we are preserving the $\rho$ notation in \eqref{20230330eq03} though this stands for  $\frac{1}{3} \cdot \frac{2^{-2k}}{t^{\frac{2}{3}}} \cdot \rho \circ t^{\frac{1}{3}}$ relative to the original notation.}
\begin{equation} \label{20230330eq03}
\Lambda_m^k (\vec{f})\approx\int_{I^{3k}} \int_{\R} f_1(x-t^{\frac{1}{3}})\, f_2 \left(x+t^{\frac{2}{3}} \right) f_3(x+t)\, f_4(x) \,2^{3k} \rho(2^{3k} t)\, dt\,dx.
\end{equation}
Following the strategy presented in Section \ref{IIPrelim}, we apply a Gabor frame decomposition of the input $f_1$ in order to obtain
$$
\Lambda_m^k (\vec{f}):=\sum_{\substack{p \sim 2^{\frac{m}{2}-2k} \\ q \sim 2^{\frac{m}{2}}}} \int_{I_p^{m, k}}  \int_{I_q^{m, 3k}} f_3(x+t)f_2(x+t^{\frac{2}{3}}) \left(\sum_{\substack{r \sim 2^{\frac{m}{2}-2k} \\ n \sim 2^{\frac{m}{2}}}} \left \langle f_1, \phi_{r, n}^k \right \rangle \phi_{r, n}^k \left(x-t^\frac{1}{3}\right) \right) f_4(x) 2^{3k} \varrho(2^{3k}t)dtdx.
$$
Next, applying a scale linearization argument in the $t-$parameter as dictated by the LGC method we deduce a first spatial correlation in the form of
$$
p \cong r+\left(2^m q \right)^{\frac{1}{3}}\,,
$$
together with the expression
\begin{equation} \label{20230407eq10}
\Lambda_m^k (\vec{f}) \sim 2^{3k} \sum_{\substack{p \sim 2^{\frac{m}{2}-2k} \\ q, n \sim 2^{\frac{m}{2}}}} \left \langle f_1, \phi_{p-(2^m q)^{\frac{1}{3}}, n}^k \right \rangle \int_{I_p^{m, 3k}} f_4(x) \calJ_{q, p+(2^m q)^{\frac{1}{3}}, n}^k (f_3, f_2)(x)dx,
\end{equation}
where
$$
\calJ_{q, p+(2^mq)^{\frac{1}{3}}, n}^k(f_3, f_2)(x):=\left(\int_{I_q^{m, 3k}} f_3(x+t)f_2(x+t^{\frac{2}{3}}) e^{-\frac{1}{3}i \cdot 2^{\frac{m}{2}+3k} t \left(\frac{q}{2^{\frac{m}{2}}} \right)^{-\frac{2}{3}}n} dt \right) \phi_{p, n}^{ k}(x) e^{-\frac{2}{3}i n \cdot (2^m q)^{\frac{1}{3}}}.
$$

Once at this point, following the spirit in Section \ref{IIlgcklarge}, we obtain the analogue of  Key Heuristic 1 therein:
$\newline$

\noindent\textbf{Key Heuristic 1$^\prime$.} \textit{Since $\supp \ \widehat{f_i} \subseteq \left[2^{m+ik}, 2^{m+ik+1} \right]$ for $i\in\{2,3\}$, via Heisenberg uncertainty principle, we have that
\begin{enumerate}
\item $f_2$ is morally a constant on intervals of length $2^{-m-2k}$;
\item $f_3$ is morally a constant on intervals of length $2^{-m-3k}$.
\end{enumerate}
In this new setting, the assumption $-k \ge \frac{m}{2}$ guarantees that $|I_p^{m, k}|=2^{-\frac{m}{2}-k} \le 2^{-m-2k} \le 2^{-m-3k}$. As a consequence, at the moral level, we have that the mapping
\begin{equation}\label{Jqnk}
x \mapsto J_{q, n}^k(f_3, f_2):=\int_{I_q^{m, 3k}} f_3(x+t)f_2(x+t^{\frac{2}{3}}) e^{-\frac{1}{3}i \cdot 2^{\frac{m}{2}+3k} t \left(\frac{q}{2^{\frac{m}{2}}} \right)^{-\frac{2}{3}}n} dt
\end{equation}
is constant on the interval $I_p^{m, k}$.
}
$\newline$

As an application of the above heuristic, we can now simultaneously preserve the correlation of the pair $(f_2,f_3)$ and decouple its information from the one carried by the pair $(f_1,f_4)$ thus obtaining the correlative time-frequency model
\begin{eqnarray*}
&& \left|\Lambda_m^k (\vec{f}) \right| \lesssim 2^{3k} \sum_{\substack{p \sim 2^{\frac{m}{2}-2k} \\ q, n \sim 2^{\frac{m}{2}}}} \left| \left \langle f_1, \phi_{p-(2^m q)^{\frac{1}{3}}, n}^k \right \rangle \right|  \left| \left \langle f_4, \phi_{p, n}^k \right \rangle \right| \\
&& \quad \quad \quad \quad \quad \cdot \frac{1}{\left|I_p^{m, k} \right|} \left| \int_{I_p^{m, k}} \int_{I_q^{m, 3k}} f_3(x+t)f_2(x+t^\frac{2}{3}) e^{-\frac{1}{3} i 2^{\frac{m}{2}+3k}t \left(\frac{q}{2^{\frac{m}{2}}} \right)^{-\frac{2}{3}} n} dtdx \right|.
\end{eqnarray*}

\subsubsection{The spatial sparse-uniform dichotomy} Following the next stage of Rank II LGC, we define the \emph{sparse index sets} relative to our input functions as follows: for $\mu \in (0, 1)$ sufficiently small, we set
$$
\calS_\mu(f_i):= \left\{r \sim 2^{m+(i-3)k}: \frac{1}{\left|I_0^{0, ik+m} \right|} \int_{3 I_r^{0, ik+m}} |M f_i|^2 \gtrsim \frac{2^{\mu m}}{\left|I^{3k} \right|}  \int_{I^{3k}} |f_i|^2  \right\}, \quad i \in \{2, 3\},
$$
and
$$
\calS_\mu(f_i):= \left\{\ell \sim 2^{\frac{m}{2}-2k}: \frac{1}{\left|I_0^{m, k} \right|} \int_{3I_\ell^{m, k}} |f_i|^2 \gtrsim \frac{2^{\mu m}}{\left|I^{3k}  \right|} \int_{I^{3k}}  |f_i|^2 \right\}, \quad i \in \{1, 4\}.
$$
Moreover, for some $\mu_0>\mu>0$ suitable small (to be chosen later), we define the \emph{uniform index sets} in an obvious way and decompose the main term $\Lambda_m^k (\vec{f})$ into the \emph{uniform component}  given by
$$
\Lambda_m^{k, \calU_\mu}(\vec{f}):=\Lambda_m^k \left(f_1^{\calU_\mu}, f_2^{\calU_{\mu_0}}, f_3^{\calU_{\mu_0}}, f_4^{\calS_\mu} \right)+ \Lambda_m^k \left(f_1^{\calS_\mu}, f_2^{\calU_{\mu_0}}, f_3^{\calU_{\mu_0}}, f_4^{\calU_\mu} \right)+\Lambda_m^k \left(f_1^{\calU_\mu}, f_2^{\calU_{\mu_0}}, f_3^{\calU_{\mu_0}}, f_4^{\calU_\mu} \right),
$$
and the \emph{sparse component} given by
$$\Lambda_m^{k, \calS_\mu}(\vec{f}):=\Lambda_m^k(\vec{f})-\Lambda_m^{k, \calU_\mu}(\vec{f})\,.$$

\subsubsection{Treatment of the sparse component $\Lambda_m^{k, \calS_\mu}$}
The estimates for the sparse component follow essentially the same strategy as the one presented in Section \ref{20230330subsec01} with only minor modifications and hence we skip the details here.

\subsubsection{Treatment of the uniform component $\Lambda_m^{k, \calU_\mu}$}
Turning now our attention to the uniform component, without loss of generality, we identify $\Lambda_m^{k, \calU_\mu}(\vec{f})$ with the representative
$$
\Lambda_m^k \left(f_1^{\calU_\mu}, f_2^{\calU_{\mu_0}}, f_3^{\calU_{\mu_0}}, f_4 \right)\,.
$$
Applying now Cauchy-Schwarz and H\"older, one deduces
\begin{eqnarray*}
&& \left| \Lambda_m^{k, \calU_\mu}(\vec{f}) \right| \lesssim 2^{3k} \sum_{\ell \sim 2^{-2k}} \left( \sum_{n \sim 2^{\frac{m}{2}}} \sum_{p=2^{\frac{m}{2}}(\ell-1)}^{2^{\frac{m}{2}}(\ell+1)}  \left| \left \langle f_1^{\calU_\mu}, \phi_{p, n}^k \right \rangle \right|^2 \right)^{\frac{1}{2}} \left( \sum_{n \sim 2^{\frac{m}{2}}} \sum_{p=2^{\frac{m}{2}}\ell}^{2^{\frac{m}{2}}(\ell+1)}  \left| \left \langle f_4, \phi_{p, n}^k  \right \rangle \right|^2 \right)^{\frac{1}{2}} \\
&& \quad \cdot \left( \frac{1}{\left|I_0^{m, k} \right|} \sup_{n \sim 2^{\frac{m}{2}}} \int_{I_{\ell}^{0, k}}  \left( \sum_{q \sim 2^{\frac{m}{2}}} \left| \int_{I_q^{m, 3k}} f^{\calU_{\mu_0}}_3(x+t)f^{\calU_{\mu_0}}_2(x+t^{\frac{2}{3}}) e^{-\frac{1}{3} i \cdot 2^{\frac{m}{2}+3k} t \left(\frac{q}{2^{\frac{m}{2}}} \right)^{-\frac{2}{3}} n} dt \right|^2 \right)dx \right)^{\frac{1}{2}}.
\end{eqnarray*}
Letting $\lambda$ be the measurable function that assumes the supremum in the last term above, we have
\begin{enumerate}
    \item [$\bullet$] $\lambda: I^{3k} \mapsto \left[2^{\frac{m}{2}}, 2^{\frac{m}{2}+1} \right] \cap \Z$;
    \item [$\bullet$] $\lambda$ is constant on intervals $\left\{I_\ell^{0, k} \right\}_{\ell \sim 2^{-2k}}$.
\end{enumerate}
With this, using Parseval, we have
\begin{eqnarray*}
&& \left|\Lambda_m^{k, \calU_\mu} \right| \lesssim 2^{3k} \sum_{\ell \sim 2^{-2k}} \left\|f^{\calU_\mu}_1 \right\|_{L^2 \left(3I_\ell^{0, k} \right)} \left\|f_4 \right\|_{L^2\left(I_\ell^{0, k}\right)} \\
&& \quad \cdot \left( \frac{1}{\left|I_0^{m, k} \right|} \int_{I_{\ell}^{0, k}}  \left( \sum_{q \sim 2^{\frac{m}{2}}} \left| \int_{I_q^{m, 3k}} f^{\calU_{\mu_0}}_3(x+t)f^{\calU_{\mu_0}}_2(x+t^{\frac{2}{3}}) e^{-\frac{1}{3} i \cdot 2^{\frac{m}{2}+3k} t \left(\frac{q}{2^{\frac{m}{2}}} \right)^{-\frac{2}{3}} \lambda(x)} dt \right|^2 \right)dx \right)^{\frac{1}{2}}.
\end{eqnarray*}
Finally, from the uniformity of $f_1$ and Cauchy-Schwarz, we conclude
$$
\left| \Lambda_m^{k, \calU_\mu} \right|  \lesssim 2^{3k} \cdot 2^{\frac{\mu m}{2}} \left\|f_1 \right\|_{L^2(3 I^{3k})} \left\|f_4 \right\|_{L^2 (3 I^{3k})}  {\bf L}_{m, k}(f_3, f_2),
$$
where
$$
{\bf L}_{m, k}^2(f_3, f_2):=\frac{1}{\left|I_0^{m, 3k} \right|} \int_{I^{3k}}  \left( \sum_{q \sim 2^{\frac{m}{2}}} \left| \int_{I_q^{m, 3k}} f^{\calU_{\mu_0}}_3(x+t)f^{\calU_{\mu_0}}_2(x+t^{\frac{2}{3}}) e^{-\frac{1}{3} i \cdot 2^{\frac{m}{2}+3k} t \left(\frac{q}{2^{\frac{m}{2}}} \right)^{-\frac{2}{3}} \lambda(x)} dt \right|^2 \right)dx.
$$
Therefore, it suffices to show the following analog of Theorem \ref{20230202thm01}.

\begin{namedthm*}{Theorem \ref{20230202thm01}$^\prime$} \label{20230330thm02}
Let $-k \ge \frac{m}{2}$. Then under the setup above, there exists an absolute $\underline{\epsilon_2}>0$, such that for any $\lambda(\cdot): I^{3k} \mapsto \left[2^{\frac{m}{2}}, 2^{\frac{m}{2}+1} \right] \cap \Z$ measurable function that is constant on intervals $\left\{I_\ell^{0, k} \right\}_{\ell \sim 2^{-2k}}$ one has
$$
{\bf L}_{m, k}(f_3, f_2) \lesssim 2^{-\underline{\epsilon_2} m} \left\|f_2 \right\|_{L^2 \left(3 I^{3k} \right)}\left\|f_3 \right\|_{L^2 \left(3 I^{3k} \right)}.
$$
\end{namedthm*}
As expected, the proof of Theorem \ref{20230202thm01}$^\prime$ follows a similar strategy with the one employed for proving Theorem \ref{20230202thm01}. In what follows, we outline the main steps. We start with the analogue of Key Heuristic 3:
$\newline$

\noindent\textbf{Key Heuristic 3$^\prime$.} \textit{Let $w \in \left\{-k, \frac{m}{2} \right\}$ and assume that
\begin{center}
$\lambda(\cdot)$ is constant on intervals of length $2^{-2k-w}$.
\end{center}
Then, for each $q \sim 2^{\frac{m}{2}}$ and $p \sim 2^w$, the continuous function
$$
F: I_0^{0, 3k+w} \ni \alpha \longmapsto \int_{I_p^{0, 3k+w}} \left| \int_{I_q^{m, 3k}} f_3(x+t+\alpha) f_2 \left(x+t^{\frac{2}{3}} \right) e^{-\frac{1}{3} i \cdot 2^{\frac{m}{2}+3k} t \left(\frac{q}{2^{\frac{m}{2}}} \right)^{-\frac{2}{3}} \lambda(x)} dt \right|^2 dx
$$
is essentially constant, in the sense described by Proposition \ref{constantprop}.}
$\newline$

Following the presentation in Section \ref{DiagUnifKlarge}--see Key Heuristic 4 therein, we have now the analogue
\medskip

\noindent\textbf{Key Heuristic 4$^\prime$.} \textit{Given $p \sim 2^w$, our approach relies on the following {\bf dichotomy:}
\begin{enumerate}
    \item [$\bullet$] either $\lambda(\cdot)$ is morally constant on $I_p^{0, 3k+w}$;
    \item [$\bullet$] or there exists some $\underline{\epsilon_3^w}>0$, such that
    $$
    \left|V_{m, k}^{p, w}(f_3, f_2) \right| \lesssim 2^{-2\underline{\epsilon_3^w}m} \left\|f_2 \right\|_{L^2(3 I^{3k})}^2 \left\|f_3 \right\|_{L^2 \left(3 I^{3k} \right)}^2,
    $$
    where
    $$
    V_{m, k}^{p, w}(f_3, f_2):=2^{\frac{m}{2}+3k+w} \sum_{q \sim 2^{\frac{m}{2}}} \int_{I_p^{0, 3k+w}} \left| \int_{I_q^{m, 3k}} f_3(x+t) f_2 \left(x+t^{\frac{2}{3}} \right) e^{-\frac{1}{3} i \cdot 2^{\frac{m}{2}+3k} t \left(\frac{q}{2^{\frac{m}{2}}} \right)^{-\frac{2}{3}} \lambda(x)} dt \right|^2 dx.
    $$
\end{enumerate}}
$\newline$

To formulate the above principle rigorously, we follow the arguments in \eqref{20230205eq11}--\eqref{20230205eq12} and define the analogues of $V_{m, k}^{p, \calH_p^w}(f_3, f_2)$ and $V_{m, k}^{p, \calL_p^w}(f_3, f_2)$ in the obvious way. The two propositions corresponding to the above dichotomy are stated as follows:

\begin{namedthm*}{Proposition \ref{mainprop01}$^{\prime}$}
Let $w \in \left\{-k, \frac{m}{2} \right\}$, then there exists some $\underline{\epsilon_3^w}>0$, such that
$$
\left|V_{m, k}^{p, \calL_p^w}(f_3, f_2) \right| \lesssim 2^{-2 \underline{\epsilon_3^w} m} \left\|f_2 \right\|^2_{L^2 \left(3 I^{3k} \right)} \left\|f_3 \right\|^2_{L^2\left(3 I^{3k} \right)}.
$$
\end{namedthm*}

{\raggedleft \textit{Observation \ref{constancy}$^{\prime}$}} [\textsf{Constancy propagation for $\lambda$}] Mirroring the approach for $k\geq\frac{m}{2}$, we apply the following bootstrap algorithm:
\begin{itemize}
\item from our initial hypothesis, we know that $\lambda$ is constant on intervals $\left\{I_\ell^{0, k} \right\}_{\ell \sim 2^{-2k}}$. We thus run Proposition \ref{mainprop01}$^{\prime}$ for $w=-k$ in order to reduce matters to the case when $\lambda$ is constant on intervals $\left\{I_\ell^{0, 2k} \right\}_{\ell \sim 2^{-k}}$.

\item feeding this back to the analogue of Proposition \ref{constantprop} followed by Proposition \ref{mainprop01}$^{\prime}$ this time for $w=\frac{m}{2}$, we can put ourselves in the situation when $\lambda$ is constant on all intervals $\left\{I_p^{m, 3k} \right\}_{p \sim 2^{\frac{m}{2}}}$.
\end{itemize}

\begin{namedthm*}{Proposition \ref{mainprop02}$^{\prime}$}
Assume that for each $p \sim 2^{-k}$, one has
$$
\lambda(x)=\lambda(y) \quad \forall x, y \in I_p^{m, 3k}.
$$
Then there exists a suitable absolute constant $\underline{\epsilon_4}>0$ such that
$$
{\bf L}_{m, k}^2(f_3, f_2) \lesssim 2^{-2\underline{\epsilon_4} m} \left\|f_1 \right\|_{L^2 \left(3 I^{3k} \right)}^2 \left\|f_2 \right\|^2_{L^2 \left(3 I^{3k} \right)}
$$
\end{namedthm*}

Since the proofs of these last two Propositions follow similar arguments with those corresponding to their counterparts for $k \ge 0$ in what follows we present only an outline/summary of the arguments.

\medskip

\noindent\textit{Proof of Proposition  \ref{mainprop01}$^{\prime}$.} Using now Key Heuristic 3$^{\prime}$ and arguing via a standard $TT^*-$method, we see that it is enough to consider
\begin{equation}\label{20230206eq03n}
\calV^w_{m, k}:=2^{\frac{m}{2}+3k+w+\del_w m} \sum_{q \sim 2^{\frac{m}{2}}} \int_{I_0^{0, 3k+w}} \left| \int_{I_q^{m, 3k}} f_3(t)f_2\left(x+t^{\frac{2}{3}} \right) e^{-\frac{i}{3} 2^{\frac{m}{2}+3k}t \left(\frac{q}{2^{\frac{m}{2}}} \right)^{-\frac{2}{3}} \lambda(x)} dt \right|^2 dx.
\end{equation}
Applying a change of variables and Cauchy-Schwarz twice, we have
\begin{equation}\label{20230206eq30n}
\calV^w_{m, k} \lesssim 2^{\frac{m}{2}+3k+w+\del_w m} \calA \calB,
\end{equation}
where
$$
\calA^2:=\iint_{\left(I^{m, 3k}_0 \right)^2} \sum_{q \sim 2^{\frac{m}{2}}} \left|f_3 \left(t+\frac{q}{2^{\frac{m}{2}+3k}} \right) \right|^2  \left|f_3 \left(s+\frac{q}{2^{\frac{m}{2}+3k}} \right) \right|^2 dtds \lesssim \frac{2^{2\mu_0 m}}{2^{\frac{m}{2}}} \left\|f_3\right\|_{L^3 \left(3 I^{3k} \right)}^4,
$$
and
\begin{eqnarray*}
&& \calB^2:= \iint_{\left(I_0^{m, 3k} \right)^2} \sum_{q \sim 2^{\frac{m}{2}}} \bigg| \int_{9I_0^{0, 3k+w}} f_2 \left(x+\frac{2t}{3} \left(\frac{q}{2^{\frac{m}{2}+3k}} \right)^{-\frac{1}{3}} \right) f_2 \left(x+\frac{2s}{3} \left(\frac{q}{2^{\frac{m}{2}+3k}} \right)^{-\frac{1}{3}} \right) \\
&& \quad \quad\quad\quad\quad\quad \cdot e^{-\frac{i}{3} 2^{\frac{m}{2}+3k}(t-s) \left(\frac{q}{2^{\frac{m}{2}}} \right)^{-\frac{2}{3}} \lambda \left(x-\left(\frac{q}{2^{\frac{m}{2}+3k}} \right)^{\frac{2}{3}} \right)} dx  \bigg|^2 dtds.
\end{eqnarray*}
Following now the arguments presented in \textsf{Step III: Treatment of $\calB$} in the proof of  Proposition  \ref{mainprop01}, we derive
\begin{equation} \label{20230207eq05n}
\calB^2 \lesssim \frakC \frakD,
\end{equation}
where
\begin{eqnarray*}
\frakC^2%
&:=& \iiint_{\left(17I_0^{0, 3k+w} \right)^2 \times 7I_0^{m, 3k}} \left|f_2(x) f_2(y) f_2(x-2^kv) f_2(y-2^kv) \right|^2 dvdxdy \\
&\lesssim& 2^{-\frac{m}{2}+3k-2w} \cdot 2^{4\mu_0 m} \left\|f_2 \right\|_{L^2 \left(3 I^{3k} \right)}^8,
\end{eqnarray*}
and
\begin{eqnarray*}
&& \frakD^2:=\iiint_{\left(17I_0^{0, 3k+w} \right)^2 \times 7I_0^{m, 3k}} dvdxdy  \\
&& \quad \left| \int_{3I_0^{m, 3k}} \left(\sum_{q \sim 2^{\frac{m}{2}}} e^{-\frac{i}{2} 2^{\frac{m}{2}+3k}v\left(\frac{q}{2^{\frac{m}{2}}} \right)^{-\frac{1}{3}} \left[ \lambda \left(x-2^k t-\left(\frac{q}{2^{\frac{m}{2}+3k}} \right)^{\frac{2}{3}} \right)- \lambda \left(y-2^k t-\left(\frac{q}{2^{\frac{m}{2}+3k}} \right)^{\frac{2}{3}} \right) \right]} \right) dt\right|^2.
\end{eqnarray*}

To estimate the term $\frakD$, we follow the arguments at {\sf Step V} and {\sf Step VI} part of the proof of Proposition \ref{mainprop01} in order to deduce
$$
\frakD^2 \lesssim \left|I_0^{m, 3k} \right|^2 \sum_{q, q_1 \sim 2^{\frac{m}{2}}} \iint_{\left(35I_0^{0, 3k+w} \right)^2 \times 5I_0^{m, k}} \left\lceil \widetilde{\Xi}_{x, y, s}(q, q_1) \right\rceil dsdxdy,
$$
where
$$
\widetilde{\Xi}_{x, y, s}(q, q_1):=\left[ \lambda(x)-\lambda(y) \right]-\frac{ \lambda \left(x+2^k s+\frac{1}{2^{2k}} \left( \left(\frac{q}{2^{\frac{m}{2}}} \right)^{\frac{2}{3}}-\left(\frac{q_1}{2^{\frac{m}{2}}} \right)^{\frac{2}{3}} \right) \right) -\lambda \left(y+2^k s+\frac{1}{2^{2k}} \left( \left(\frac{q}{2^{\frac{m}{2}}} \right)^{\frac{2}{3}}-\left(\frac{q_1}{2^{\frac{m}{2}}} \right)^{\frac{2}{3}} \right) \right)}{\left(q_1/q \right)^{\frac{1}{3}}}.
$$
Applying the changes of variable $2^k s \to 2^k z-\frac{1}{2^{2k}} \left( \left(\frac{q}{2^{\frac{m}{2}}} \right)^{\frac{2}{3}}-\left(\frac{q_1}{2^{\frac{m}{2}}} \right)^{\frac{2}{3}} \right)+\frac{1}{2^{\frac{m}{2}+2k}} \left\lfloor 2^{\frac{m}{6}} \left(q^{\frac{2}{3}}-q_1^{\frac{2}{3}} \right) \right\rfloor$ followed by $u=\left\lfloor 2^{\frac{m}{6}} \left(q^{\frac{2}{3}}-q_1^{\frac{2}{3}} \right) \right\rfloor$, the above expression becomes
$$
\Xi_{x, y, z, u}(q_1):=\left[\lambda(x)-\lambda(y) \right]- \sqrt{ \frac{u}{2^{\frac{m}{6}} q_1^{\frac{2}{3}}}+1} \left[\lambda \left(x+2^k z+\frac{u}{2^{\frac{m}{2}+2k}} \right)-\lambda \left(y+2^k z+\frac{u}{2^{\frac{m}{2}+2k}} \right) \right].
$$
Since the above expression is in the same spirit with the one at {\bf Step IV \& V} in the proof of Proposition \ref{mainprop01}, one can now apply a similar analysis--whose details will be skipped here for brevity reasons.  \hfill\fbox

\medskip

\noindent\textit{Proof of Proposition  \ref{mainprop02}$^{\prime}$.} For the clarity of the exposition, we set as before $f_2=f_2^{\calU_{\mu_0}}$ and $f_3=f_3^{\calU_{\mu_0}}$ and, more importantly, we  choose to skip/ignore the preliminary reductions presented at \textsf{Step I} in the proof of Proposition \ref{mainprop02} thus avoiding the references to the parameters $\sigma$ and $\tau$ therein as well as to the latter introduced parameter $\varsigma$ (though of course, if one wants to produce a rigorous proof than the outline presented below must be amended in order to include all the these parameters).

With these clarified, the ``picture from above" can be presented as follows:
$$
{\bf L}_{m, k}^2(f_3, f_2) \lesssim 2^{\frac{m}{2}+3k} \frakE \frakF,
$$
where
$$
\frakE^2:=\int_{I^{3k}} \iint_{\left(I_0^{m, 3k}\right)^2} \left|f_2 \left(x+\frac{2 \cdot 2^k}{3}t  \right) f_2 \left(x+\frac{2 \cdot 2^k}{3}s  \right) \right|^2 dtdsdx \lesssim \frac{2^{2\mu_0 m}}{2^{m+3k}} \left\|f_2 \right\|_{L^2(3 I^{3k})}^4,
$$
and
\begin{eqnarray*}
&& \frakF^2:=\int_{I^{3k}} \iint_{\left(I_0^{m, 3k} \right)^2} \Bigg| \sum_{q \sim 2^{\frac{m}{2}}} f_3 \left(x+t\left(\frac{q}{2^{\frac{m}{2}}} \right)^{\frac{1}{3}}+\frac{q}{2^{\frac{m}{2}+3k}} \right) \\
&& \quad \quad \quad \quad \quad \quad \quad \quad \quad \cdot f_3 \left(x+s\left(\frac{q}{2^{\frac{m}{2}}} \right)^{\frac{1}{3}}+\frac{q}{2^{\frac{m}{2}+3k}} \right) e^{-\frac{i}{3} 2^{\frac{m}{2}+3k}(t-s) \left(\frac{q}{2^{\frac{m}{2}}} \right)^{-\frac{1}{3}} \lambda(x)} \Bigg|^2 dtdsdx.
\end{eqnarray*}
To estimate the term $\frakF$, we follow the arguments presented in \eqref{20230213eq10} and \eqref{20230215eq01} to deduce
$$
\frakF^2 \lesssim  \frakG \frakH,
$$
where
$$
\frakG^2:=\int_{I^{3k}} \iint_{\left(I_0^{m, 3k} \right)^2} \sum_{r \sim 2^{\frac{m}{2}}} \left| f_3 \left(x+\left(\frac{r}{2^{\frac{m}{2}}} \right)^{-\frac{1}{3}} t \right) f_3 \left(x+\left(\frac{r}{2^{\frac{m}{2}}} \right)^{-\frac{1}{3}} s \right) \right|^2 dtdsdx \lesssim 2^{-\frac{m}{2}-3k} \cdot 2^{2\mu_0 m} \left\|f_3 \right\|_{L^2 \left(3 I^{3k} \right)}^4,
$$
and
$$
\frakH^2:=\int_{I^{3k}} \iint_{\left(I_0^{m, 3k} \right)^2} \sum_{r \sim 2^{\frac{m}{2}}} \left| \sum_{p \sim 2^{\frac{m}{2}}} f_3\left(x+t-\frac{p}{2^{\frac{m}{2}+3k}} \right) f_3\left(x+s-\frac{p}{2^{\frac{m}{2}+3k}} \right) e^{-\frac{i}{3} \cdot 2^{\frac{m}{2}+3k}(t-s)\Phi_{r, p}(x)} \right|^2 dtdsdx,
$$
where
\begin{equation} \label{20230215eq01n}
\Phi_{r, p}(x):=\left(\frac{1-\frac{r}{2^{\frac{m}{2}}}}{\frac{pr}{2^m}} \right)^{\frac{1}{3}} \left[\left(\frac{1-\frac{r}{2^{\frac{m}{2}}}}{\frac{p}{2^\frac{m}{2}}}\right)^{\frac{1}{3}}-\left(\frac{1-\frac{r}{2^{\frac{m}{2}}}}{\frac{pr}{2^m}} \right)^{\frac{1}{3}} \right] \lambda \left(x-\frac{\frac{p}{2^{\frac{m}{2}}}}{2^{3k} \left(1-\frac{r}{2^{\frac{m}{2}}} \right)} \right).
\end{equation}
Now by applying a $TT^*$ argument followed by Cauchy-Schwarz, we see that
$$
\frakH^2 \lesssim \calI \calJ,
$$
where
\begin{eqnarray*}
&& \calI^2:=\sum_{p_1, p_2 \sim 2^{\frac{m}{2}}} \int_{I^{3k}} \iint_{\left(I_0^{m, 3k} \right)^2} \bigg| f_3\left(x+t-\frac{p_1}{2^{\frac{m}{2}+3k}} \right) f_3\left(x+s-\frac{p_1}{2^{\frac{m}{2}+3k}} \right) \\
&& \quad \quad \quad \quad \quad \cdot f_3\left(x+t-\frac{p_2}{2^{\frac{m}{2}+3k}} \right) f_3\left(x+s-\frac{p_2}{2^{\frac{m}{2}+3k}} \right) \bigg|^2 dtdsdx \lesssim 2^{3k} \cdot 2^{4\mu_0 m} \left\|f_3 \right\|_{L^2 \left(3 I^{3k} \right)}^8
\end{eqnarray*}
and
$$
\calJ^2:=\sum_{p_1, p_2 \sim 2^{\frac{m}{2}}} \int_{I^{3k}} \iint_{\left(I_0^{m, 3k} \right)^2} \left| \sum_{r \sim 2^{\frac{m}{2}}} e^{-\frac{i}{3} \cdot 2^{\frac{m}{2}+3k}(t-s) \left(\Phi_{r, p_1}(x)-\Phi_{r, p_2}(x) \right)} \right|^2 dtdsdx.
$$
Following the arguments in \eqref{20230215eq10}, we further see that
\begin{equation} \label{20230215eq10n}
\calJ^2 \lesssim 2^{-m-6k} \sum_{p_1, p_2, r_1, r_2 \sim 2^{\frac{m}{2}}} \int_{I^{3k}} \left\lceil V_{p_1, p_2, r_1, r_2}(x) \right\rceil dx,
\end{equation}
where $V_{p_1, p_2, r_1, r_2}(x):=\Phi_{r_1, p_1}(x)-\Phi_{r_1, p_2}(x)-\Phi_{r_2, p_1}(x)+\Phi_{r_2, p_2}(x)$. To treat the term $V_{p_1, p_2, r_1, r_2}$, we proceed as follows (up to an admissible global error term):
\begin{itemize}
\item  For $r_1$ being fixed, apply the change variables:
$$
u_1=\left\lfloor \frac{2^{\frac{m}{2}}p_1}{2^{\frac{m}{2}}-r_1}\right\rfloor, \quad u_2=\left\lfloor \frac{2^{\frac{m}{2}}p_2}{2^{\frac{m}{2}}-r_1}\right\rfloor, \quad  \textrm{and} \quad v_1=\left\lfloor \frac{2^{\frac{m}{2}}p_1}{2^{\frac{m}{2}}-r_2}\right\rfloor;
$$
\item Multiply the term $\left(\frac{u_1}{v_1} \left(\frac{r_1}{2^{\frac{m}{2}}}-1 \right)+1 \right)^{\frac{2}{3}} \Big\slash \left(\frac{r_1}{2^{\frac{m}{2}}} \right)^{\frac{1}{3}}$;

\item For $u_1, u_2, v_2$ being fixed, apply the change variable $\frac{s_1^3}{2^{\frac{3m}{2}}}=\frac{u_1}{v_1} \left(1-\frac{2^{\frac{m}{2}}}{r_1} \right)+\frac{2^{\frac{m}{2}}}{r_1}+O \left(\frac{1}{2^{\frac{m}{2}}} \right)$;

\item Multiply the term $\left(\frac{v_1}{2^{\frac{m}{2}}} \right)^{\frac{2}{3}}$ followed by the change variable $x \to x+\frac{v_1}{2^{\frac{m}{2}+3k}}$.
\end{itemize}

\medskip

As a consequence, the term $V_{p_1, p_2, r_1, r_2}$ becomes
$$
V_{u_1, u_2, v_1, s_1}(x):= \left| \underbrace{\frac{A_{u_1, u_2, v_1}(x)s_1^2}{2^m}+\frac{B_{u_1, u_2, v_1}(x)s_1}{2^{\frac{m}{2}}}}_{:=C_{u_1,u_2,v_1}(x)[s_1]}+ \underbrace{\left(B_{u_1, u_2, v_1}(x)+\frac{A_{u_1, u_2, v_1}(x)s_1^2}{2^m} \right) \left(\frac{s_1^3}{2^{\frac{3m}{2}}}-\frac{u_1}{v_1} \right)^{\frac{1}{3}}\left(\frac{v_1}{v_1-u_1} \right)^{\frac{1}{3}}}_{:=D_{u_1,u_2,v_1}(x)[s_1]} \right|,
$$
where
$$
A_{u_1, u_2, v_1}(x):=\left(\frac{v_1}{u_1} \right)^{\frac{2}{3}} \lambda \left(x-\frac{u_1-v_1}{2^{\frac{m}{2}+3k}} \right)-\left(\frac{v_1}{u_2} \right)^{\frac{2}{3}} \lambda \left(x-\frac{u_2-v_1}{2^{\frac{m}{2}+3k}} \right)
$$
and
$$
B_{u_1, u_2, v_1}(x)=\lambda(x)-\left(\frac{u_1}{u_2} \right)^{\frac{2}{3}} \lambda \left(x-\frac{v_1}{2^{\frac{m}{2}+3k}} \left(\frac{u_2}{u_1}-1 \right) \right).
$$
Moreover, by \eqref{20230215eq10n}, we have
\begin{equation} \label{20230608eq01}
\calJ^2 \lesssim 2^{-m-6k} \sum_{u_1, u_2, v_1, s_1 \sim 2^{\frac{m}{2}}} \int_{3I^{3k}} \left\lceil V_{u_1, u_2, v_1, s_2}(x) \right\rceil dx.
\end{equation}
To estimate the above expression, we first fix $u_1,\,u_2,\,v_1$ and $x$ and then note that it is sufficient to consider the case when
\begin{equation} \label{20230407eq01}
D_{u_1,u_2,v_1}(x)[s_1]=-C_{u_1,u_2,v_1}(x)[s_1]\,+\,O( 2^{\epsilon_5 m})
\end{equation}
for a suitable $\epsilon_5>0$ being sufficiently small. Taking a cubic power in \eqref{20230407eq01} and dividing by $2^m\,\left(\frac{v_1}{v_1-u_1} \right)$ on both sides of the equality, we deduce that\footnote{The expression on the left-hand side of \eqref{20230407eq02} is a polynomial of degree $9$ in $\frac{s_1}{2^{\frac{m}{2}}}$. Since we only need to use the coefficients of either the constant term or of the $9^{th}$ degree term, we omit displaying the explicit form of the remaining terms.}
\begin{equation} \label{20230407eq02}
\left|\frac{A_{u_1, u_2, v_1}(x)^3}{2^m} \cdot \frac{s_1^9}{2^{\frac{9m}{2}}}+\frac{3B_{u_1, u_2, v_1}(x)A_{u_1, u_2, v_1}(x)^2}{2^m} \cdot \frac{s_1^7}{2^{\frac{7m}{2}}}+ \dots -\frac{B_{u_1, u_2, v_1}(x)^3}{2^m} \cdot \frac{u_1}{v_1}\right| \lesssim 2^{10\epsilon_5 m}.
\end{equation}
Without loss of generality, we may assume that there are $\gtrsim 2^{\frac{m}{2}-\epsilon_6 m}$ many $s_1 \sim 2^{\frac{m}{2}}$ such that \eqref{20230407eq02} holds for some suitably small $0<\epsilon_6<20\epsilon_5$. Using this, together with repeated applications of the mean value theorem, we deduce that
\begin{equation} \label{20230216eq22n}
\left|A_{u_1, u_2, v_1}(x) \right|, \left|B_{u_1, u_2, v_1}(x) \right| \lesssim 2^{\frac{m}{3}+100\epsilon_5 m}.
\end{equation}
As a consequence, from \eqref{20230608eq01} and up to an admissible global error,  we have
\begin{eqnarray*}
\calJ^2%
& \lesssim & 2^{-m-6k} \sum_{u_1, u_2, v_1, s_1 \sim 2^{\frac{m}{2}}} \int_{3I^{3k}} \one_{ \left\{ x \in 3 I^{3k}, (u_1, u_2, v_1) \in \left[2^{\frac{m}{2}}, 2^{\frac{m}{2}+1} \right]^3: \left|B_{u_1, u_2, v_1}(x) \right| \lesssim 2^{\frac{m}{3}+100\epsilon 5 m} \right\}} (x) dx  \\
& \lesssim & 2^{-\frac{2m}{3}-6k+100\epsilon_5 m} \sum_{u_1, u_2, v_1, s_1 \sim 2^{\frac{m}{2}}} \int_{3I^{3k}} \left\lceil B_{u_1, u_2, v_1}(x) \right\rceil dx.
\end{eqnarray*}
Once at this point, one can apply similar reasonings with the ones used for finalizing \textsf{Step III} within the proof of Proposition \ref{mainprop02}. We leave these details to the interested reader.

\medskip

\subsection{The case $0 \le -k \le \frac{m}{2}$.}
The main goal of this section is to provide a sketch for the proof of the following:

\begin{thm} \label{20230407thm01}
Recall the expression of  $\Lambda_m^k (\vec{f})$ in \eqref{20230407eq10}. Then, there exists an $\underline{\epsilon}>0$ such that
$$
\left|\Lambda_m^k \left(\vec{f}\right) \right| \lesssim (1+|k|)\,2^{-\underline{\epsilon} \min \left\{-k, \frac{m}{2}+k \right\}} \left\|\vec{f} \right\|_{L^{\vec{p}} \left(3 I^{3k} \right)}\quad \forall p \in {\bf H}^{-}\,.
$$
\end{thm}

Following now the strategy presented in Section \ref{Diagkpsmall}, we start with a preliminary frequency foliation:
\begin{itemize}
\item for the functions $f_2$ and $f_3$ we have
$$
\supp\,\widehat{f_j} \subseteq \left [2^{m+jk}, 2^{m+jk+1} \right ]=\bigcup_{\ell_j \sim 2^{\frac{m}{2}+k}} \omega_j^{\ell_j}:=\bigcup_{\ell_j \sim 2^{\frac{m}{2}+k}} \left[\ell_j 2^{\frac{m}{2}+(j-1)k}, \left(\ell_j+1 \right)2^{\frac{m}{2}+(j-1)k} \right], \quad j \in \{2, 3\}.
$$
\item for the functions $f_1$ and $f_4$, we have
$$
\supp\,\widehat{f_j} \subset \left[2^{m+k}, 2^{m+k+1} \right]=\bigcup_{\ell_j \sim 2^{\frac{m}{2}+k}} \omega_j^{\ell_j}:=\bigcup_{\ell_j \sim 2^{\frac{m}{2}+k}} \left[\ell_j2^{\frac{m}{2}}, \left(\ell_j+1 \right) 2^{\frac{m}{2}} \right], \quad j \in \{1, 4\}.
$$
\end{itemize}
Using the above decomposition, and recalling the definition of $\underline{f}_j^{\omega_j^{\ell_j}}$ in \eqref{20230408eq01}, we deduce that
$$
f_i=\sum_{\ell_i \sim 2^{\frac{m}{2}+k}} f_i^{\omega_i^{\ell_i}}, \quad i\in\{1, 2, 3, 4\}\,.
$$
As a consequence, from \eqref{Jqnk} and the above, we have
$$
J_{q, n}^k \left(f_3, f_2 \right)=\sum_{\ell_2, \ell_3 \sim 2^{\frac{m}{2}+k}} e^{i \left(\ell_2 2^{\frac{m}{2}+k}+\ell_3 2^{\frac{m}{2}+2k} \right)x} \underline{J}_{q, n}^k \left(\underline{f}_3^{\omega_3^{\ell_3}}, \underline{f}_2^{\omega_2^{\ell_2}} \right)(x),
$$
where $\underline{f}_j^{\omega_j^{\ell_j}}, j=1, 2$ are defined as in \eqref{20230408eq01} and
$$
\underline{J}_{q, n}^k \left(\underline{f}_3^{\omega_3^{\ell_3}}, \underline{f}_2^{\omega_2^{\ell_2}} \right)(x):=\int_{I_q^{m, 3k}} \underline{f}_3^{\omega_3^{\ell_3}}(x+t) \underline{f}_2^{\omega_2^{\ell_2}}(x+t^{\frac{2}{3}})\, e^{i \left(\ell_2 2^{\frac{m}{2}+k}\,t^{\frac{2}{3}}+\ell_3 2^{\frac{m}{2}+2k}\,t \right)}\, e^{-\frac{i}{3} 2^{\frac{m}{2}+3k} t \left(\frac{q}{2^{\frac{m}{2}}} \right)^{-\frac{2}{3}} n_1} dt.$$
\medskip
\noindent\textit{Observation \ref{keyobs-02}$^{\prime}$.} \textit{
The mapping
$$
x \longmapsto \underline{J}_{p, n}^k\left(\underline{f}_3^{\omega_3^{\ell_3}}, \underline{f}_2^{\omega_2^{\ell_2}} \right)(x)
$$
is morally constant on intervals $\left\{I_p^{m, k} \right\}_{p \sim 2^{\frac{m}{2}-2k}}$.}
\medskip

Next, we need to formulate the \emph{the (rescaled) time-frequency correlation set} for our present case. By running first the Rank I LGC, after linearization and the Gabor frame decomposition of the inputs
$$f_j=\sum_{\substack{p_j \sim 2^{\frac{m}{2}-(3-j)k} \\ n_j \sim 2^{\frac{m}{2}}}} \left\langle f_j, \phi_{p_j, n_j}^{jk} \right\rangle  \phi_{p_j, n_j}^{jk}, \quad j\in\{1, 2, 3\}\,,$$
we deduce the \emph{time-frequency correlations for $k<0$} as
$$
p_1 \cong p-(q 2^m)^{\frac{1}{3}}, \quad p_2 \cong p2^k+(q2^m)^{\frac{2}{3}}, \quad p_3 \cong p2^{2k}+q,
$$
and
$$
n_1-2\left(\frac{q_1}{2^{\frac{m}{2}}} \right)^{\frac{1}{3}}n_2+3 \left(\frac{q}{2^{\frac{m}{2}}} \right)^{\frac{2}{3}} n_3 \cong 0.
$$
For $j \in \{1, 2, 3\}$ notice that if $2^{\frac{m}{2}+jk}n_j \in \omega_j^{\ell_j}$ then $n_j \in \widetilde{\omega_j^{\ell_j}} :=\left[\ell_j 2^{-k}, (\ell_j+1)2^{-k} \right]$. With these, we have

\begin{namedthm*}{Lemma \ref{Rtfc}$^{\prime}$}
Let $\ell_1, \ell_2, \ell_3$ and $q$ be defined as above, then
$$
\ell_1-2\left(\frac{q}{2^{\frac{m}{2}}} \right)^{\frac{1}{2}}\ell_2+3 \left(\frac{q}{2^{\frac{m}{2}}} \right)^{\frac{2}{3}}\ell_3 \cong 0.
$$
\end{namedthm*}
We now define the \emph{time-frequency correlation set} for $k<0$ as follows: for each $q \sim 2^{\frac{m}{2}}$, we let
$$
\TFC_q:=\left\{ \substack{(\ell_1, \ell_2, \ell_3) \\ \ell_1, \ell_2, \ell_3 \sim 2^{\frac{m}{2}+k}}: \ell_1-2\left(\frac{q}{2^{\frac{m}{2}}} \right)^{\frac{1}{2}}\ell_2+3 \left(\frac{q}{2^{\frac{m}{2}}} \right)^{\frac{2}{3}}\ell_3 \cong 0 \right\}.
$$
This further gives
\begin{eqnarray*}
\left|\Lambda_m^k \left(\vec{f}\right) \right|%
&\lesssim& 2^{3k} \sum_{\substack{p \sim 2^{\frac{m}{2}-2k} \\ q \sim 2^{\frac{m}{2}}}} \sum_{\left(\ell_1, \ell_2, \ell_3\right) \in \TFC_q} \sum_{n \in \widetilde{\omega_1^{\ell_1}}}  \left | \left \langle f_1^{\omega_1^{\ell_1}}, \phi_{p-(2^m q)^{\frac{1}{3}}, n}^k \right \rangle \right| \left| \left\langle f_4 e^{i2^{\frac{m}{2}+k}(\ell_2+\ell_32^k)x}, \phi_{p, n}^k \right \rangle \right|  \\
& \cdot& \frac{1}{\left|I_p^{m, 3k} \right|} \left| \int_{I_p^{m, k}} \int_{I_q^{m, k}} \underline{f}_3^{\omega_3^{\ell_3}}(x+t) \underline{f}_2^{\omega_2^{\ell_2}} (x+t^{\frac{2}{3}}) e^{i \left(\ell_2 2^{\frac{m}{2}+k}t^{\frac{2}{3}}+\ell_32^{\frac{m}{2}+2k}t-\frac{1}{3} 2^{\frac{m}{2}+3k}t \left(\frac{q}{2^{\frac{m}{2}}} \right)^{-\frac{2}{3}}n_1 \right)} dtdx \right|,
\end{eqnarray*}
which, by a change variable $n \to n+\ell_12^{-k}$ and a linearization of the term $t^{\frac{2}{3}}$, is further bounded above by
\begin{eqnarray*}
&& 2^{3k} \sum_{\substack{p \sim 2^{\frac{m}{2}-2k} \\ q \sim 2^{\frac{m}{2}} \\ n \sim 2^{-k}}}  \sum_{\left(\ell_1, \ell_2, \ell_3\right) \in \TFC_q}  \left | \left \langle f_1^{\omega_1^{\ell_1}}, \phi_{p-(2^m q)^{\frac{1}{3}}, n+\ell_12^{-k}}^k \right \rangle \right| \left| \left\langle f_4, \phi_{p, n+\ell_12^{-k}+\ell_2+\ell_32^k}^k \right \rangle \right|  \\
&& \cdot  \frac{1}{\left|I_p^{m, k} \right|} \left| \int_{I_p^{m, k}} \int_{I_q^{m, 3k}} \underline{f}_3^{\omega_3^{\ell_3}}(x+t) \underline{f}_2^{\omega_2^{\ell_2}} (x+t^{\frac{2}{3}}) e^{i2^{\frac{m}{2}+3k}t \left(\frac{q}{2^{\frac{m}{2}}} \right)^{-\frac{2}{3}}n} e^{ i 2^{\frac{m}{2}+3k} t 2^{-k}\left(\frac{2}{3}\ell_2 \left(\frac{q}{2^{\frac{m}{2}}}\right)^{-\frac{1}{3}}+\ell_3-\frac{1}{3} \left(\frac{q}{2^{\frac{m}{2}}} \right)^{-\frac{2}{3}} \ell_1 \right) } dtdx \right|.
\end{eqnarray*}
From the above expression, we have the following analogue of  Key Heuristic 2:
\medskip

\noindent\textbf{Key Heuristic 2$^{\prime}$.} \textit{Isolate the term
\begin{equation} \label{20230409eq01}
 \left| \left\langle f_4 e^{i2^{\frac{m}{2}+k}(\ell_2+\ell_32^k)x}, \phi_{p, n+\ell_12^{-k}}^k \right \rangle \right|.
\end{equation}
In analyzing the Gabor (local Fourier) coefficient in \eqref{20230409eq01}, we have
\begin{enumerate}
\item [$\bullet$] $\ell_2$ ranges over an interval of length $\sim 2^{\frac{m}{2}+k}$;
\item [$\bullet$] $n$ ranges over an interval of length $\sim 2^{-k}$.
\end{enumerate}
This now again gives us two different cases depending on the regime of dominance:
\\
\textsf{Case I}: $0 \le -k \le \frac{m}{4}$; this corresponds to $\ell_2$ dominance; \\
\textsf{Case II}: $\frac{m}{4} \le -k \le \frac{m}{2}$; this corresponds to $n$ dominance.}
\medskip

As in the situation $k \ge 0$, the proofs of the two cases above require first a time-frequency sparse-uniform dichotomy. Maintaining the brevity of our outline, for $\mu \in (0, 1)$ a small absolute constant, we first define the \emph{sparse set} of indices as follows:
\begin{enumerate}
    \item [$\bullet$] in \textit{\textsf{Case I}}, for $j \in \{2, 3\}$, we have
    $$
    \bbS_\mu(f_j):= \left\{(q_j, L_j): \int_{3 I_{q_j}^{0, (j-1)k}} \left|f_j^{\Omega_j^{L_j}}\right|^2 \gtrsim 2^{3\mu k} \int_{3 I_{q_j}^{0, (j-1)k}} \left|f_j\right|^2, \quad q_j \sim 2^{(j-4)k}, L_j \sim 2^{-k} \right\},
    $$
    and for $j \in \{1, 4\}$
    $$
    \bbS_\mu(f_j):=\left\{(q_j, L_j): \int_{3 I_{q_j}^{0, k}} \left|f_j^{3 \Omega_j^{L_j}} \right|^2 \gtrsim 2^{\mu k} \int_{3 I_{q_j}^{0, k}} \left|f_j \right|^2, \quad q_j \sim 2^{-2k}, L_j \sim 2^{-k} \right\}.
    $$
    \item [$\bullet$] in \textit{\textsf{Case II}}, for $j \in \{2, 3\}$, we have
    $$
    \bbS_\mu \left(f_j \right):= \left\{(q_j, L_j): \int_{3 I_{q_j}^{0, (j-1)k}} \left|f_j^{\omega_j^{\ell_j}}\right|^2 \gtrsim 2^{-3\mu \left(\frac{m}{2}+k \right)} \int_{3 I_{q_j}^{0, (j-1)k}} \left|f_j\right|^2, \quad q_j \sim 2^{(j-4)k}, \ell_j \sim 2^{\frac{m}{2}+k} \right\},
    $$
    and for $j \in \{1, 4\}$
    $$
    \bbS_\mu(f_j):=\left\{(q_j, L_j): \int_{3I_{q_j}^{0, k}} \left|f_j^{3 \Omega_j^{L_j}} \right|^2 \gtrsim 2^{-\mu \left(\frac{m}{2}+k \right)} \int_{3 I_{q_j}^{0, k}} \left|f_j \right|^2, \quad q_j \sim 2^{-2k}, L_j \sim 2^{-k} \right\}.
    $$
\end{enumerate}
Here, for $j\in\{1, 2, 3\}$, $\Omega_j^{L_j}:=\left[L_j2^{m+(j+1)k}, (L_j+1)2^{m+(j+1)k} \right]$ while $\Omega_4^{L_4}:=\left[L_42^{m+2k}, (L_4+1)2^{m+2k} \right]$ with $L_{(\cdot)} \sim 2^{-k}$.

Following now the arguments presented in Section \ref{Sec05} and Section \ref{Sec06}, respectively, we see that the proofs of the above two cases reduce to estimating the cardinality of certain time-frequency correlation sets:
\begin{enumerate}
    \item [$\bullet$] in \textit{\textsf{Case I}}, one has to show that there exists some $\epsilon>0$, such that
    \begin{equation}\label{tfc1}
    \# \sTFC \lesssim 2^{-(3-\epsilon)k}\,,\qquad\textrm{where}
    \end{equation}
    $$\sTFC:= \left\{ \substack{(p, q) \\ p \sim 2^{-2k}, q \sim 2^{-k}}: \frac{2}{3} \left(2^kq \right)^{-\frac{1}{3}} L_2 \left(p+\frac{q^{\frac{2}{3}}}{2^{\frac{k}{3}}} \right)-\frac{1}{3} \left(2^k q \right)^{-\frac{2}{3}} L_1(p)=L_3 \left(2^k p+q \right) \right\},$$
    with $L_1(\cdot), L_2(\cdot), L_3(\cdot)$ measurable functions obeying in the usual fashion
    $$L_1, L_2: \left[2^{-2k}, 2^{-2k+1}\right] \cap \Z \longmapsto \left[2^{-k}, 2^{-k+1} \right] \cap \Z\qquad\textrm{and}$$
    $$L_3: \left[2^{-k}, 2^{-k+1} \right] \cap \Z \longmapsto \left[2^{-k}, 2^{-k+1} \right] \cap \Z.$$
    \item [$\bullet$] in \textit{\textsf{Case II}}, one has to show that there exists some $\epsilon>0$, such that
    \begin{equation}\label{tfc2}
    \# \sTFC \lesssim |k| 2^{-3k-\epsilon \left(\frac{m}{2}+k \right)}\,,\qquad\textrm{where}
    \end{equation}
    $$\sTFC:= \left\{ \substack{(p, q) \\ p \sim 2^{-2k}, q \sim 2^{-k}}: \frac{2}{3} \left(2^kq \right)^{-\frac{1}{3}} \ell_2 \left(p+\frac{q^{\frac{2}{3}}}{2^{\frac{k}{3}}} \right)-\frac{1}{3} \left(2^k q \right)^{-\frac{2}{3}} \ell_1(p)=\ell_3 \left(2^k p+q \right) \right\},$$
    with $\ell_1(\cdot), \ell_2(\cdot), \ell_3(\cdot)$ measurable functions obeying
    $$\ell_1, \ell_2: \left[2^{-2k}, 2^{-2k+1}\right] \cap \Z \longmapsto \left[2^{\frac{m}{2}+k}, 2^{\frac{m}{2}+k+1} \right] \cap \Z \qquad \textrm{and}$$
    $$\ell_3: \left[2^{-k}, 2^{-k+1} \right] \cap \Z \longmapsto \left[2^{\frac{m}{2}+k}, 2^{\frac{m}{2}+k+1} \right] \cap \Z\,.$$
\end{enumerate}
The proof of the estimates \eqref{tfc1} and \eqref{tfc2} follow essentially the same arguments as in Subsection \ref{20230306subsec01}.

\section{The Banach boundedness range} \label{20230411Sec01}

In this section we sketch how to put all the previous estimates together in order to obtain the desired Banach range bounds for the quadrilinear form $\Lambda$ associated to our curved tri-linear Hilbert transform $H_{C}$. Recalling the decomposition in Section \ref{basicred}
\begin{equation}\label{LlohiL}
\Lambda=\Lambda^{Hi}\,+\,\Lambda^{Lo}=\Lambda^{Hi}_{+}\,+\,\Lambda^{Hi}_{-}\,+\,\Lambda^{Lo}\,,
\end{equation}
we proceed by analyzing each of the terms above.

\subsection{The high oscillatory component $\Lambda_{+}^{Hi}$}\label{Boundhiplus}

 To begin with, recall that the main term that we wish to understand is given by
\begin{equation}\label{hiplus}
\Lambda_{j, l, m}^k (\vec{f})=\iint_{\R^2} f_{1, j}^k(x-t) f_{2, l}^k(x+t^2) f_{3, m}^k(x+t^3)f_4(x) 2^k \varrho(2^k t)dtdx,
\end{equation}
with $j,l,m\in\N$, $k\in\N$ and
$$
f^k_{1, j}=\calF^{-1} \left( \widehat{f_1} \phi \left(\frac{\cdot}{2^{j+k}} \right) \right), \quad f^k_{2, l}=\calF^{-1} \left( \widehat{f_2} \phi \left(\frac{\cdot}{2^{l+2k}} \right) \right) \quad \textrm{and} \quad f^k_{3, m}=\calF^{-1} \left( \widehat{f_3} \phi \left(\frac{\cdot}{2^{m+3k}} \right) \right).
$$

In a first instance, preserving a faithful representation of the work performed in the previous sections, we collect all the key estimates we have obtained: there exists some absolute $\epsilon>0$ sufficiently small, such that for any $\vec{p} \in {\bf H}^{+}$, one has
\begin{enumerate}
    \item [$\bullet$] \textbf{Main diagonal term $\Lambda^{D}$ restricted to $k \ge 0$}:
    \begin{enumerate}
        \item [(1)] if $k \ge \frac{m}{2}$, then $|\Lambda^k_{m, m, m}(\vec{f})| \lesssim 2^{-\epsilon m} \|\vec{f}\|_{L^{\vec{p}}}$;
         \item [(2)] if $0 \le k \le \frac{m}{4}$, then $|\Lambda^k_{m, m, m}(\vec{f})| \lesssim 2^{-\epsilon k} \|\vec{f}\|_{L^{\vec{p}}}$;
         \item [(3)] if $\frac{m}{4} \le k \le \frac{m}{2}$, then $|\Lambda^k_{m, m, m}(\vec{f})| \lesssim k\, 2^{-\epsilon \left(\frac{m}{2}-k \right)} \|\vec{f}\|_{L^{\vec{p}}}$.
    \end{enumerate}
    \item [$\bullet$] \textbf{Main off-diagonal term $\Lambda^{\not{D}}$ restricted to $k \ge 0$ and $j=l>m+100$}:
    \begin{enumerate}
            \item [(4)] if $0 \le m \le \frac{j}{2}$ and $0 \le k \le \frac{j-m}{2}$, then $|\Lambda^k_{j, j, m} (\vec{f})|\lesssim 2^{-\epsilon \left(\frac{j}{2}-k \right)} \|\vec{f}\|_{L^{\vec{p}}}$;
            \item [(5)] if $0 \le m \le \frac{j}{2}$ and $\frac{j-m}{2} \le k \le j-m$, then $|\Lambda^k_{j, j, m} (\vec{f})| \lesssim 2^{-\epsilon \left(j-m-k \right)} \|\vec{f}\|_{L^{\vec{p}}}$;
            \item [(6)] if $0 \le m \le \frac{j}{2}$ and $k \ge j-m$, then $|\Lambda^k_{j, j, m} (\vec{f})| \lesssim 2^{-\epsilon j} \|\vec{f}\|_{L^{\vec{p}}}$;
            \item [(7)] if $\frac{j}{2} \le m \le j-100$ and $0 \le k \le j-m$, then $|\Lambda^k_{j, j, m} (\vec{f})| \lesssim 2^{-\epsilon (j-m-k)} \|\vec{f}\|_{L^{\vec{p}}}$;
             \item [(8)] if $\frac{j}{2} \le m \le j-100$ and $k \ge j-m$, then $|\Lambda^k_{j, j, m} (\vec{f})|  \lesssim 2^{-\epsilon (j-m)} \|\vec{f}\|_{L^{\vec{p}}}$.
    \end{enumerate}
    \item [$\bullet$] \textbf{Main off-diagonal term $\Lambda^{\not{D}}$ restricted to $k \ge 0$ and $l=m>j+100$}:
    \begin{enumerate}
        \item [(9)] if $0 \le j \le \frac{m}{2}$,  then $|\Lambda^k_{j, m, m} (\vec{f})| \lesssim 2^{-\epsilon m} \|\vec{f}\|_{L^{\vec{p}}}$;
        \item [(10)] if $\frac{m}{2} \le j \le m-100$,  then $|\Lambda^k_{j, m, m} (\vec{f})|\lesssim 2^{-\epsilon (m-j)} \|\vec{f}\|_{L^{\vec{p}}}$.
    \end{enumerate}

        \item [$\bullet$] \textbf{Main off-diagonal term $\Lambda^{\not{D}}$ restricted to $k \ge 0$ and $j=m>l+100$}:
    \begin{enumerate}
        \item [(11)] if $0 \le l \le \frac{m}{2}$,  then $|\Lambda^k_{m, l, m} (\vec{f})| \lesssim 2^{-\epsilon m} \|\vec{f}\|_{L^{\vec{p}}}$;
        \item [(12)] if $\frac{m}{2} \le j \le m-100$,  then $|\Lambda^k_{m, l, m}(\vec{f})| \lesssim 2^{-\epsilon (m-l)} \|\vec{f}\|_{L^{\vec{p}}}$.
    \end{enumerate}
\end{enumerate}
Once here, we make the following relevant

\begin{obs}\textsf{[Continuous bounds]}\label{ContBound}  Inspecting the above list of bounds one notices that many of the end-points for the various cases dependent on the ranges of $k,j,l,m$ appear as it would be ``discontinuous"/have jumps. Of course, this does not happen in reality but is rather an artifact of the distinct methods used to obtain these bounds. It is thus not hard to see that, given any transition point within the discussed range, one can extend within an $\epsilon$ neighborhood of that point the better bound by simply ``perturbing" the method employed for obtaining it. This fact would be exemplified in the Appendix---see Section \ref{AppendixC}. As a consequence, we have the following upgraded estimates: there exists $\epsilon>0$ such that
\begin{itemize}
\item for the main diagonal term, one has
$$
\left| \Lambda_{j, l, m}^k (\vec{f}) \right| \lesssim 2^{-\epsilon \min \left\{ 2|k|, m \right\}} \|\vec{f}\|_{L^{\vec{p}}} \quad \forall \,\vec{p} \in {\bf H}^{+},
$$
\item for the main off-diagonal term, one has
$$
\left| \Lambda_{j, l, m}^k (\vec{f}) \right| \lesssim 2^{-\epsilon \max \left\{|j-l|, |l-m|, |j-m| \right\}} \|\vec{f}\|_{L^{\vec{p}}} \quad \forall\,\vec{p} \in {\bf H}^{+}.
$$
\end{itemize}
\end{obs}
\medskip

Combining the above estimates (by possibly making $\epsilon$ smaller) we  obtain the global estimate
\begin{equation} \label{20230323eq01}
\left| \Lambda_{j, l, m}^k (\vec{f}) \right| \lesssim 2^{-\epsilon \max \left\{  \min \left\{ 2|k|, \max\{j, l, m\} \right\}, \max \left\{|j-l|, |l-m|, |j-m| \right\} \right\}} \|\vec{f}\|_{L^{\vec{p}}} \quad \forall\,\vec{p} \in {\bf H}^{+}.
\end{equation}
Note that the role of $j, l$ and $m$ is symmetric in the above estimate \eqref{20230323eq01}, hence it suffices to control
\begin{equation} \label{dec1}
\Lambda_{\ge 0}(\vec{f}):=\sum_{\substack{0 \le j \le l  \le m \\ k \ge 0}} \left|  \Lambda_{j, l, m}^k (\vec{f}) \right|=\Lambda_{\ge 0}^1(\vec{f})+\Lambda_{\ge 0}^2(\vec{f})+\Lambda_{\ge 0}^3(\vec{f})+\Lambda_{\ge 0}^4(\vec{f}),
\end{equation}
where
\begin{equation} \label{dec2}
\Lambda_{\ge 0}^1(\vec{f}):=\sum_{0 \le j \le l \le m \le 2k} \left|  \Lambda_{j, l, m}^k (\vec{f}) \right|, \quad \Lambda_{\ge 0}^2(\vec{f}):=\sum_{0 \le j \le l \le 2k \le m} \left|  \Lambda_{j, l, m}^k (\vec{f}) \right|,
\end{equation}
and
\begin{equation} \label{dec3}
\Lambda_{\ge 0}^3(\vec{f}):=\sum_{0 \le j \le 2k \le l \le m} \left|  \Lambda_{j, l, m}^k (\vec{f}) \right| \quad \textrm{and} \quad \Lambda_{\ge 0}^4(\vec{f}):=\sum_{0 \le 2k \le j \le l \le m } \left|  \Lambda_{j, l, m}^k (\vec{f}) \right|.
\end{equation}
In what follows we will discuss each of the four terms $\Lambda_{\ge 0}^i, i\in\{1, \ldots,4\}$. Without loss of generality, we may assume $\vec{p}=(2, 2, \infty, \infty)$.

\medskip

\textit{Estimate of $\Lambda_{\ge 0}^1(\vec{f})$.} From \eqref{20230323eq01}, Cauchy--Schwarz and Parseval, we have
\begin{eqnarray*}
\Lambda_{\ge 0}^1(\vec{f})%
&\lesssim& \sum_{0 \le j \le l \le m} 2^{-\epsilon m} \left\|f_3 \right\|_{L^\infty} \left\|f_4 \right\|_{L^\infty} \left(\sum_{k \ge 0}  \left\|f_{1, j}^k \right\|_{L^2}^2\right)^{\frac{1}{2}}\,\left(\sum_{k \ge 0}  \left\|f_{2, l}^k \right\|_{L^2}^2  \right)^{\frac{1}{2}}\\
&\lesssim& \sum_{m\geq 0} m^2\,2^{-\epsilon m} \left\|f_1\right\|_{L^2} \left\|f_2 \right\|_{L^2} \left\|f_3 \right\|_{L^\infty} \left\|f_4 \right\|_{L^\infty} \lesssim  \|\vec{f}\|_{L^{\vec{p}}}.
\end{eqnarray*}

\medskip

\textit{Estimate of $\Lambda_{\ge 0}^2(\vec{f})$.} In this situation, we further split $\Lambda_{\ge 0}^2$ into two components:
$$
\Lambda_{\ge 0}^2=\Lambda_{\ge 0}^{2, 1}+\Lambda_{\ge 0}^{2, 2},
$$
where
$$
\Lambda_{\ge 0}^{2, 1}(\vec{f}):=\sum_{\substack{0 \le j \le l \le 2k \le m \\ m-j \le 2k}} \left|  \Lambda_{j, l, m}^k (\vec{f}) \right| \quad \textrm{and} \quad \Lambda_{\ge 0}^{2, 2}:=\sum_{\substack{0 \le j \le l \le 2k \le m \\ m-j > 2k}} \left|  \Lambda_{j, l, m}^k (\vec{f}) \right|.
$$
We first treat the term $\Lambda_{\ge 0}^{2, 1}$: from \eqref{20230323eq01}, we have
\begin{eqnarray*}
\Lambda_{\ge 0}^{2, 1}(\vec{f})
&\lesssim&  \sum_{\substack{0 \le j \le l \le 2k \le m \\ 0 \le m \le 4k}} 2^{-\epsilon k}  \left\|f_{1, j}^k \right\|_{L^2 } \left\|f_{2, l}^k \right\|_{L^2 } \left\|f_3 \right\|_{L^\infty} \left\|f_4 \right\|_{L^\infty} \\
&\lesssim& \sum_{k \ge 0}  \sum_{0 \le j \le l \le 2k} k 2^{-\epsilon k} \left\|f_1\right\|_{L^2} \left\|f_2 \right\|_{L^2} \left\|f_3 \right\|_{L^\infty} \left\|f_4 \right\|_{L^\infty}\lesssim \|\vec{f}\|_{L^{\vec{p}}}.
\end{eqnarray*}
Next, we estimate the term $\Lambda_{\ge 0}^{2, 2}$: using again \eqref{20230323eq01} we have
\begin{eqnarray*}
\Lambda_{\ge 0}^{2, 2}(\vec{f})
&\lesssim& \sum_{\substack{0 \le j \le l \le 2k \le u+j \\ u \ge 2k}} 2^{-\epsilon u} \left\|f_{1, j}^k \right\|_{L^2 } \left\|f_{2, l}^k \right\|_{L^2 } \left\|f_3 \right\|_{L^\infty} \left\|f_4 \right\|_{L^\infty} \\
&\lesssim& \sum_{k \ge 0} \sum_{\substack{0 \le j \le l \le 2k \\ u \ge 2k}} 2^{-\epsilon u} \left\|f_1 \right\|_{L^2}\left\|f_2 \right\|_{L^2}\left\|f_3 \right\|_{L^\infty}\left\|f_1 \right\|_{L^\infty} \lesssim  \|\vec{f}\|_{L^{\vec{p}}}.
\end{eqnarray*}
Combining both estimates for $\Lambda_{\ge 0}^{2, 1}$ and $\Lambda_{\ge 0}^{2, 2}$ we deduce $\Lambda_{\ge 0}^2(\vec{f}) \lesssim \|\vec{f}\|_{L^{\vec{p}}}$.

\medskip

\textit{Estimate of $\Lambda_{\ge 0}^3(\vec{f})$.} The treatment of the term $\Lambda_{\ge 0}^3$ is similar to the one of $\Lambda_{\ge 0}^2$. Indeed, we again split
$$
\Lambda_{\ge 0}^3(\vec{f})=\Lambda_{\ge 0}^{3, 1}(\vec{f})+\Lambda_{\ge 0}^{3, 2}(\vec{f}),
$$
with
$$
\Lambda_{\ge 0}^{3, 1}(\vec{f}):=\sum_{\substack{0 \le j \le 2k \le l \le m \\ m-j \le 2k}} \left|  \Lambda_{j, l, m}^k (\vec{f}) \right| \quad \textrm{and} \quad \Lambda_{\ge 0}^{3, 2}(\vec{f}):=\sum_{\substack{0 \le j \le 2k \le l \le m \\ m-j > 2k}} \left|  \Lambda_{j, l, m}^k (\vec{f}) \right|.
$$
The term $\Lambda_{\ge 0}^{3, 1}$ can be estimated in the same fashion as $\Lambda_{\ge 0}^{2, 1}$. For the term $\Lambda_{\ge 0}^{3, 2}$, from \eqref{20230323eq01}, we deduce that
\begin{eqnarray*}
\Lambda_{\ge 0}^{3, 2}(\vec{f})&\lesssim& \sum_{k \ge 0} \sum_{\substack{0 \le j \le 2k \\ u \ge 2k}} 2^{-\epsilon u} \sum_{2k \le l \le u+2k} \left\|f_{1, j}^k \right\|_{L^2 } \left\|f_{2, l}^k \right\|_{L^2 } \left\|f_3 \right\|_{L^\infty} \left\|f_4 \right\|_{L^\infty}\\
&\lesssim& \sum_{k \ge 0} \sum_{\substack{ 0 \le j \le 2k \\u \ge 2k}} u2^{-\epsilon u} \left\|f_1 \right\|_{L^2}\left\|f_2 \right\|_{L^2}\left\|f_3 \right\|_{L^\infty}\left\|f_1 \right\|_{L^\infty}\lesssim \|\vec{f}\|_{L^{\vec{p}}}.
\end{eqnarray*}
As a consequence, we conclude $\Lambda_{\ge 0}^3(\vec{f}) \lesssim \|\vec{f}\|_{L^{\vec{p}}}$ as desired.

\medskip

\textit{Estimate of $\Lambda_{\ge 0}^4(\vec{f})$.} In this case, we write
$$
\Lambda_{\ge 0}^4=\Lambda_{\ge 0}^{4, 1}+\Lambda_{\ge 0}^{4, 2},
$$
with
$$
\Lambda_{\ge 0}^{4, 1}(\vec{f}):=\sum_{\substack{0 \le 2k \le j \le l \le m  \\ m \le 4k}} \left|  \Lambda_{j, l, m}^k (\vec{f}) \right| \quad \textrm{and} \quad \Lambda_{\ge 0}^{4, 2}:=\sum_{\substack{0 \le 2k \le j \le l \le m  \\ m > 4k}} \left|  \Lambda_{j, l, m}^k (\vec{f}) \right|.
$$
For the term $\Lambda_{\ge 0}^{4, 1}$, note that if $m \le 4k$, then $m-j \le 2k$, and hence from \eqref{20230323eq01}, one has
\begin{equation*}
\Lambda_{\ge 0}^{4, 1}(\vec{f})
\lesssim \sum_{k \ge 0} \sum_{j, l, m \le 4k} 2^{-\epsilon k} \left\|f_1 \right\|_{L^2}\left\|f_2 \right\|_{L^2}\left\|f_3 \right\|_{L^\infty}\left\|f_1 \right\|_{L^\infty}
\lesssim \|\vec{f}\|_{L^{\vec{p}}}.
\end{equation*}
For the term $\Lambda_{\ge 0}^{4, 2}$, using again \eqref{20230323eq01}, we see that
\begin{equation*}
\Lambda_{\ge 0}^{4, 2}(\vec{f})\lesssim \sum_{m \ge 0} \sum_{j, l, k \le m} 2^{-\epsilon m} \left\|f_1 \right\|_{L^2}\left\|f_2 \right\|_{L^2}\left\|f_3 \right\|_{L^\infty}\left\|f_4 \right\|_{L^\infty}\lesssim  \|\vec{f}\|_{L^{\vec{p}}}.
\end{equation*}
The final estimate for the term $\Lambda_{\ge 0}^4$ follows then immediately from the above two estimates.

\medskip

Thus, combining all the estimates of $\Lambda_{\ge 0}^i, i\in\{1, 2, 3, 4\}$, we get
\begin{equation}\label{ctrlplus}
\Lambda_{\ge 0}(\vec{f}) \lesssim \|\vec{f}\|_{L^{\vec{p}}} \quad\forall\, \vec{p} \in {\bf H}^{+}.
\end{equation}

Once at this point, by following the same type of reasonings--see the shifted square function argument therein--as in the proof of Lemma \ref{offdiagnonst}, we immediately deduce

\begin{lem}\label{fixjlm} Assume $j,l,m\in\N$ fixed and set
\begin{equation}\label{jlmplus}
 \Lambda_{j, l, m}^{\pm} (\vec{f}):=\sum_{k\in \Z_{\pm}}  |\Lambda_{j, l, m}^k (\vec{f})|\,.
\end{equation}
Then, for any $k\in\Z$ and $\vec{p}\in\overline{\textbf{B}}$, the following holds:
\begin{equation}
\left| \Lambda_{j, l, m}^{k} (\vec{f})\right|\lesssim_{\vec{p}} \|\vec{f}\|_{L^{\vec{p}}}\qquad\textrm{and}\qquad\left| \Lambda_{j, l, m}^{\pm} (\vec{f})\right|\lesssim_{\vec{p}} \left(\max\{j,l,m\}\right)^3\,\|\vec{f}\|_{L^{\vec{p}}}\,.
\end{equation}
\end{lem}

Combining now the proof of \eqref{ctrlplus} with Lemma \ref{fixjlm} and multi-linear interpolation applied within each of the decompositions \eqref{dec1}--\eqref{dec3}, we conclude that
$$
\Lambda_{\ge 0}(\vec{f}) \lesssim_{\vec{p}} \|\vec{f}\|_{L^{\vec{p}}}\quad \forall\,\vec{p} \in {\bf B}^+,
$$
where
$$
{\bf B}^+:=\left\{\substack{(p_1, p_2, p_3, p_4) \\ 1<p_i \le \infty, \ i\in\{1, 2, 3, 4\}\\p_1\not=\infty}: \frac{1}{p_1}+\frac{1}{p_2}+\frac{1}{p_3}+\frac{1}{p_4}=1 \right\}\,.
$$

\subsection{The high oscillatory component $\Lambda_-^{Hi}$}
The treatment of $\Lambda_-^{Hi}$ mirrors the one of $\Lambda_{+}^{Hi}$. Indeed, we first notice that in the regime $j,l,m\in\N$ and $k\in \Z_{-}$ the analogue of \eqref{hiplus} is given by
\begin{equation}\label{himinus}
\Lambda_{j, l, m}^k (\vec{f})=\iint_{\R^2} f^k_{1, j}\left(x-t^{\frac{1}{3}} \right) f^k_{2, l} \left(x+t^{\frac{2}{3}} \right) f^k_{3, m} \left(x+t \right) f_4(x) 2^{3k} \rho(2^{3k}t) dtdx.
\end{equation}
Combining now the estimates from Section \ref{Negative} with the $k<0$ analogues of those obtained above for the stationary off-diagonal term $\Lambda^{\not{D},S}$ and using similar reasonings with those stated in Observation \ref{ContBound} we obtain that there exists $\epsilon>0$ such that for any $j,l,m\in\N$ and $k\in\Z_{-}$ the following global estimate holds:
$$
\left| \Lambda_{j, l, m}^k (\vec{f}) \right| \lesssim 2^{-\epsilon \max \left\{  \min \left\{ 2|k|, \max\{j, l, m\} \right\}, \max \left\{|j-l|, |l-m|, |j-m| \right\} \right\}} \|\vec{f}\|_{L^{\vec{p}}} \quad \forall\, \vec{p} \in {\bf H}^{-}.
$$
Following now same type arguments as in the case $k \ge 0$, essentially by interchanging the role of $k$ there by $-k$, we get
$$
\Lambda_{\le 0}(\vec{f}):=\sum_{\substack{0 \le j \le l \le m \\ k \ge 0}} \left|\Lambda_{j, l, m}^k (\vec{f}) \right| \lesssim \|\vec{f}\|_{L^{\vec{p}}} \quad \forall\, \vec{p} \in {\bf H}^{-},
$$
which, as before, together with Lemma \ref{fixjlm} and a multi-linear interpolation argument, gives
$$
\Lambda_{\le 0}(\vec{f}) \lesssim_{\vec{p}} \|\vec{f}\|_{L^{\vec{p}}} \quad  \forall\, \vec{p} \in {\bf B}^{-},
$$
where
$$
{\bf B}^{-}:=\left\{ \substack{(p_1, p_2, p_3, p_4) \\ 1<p_i \le \infty, i\in\{1, 2, 3, 4\}\\p_3 \not=\infty}: \frac{1}{p_1}+\frac{1}{p_2}+\frac{1}{p_3}+\frac{1}{p_4}=1\right\}.
$$

\subsection{The low oscillatory component $\Lambda^{Lo}$}
In the third part of this section, we simply restate the main result in Section \ref{20230412Sec01}--see Theorem \ref{20230326lem01} therein--which asserts
$$
\Lambda^{Lo}(\vec{f}) \lesssim_{\vec{p}} \|\vec{f}\|_{L^{\vec{p}}} \quad \forall\,p \in {\bf B},
$$
where we recall that
$$
\Lambda^{Lo}(\vec{f}):=\sum_{k \in \Z} \sum_{(j, l, m) \in \Z^3 \backslash \N^3} \Lambda_{j, l, m}^k(\vec{f})\,.
$$

\subsection{Final range}
Combining all the estimates for $\Lambda_{+}^{Hi}$, $\Lambda_{-}^{Hi}$, and $\Lambda^{Lo}$, we conclude that
$$
|\Lambda(\vec{f})| \lesssim_{\vec{p}} \|\vec{f}\|_{L^{\vec{p}}} \quad \forall\,\vec{p} \in {\bf B}={\bf B}^{+} \cap {\bf B}^{-}.
$$
(see, Figure \ref{Fig01}).

\begin{center}
\begin{figure}[ht]
\begin{tikzpicture}[scale=0.9]
\pgfmathsetmacro{\factor}{.5};
\coordinate [label=right:] (A) at (2,0,-2*\factor);
\fill (2,0,-2*\factor) node [right] {$(1, 0, 0, 0)$};
\coordinate [label=left:] (B) at (-2,0,-2*\factor);
\fill (-2,0,-2*\factor) node [left] {$(0, 1, 0, 0)$};
\coordinate [label=above:] (C) at (0,2,2*\factor);
\fill (0, 2,2*\factor) node [above] {$(0, 0, 0, 1)$};
\coordinate [label=below:] (D) at (0,-2,2*\factor);
\fill (0,-2,2*\factor) node [below] {$(0, 0, 1, 0)$};
\coordinate [label=right:] (E) at (1,-1, 0*\factor);
\fill (1.1,-1,0*\factor) node [right] {$\left(\frac{1}{2}, 0, \frac{1}{2}, 0\right)$};
\foreach \i in {A,B,C,D}
\draw[-, fill=red!15, opacity=.1] (A)--(D)--(B)--(C)--cycle;
\draw[-, fill=green!50, opacity=.3] (A)--(D)--(B)--cycle;
\draw[-, fill=green!50, opacity=.3] (A)--(C)--(D)--cycle;
\fill (0, 0, 0*\factor) node [right] {${\bf B}$};
\draw [red] (C)--(D);
\draw [red, dashed] (A)--(B);
\draw [red] (A)--(C);
\draw [green] (A)--(D);
\draw [red] (B)--(C);
\draw [red] (B)--(D);
\fill [opacity=1, purple] (A) circle [radius=2pt];
\fill [opacity=1, purple] (B) circle [radius=2pt];
\fill [opacity=1, purple] (C) circle [radius=2pt];
\fill [opacity=1, purple] (D) circle [radius=2pt];
\fill [opacity=1, teal] (E) circle [radius=2pt];
\end{tikzpicture}
\caption{Banach boundedness range for $\Lambda$}
\label{Fig01}
\end{figure}
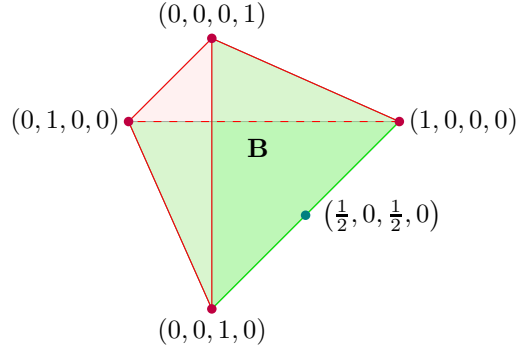
\end{center}

\medskip
\section{Appendix}\label{Appendix}

In this section we address three technical points that we encountered in some key moments along our proof:
\begin{itemize}
\item Localization of the time-frequency correlation sets;

\item A rigorous justification for Key Heuristic 1;

\item Continuity of the $\Lambda_{j, l, m}^k-$bounds.
\end{itemize}
Our presentation will be kept short and will follow the order listed above:

\medskip
\subsection{Localization of the time-frequency correlation set(s)} \label{AppendixB}

In this section we want to provide more evidence on how one can reduce \eqref{20230222eq21} to the main contribution represented by \eqref{LGC1} fact that requires some elaboration on the definition of the time-frequency correlation set \eqref{TFCqq}.

To begin with, we let
$$
\one_{I^k}(t)=\sum_{q \sim 2^{\frac{m}{2}}} \psi_{I_q^{m, k}}(t)
$$
be a smooth partition of unity adapted to the interval $I^k$. Using this, one can rewrite \eqref{20230222eq10} as
\begin{eqnarray*}
&& \Lambda_m^k(\vec{f})= 2^k \sum_{\substack{p \sim 2^{\frac{m}{2}+2k} \\ q \sim 2^{\frac{m}{2}}}}  \sum_{\substack{p_j \sim 2^{\frac{m}{2}+\left(j-1 \right)k}, \\ n_j \sim 2^{\frac{m}{2}},  \ j\in\{1, 2, 3\}}} \left \langle f_1, \phi_{p_1, n_1}^k \right \rangle \left \langle f_2, \phi_{p_2, n_2}^{2k} \right \rangle \left \langle f_3, \phi_{p_3, n_3}^{3k} \right \rangle \\
&&\quad \cdot \int_{I_p^{m, 3k}} \int_{\R} \phi_{p_1, n_1}^k \left(x-t \right) \phi_{p_2, n_2}^{2k} \left(x+t^2 \right) \phi_{p_3, n_3}^{3k} \left(x+t^3 \right) \psi_{I_q^{m, k}}(t) f_4(x) dtdx
\end{eqnarray*}
Now following the same reasoning as in {\sf Step 3} of Section \ref{LGC(I)}, we notice that the term \eqref{20230422eq01} takes the form
$$
2^{-\frac{m}{2}-k} \cdot C_{q, n_1, n_2, n_3} \cdot \int_{\R} e^{-it \left(n_1-\frac{2q}{2^{\frac{m}{2}}}n_2-\frac{3q^2}{2^m}n_3 \right)} \psi_{[0, 1]}(t)dt=2^{-\frac{m}{2}-k} \cdot C_{q, n_1, n_2, n_3}   \cdot \widehat{\psi_{[0, 1]}} \left(n_1-\frac{2q}{2^{\frac{m}{2}}}n_2-\frac{3q^2}{2^m}n_3 \right),
$$
where $\psi_{[0, 1]}$ is a smooth bump function supported near $0$. Since $\widehat{\psi_{[0, 1]}}$ is a Schwartz function, the main contribution from the above expression comes from those triples $(n_1, n_2, n_3)$ that satisfy the relation ``$n_1-\frac{2q}{2^{\frac{m}{2}}}n_2-\frac{3q^2}{2^m}n_3$ is close to the origin".  In order to make this more precise, one can define an $\bar{\epsilon}$-thickening of the original time-frequency correlation set
$$
\TFC_q^I(\bar{\epsilon}):=\left\{ \substack{(n_1, n_2, n_3) \\ n_1 ,n_2, n_3 \sim 2^{\frac{m}{2}}}: \left| n_1-\frac{2q}{2^{\frac{m}{2}}} n_2-\frac{3q^2}{2^m}n_3 \right|<2^{\bar{\epsilon} \min \left\{k, \frac{m}{2}\right\}} \right\},
$$
where $\bar{\epsilon}>0$ is an absolute constant smaller than the constant $\epsilon$ appearing in our main Theorem \ref{mainthm-2022}.

Once at this point, the strategy to follow is pretty straightforward:
\begin{itemize}
\item if $(n_1, n_2, n_3) \in \TFC_q^I(\bar{\epsilon})$ then one can follow with only trivial changes the proof of Theorem \ref{mainthm-2022};

\item if $(n_1, n_2, n_3) \notin \TFC_q^I(\bar{\epsilon})$ the one can prove the desired estimate by just applying Rank I LGC and exploiting the rapid decay of the expression $\widehat{\psi_{[0, 1]}} \left(n_1-\frac{2q}{2^{\frac{m}{2}}}n_2-\frac{3q^2}{2^m}n_3 \right)$.
 \end{itemize}

\subsection{A rigorous justification for Key Heuristic 1} \label{AppendixA}

We focus our attention on the claim \eqref{20230131eq04} in Key Heuristic 1 which states that for any $p \sim 2^{\frac{m}{2}+2k}$, the mapping
\begin{equation} \label{20230421eq01}
x \mapsto J_{q, n}^k(f_1, f_2)(x)=\int_{I_q^{m, k}} f_1(x-t)f_2(x+t^2) e^{3i \cdot 2^{\frac{m}{2}+k}t \left(\frac{q}{2^{\frac{m}{2}}} \right)^2 n} dt
\end{equation}
is morally constant on intervals of the form $I_p^{m, 3k}$ under the assumption that $\supp \ \widehat{f_i} \subseteq \left[2^{m+ik}, 2^{m+ik+1} \right]$ and $k \ge \frac{m}{2}$.

Fix now $p \sim 2^{\frac{m}{2}+2k}$, $x\in I_p^{m, 3k}$ and $t\in I_q^{m, k}$. Applying a Taylor series argument, we see that
\begin{equation} \label{20230421eq03}
f_2(x+t^2)=\int_{\R} \widehat{f_2}(\eta) \phi \left(\frac{\eta}{2^{m+2k}} \right) e^{i \left(\frac{p}{2^{\frac{m}{2}+3k}}+t^2 \right) \eta} e^{i \left(x-\frac{p}{2^{\frac{m}{2}+3k}} \right) \eta} d\eta=f_2 \left(\frac{p}{2^{\frac{m}{2}+3k}}+t^2 \right)+\textrm{Err}_p f_2(x, t),
\end{equation}
where
$$
\textrm{Err}_p f_2 (x, t):= \sum_{j=1}^{\infty} \frac{i^j}{j!} \left[2^{m+2k} \left(x-\frac{p}{2^{\frac{m}{2}+3k}} \right) \right]^j \int_{\R} \widehat{f_2}(\eta) \left(\frac{\eta}{2^{m+2k}} \right)^j \phi \left(\frac{\eta}{2^{m+2k}} \right) e^{i \left(\frac{p}{2^{\frac{m}{2}+3k}}+t^2 \right) \eta} d\eta.
$$
Plugging \eqref{20230421eq03} back to \eqref{20230421eq01}, we see that, for $x \in I_p^{m, 3k}$, it is sufficient to consider the situation when $J_{q, n}^k(f_1, f_2)(x)$ takes the form
\begin{equation} \label{20230421eq02}
\int_{I_q^{m, k}} f_1(x-t) f_2 \left(\frac{p}{2^{\frac{m}{2}+3k}}+t^2 \right) e^{3i \cdot 2^{\frac{m}{2}+k}t \left(\frac{q}{2^{\frac{m}{2}}} \right)^2 n} dt\,.
\end{equation}
Indeed, this is because the remaining term corresponding to $\textrm{Err}_p f_2 (x, t)$ is given by
$$\sum\limits_{j=1}^{\infty} \frac{i^j}{j!} \left[2^{m+2k} \left(x-\frac{p}{2^{\frac{m}{2}+3k}} \right) \right]^i J_{q, n, j}^k (f_1, f_2)(x)\,,$$
where $J_{q, n, j}^k (f_1, f_2)(x)$ is similar to \eqref{20230421eq02}. Thus, once at this point, one can simply transfer all the  arguments involving \eqref{20230421eq02} to the expression  $J_{q, n, j}^k(f_1, f_2)(x)$.

Finally, one can repeat the same argument for the term $f_1(x-t)$ and thus conclude that it is sufficient to consider the case when $J_{q, n}^k(f_1, f_2)(x)$ takes the form
$$
\int_{I_q^{m, k}} f_1\left(\frac{p}{2^{\frac{m}{2}+3k}}-t \right) f_2 \left(\frac{p}{2^{\frac{m}{2}+3k}}+t^2 \right) e^{3i \cdot 2^{\frac{m}{2}+k}t \left(\frac{q}{2^{\frac{m}{2}}} \right)^2 n} dt \quad \textrm{for} \ x \in I_p^{m, 3k},
$$
which proves the desired claim.

\medskip
\subsection{Continuity of the $\Lambda_{j, l, m}^k-$bounds.} \label{AppendixC}

In this section we provide a quick outline of how ``to glue" the decay estimates for $\Lambda_{j, l, m}^k$ listed at the beginning of Section \ref{Boundhiplus}. Indeed, recalling Observation \ref{ContBound} and its consequence in the form of the global estimate \eqref{20230323eq01} we now provide a justification of why these decay estimates are indeed ``continuous" with respect to $k,j,l,m$. We will take as a model\footnote{The rest of these decay estimates can be proved in a similar fashion.} case the items (1) and (3) in the list displayed in Section \ref{Boundhiplus}: we now claim that there exists $\del>0$ sufficiently small, such that in a small neighborhood of $k=\frac{m}{2}$, say $V_{\del m} \left(\frac{m}{2} \right):=\left\{k: \left|k-\frac{m}{2} \right|< \del m \right\}$ one has that
\begin{equation} \label{20230422eq02}
\Lambda_m^k (\vec{f}) \lesssim 2^{-\frac{\epsilon m}{2}} \|\vec{f}\|_{L^{\vec{p}}}, \quad  \vec{p} \in {\bf H}^{+}.
\end{equation}

Notice that it is sufficient to consider the case when $\frac{m}{2}-\del m \le k <\frac{m}{2}$. Since $k<\frac{m}{2}$, we apply the frequency decomposition in Section \ref{Diagkpsmall} and remark that we have $\ell_1, \ell_2, \ell_3 \sim 2^{\frac{m}{2}-k}<2^{\del m}$. Fixing now such a triple $(\ell_1, \ell_2, \ell_3)$ and following the arguments in Section \ref{Model}, we reach to the following expression
\begin{eqnarray*}
&&\Lambda_{m; \ell_1, \ell_2, \ell_3}^k(\vec{f}):= 2^k \sum\limits_{\substack{p \sim 2^{\frac{m}{2}+2k}\\ q \sim 2^{\frac{m}{2}} \\ n \sim 2^k}} \left| \left\langle f_3^{\omega_3^{\ell_3}}, \phi_{p+\frac{q^3}{2^m}, \ell_3 2^k+n}^{3k} \right \rangle \right|  \left| \left \langle f_4 e^{i 2^{\frac{m}{2}+3k} \left(\ell_1 2^{-k}+\ell_2 \right) \left( \cdot \right) }, \phi_{p, \ell_3 2^k+n}^{3k} \right \rangle \right| \nonumber \\
&& \cdot \frac{1}{\left|I_p^{m, 3k} \right|} \int_{I_p^{m, 3k}} \left| \int_{I_q^{m, k}}  \underline{f}_1^{\omega_1^{\ell_1}}(x-t) \underline{f}_2^{\omega_2^{\ell_2}}(x+t^2) e^{i 2^{\frac{m}{2}+k}t \cdot \frac{3q^2}{2^m}n}  e^{i 2^{\frac{m}{2}+k}t \cdot 2^k \left(-\ell_1 +\ell_2 \frac{2q}{2^{\frac{m}{2}}}+ \ell_3 \frac{3q^2}{2^m}\right)} dt\right| dx.
\end{eqnarray*}
Letting now $f_1:=\underline{f}_1^{\omega_1^{\ell_1}} \cdot  e^{-i 2^{\frac{m}{2}+k} (\cdot) \cdot 2^k \left(-\ell_1 +\ell_2 \frac{2q}{2^{\frac{m}{2}}}+ \ell_3 \frac{3q^2}{2^m}\right)}$, $f_2:=\underline{f}_2^{\omega_2^{\ell_2}}$, $f_3:=f_3^{\omega_3^{\ell_3}}$, and $f_4:=f_4 e^{i 2^{\frac{m}{2}+3k} \left(\ell_1 2^{-k}+\ell_2 \right) \left( \cdot \right) }$, we see that the above expression is similar to the main diagonal term \eqref{20230131main01} for $k \ge \frac{m}{2}$. Therefore, one can apply a similar argument with the one in Section \ref{MaindiagKpositive} in order to bound the term $\Lambda_{m; \ell_1, \ell_2, \ell_3}^k \left(\vec{f} \right)$. The desired decay \eqref{20230422eq02} follows now from the observation $\# (\ell_1, \ell_2, \ell_3) \lesssim 2^{3\del m}$ together with the choice $\del<\frac{\epsilon}{100}$.

\end{document}